\documentclass[review,onefignum,onetabnum]{siamart190516}
\usepackage{amsmath}
\usepackage{geometry}
\usepackage{amssymb,latexsym}
\usepackage{verbatim}
\usepackage{extarrows}
\usepackage{enumerate}
\usepackage{txfonts}
\usepackage{mathtools}
\usepackage{booktabs}
\usepackage{bbm}

\usepackage[font=small,labelfont=bf, labelsep = quad]{caption}
\usepackage{subfigure}
\usepackage{graphicx}
\usepackage{caption}
\usepackage{mwe}    
\usepackage{hyperref}
\hypersetup{
    colorlinks= True,
    linkcolor=blue,
    filecolor=magenta,      
    urlcolor=cyan,
}
\usepackage{xurl}
\usepackage{xcolor}
\renewcommand{\c}[1]{\mathcal{#1}}

\usepackage{mathtools}
\usepackage[tableposition=top]{caption}
\usepackage{booktabs,dcolumn}

\theoremstyle{plain}

\theoremstyle{definition}

\newtheorem{assumption}[theorem]{Assumption}

\usepackage{bbm}

\newcommand\R{\mathbb{R}}

\renewcommand{\c}[1]{\mathcal{#1}}

\usepackage{float}
\usepackage{placeins}

\newcommand*\sfref[1]{\ref{#1}}
 
\newcommand*\stref[1]{SI Table \ref{#1}}

\newcommand*\ssref[1]{SI Sec. \ref{#1}}

\usepackage{url}

\usepackage{xparse}
\usepackage{tikz}
\usetikzlibrary{matrix,backgrounds}
\pgfdeclarelayer{myback}
\pgfsetlayers{myback,background,main}

\tikzset{mycolor/.style = {line width=1bp,color=#1}}%
\tikzset{myfillcolor/.style = {draw,fill=#1}}%

\NewDocumentCommand{\highlight}{O{blue!40} m m}{%
\draw[mycolor=#1] (#2.north west)rectangle (#3.south east);
}

\def\ao #1{{#1}}
\def\aob #1{{#1}}
\def\aon #1{{#1}}
\def\aoa #1{{#1}}

\begin{document}
\nolinenumbers
\title{Optimal Lockdown for Pandemic Control} 
\author{Qianqian Ma\thanks{Department of Electrical and Computer Engineering, Boston
  University, Boston, MA, USA } \and Yang-Yu Liu\thanks{Channing Division of Network Medicine,
  Brigham and Women's Hospital, Harvard Medical School, Boston, MA
  02115, USA} \and Alex Olshevsky\thanks{Department of Electrical and Computer Engineering and Division
  of Systems Engineering, Boston University, Boston, MA, USA}}
\maketitle

\begin{abstract} 
As a common strategy of contagious disease containment, lockdowns will inevitably weaken the economy. The ongoing COVID-19 pandemic underscores the trade-off arising from public health and economic cost. An optimal lockdown policy to resolve this trade-off is highly desired. Here we propose a mathematical framework of pandemic control through an optimal stabilizing non-uniform lockdown, where our goal is to reduce the economic activity as little as possible while decreasing the number of infected individuals at a prescribed rate. This framework allows us to efficiently compute the optimal stabilizing lockdown policy for general epidemic spread models, including both the classical SIS/SIR/SEIR models and a new model of COVID-19 transmissions. We demonstrate the power of this framework by analyzing publicly available data of inter-county travel frequencies to analyze a model of COVID-19 spread in the 62 counties of New York State. We find that an optimal stabilizing lockdown based on epidemic status in April 2020 would have reduced economic activity more stringently outside of New York City compared to within it, even though the epidemic was much more prevalent in New York City at that point. Such a counterintuitive result highlights the intricacies of pandemic control and sheds light on future lockdown policy design.
\end{abstract} 

\tableofcontents
\section{Introduction} The COVID-19 pandemic has resulted in more than 92.3M confirmed cases and 2.0M deaths (up to Jan 13th, 2021) \cite{JHU} and has impacted the lives of more than 90\% global population \cite{pew,hospitalization}. 
Curbing the spread of the pandemic like COVID-19 depends critically on the successful implementation of non-pharmaceutical interventions such as lockdowns, social distancing, shelter in place orders, contact tracing, isolation, and quarantine \cite{chinazzi2020effect, pan2020association, flaxman2020estimating, germann2006mitigation}. However, these interventions can also  lead to substantial economic damage,
motivating us to investigate the problem of curbing  pandemic spread while minimizing the induced economic losses. 

\sloppy

We consider the problem of designing an optimal \aob{stabilizing} lockdown that minimizes the economic damage while reducing  the number of new infections to zero at a prescribed rate.  Such a lockdown should be {\em non-uniform}, because shutting down different locations has different implications both for the economic cost and for pandemic spread. The difficulty is whereas a uniform lockdown can be found through a search over a single parameter, a non-uniform lockdown is parametrized by many parameters associated with different locations. Despite of its significance and implications, a computationally efficient framework to design \aob{stabilizing} lockdown strategies is still lacking.    


Here we propose such a framework by mapping the design of optimal lockdown policy to a classical problem in control theory --- design  an intervention that affects the eigenvalues of a matrix governing the dynamics of a dynamical system. It turns out that, even though general epidemic spreading dynamics are highly nonlinear, an eigenvalue bound for a linear approximation of the spreading dynamics nevertheless forces the number of infections at each location to go to zero asymptotically at a prescribed rate \ao{for all time}. We provide two polynomial-time algorithms that design the optimal lockdown to achieve such an eigenvalue bound. 

We apply these algorithms to design an optimal \aob{stabilizing} lockdown on both synthetic and real data (using  data from SafeGraph \cite{travel} to fit a county-level model of New York State) for epidemic spread models of COVID-19 using disease parameters from the literature \cite{giordano2020modelling,bertozzi2020challenges,birge2020controlling}. 
Unsurprisingly, we find that the heterogeneous lockdown is far more economical than a homogeneous lockdown. However, we find additional features of the optimal \aob{stabilizing} lockdown that are counter-intuitive. For example, we find that in models of random graphs, degree centrality and population do not affect the strength of the lockdown of a location unless its population (or degree centrality) takes extremely smaller (or larger) values than others. Most surprisingly, we show that an optimal \aob{stabilizing} lockdown based on the epidemic status in April 2020 would have reduced activity more strongly outside of New York City (NYC) compared to within it, even though the epidemic was much more prevalent in NYC at that point.


\section{Results}

All the epidemic spread models considered in this work are compartmental or network models \cite{birge2020controlling,7277005} with ``locations'' corresponding to  neighborhoods,  counties, or other geographic subdivisions. 
We consider $n$ locations, with the variable $x_i$ denoting the proportion of infected population at location $i$. Our framework can be applied to general epidemic spread models. For demonstration purpose, here we consider a simple model of COVID-19 which contains the classical Susceptible-Infectious-Recovered (SIR) model and the Susceptible-Exposed-Infectious-Recovered (SEIR) model as special cases. Besides the COVID-19 model, we also consider the classical susceptible-infectious-susceptible (SIS) model; details about the SIS model can be found in \ssref{sec: SIS}. The optimal lockdown issue we consider is summarized in Fig.\ref{Fig: framework}.

\smallskip

\textbf{A network model of  COVID-19.} We consider a simple model (similar to models in literature \cite{khanafer2014optimal, giordano2020modelling,birge2020controlling, pagliara2020adaptive, carli2020model, zino2020assessing}) of COVID-19 spreading that breaks infected individuals into two types: asymptomatic and symptomatic. This model allows individuals transmit the infection at different rates: 
\begin{equation}\label{eq:COVID}
	\begin{aligned}
		\dot{s}_i & =  - s_i \sum_{j=1}^n a_{ij} (\beta^{\text a} x^{\text a}_j + \beta^{\text s} x^{\text s}_j) \\ 
		\dot{x}^{\text a}_i & =  s_i \sum_{j=1}^n a_{ij} (\beta^{\text a} x^{\text a}_j + \beta^{\text s} x^{\text s}_j) - (\epsilon + r^{\text a}) x^{\text a}_i \\ 
		\dot{x}^{\text s}_i & =  \epsilon x^{\text a}_i - r^{\text s} x^{\text s}_i 
	\end{aligned}.
\end{equation}
Here $s_i$ ($x_i^{\text a}$ or $x_i^{\text s}$) stands for the proportion of susceptible (asymptomatic or symptomatic infected, respectively) population at location $i$, $a_{ij}$ captures the rate at which infection flows from location $j$ to location $i$, $\beta^{\text a}$ (or $\beta^{\text s}$) is the transmission rate of asymptomatic (or symptomatic) infected individuals, $r^\text{a}$ (or $r^\text{s}$) is the recovery rate of asymptomatic (or symptomatic) infected individuals. We assume infected individuals are asymptomatic at first and $\epsilon$ is the rate at which they develop symptoms. We use different parameters for symptomatic and asymptomatic individuals because a recent study \cite{Kissler2020.10.21.20217042} reported that asymptomatic individuals have viral load that drops more quickly, so they not only recover faster, but also are probably less contagious.

Note that our model of COVID-19 spreading can be considered as a generalization of the classical SIR model and the SEIR model of epidemic spread. Indeed, by setting $\beta^{\text s}=\epsilon= r^{\text s} = 0$, we recover the SIR model; and by setting $\beta^{\text a}=r^{\text a} = 0$, we recover the SEIR model. However, neither the SIR nor the SEIR model captures the existence of two classes of individuals who transmit infections at different rates as above. Our model can also be considered as a simplification of existing models in studying COVID-19 spreading \cite{birge2020controlling, giordano2020modelling}. For example, in \cite{birge2020controlling}, asymptotic stability was considered in a slightly more general model including both births and deaths. Here for simplicity in our model we consider a fixed population size. 
In \cite{giordano2020modelling} eight classes of patients (instead of two) were introduced, depending on whether the infection is diagnosed, whether the patient is hospitalized, as well as other factors. 

In matrix form, we can write our model as 
\begin{equation} \label{COVID_19}
	\renewcommand{\arraystretch}{0.75}
	\left( 
	\begin{array}{c} 
		\dot{s} \\
		\dot{x}^{\text a} \\ 
		\dot{x}^{\text s}
	\end{array} 
	\right) = 
	\begin{tikzpicture}[baseline=-\the\dimexpr\fontdimen22\textfont2\relax ]
		\matrix (m)[matrix of math nodes, left delimiter=(,right delimiter=), nodes={minimum width=5em, minimum height=1.4em}]
		{
			0 & -\beta^{\text a} {\rm diag}(s) A &  - \beta^{\text s} {\rm diag}(s) A \\
			0 & \beta^{\text a} {\rm diag}(s) A - (\epsilon + r^{\text a}) &  \beta^{\text s} {\rm diag}(s) A \\
			0 & \epsilon & - r^{\text s}\\
		};
		
		\begin{pgfonlayer}{myback}
			\highlight[gray]{m-2-2}{m-3-3}
		\end{pgfonlayer}
	\end{tikzpicture}
	\left( 
	\begin{array}{c} 
		s \\
		x^{\text a} \\ 
		x^{\text s}
	\end{array} 
	\right),
\end{equation}
where scalars in the matrix should be understood as multiplying the identity matrix. Let us write $M(t)$ for the bottom right $2n \times 2n$ submatrix (outlined by a  box) in Eq. \eqref{COVID_19}.

It turns out that, if we want the number of infections at each location (or a linear combination of those numbers) to go to zero at a prescribed rate $\alpha$, we just need to ensure that the linear eigenvalue condition $\lambda(M(t_0)) \leq -\alpha$ holds (see \ssref{sec: stability} for a formal proof). \ao{Note that this is quite different from what usually happens in nonlinear systems when we pass to an eigenvalue bound of at a point:  here as long as the eigenvalue condition $\lambda(M(t_0)) \leq -\alpha$ is satisfied, we obtain that infections go to zero at rate $\alpha$ over {\em all} times $t \geq t_0$}. 

\ao{In the remainder of this paper, we will attempt to design strategies that enforce decay of infections with a prescribed rate by modifying the matrix $A$ through lockdowns to satisfy such an eigenvalue bound.  This is different from the more traditional approach of optimal control of network epidemic processes \cite{bock2018optimal, NBERw27102, fajgelbaum2020optimal, alvarez2020simple} in a number of ways.  First, this gives rise to a fixed lockdown, whereas a traditional optimal control approach would result in a lockdown that is different at every time $t$, which is obviously unrealistic. \aon{If the time-varying lockdown is approximated through a series of infrequently changing fixed lockdowns, the optimality guarantee is lost.} 
}\aon{Second, the optimal control approach results in lockdowns that relax in strength as the number of infections decreases. If policymakers are tasked with repeated lockdown relaxations, a potential danger is that political considerations will result in lockdowns that are too loose, leading cases to increase again. For example, a recent CDC report found that pre-mature relaxations of restrictions drove an increase of cases throughout the United States in 2020\cite{guy2021association}.  }Finally, as we will see later, one of  the main benefits of our approach is the guaranteed scalability: the main result  of this paper is a nearly linear time algorithm. \aon{By contrast,}  optimal control of epidemic processes is based on methods which are either known to be non-scalable or which sometimes fail to converge \aon{at all}. We discuss  this at more length in Section I of the Supplementary Information.

\medskip

\textbf{Lockdown model.}  Methods of constructing $A$ capturing spatial heterogeneity have been well studied \cite{gross2020epidemic, longini1988mathematical, sattenspiel1995structured,arino2003multi, bayham2015measured, pare2018virus, bussell2019applying, della2020network}. Here we follow a recent work \cite{birge2020controlling}that is particularly well-suited to model the lockdowns in curbing COVID-19. 
Denote the fixed population size at location $i$ as $N_i$, and assume people travel from location $i$ to location $j$ at rate $\tau_{ij}$. It is well accepted that such travel rates determine the evolution of an epidemic. For example, it has been reported that regional progress of influenza is much more correlated with the movement of people to and from their workplaces rather than geographic distances \cite{viboud2006synchrony}; \ao{in the context of COVID-19, mobility based on cell-phone data has been predictive as a measure of epidemic spread \cite{chang2021mobility, glaeser2020jue}}. The quantities $a_{ij}$ can then be determined as (see \ssref{sec: A_construct} for details):  
\begin{align}\label{eq: a_ij}
	a_{ij} = \sum_{l=1}^n  \tau_{il} \tau_{jl} \frac{N_j}{\sum_{k=1}^n N_k \tau_{kl}} .  
\end{align}
It is intuitive that $a_{ij}$ is the sum of the terms involving $\tau_{il} \tau_{jl}$ since this product captures the interactions between people from locations $i$ and $j$ through visits to location $l$. Eq. \eqref{eq: a_ij} can also be written in matrix form as $A = C B^{\top}$ with 
\begin{equation} \label{eq:bc} C =  \tau, B^{\top} = D_1  \tau^{\top} D_2, 
\end{equation} where $\tau = (\tau_{ij})$, $D_1 = {\rm diag}(\sum_{k} N_k \tau_{kl})^{-1}$ while $D_2 =  {\rm diag}(N_1, \ldots, N_n)$.  


When a lockdown is ordered heterogeneously across  different locations, this has two consequences. First, the transmission rates will be altered.  For instance, ensuring that all buildings have a maximum enforced density limits the rate at which people can interact, as do mandatory face-covering, and other measures, resulting in a number of transmissions that is a fraction of what they otherwise would have been. We may account for this as follows. From Eq. \eqref{eq: a_ij}, we have that 
\[ \beta^{\text a} a_{ij}  = \sum_{l=1}^n  \beta^{\text a} \tau_{il} \tau_{jl} \frac{N_j}{\sum_{k=1}^n N_k \tau_{kl}}. \] The effect of the lockdown is to replace $\beta^{\text a}$ in each term of the sum by $\beta^{\text a} f_l$, for some location-dependent $f_l \in [0,1]$. The effect on $\beta^{\text s} a_{ij}$ is similar. 
Secondly,  travel rates to location $l$ are also a fraction of what they were before since there is now reduced inducement to travel, i.e., $\tau_{il}$ should be replaced by $\tau_{il} g_l$ with $g_l \in [0,1]$ for each location $l$.

To avoid overloading the notation, we will not change the definitions of $\beta^{\text a}, \beta^s$ or the travel rates $\tau_{il}$ but instead achieve the same effect by changing the definition of $a_{ij}$ as:
\[ a_{ij} =  \sum_{l=1}^n z_l \frac{N_j}{\sum_{k=1}^n N_k \tau_{kl}} \tau_{il} \tau_{jl}, \]
where $z_l = f_l g_l\in [0, 1]$. 
In matrix notation, the post-lockdown $A$ matrix is  \begin{equation} \label{sislockdown} A = C {\rm diag}(z) B^{\top}.
\end{equation} The quantities $z_1, \ldots, z_n$ can be thought of as measuring the intensity of the lockdown at each location.



\medskip

\textbf{Lockdown cost.} Clearly, setting $z_l=1$ corresponds to doing nothing and should have a zero economic cost. On the other hand, choosing $z_l=0$ corresponds to a complete lockdown and should be avoided. We will later apply our framework to real data collected from counties in New York State; shutting down a county entirely would result in people being unable to obtain basic necessities, and thus the economic cost should approach $+\infty$ as $z_l \rightarrow 0$. 
With these considerations in mind, a natural choice of lockdown cost is \begin{equation} \label{eq:lockdown_cost} c(z_1, \ldots, z_n) = \sum_{i=1}^n c_l \left( \frac{1}{z_l} - 1 \right) . \end{equation} Here $c_l$ captures the relative economic cost of closing down location $l$. Throughout this paper, we will choose $c_l$ to be the employment at local $l$, but other choices are also possible (e.g., $c_l$ could be the GDP generated at location $l$).

Besides the cost function in Eq. (\ref{eq:lockdown_cost}), we will also consider cost functions that blow up with  different exponents as $\sum_{i=1}^n c_l \left( z_l^{-k} - 1 \right)$, as well as costs which threshold as $\sum_{i=1}^n \min  \left( c_l (z_l^{-1} - 1), C_l \right)$ which alter our cost function by saturating at some location-dependent cost $C_l$ rather than blowing up as $z_l \rightarrow 0$. 

\ao{One of the advantage of these cost functions is that they inherently discourage extreme disparities among nodes. Indeed,  a lockdown that places all the burden on a small collection of nodes by setting their $z_i$ close to zero will have cost that blows up. By contrast, some previous works such as \cite{birge2020controlling, carli2020model} used cost functions $\sum_{i} c_i (1-z_i)$ associated with lockdown strengths, which do not have this feature.}



\subsection{Analytical Results.}

\ao{Th optimal stabilizing lockdown problem put together all the features we have outlined above: we are looking for a lockdown enforcing a guaranteed decay rate through eigenvalue bounds on the matrix $A$ of minimum cost. Note that our problem formulation puts a cost on the lockdown strength and puts a decay condition on the number of infections as a {\em constraint}. It is therefore slightly different from approaches which put both of these into the cost, though it should be noted that via Lagrange multiplier arguments such problem variations are typically equivalent.}

We will consider two variations of the lockdown problem, whose difference is whether an optimal \aob{stabilizing} lockdown is allowed to increase activity in certain locations. \textit{Constrained lockdown}: we seek to find a vector $z$ with entries in $[0,1]$ minimizing lockdown cost determined by Eq. (\ref{eq:lockdown_cost})  subject to the bound $\lambda_{\rm max}(A-\gamma I) \leq \alpha$ in the SIS model, and $\lambda_{\rm max} (M(t)) \leq \alpha$ in the case of the COVID-19 model of Eq. \eqref{COVID_19}. \textit{Unconstrained lockdown}: same as above, but we do not constrain the entries of $z$ to lie in $[0,1]$. (The two factors $f_l$ and $g_l$ described in the previous subsection will not be constrained to lie in [0,1] either.)
Indeed, if certain locations contribute little to disease spread but have very high relative economic cost of lockdown, one could even increase activity in these locations to allow for a harsher lockdown elsewhere with the same overall cost. While our methods work for both variations, all of our simulations and empirical results will consider the constrained version. Our approach of stabilizing the system by forcing the eigenvalues to have negative real part is a standard heuristic in control theory \cite{levine1996control, ogata2002modern,  gopal2002control}. This approach comes with a caveat---if pushed to the extreme by moving the eigenvalues further and further towards negative infinity, the asymptotically better performance will  start coming at the expense of the non-asymptotic behavior of the system.



\bigskip
We next turn to describe our main results.  Our first step is to discuss an assumption required by one of our algorithms, i.e., the recovery rate $\gamma$ has to be small relative to the entries in the matrices $C$ and $B$. We call this ``high-spread assumption" (see Assumption \ref{as:spread} in \ssref{sec: lock down} for the formal description). We will later show that, under this assumption, the constrained and unconstrained lockdown problems are equivalent. This is quite intuitive: if the epidemic spreads sufficiently fast everywhere, the unconstrained shutdown will never choose to increase the activity of any location.

\smallskip

\textbf{Main theoretical contribution.} With the above preliminaries in place, we can now state our main theoretical contribution. Our main theorem provides efficient algorithms for both the unconstrained and constrained lockdown problems (see Theorem \ref{thm:mainthm} in \ssref{sec: lock down} for the formal description of this theorem). In particular, we first prove that the unconstrained lockdown problem for both SIS and COVID-19 models can be exactly mapped to the classical matrix balancing problem (see \ssref{sec: matrix balancing} for details) and solved with nearly linear time complexity. Moreover, we prove that if the “high-spread” assumption holds, then the constrained lockdown problem is equivalent to the unconstrained lockdown problem and consequently is also reducible to matrix balancing. Even if the “high-spread” assumption does not hold, we prove that under certain conditions the constrained lockdown problem for the SIS and COVID-19 models can still be solved by applying the covering semi-definite program with polynomial time complexity of ~$\tilde{O}(n^3)$, where the tilde hides factors logarithmic in model parameters. To summarize, we give three separate cases that cover all possible scenarios. In two of these cases, the optimal \aob{stabilizing} lockdown problem is solvable in nearly linear time, and in the remaining case, it is solvable in $\tilde{O}(n^3)$. 

In practice, we find the optimal lockdown problem is solvable in linear time in the vast majority of the cases. Specifically, when we fit the models to New York State data, in 23 experiments out of 27, the linear time algorithm gave the correct answer.

\subsection{Empirical Application.}\label{applications} \quad We now apply the algorithms we've developed to design an optimal \aob{stabilizing} lockdown policy for the 62 counties in the State of New York (NY). Our goal is to reduce activity in each county in a non-uniform way to curb the spread of COVID-19 while simultaneously minimizing economic cost. The data sources we employed are presented in \ssref{data_sources}. 

We consider only the constrained lockdown problem here. When the ``high-spread" assumption is satisfied, we will apply the matrix-balancing algorithm, otherwise we will apply the covering semi-definite program. To provide valid estimation results, we employed three different sets of disease parameters provided in literature \cite{bertozzi2020challenges}, \cite{birge2020controlling},  \cite{giordano2020modelling} (See Supplementary Table \ref{tab: model_parameters}).

\medskip 

\textbf{Comparison with other lockdown policies.} We used the data of the 62 counties in NY on April 1st, 2020 as initialization and estimated the number of active cases over $300 \sim 800$ days and the number of  cumulative cases over $500 \sim 1,500$ days  with different lockdown policies. Fig.\ref{Fig: active_acc_Covid_19}a-c
show the estimated active cases over times, and Fig.\ref{Fig: active_acc_Covid_19}d-e
show the estimated cumulative cases over time.  Here results from different columns of Fig.2  were calculated by using different sets of disease parameters adopted from literature \cite{birge2020controlling,giordano2020modelling, bertozzi2020challenges}.


We compared the optimal \aob{stabilizing} lockdown policy calculated by our methods with several other benchmark policies: (1) no lockdown, i.e., $z_l=1$ for all locations; (2) random lockdown, $z_l$ is randomly chosen from a uniform distribution $\mathcal{U}[a, b]$ with the lower- and upper-bounds $a$ and $b$ chosen such that the overall cost of this policy is the same as that of our policy; (3) uniform lockdown, where $z_l=z$ is the same for all the locations and $z$ is chosen such that the overall economic cost is the same as that of our policy; (4) uniformly-bounded-decline locdown, where  the ``decline'' is uniformly bounded across locations, i.e., the decay rate of the infections in each location is bounded by a constant $\alpha$, where $\alpha$ is chosen such that the economic cost of this policy is the same as that of our policy. Note that among the four benchmark policies both the uniformly-bounded-decline lockdown and the random lockdown are heterogeneous locationwise.


From  Fig.\ref{Fig: active_acc_Covid_19}, we can see that our optimal \aob{stabilizing} lockdown policy outperforms all other lockdown polices  in terms of the total final number of cumulative cases. Similar findings for the SIS and SIR models are reported in Fig.\sfref{Fig: active_acc_SIS} and Fig.\sfref{Fig: active_acc_SIR}. 

\medskip 

\textbf{Optimal \aob{stabilizing} lockdown rate $z_l^*$ for each county.} Fig.\ref{Fig: covid_19_each_county_ours_uni} shows lockdown-rate profiles $z_l$ calculated by various policies.
First, we found that the optimal \aob{stabilizing} lockdown profile is quite sensitive to the disease parameters. 
Second, surprisingly, the values of $z_l^*$ for counties in NYC are relatively higher (corresponding to a less stringent lockdown) than that of counties outside NYC, regardless of the disease parameters. 
This is a  counter-intuitive result: even though the epidemic was largely localized around NYC on the date we used to initialize the infection rates, the calculated optimal \aob{stabilizing} lockdown profile indicates that it is cheaper to reduce the spread of COVID-19 by being harsher on neighboring regions with smaller populations.  It can be observed from Fig. \ref{Fig: covid_19_each_county_ours_uni} that this pattern only appears in our heterogenous optimal \aob{stabilizing} lockdown policy. 

\ao{To see why this is counter-intuitive, consider the case of a single-node (i.e., non-network) model. It is easy to see that the optimal \aob{stabilizing} lockdown is insensitive to population. Intuitively, doubling population doubles the cost of the lockdown and also doubles the benefits in terms of lives saved. In terms of our model, the stabilization constraint is on the {\em proportion} of infected, so doubling the population may change the optimal cost but does not change the optimal solution. Furthermore, the strength of the optimal \aob{stabilizing} lockdown is increasing in $s(t_0)$: harsher restrictions are needed to achieve the same decay rate if more people are infected. Thus it is surprising that when we consider a network model of New York State, the region with the highest population and highest proportion of infected is treated the lightest under the optimal \aob{stabilizing} shutdown.
}

In \ssref{city_suburb}, we further replicate the same finding in a much simpler city-suburb model: we consider a city with large population and a neighboring suburb with small population and observe that the optimal \aob{stabilizing} lockdown will choose to shut down the suburb more stringently. In \ssref{sec: extended economic cost functions}, we further confirmed the same counterintuitive phenomenon using other cost functions.  In \ssref{Sec: Robustness}, we checked the robustness of this counterintuitive phenomenon with respect to the uncertainty of the travel rate matrix $\tau$ by adding noises or removing part of the travelling data. It turns out that this phenomenon is quite robust against the uncertainty of the matrix $\tau$ . \ao{In particular, perturbing each $\tau_{ij}$ by noise with variance up to $10 \tau_{ij}^2$, preserves the result, as does randomly setting half of the $\tau_{ij}$ to zero (See Fig. \sfref{robustness_check}). } 

We investigate this finding further in \ssref{two-parameters}, where we let the metric to be optimized be the  total number of infections. As a counterpoint, we do a greedy search over all two-parameter lockdowns that shut down NYC harder than the rest of New York State. 
Our results show that our lockdown has a smaller number of infections than the best two-parameter lockdown of the same cost.

This phenomenon likely occurs because shutting down a suburb with small population yields benefits proportional to the much larger population of the city, since a shutdown in the suburb affects the rate at which infection spreads in the city as city residents can infect each other through the suburb. Similarly, a possible explanation for the phenomenon we observe on New York State data is that shutdowns outside of NYC may be a cheaper way to curb the spread of infection within NYC.
\ao{We stress that this effect is due to the network interactions. In particular, this counterintuitive phenomenon does not occur in a hypothetical model of New York State where residents always stay within their own county: in that case, the lockdown problem reduces to a collection of single-node models which do not interact, and optimal \aob{stabilizing} shutdown will be increasing in the proportion of infected (and insensitive to population), thus hitting NYC harder than the rest of New York State.}



\medskip

\textbf{Additional observations.} To fully understand the effect of the disease-related parameters to the optimal \aob{stabilizing} lockdown-rate profile $\{ z_l^* \}$ and the economic cost, we implemented additional numerical experiments to analyze the sensitivity. The results are shown in SI Fig.\sfref{Fig: sensitivity}. It can be seen that the value of $z_l^*$ and the corresponding economic cost are sensitive to recovery rate and the initial growth rate but not to other parameters. 

We also studied the relationship of the value of $z_l^*$ and the structure of the underlying graph. We plotted the obtained $z_l^*$ with respect to degree, home-stay rate, population, and employment in Fig.\sfref{Fig: 2020_parameters}. We also implemented random permutation experiments (where we randomly permute one parameter while fixing everything else) in terms of degree, home-stay rate, population, employment and initial susceptible rate, and the results are shown in Fig.\sfref{Fig: rand_perm1} and Fig.\sfref{Fig: rand_perm2}. 
From these experiments, we found the distribution of $z_l^*$ can be strongly affected by permutations of centrality, population, and the home stay rate. However, the distribution of $z_l^*$ is not altered much by permuting employment and the initial susceptible rate. 
More details about these experiments are presented in \ssref{exp: extended}.

\subsection{Numerical Simulations}\label{syn_experiments}
From the empirical analysis of data in New York State, we hypothesize that  the home-stay rate, degree centrality, and population are three major parameters that impact the optimal  lockdown rate $z_l^*$ of county-$l$. However, no inferences can be made about the effect of these parameters from empirical data because all of them vary together. To study how the value of $z_l^*$ is related to these parameters, we implement  experiments on synthetically generated data. We describe the experiments and results next, with full details provided in the supplementary \ssref{exp: extended}.


\medskip

\textbf{Impact of degree centrality.} To study the impact of degree centrality, we considered geometric random graphs (see \ssref{sec: centrality} for other graphs).  The population, the home-stay rate, and the initial susceptible rate of different nodes are set as the same values across all the nodes. The  simulation results are presented in Fig.\ref{Fig: synthetic}a-b. We found that degree centrality only matters for the value of $z_l^*$ when there exist hotspots (i.e., hubs node with very high degrees). 
Beyond such hotspots, the effect of degree centrality is essentially ignorable.

\medskip

\textbf{Impact of  population.} To study the impact of  population, we fix all other model parameters and vary the population of the nodes. We again considered geometric random graphs, where node degrees are similar. The simulation results are presented in Fig.\ref{Fig: synthetic}c. We found that nodes with  small populations are assigned smaller values of $z_l^*$, but once the population is large enough, $z_l^*$ is almost independent of the population size.

\medskip

\textbf{Impact of  home-stay rate.} To study the impact of  home-stay rate, we fixed all other model parameters and tuned the home-stay rate of the nodes. Our simulation results based on the geometric random graphs were presented in Fig.\ref{Fig: synthetic}d. We found that $z_l^*$ increases with increasing home-stay rate, which agrees well with our intuition. 

\section{Discussion}
\ao{The main contribution of this paper is two-fold. Our first contribution is methodological: we give a modeling framework that  gives rise to efficient methods for pandemic control through fixed lockdowns, with our main algorithm taking nearly linear time. The linear time nature of the methods allows us to scale up in a way that is not known for any other method. For example, at present data is not available to design lockdowns at the city level, but the method presented here can be scaled up to design a lockdown even at the neighborhood level for the entire United States. Although this is highly unlikely to happen for the COVID-19 pandemic, it may be of use in future outbreaks. } 

\ao{Our second contribution is to use our algorithms to observe counter-intuitive properties of lockdowns. In particular, we observe that a model of epidemic spread in New York State will tend to shut down  outside of NYC more stringently that NYC itself, even if the epidemic is largely localized to NYC. 
	We compared the lockdown found our model against exhaustive search of all two-parameter lockdowns which shut down NYC harder than the rest of New York State to verify that indeed it outperforms. We further found that this result is robust against significant perturbations to the travel matrix, the epidemic parameters, as well as the epidemic model.} 


While we have focused on simple models in this work, our methods can be applied to more complex models, such as the SIDHARTE model \cite{giordano2020modelling},  which contains more than two classes (asymptomatic/symptomatic) of infected people. If additional data sets become available on variations of activity by age and location in the future, it is possible to incorporate this as in \cite{Brittoneabc6810}; one could, for example, split each city into multiple nodes, with each node corresponding to a different age group residing in that city. 

\ao{We next briefly discuss future directions. The main limitations of research on lockdown at the present time is the lack of available data. This limitation drove a number of the modeling choices made in the present work, as we describe next.} 

\ao{For example, it is tempting to divide trips into several different types (e.g., work, family, entertainment, etc), and argue that different types have different likelihood of leading to infections. Unfortunately, we are not aware of any data source which either counts such trips for different locations through the United States or estimates the infectivity differentials across types. One might further attempt to fit different transmissibility parameters to different counties, but again such data does not appear to be available \aon{at this level of granularity}. In general, there are many ways to build sophisticated models but the primary limitation is lack of data.} 

\ao{Even the SafeGraph data we have relied on here \aon{is not without limitations}. Indeed, SafeGraph estimates are based on cell phone data, and by definition do not sample people without cell phones. Further, we do not even know how many of the minutes recorded as travel outside might have been spent alone in a vehicle, with no possibility for transmission. \aon{However,}  as a counterpoint, mobility from cell phone data has been highly predictive in modeling COVID-19 spread as reported in \cite{chang2021mobility, glaeser2020jue}.  Finally, appropriately anonymized public data sets would allow us to better understand how people respond to lockdown and estimate a model with a more heterogeneous response compared to our model here. Future work is likely to be driven by fitting finer models as more data becomes available.}

\ao{We conclude by mentioning that, because of all of these considerations, we have focused on qualitative patterns of the lockdown which hold across different models and parameter values.   For example, as mentioned earlier, our finding that it is better to have a tighter lockdown in NYC holds when each entry of the travel matrix $\tau_{ij}$ is perturbed with noise of variance  up to $10 \tau_{ij}^2$ (see Supplementary Figure 18, and a detailed explanation of this experiment in Section 10 of the Supplementary Information).  The same finding is also unaffected by setting half of $\tau_{ij}$ randomly to zero. Thus the precise values in the travel matrix do not appear to be important for this finding as long as the zero-nonzero structure remains very broadly similar. Finally, the same finding remains true in the SIS/SIR models. Thus our main empirical result is a robust property of the optimal stabilizing shutdown that  holds even in the presence of severe data and model mis-specifications. }

\begin{figure}[!htb]
	\centering
	
	\begin{minipage}[b]{0.5\linewidth}
		\centering
		\includegraphics[width=1.0\linewidth]{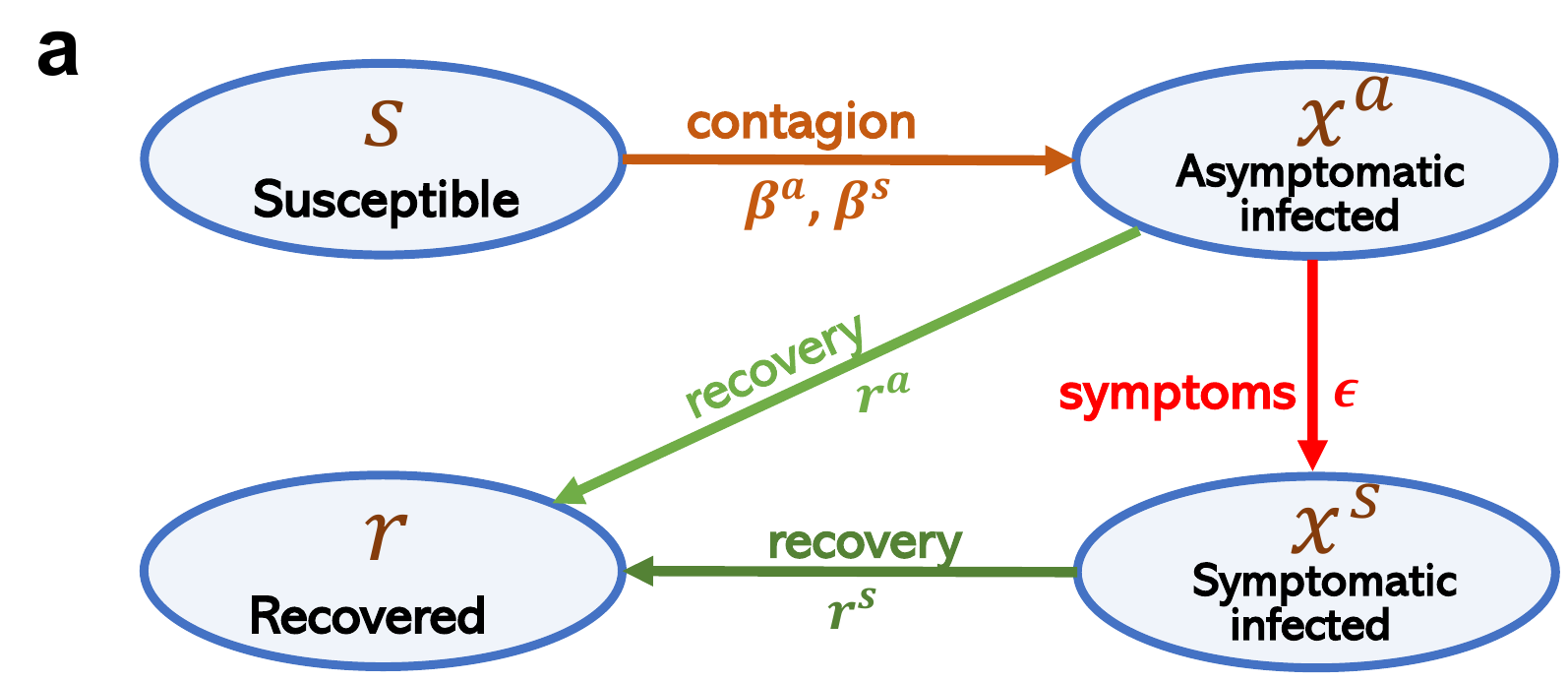}
	\end{minipage}%
	\vspace{3mm}
	
	\begin{minipage}[b]{1.0\linewidth}
		\centering
		\includegraphics[width=1.0\linewidth]{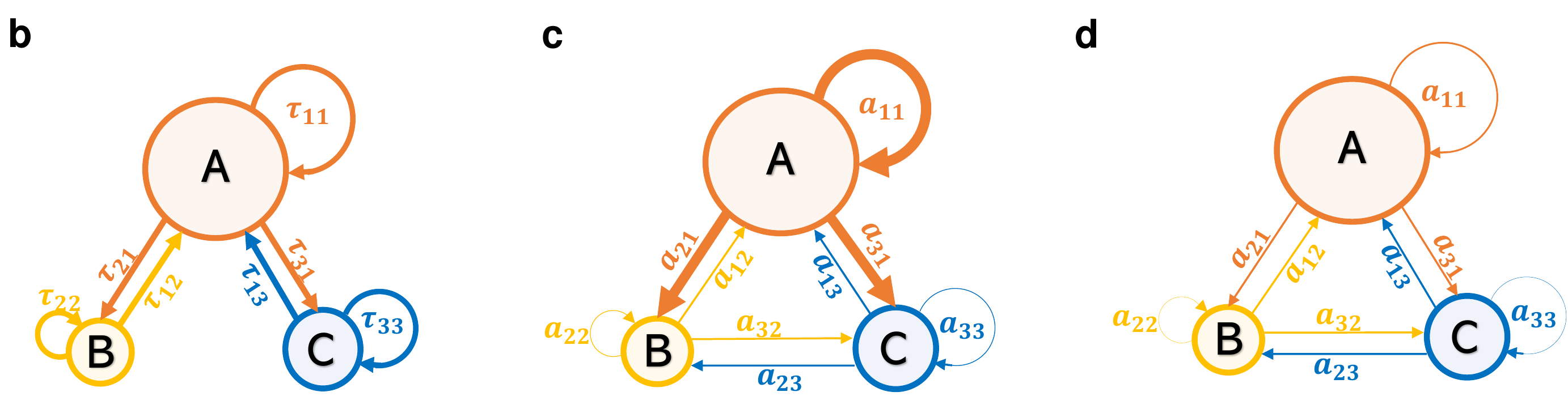}
	\end{minipage}%
	\vspace{3mm}
	
	\begin{minipage}[b]{0.75\linewidth}
		\centering
		\includegraphics[width =1.0\linewidth]{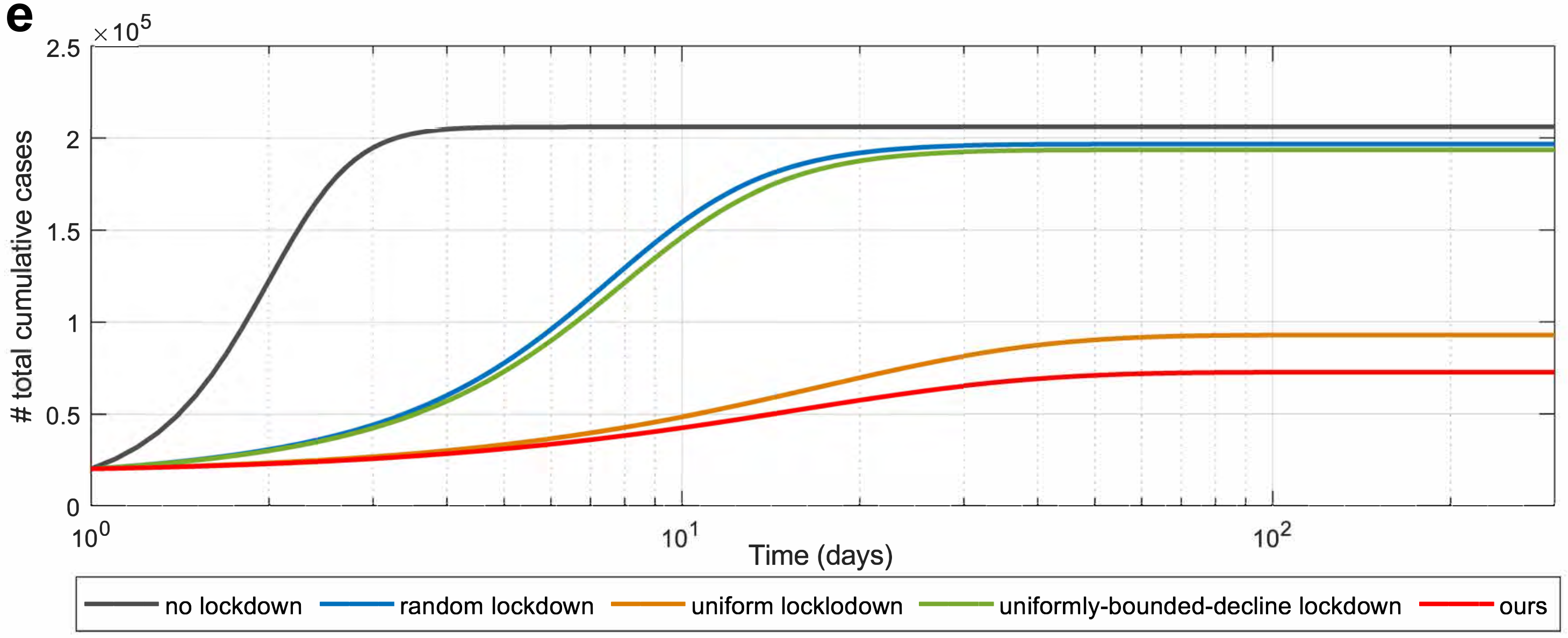}
	\end{minipage}%
	\caption{Framework of the optimal \aob{stabilizing} lockdown design. {\bf a},  the COVID-19 model we consider, which corresponds to Eq. \eqref{eq:COVID}. {\bf b},  the travel pattern of a three-nodes network, where A represents a city with large population, B and C represent two suburbs with small population, $\tau_{ij}$ represents the travel rate from location $i$ to location $j$ in a day. {\bf c}, the infection flow pattern of the three-nodes network before the lockdown, $a_{ij}$ represents the infection flow as
		described in Eq.\eqref{eq:COVID}. Note that although no travel occurs between nodes $B$ and $C$, the corresponding entries $a_{23}$ and $a_{32}$ are nonzero since people from these locations can meet each other in location $A$. {\bf d}, the infection flow pattern of the same network after the lockdown. The widths of the edges in {\bf c}, {\bf d} are propotional to the value of $a_{ij}$.
		When lockdown policy are implemented, the value of $a_{ij}$ will decrease, which lead to the control of the epidemics. {\bf e}, the estimated number of cumulative cases of  different lockdown polices. In this figure, ``ours" represents the heterogeneous optimal \aob{stabilizing} lockdown calculated by our method, ``uniform lockdown" represents a uniform policy with the same economic cost as our lockdown. Random lockdown rate $z_l$ is randomly chosen from a uniform distribution $\mathcal{U}[a, b]$ with the lower- and upper-bounds $a$ and $b$ chosen such that the overall cost of this policy is the same as that of our policy.  Uniformly-bounded-decline lockdow implies a policy where the decay rate of the infections in each county is bounded by a constant $\alpha$, where $\alpha$ is chosen such that the economic cost of this policy is the same as that of our policy.
		For this small network, the lockdown rates of our heterogeneous policy are $z_1 = 0.21,~z_2 = 0.06,~ z_3 = 0.06$. By contrast, the lockdown rates of the uniform policy with the same cost are $z_1 = 0.16,~ z_2 = 0.16,~ z_3 = 0.16$. It can be seen that our policy outperforms all the other lockdown policies. }\label{Fig: framework}
	\centering
	\vspace{-1mm}
\end{figure}

\vspace{\fill}
\clearpage


\begin{figure}[!htb]
	\centering
	\begin{minipage}[b]{0.95\linewidth}\label{fig: active_covid_19}
		\centering
		\includegraphics[width=1.0\linewidth]{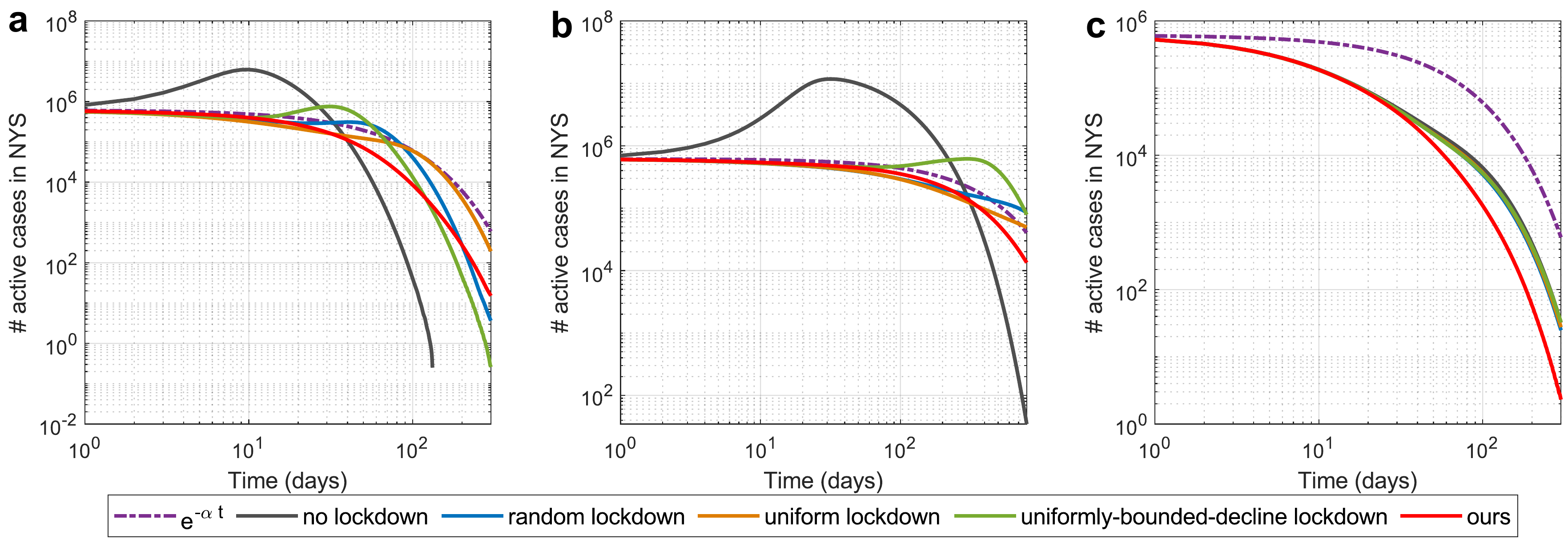}			
	\end{minipage}\hfill
	
	\vspace{3mm}
	
	\begin{minipage}[b]{0.95\linewidth}\label{fig: acc_covid_19}
		\centering
		\includegraphics[width=1.0\linewidth]{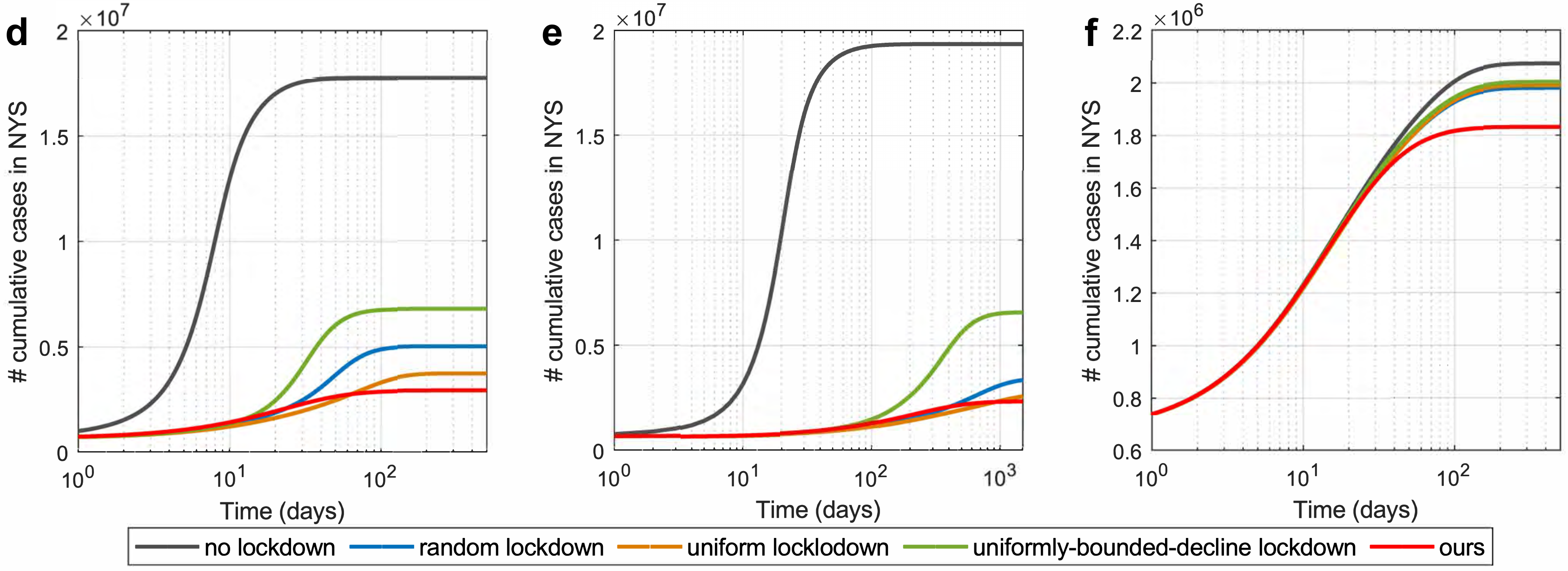}
	\end{minipage}%
	\caption{Estimated number of active casesand cumulative cases for the COVID-19 model by applying different lockdown policies based on available data about COVID-19 outbreak in NY on April 1st, 2020.  {\bf a-c},  the estimated number of active cases in NY. {\bf d-f}, the estimated cumulative cases in NY. In {\bf a}, {\bf d}, the disease parameters are set as in \cite{bertozzi2020challenges}, the decay rate $\alpha$ is chosen as 0.0231 which corresponds to halving every 30 days. In {\bf b}, {\bf e}, the disease parameters are set as in \cite{giordano2020modelling}, the decay rate is chosen as $\alpha = 0.2 r^{\text s} = 0.0034$ so that $\alpha < \min (r^{\text a}, r^{\text s})$. In {\bf c}, {\bf f}, the disease parameters are set as in \cite{birge2020controlling}, the decay rate $\alpha$ is chosen $0.0231$ that corresponds to halving every 30 days. Uniform lockdown, random lockdown, and uniformly-bouded-decline lockdown are defined as in Fig \ref{Fig: framework}. 
	}\label{Fig: active_acc_Covid_19}
	\centering
	\vspace{-1mm}
\end{figure}

\begin{figure}[!htb]
	\centering
	\begin{minipage}[b]{0.85\linewidth}
		\centering
		\includegraphics[width=1.0\linewidth]{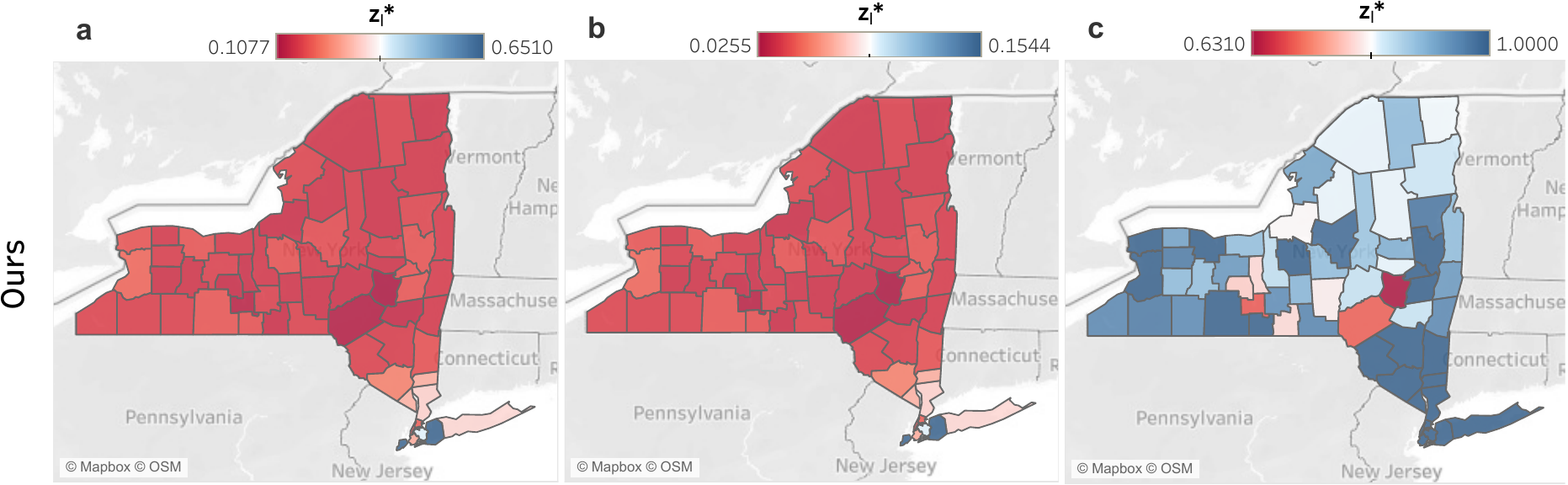}
	\end{minipage}

	\begin{minipage}[b]{0.85\linewidth}
		\centering
		\includegraphics[width=1.0\linewidth]{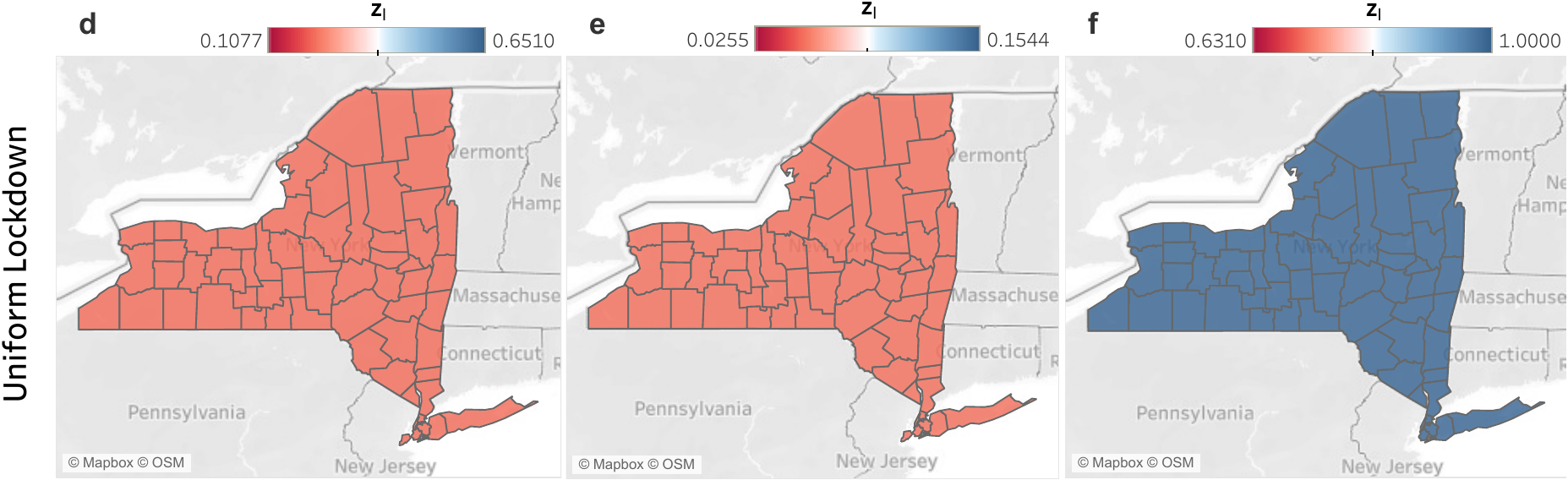}
	\end{minipage}%
	\centering

	\begin{minipage}[b]{0.85\linewidth}
		\centering
		\includegraphics[width=1.0\linewidth]{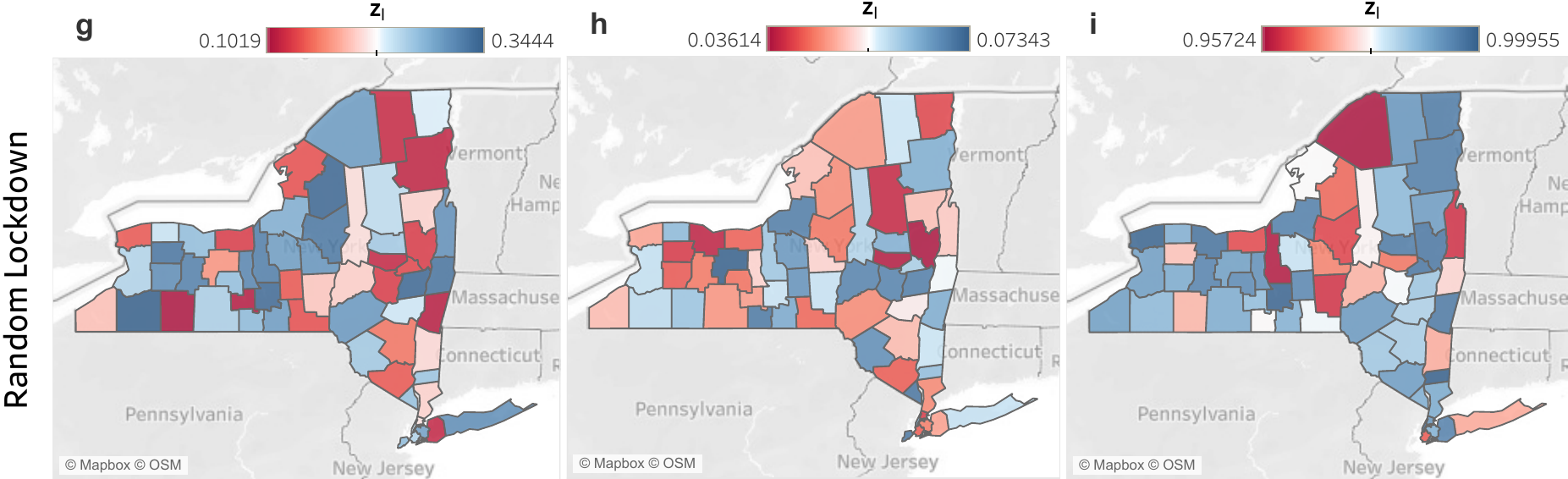}
	\end{minipage}%
	\centering

	\centering
	\begin{minipage}[b]{0.85\linewidth}
		\centering
		\includegraphics[width=1.0\linewidth]{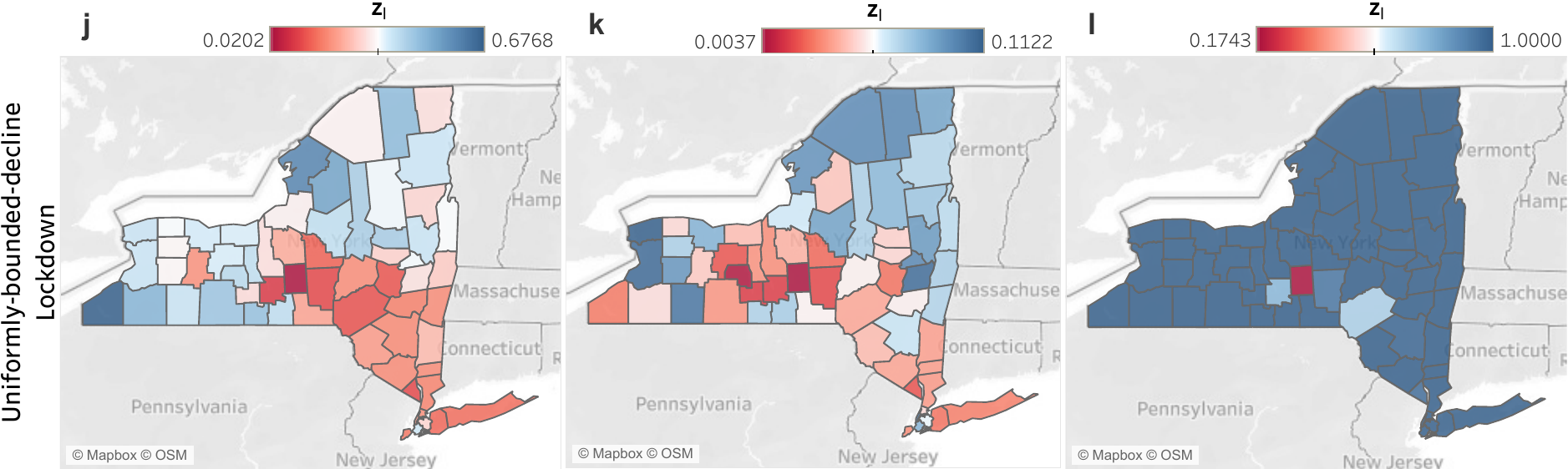}
	\end{minipage}
	\caption{Lockdown rate of each county given by different polices for the COVID-19 model based on available data about COVID-19 outbreak in NY on April 1st, 2020. {\bf a-c}, optimal lockdown rate $z_l^*$ given by our method . {\bf d-f}, uniform lockdown rate $z_l$. {\bf g-i}, random lockdown rate $z_l$. {\bf j-l}, uniformly-bounded-decline lockdown rate $z_l$.  Uniform lockdown, random lockdown, and uniformly-bouded-decline lockdown are defined as in Fig \ref{Fig: framework}.
		In {\bf a}, {\bf d}, {\bf g}, {\bf j}, the disease parameters are set as in \cite{bertozzi2020challenges}, the decay rate $\alpha$ is chosen as 0.0231 which corresponds to halving every 30 days. In {\bf b}, {\bf e}, {\bf h}, {\bf k} the disease parameters are set as in \cite{giordano2020modelling}, the decay rate is chosen as $\alpha = 0.2 r^{\text s} = 0.0034$ so that $\alpha < \min (r^{\text a}, r^{\text s})$. In {\bf c}, {\bf f}, {\bf i}, {\bf l}, the disease parameters are set as in \cite{birge2020controlling}, the decay rate $\alpha$ is chosen $0.0231$ that corresponds to halving every 30 days.  It can seen from {\bf a-c} that the value of $z_l^*$ for counties in NYC are relatively higher than other counties in New York State, which implies we should shutdown the outside of NYC harder than itself. Besides, it can be seen that such counter-intuitive phenomenon does not appear in any other lockdown polices.
	}\label{Fig: covid_19_each_county_ours_uni}
	\centering
	\vspace{-1mm}
\end{figure}

\clearpage

\begin{figure}[!htb]
	\centering
	\begin{minipage}[b]{0.24\linewidth}\label{fig: central_1}
		\centering
		\includegraphics[width= 1.0\linewidth]{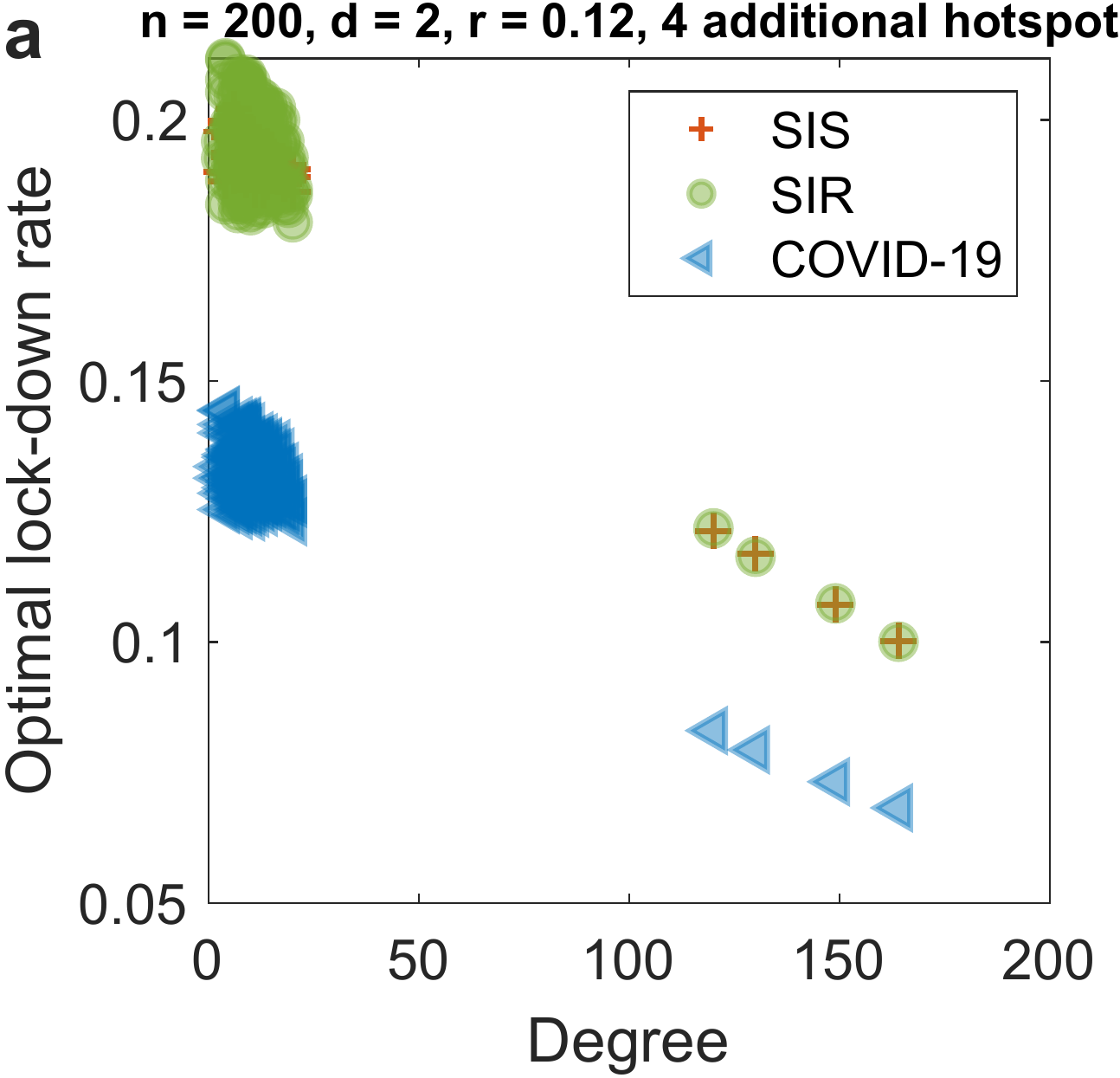}
	\end{minipage}%
	\hspace{1mm}
	\begin{minipage}[b]{0.24\linewidth}\label{fig: central_2}
		\centering
		\includegraphics[width =1.0\linewidth]{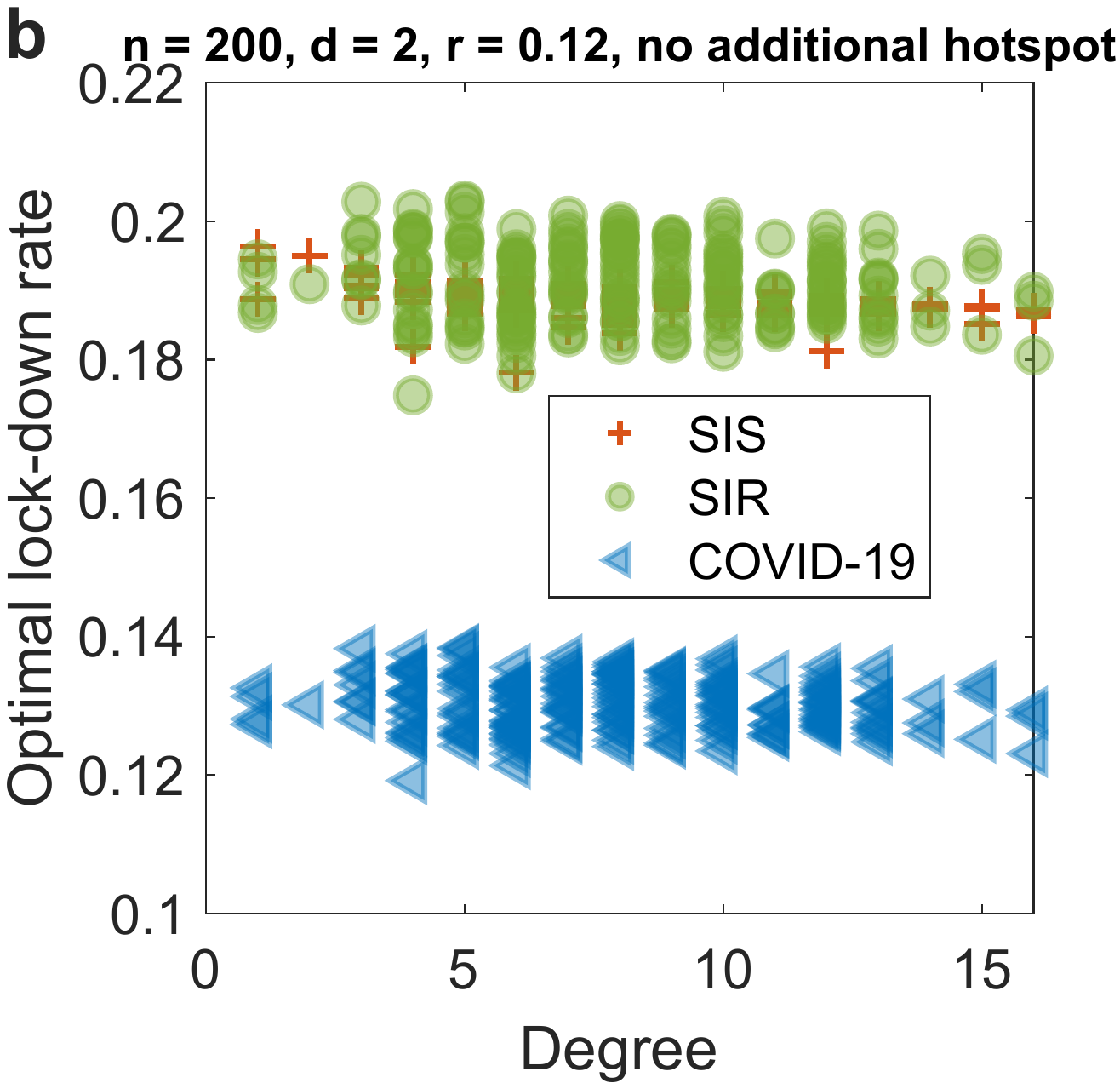}
	\end{minipage}%
	\begin{minipage}[b]{0.24\linewidth}\label{fig: population_new}
		\centering
		\includegraphics[width=0.91\linewidth]{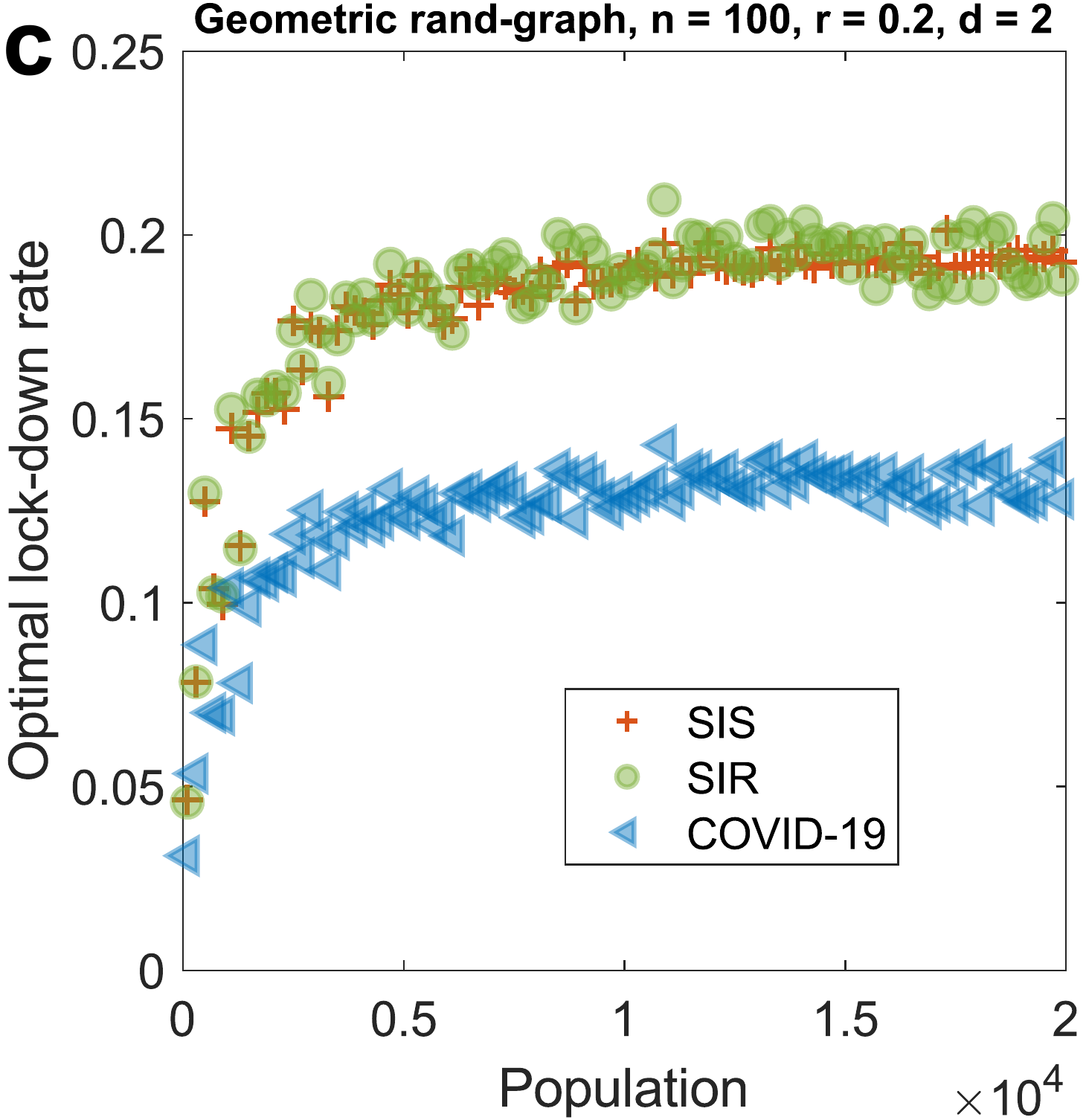}
	\end{minipage}%
	\begin{minipage}[b]{0.24\linewidth}\label{fig: home_rate_new}
		\centering
		\includegraphics[width=0.95\linewidth]{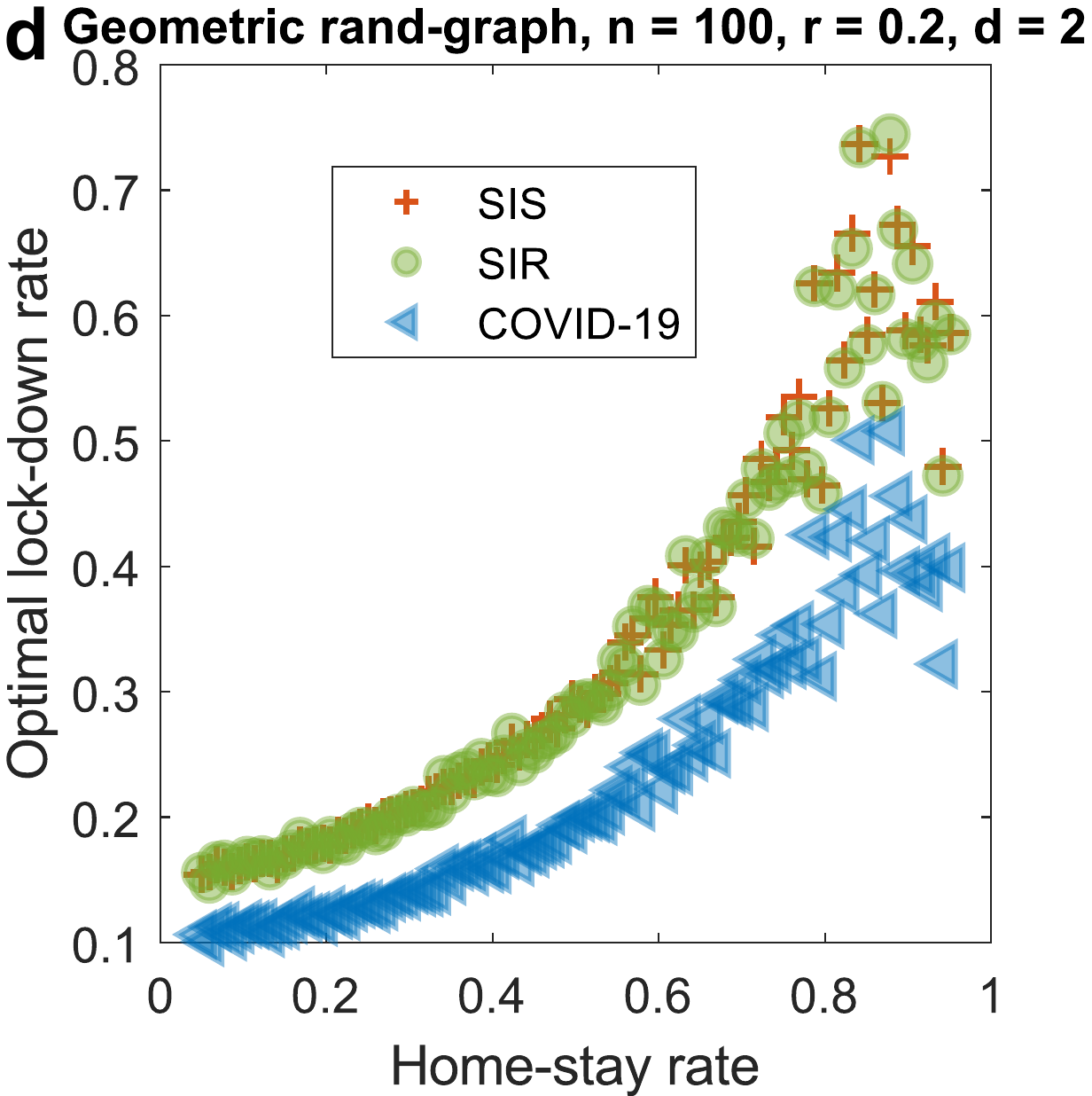}
	\end{minipage}
	\centering
	\vspace{-1mm}
	\caption{Numerical results of optimal lockdown rates on synthetic networks. {\bf a}, {\bf b}, the impact of the degree centrality on the geometric random graph.
		Each point represents a node in the network. It can be observed that centrality only matters for the value of $z_l^*$ when there exist hotspots in all the three models. 
		Without such hotspots, the effect of degree centrality is essentially ignorable. {\bf c}, the impact of population. Each point represents a node in the network. Nodes with small populations are assigned smaller values of $z_l^*$. Surprisingly, once the population is big enough, it doesn't affect $z_l^*$ too much. {\bf d}, the impact of  home-stay rate. Each point represents a node in the network. It can be observed that $z_l^*$increases as the home-stay rate increases, and in fact the home-stay rate has by far the largest influence on $z_l^*$ compared to the other parameters we consider.}\label{Fig: synthetic}
	\vspace{-1mm}
\end{figure}

\noindent {\bf Author contributions.} All authors designed and did the research. Q.M. performed all the calculations and wrote the manuscript. Y.-Y.L and A.O. edited the manuscript.  
\bigskip

\noindent {\bf Competing interests statement.} The authors declare no competing interests. 

\nocite{*}
\bibliography{pandemic}{}
\bibliographystyle{abbrv}


\clearpage

\begin{center}
  \vskip 60pt
\large \MakeUppercase{\textbf{Optimal Lockdown for Pandemic Control}}\\
\MakeUppercase{\textbf{---Supplementary Information---}} 
  \vskip 1em
  \large QIANQIAN MA\footnote{Department of Electrical and Computer Engineering, Boston
University, Boston, MA USA},  \ 
YANG-YU LIU\footnote{Channing Division of Network Medicine, Department of Medicine, Brigham and Women's
Hospital, Harvard Medical School, Boston, MA 02115, USA}, 
and ALEX OLSHEVSKY \footnote{Department of Electrical and Computer Engineering and Division of
Systems Engineering, Boston University, Boston, MA USA}
  \vskip 20pt
\end{center}

\setcounter{section}{0}
\setcounter{figure}{0}

\section{Related work}



\section{Related work}
	


	Our work is related to a number of recent papers motivated by the spread COVID-19, as well as some older work.  Indeed, spatial spread of epidemic admits a natural network representation, where nodes represent different locations and edges encode traveling of residents between the locations. Such spatial epidemic network model has received attention in the studies of COVID-19 recently \cite{jia2020population,chinazzi2020effect, bongiorno2021multi, zino2021analysis, parino2020modelling}. We begin by discussing several papers most closely related to our work. 
	
	Our paper builds on the results of \cite{birge2020controlling}, which proposed a spatial epidemic transmission model and consider the effect of lockdowns; we use the same model of lockdown of \cite{birge2020controlling} in this work.  The major difference between this work in \cite{birge2020controlling} is two-fold. First, we do not consider asymptotic stability in a model with births and deaths as our focus is on a shorter scale. Second, we propose new algorithms with improved running times; in particular, our main contribution is a linear time method that is applicable to the vast majority of cases we have considered. Similarly, the main difference of this work relative to \cite{nowzari2016analysis} and the references therein are new algorithms (though the lockdown models differ somewhat), as well as the new observations on counter-intuitive phenomena satisfied by the optimal lockdown.  Our work has some similarities with literature \cite{Brittoneabc6810}, which divided the population into 18 compartments and found that  population heterogeneity could significantly impact disease-induced herd immunity.  
	
	\subsection{Our Approach vs Traditional Optimal Control} \ao{An alternative approach would be to approach the lockdown problem using the techniques of optimal control. This approach is explored by a number of papers \cite{khanafer2014optimal, bock2018optimal, carli2020model, NBERw27102, fajgelbaum2020optimal, alvarez2020simple, borri2020optimal}. }
	
	As mentioned in the main body of the paper, our work has several main advantages over the traditional optimal control approach. The first is that we are looking for a fixed lockdown, whereas an optimal control based approach would offer a lockdown which varies for all time $t$. The second is that an approach based on optimal control would ask policymakers to repeatedly design lockdown relaxations when cases begin to decrease. As the public grows more impatient with lockdowns, political constraints could easily result in poor decision-making; one could argue this is what happened in the United States in 2020 \cite{guy2021association}. A single-fixed lockdown implemented when the number of cases is growing and maintained until the epidemic is extinct does not have this problem. Finally,  scalability is central to our results:  our main result is a nearly linear time algorithm. This is not the case for the optimal control approach. For example, Khanafer and Ba$\c{s}$ar  \cite{khanafer2014optimal} wrote down the optimal control formulation for the SIS case and remarked that solving the resulting equations ``is intractable.''
	
	\ao{As Khanafer and Ba$\c{s}$ar pointed out, there are no methods that are guaranteed to find the optimal control efficiently. Nevertheless, in a number of recent works promising numerical results are obtained. For a direct formulation of the problem, one can turn to  \cite{carli2020model}, which is also in the same spirit as our work, in that it studies control of COVID-19 using, among other things, movement restrictions. This results in a mixed-integer non-linear programming problem. Solving such problems is generally intractable, so \cite{carli2020model} used a genetic algorithm as a heuristic.} 
	
	\ao{Another possibility might be to use indirect methods (i.e., relying on the maximum principle) to solve the optimal control problem. The same complexity considerations, however, come up in this context again. Most of the literature on the optimal control of epidemics over networks seems to use variants of the forward-backward sweep method \cite{liu2016optimal,  fajgelbaum2020optimal, laaroussi2018optimal, bolzoni2017time} to solve the equations arising from the maximum principle, but it is known that this method can diverge even for simple examples \cite{mcasey2012convergence}. Nevertheless, on many examples the method converges in reasonable time: for example, \cite{liu2016optimal} reports excellent results on random networks of large size.}

	Recent papers studying control of COVID-19 using numerical optimal control methods are  \cite{NBERw27102, fajgelbaum2020optimal, alvarez2020simple, borri2020optimal}. These papers assumed a cost of a human life lost (sometimes set based on the average lifetime earnings) and considered the discounted total cost over the entire epidemic (alternatively, \cite{NBERw27102} considered the frontier of possible strategies over all possible ways to value life). This framework is conceptually different from ours: when modified to fit into our framework, strategies derived in this way will begin to relax the lockdown once the epidemic drops below a certain threshold as the number of lives lost comes to balance the cost of the lockdown, ultimately driving the epidemic to an endemic state; by contrast, our framework is designed to send the number of infections to zero. 
	
	More generally, the scalability of these approaches is unclear, due to their reliance on numerical methods without a clear convergence theory. 
	For example, solving for the solution of this nonlinear optimal control problem as in \cite{fajgelbaum2020optimal} requires an iterative method closely related to forward-backward sweeping. Each step requires the numerical solution of a system of differential equations, a matrix inversion, and a maximization of the Hamiltonian, without any a-priori bounds on the total number of steps the procedure will take, or any guarantee that the procedure will converge. 
	

	\section{Mathematical Background} The supplementary information will provide the proofs of the main results of the paper, as well as give details of many of our empirical and numerical results that were summarized in the main text. We begin with some definitions. 
	
	A matrix is called continuous time stable if all of its eigenvalues have nonpositive real parts. A matrix is called discrete time stable if all of its eigenvalues are upper bounded by one in magnitude. A central concern of this paper is to get certain quantities of interest (e.g., number of infected individuals) to decay at prescribed exponential rates. We will say that {\em $y(t)$ decays at rate $\alpha$ beginning at $t_0$} if $y(t) \leq y(t_0) e^{-\alpha t}$ for all $t \geq t_0$. Note that the decay in this definition  is not asymptotic but results in a decrease starting  at time $t_0$. 
	
	We will associate to every matrix $A \in \R^{n \times n}$ the graph $G(A)$ corresponding to its nonzero entries: the vertex set of $G(A)$ will be $\{1, \ldots, n\}$ while $(i,j) \in G(A)$ if and only if $A_{ji}\neq 0$. Informally, $(i,j)$ is an edge in $G(A)$ when the variable $j$ ``is influenced by'' variable $i$. We will say that $A$ is strongly or weakly
	connected if the graph $G(A)$ has this property.

	\subsection{Covering Semi-definite Program\label{sec:covering}} A covering semi-definite program has the form
	\[ \min \sum_{i=1}^d c_i x_i \] 
	\[ \sum_{i=1}^d x_i C_i \geq I \] 
	\[ x \geq 0. \] Here $c_i$ are nonnegative scalars and $C_i$ are positive semi-definite matrices. It turns out that covering semi-definite programs can be solved considerably faster than general semi-definite programs. Indeed, the recent paper \cite{jambulapati2020positive} showed that to compute a fixed-accuracy additive approximation of the optimal solution takes $\widetilde{O} \left( n^{\omega} + \sum_{i=1}^d {\rm nnz}(C_i) \right)$, where  $\omega$ is the exponent of matrix multiplication. 
	
	With these preliminaries in place, we next turn to justifying the analytical claims made in the main body of the paper. 
	
	\subsection{Matrix Balancing}\label{sec: matrix balancing} 
	The matrix balancing problems plays a fundamental role in our main results, and we briefly introduce it here. 
	Given a nonnegative matrix $P \in \R^{n \times n}$, we say it is balanced if it has identical row and column sums
	. The matrix balancing problem is, given a nonnegative $P$, to find a nonnegative diagonal matrix $D$ such that $D P D^{-1}$ is balanced. 
	
	The problem of matrix balancing is quite old; for example, an asymmetric version of this problem was introduced in the classic work of Sinkhorn and Knopp in the 1960s \cite{sinkhorn1967concerning}. It is impossible to survey all the literature on matrix balancing and related problems, though we refer the reader to \cite{idel2016review}. Recently, a powerful  algorithm for matrix balancing was given in \cite{cohen2017matrix}. It was shown in that work that this problem can be solved in linear time, understood as follows: solving the problem to accuracy $\epsilon$ requires only $\tilde{O}( {\rm nnz}(P) \log \kappa \log \epsilon^{-1})$ where $ {\rm nnz}(P)$ is the number of nonzero entries in the matrix $A$, $\kappa = D_{\rm max}^*/D_{\rm min}^*$ is the imbalance of the optimal solution, and the $\widetilde{O}$ hides logarithmic factors. Thus matrix balancing problems can be solved in nearly the same time as it takes to simply read the data, provided $\kappa$ is bounded away from zero. \aoa{In the event that we do not have an a-prior bound on $\kappa$, \cite{cohen2017matrix} give complexity bounds of $\tilde{O}({\rm nnz}(P)^{3/2})$ and $\tilde{O}({\rm nnz}(P) {\rm diam}(A))$ where ${\rm diam}(A)$ is the diameter of the graph corresponding to the matrix $A$.} 
	
	We will summarize this complexity by saying that the running time ``explicitly scales nearly linearly in the number of nonzero entries.'' The ``nearly'' comes from the logarithmic terms; the word ``explicit'' comes because the scaling also depends on the imbalance of the optimal lockdown $\kappa$, and one can construct families of examples where the $\kappa$ will have some kind of scaling with network size. 
	
	\smallskip 
	
	\section{Analytical Calculations} \label{SI: anlytical_calculations}
	In this section, we first present the details of the SIS model \cite{6763066} as well as how the matrix $A$ is constructed. Next, we justify the main theoretical achieved via eigenvalue bounds;  that  optimal lockdown for  these models can be reduced to a covering semi-definite program (which can be solved in matrix multiplication time); and that, under the high spread condition, optimal lockdown for  these models reduces to a matrix balancing problem (which can be solved in linear time). 
	
	\bigskip
	
	\subsection{Network SIS Model}\label{sec: SIS} is described by the following set of ordinary  differential equations
	\begin{equation} \label{eq:standardsis} \dot{x}_i = (1-x_i) \sum_{j=1}^n \beta a_{ij} x_j - \gamma x_i, ~~~ i = 1,\ldots, n.  
	\end{equation}  
	Here $\beta$ denotes the transmission rate, which captures the rate at which an infected individual infects others, $\gamma$ denotes the recovery rate, and $a_{ij}$ captures the rate at which infection flows from the population at location $j$ to location $i$.  Because $\dot{x}_i$ scales with $1-x_i$, the SIS model assumes that everyone who is not infected is susceptible.

	For simplicity of notation, we can stack up the coefficients $a_{ij}$ into a matrix as as $A=[a_{ij}]$. Then we can write the network SIS model as
	\[ \dot{x} = {\rm diag}({\bf 1} - x) \beta A x - \gamma x, \] where ${\bf 1}$ denotes the vector of all-ones while ${\rm diag}(u)$ makes a diagonal matrix out of the vector $u$. 
	
	It is desirable to have  $x_i(t) \rightarrow 0$ for all $i=1, \ldots, n$, i.e., to have the infection die out. It is mathematically convenient to encode this into the following equivalent condition: we will require that there exists some linear combination of the quantities $x_i(t)$ with positive coefficients which approaches zero, which happens if and only if the matrix $\beta A-\gamma I$ is continuous-time stable \cite{lajmanovich1976deterministic, aronsson1980deterministic, guo2006global, khanafer2016stability, mei2017dynamics}. 
	To achieve an exponential decay rate $\alpha$ of each $x_i(t)$, we require that there exists a positive linear combination of $x_i(t)$ decaying at that rate, which is guaranteed if $\lambda_{\rm max}(\beta A-\gamma I) \leq \alpha$ (see formal proof in \ssref{SI: anlytical_calculations}). Note that even though the network SIS dynamics is nonlinear, the asymptotic convergence  nevertheless reduces to a linear eigenvalue problem. 
	
	\smallskip
	
	\subsection{Construction of The Matrix $A$} \label{sec: A_construct}
	We next describe how the matrix $A$ is constructed. Our discussion will only be for the SIS case, as the COVID-19 case is similar. 
	
	Observe that the susceptible individuals at location $i$ can be infected at location $i$ as well as in other locations, the flow of susceptible population from location $i$ to location $l$ is $(1- x_i)\tau_{il}$. Besides, the rate of infection at location $i$ is proportional to the fraction of infected people in the total population of location $i$. Then we can rewrite the SIS model as
	\begin{align}\label{eq: SIS_rewrite}
		\dot{x}_i = \sum_{l =  1}^n (1 - x_i) \tau_{il} \frac{\sum_{j = 1}^{n}N_j \tau_{jl}x_j}{\sum_{k = 1}^n N_k \tau_{kl}} \beta - \gamma x_i.
	\end{align}
	Let $m(l) = \sum_{k = 1}^n N_k \tau_{kl}$, then \eqref{eq: SIS_rewrite} can be written as
	\begin{align*}
		\dot{x}_i &= \sum_{l =  1}^n (1 - x_i) \tau_{il} \frac{\sum_{j = 1}^{n}N_j \tau_{jl}x_j}{m(l)} \beta - \gamma x_i\\
		&= (1 - x_i) \sum_{l =  1}^n \sum_{j = 1}^n\frac{N_j \tau_{il}\tau_{jl}x_j }{m(l)}\beta- \gamma x_i\\
		& = (1- x_i)\sum_{j = 1}^n a_{ij} \beta x_j - \gamma x_i,
	\end{align*}
	where 
	\begin{align}\label{aij equation}
		a_{ij} = \sum_{l =  1}^n \frac{N_j \tau_{il}\tau_{jl} }{m(l)} = \sum_{l =  1}^n\frac{N_j}{\sum_{k = 1}^n N_k \tau_{kl}} \tau_{il}\tau_{jl}.     
	\end{align}
	As already remarked, this approach is not original to our work and is taken from \cite{birge2020controlling}. Note that in COVID-19 case, via similar process, we can obtain the same $a_{ij}$ as in Eq. \eqref{aij equation}.

	\subsection{Stability of The Network SIS and COVID-19 Models}\label{sec: stability}
	
	Recall that the network SIS model is given by the system of equations
	\begin{equation} \label{eq:standardsis2} \dot{x}_i = (1-x_i) \sum_{j=1}^n \beta a_{ij} x_j - \gamma x_i, ~~~ i = 1,\ldots, n,  
	\end{equation} where $n$ is the number of nodes in the underlying graph and $x_i(t)$ is the proportion of infected individuals at node $i$. By contrast,  the network COVID-19 model is given by 
	\begin{equation} \label{COVID_19-2}
		\left( 
		\begin{array}{c} 
			\dot{s} \\
			\dot{x^{\text a}} \\ 
			\dot{x}^{\text s} 
		\end{array} 
		\right) = 
		\left(
		\begin{array}{ccc} 
			0 & -\beta^{\text a} {\rm diag}(s) A &  - \beta^{\text s} {\rm diag}(s) A \\
			0 &\beta^{\text a} {\rm diag}(s) A - (\epsilon + r^{\text a}) &  \beta^{\text s} {\rm diag}(s) A \\
			0 & \epsilon & - r^{\text s}
		\end{array}
		\right)\left( 
		\begin{array}{c} 
			s \\
			x^{\text a} \\ 
			x^{\text s}
		\end{array} 
		\right),
	\end{equation} where now $x^{\text a}_i, x^{\text s}_i$ are the proportion of infected/asymptomatic and infected/symptomatic individuals at node $i$. This is a system of $3n$ equations, with three equations per node of the network. Recall also  the notation $M(t)$ used to denote the bottom $2n\times2n$ submatrix of the above matrix. In the main text, we stated that stability and decay rate of the network SIS model is equivalent to the eigenvalues of the matrix $A-\gamma I$, while the stability of the network COVID-19 model can be  ensured by bounding the eigenvalues of $M(t)$. We next give a pair of propositions formally justifying these assertions. 
	
	\medskip

	\begin{proposition} \label{prop:expdecay} Suppose the matrix $A$ is strongly connected. If $\lambda_{\rm max}(\beta A - \gamma I) \leq - \alpha$, then for the network SIS dynamics of Eq. (\ref{eq:standardsis2}),  there exists a positive linear combination of the quantities $x_i(t)$ that decays to zero at rate $\alpha$ starting at any time $t_0$. Conversely, if $\lambda_{\rm max}(\beta A - \gamma I) > -\alpha$, then there exists an initial condition in $x(0) > 0$ so that every positive linear combination of the quantites $x_i(t)$ fails to decay at rate $\alpha$.   
	\end{proposition}

	\begin{proposition} Suppose the matrix $A$ is strongly connected, and $s(t_0) > 0$. If $\lambda_{\rm max}(M(t_0)) \leq -\alpha$, then there exists a positive linear combination of the quantities $x^{\text a}_i(t), x^{\text s}_i(t)$ which decays  at rate $\alpha$ starting at time $t_0$. \label{prop:COVID} 
	\end{proposition} 
	
	\medskip
	
	The idea behind these propositions is standard in control theory: to get your system to decay at a rate of $e^{-\alpha t}$, make sure your eigenvalues have real parts that are at most $-\alpha$. For linear systems, this is guaranteed to work after a transient time, and for nonlinear systems, the situation is complicated. Fortunately, the nonlinear systems corresponding to network SIS and COVID dynamics have favorable properties, so that it is possible to draw conclusions about global behavior from the eigenvalues of a linear approximation at a point. 
	
	To get each of the quantities $x^{\text a}_i(t), x^{\text s}_i(t)$ to converge to zero at an asymptotic rate of $e^{-\alpha t}$, it suffices to have an arbitrary positive combination of them decay at rate $\alpha$. Interestingly, for the SIS case, this happens {\em if and only if} an eigenvalue condition is satisfied. In the COVID-19 case, the eigenvalue condition merely suffices to establish this.

	We now turn to the proof of these propositions. Our first step is to restate the Perron-Frobenius theorem in a form that will be particularly useful to us. 
	
	\begin{lemma}[A Version of Perron-Frobenius] \label{lemm:pf} Let us suppose $P$ is a strongly connected matrix whose off-diagonal elements are nonnegative. Then there exists a real eigenvalue of $P$ which is as large as the real part of any other eigenvalue of $P$. This eigenvalue is simple and the eigenvector corresponding to it is positive.  
	\end{lemma} 
	
	This lemma follows by observing that $A+\alpha I$ is nonnegative for a large enough choice of $\alpha$, so we can apply the Perron-Frobenius theorem (in the form of Theorem 2 in Section 8.2 of \cite{gantmacher2005applications}) to $A+\alpha I$. 
	
	\begin{definition} We will use $\lambda_{\rm max}(P)$ and $v_{\rm max}(P)$ to denote the eigenvalue/eigenvector described in Lemma \ref{lemm:pf}. 
	\end{definition} 
	
	We next give a proof of Proposition \ref{prop:expdecay}.

	\begin{proof}[Proof of Proposition \ref{prop:expdecay}] Suppose $P = \beta A- (\gamma - \alpha) I$. Let $\lambda_{\rm max}(P)$ and $v_{\rm max}(P)$ be the corresponding eigenvalue/eigenvector pair. Note that, by our assumptions, $\lambda_{\rm max} \leq 0$.  
		We thus have that 
		\begin{eqnarray*} \frac{d}{dt} v_{\rm max}^{\top} x & = & v_{\rm max}^{\top} \left[ {\rm diag} ( {\bf 1} - x) \beta A x - \gamma  x \right]\\ 
			& \leq & v_{\rm max}^{\top} \left[\beta Ax - \gamma x \right] \\ 
			& = & v_{\rm max}^{\top} \left[\beta Ax - \gamma x + \alpha I x \right] - v_{\rm max}^{\top} \alpha x \\
			& =  & \lambda_{\rm max} v_{\rm max}^{\top} x - \alpha v_{\rm max}^{\top} x\\ 
			& \leq & - \alpha v_{\rm max}^{\top} x
		\end{eqnarray*} where the second line used that $\beta, v_{\rm max}, A, x$ are nonnegative while ${\bf 1}-x \in [0,1]^n$; and the last line used that $\lambda_{\rm max} \leq 0$.  We conclude that $v_{\rm max}^{\top} x$ decays at a rate of $\alpha$ starting at any time. 
		
		On the other hand, suppose $\lambda_{\rm max}(\beta A-\gamma I) > -\alpha$. 
		Observe that the linearization of Eq. (\ref{eq:standardsis}) around the origin is $\dot{x}=(\beta A-\gamma I)x$.  Let us write the Jordan normal form,  
		\[ \beta A - \gamma I = T \left( {\rm diag}(\lambda_1, \ldots, \lambda_n) + J \right) T^{-1}, \] where $J$ is the upper-triangular matrix $\lambda_1, \ldots, \lambda_n$ are the eigenvalues of $\beta A-\gamma I$, with $\lambda_1$ being the Perron-Frobenius eigenvalue. Because the Perron Frobenius argument is simple, we have that the top Jordan block is $1 \times 1$; so that $J$  and all of its powers have zero entries in the first row. Using the standard formula for the matrix exponential of a Jordan block, we next decompose as 
		\begin{align*} e^{t(\beta A - \gamma I)}  & = T \left( {\rm diag}(e^{\lambda_1 t}, e^{\lambda_2 t}, \ldots, e^{\lambda_n t})  \left( I +  U_t \right) \right) T^{-1},  
		\end{align*} where $U_t$ is an upper triangular matrix, depending on $t$, but having only zero entries in its first row. We now choose the initial condition $x_0 = l T  {\bf e}_1$, where ${\bf e}_1$ is the first basis vector, and $l$ is a small-enough positive scalar; we'll discuss how small $l$ has to be later. We then have $T^{-1} x_0 =  l {\bf e_1}$ and therefore under the flow $\dot{x}= \beta A - \gamma I$ it holds that 
		\begin{align} \lim \sup_{t \rightarrow \infty} ||x(t)||_2^{1/t} & = \lim \sup_{t \rightarrow \infty} ||T e^{\lambda_1 t} l {\bf e}_1||_2^{1/t} \nonumber \\ & = \lambda_1 \nonumber \\
			& > -\alpha, \label{eq:linearbound}
		\end{align} 
		
		This establishes the property we want for the flow of the linear system $\dot{x}=(\beta A - \gamma I)x$, but we need to establish that the network SIS dynamics has the same property. This can be done via a particular form of the Hartman-Grobman theorem. Indeed, let $y(t)$ be the network SIS trajectory starting from $y_0$. The Harman-Grobman theorem, in the form of Theorem 3 of \cite{guysinsky2003differentiability}, guarantees the existence of a homeomorphism $h: \R^n \rightarrow \R^n$ such that 
		\[ y(t) = h^{-1} (e^{t(\beta A - \gamma I)} h(y_0)), \] and, as proved in  \cite{guysinsky2003differentiability},  because the network SIS dynamics are infinitely differentiable, we can further take $h$ to be differentiable at the origin with the derivative at the origin equalling identity: \[ h(y) = y + o(||y||_2). \] This implies that 
		\[ y = h^{-1}(y) + o(||y||_2),\] which we rearrange as 
		\begin{equation} \label{eq:hinverse} h^{-1}(y) = y + o(||y||_2)
		\end{equation}
		Morever,  for small enough $||x||_2,$ we have 
		\[ \frac{1}{2} ||x||_2 \leq ||h(x)||_2 \leq 2 ||x||_2. \] This further implies that for small enough $||x||_2$,
		\[ \frac{1}{2} ||x||_2 \leq ||h^{-1}(x)||_2 \leq 2 ||x||_2. \]

		With these observations in mind, we now choose the initial condition $y_0 = h^{-1}(x_0)$, where $x_0 = l T {\bf e}_1$ as above. We then have that   
		\begin{eqnarray*} \lim \sup_{t \rightarrow \infty} ||y(t)||_2^{1/t} & = & \lim \sup_{t \rightarrow \infty} ||h^{-1} (e^{t(\beta A-\gamma I)} h(y_0)||_2^{1/t} \\ 
			& \geq & \lim \sup_{t \rightarrow \infty}  (1/2)^{1/t} ||e^{t(\beta A-\gamma I)} x_0||_2^{1/t} \\ 
			& > & -\alpha,
		\end{eqnarray*} appealing in the last step to  Eq. (\ref{eq:linearbound}).
		
		We conclude the proof by arguing that our initial condition $y_0$ is nonnegative provided we choose $l$ small enough. Indeed, we first argue that $T {\bf e}_1$ is strictly positive. Indeed, observe that this is the first column of $T$, and since $\beta A-\gamma I  = T D T^{-1}$ can be rewritten as $(\beta A - \gamma I) T=TD$, we obtain that the first column of $T$ is the Perron-Frobenius eigenvector of $A$, which is positive by Lemma \ref{lemm:pf}.  Finally, $$y_0 = h^{-1}(l T {\bf e}_1) = l T {\bf e}_1 + o(l ||T||_2 \sqrt{n}) = l T {\bf e}_1 + o(l),$$ by Eq. (\ref{eq:hinverse});  and for small enough $l$, this has to be strictly positive by strict positivity of $T {\bf e}_1$.  This concludes the proof. 
	\end{proof}

	
	\medskip
	
	We next give the proof of Proposition \ref{prop:COVID}. Since, unlike in the SIS case, only one direction must be proven, the proof is straightforward. 
	
	\medskip
	
	\begin{proof}[Proof of Proposition \ref{prop:COVID}] By Lemma \ref{lemm:pf}, the matrix $M(t_0)$ has a left-eigenvector $v_{\rm max}$ which is positive. Let $\lambda_{\rm max}$ be the corresponding eigenvalue; by assumption $\lambda_{\rm max} \leq -\alpha$. Let us define $p(t) = [x(t), d(t)]^{\top}$, where $x(t)$ stacks up all the $x_i(t)$, and likewise for $d(t)$. Since it is immediate that the dynamics of Eq. (\ref{eq:COVID}) result in $s(t)$ non-increasing, we have that
		\begin{eqnarray*} \frac{d}{dt} v_{\rm max}^{\top} p(t) & = &  v_{\rm max}^{\top} M(t) p(t) \\ 
			& \leq & v_{\rm max}^{\top} M(t_0)) p(t) \\ 
			& = & \lambda_{\rm max} v_{\rm max}^{\top} p(t) \\ 
			& \leq & - \alpha v_{\rm max}^{\top} p(t)
		\end{eqnarray*} where the second equation uses that $v_{\rm max}$ and $p(t)$ are nonnegative, and as a consequence of the non-increasing of $s(t)$, $v_{\rm max}^{\top} M(t) p(t) \leq v_{\rm max}^{\top} M(t_0) p(t)$; while the final equation used $\lambda_{\rm max} \leq - \alpha$. We conclude that $v_{\rm max}^{\top} p(t)$ decreases at rate $\alpha$ starting from any time.
	\end{proof}

	
	\subsection{Lockdown Design}\label{sec: lock down}
	We now turn to the algorithmic question of designing an optimal lockdown. We will first present our main results. Next we will present a string of lemmas and observations which will culminate in the proof of the main Theorem. It is here that we will perform the reduction from the problem of computing the optimal lockdown to matrix balancing and covering semi-definite programs.

	Our first step is to discuss an assumption required by one of our algorithms. The formal statement of the assumption is as follows.

	\begin{assumption}[High spread assumption] $~$
		\begin{enumerate} \item In the network SIS model, we have ${\rm diag}(B^{\top}C) \geq \gamma$
			\item In the network COVID-19 model, we must have $$\left(\beta^{\text a} + \beta^{\text s} \frac{\epsilon}{r^{\text s}}\right) B^{\top} {\rm diag}(s(t_0)) C \geq \epsilon + r^{\text a}.$$ 
		\end{enumerate} \label{as:spread}
	\end{assumption} 
	
	To see why this condition is satisfied in a ``high spread'' regime, note that $A=C B^{\top}$ and, given our choices of $C$ and $B$ in Eq. (\ref{eq:bc}), the entry $C_{ii} B_{ii}$ corresponds to the spread of the epidemic in location $i$ purely from the same-location trips of residents of location $i$. This has to be bigger than the natural rate $\gamma$ at which people recover. In other words, the natural rate of spreading of the epidemic has to be high everywhere. Assumption \ref{as:spread} is actually somewhat looser than this, as what must be bounded below by $\gamma$ is ${\rm diag}(B^{\top} C)$, which is a sum of $O(n)$ products, only one of which is $C_{ii} B_{ii}$. Note that this is not quite the same as requiring that $A_{ii} \geq \gamma$ since $A+C B^{\top}$, and we've flipped the order of multiplication on the condition.    In the COVID-19 case, the interpretation is similar: the recovery rates $r^{\text a}, r^{\text s}$ need to be small relative to $s_i(t_0) C_{ii} B_{ii}$ as well as  the spread parameters $\beta^{\text a}, \beta^{\text s}, \epsilon$, though the relation is now more involved.

	\smallskip

	\textbf{Main theoretical contribution.} Our main theorem provides algorithms for the unconstrained and constrained lockdown problems in the cases when Assumption \ref{as:spread} does and does not hold. Our key contribution is to give an algorithm for optimal stabilizing lockdown whose complexity has an explicit scaling which is nearly linear in the number of nonzero entries of the matrix $A$. That is to say, not only can the optimal heterogeneous lockdown be computed exactly, but doing so takes nearly as much time as just reading through the data. 
	
	\begin{theorem} \label{thm:mainthm} Suppose the graph corresponding to positive entries of the matrix $A$ is strongly connected and $s(t_0)>0$. Then: \begin{enumerate} \medskip \item The unconstrained lockdown problem for both SIS and COVID-19 models can be reduced to matrix balancing. 
			\medskip
			\item Suppose further Assumption \ref{as:spread} holds. Then the constrained lockdown problem is equivalent to the unconstrained lockdown problem and consequently is also reducible to matrix balancing. 
			\medskip
			\item If Assumption \ref{as:spread} does not hold but the matrices $C$ and $B$ are given by Eq. (\ref{eq:bc}) and $\tau$ is strongly connected with positive diagonal, then the constrained lockdown problem for the SIS, SIR, COVID-19 models can be solved in $\widetilde{O}(n^3)$ time, where $\widetilde{O}(\cdot)$ hides factors which are logarithmic in the remaining problem parameters. 
		\end{enumerate}
	\end{theorem}

	Next, we will prove Theorem \ref{thm:mainthm}. Our starting point will be the following lemma on a ``splitting'' of a positive matrix. 
	
	\smallskip
	
	\begin{lemma} $~$\label{lemma:d}
		\begin{enumerate} \item A strongly connected matrix $P$ with non-negative off-diagonal elements is continuous-time stable if and only if there exists $d>0$ such that $Pd \leq 0$.
			\item The nonnegative strongly connected matrix $B$ is discrete-time stable if and only if there exists $d>0$ such that $Bd \leq d$. 
			\item Suppose $P = L-D$ where $L$ is nonnegative while $D$ is a matrix with nonpositive off-diagonal elements whose inverse is elementwise nonnegative.   Suppose further that both $P$ and $D^{-1} L$ are both strongly connected. Then $P$ is continuous time stable if and only if $B=D^{-1}L$ is discrete-time stable. 
		\end{enumerate}
	\end{lemma} 
	
	\smallskip
	
	This lemma is a small variation on a well-known fact: usually, parts 1 and 2 are stated for strictly stable matrices, in which case all the inequalities need to be strict (see \cite{bullo2019lectures}, Theorem 15.17 for the strict version of part (i) and Proposition 1 of \cite{rantzer2011distributed} for the strict version of part (ii)). The non-strict version additionally requires that the matrices be strongly connected, which is not needed for the nonstrict version of this problem. Note that we do not claim that any part of this lemma is novel. For completeness, we nevertheless give a proof next.

	\begin{proof}[Proof of Lemma \ref{lemma:d}] For part (1), observe that $Pd \leq 0$ if the same as $P{\rm diag}(d) {\bf 1} \leq 0$, which is equivalent to ${\rm diag}(d)^{-1} P {\rm diag}(d)$ being a matrix with nonpositive row sums. Since we have assumed that $P$ has nonnegative off-diagonal elements, this implies that the diagonal elements of ${\rm diag}(d)^{-1} P {\rm diag}(d)$ are non-positive. By Gershgorin circles, ${\rm diag}(d)^{-1} P {\rm diag}(d)$ must be continuous-time stable. Since its eigenvalues are the same as the eigenvalues of $P$, we conclude that $P$ is also continuous-time stable. Similarly, for part (2), if such a $d$ exists, then ${\rm diag}(d)^{-1} B {\rm diag}(d)$ is a nonnegative matrix whose row sums are upper bounded by one, and must be discrete-time stable again by Gershgorin circles. This proves the ``if'' part of parts (1) and (2).
		
		For the ``only if'' parts, suppose $P$ is continuous-time stable and strongly connected with non-negative off-diagonal entries. By Lemma \ref{lemm:pf}, we have that there exists a positive vector $d$ such that $Pd =\lambda d$ and $\lambda \leq 0$;  this proves the ``only if'' of part 1. For part 2, if $B$ is discrete-time stable and strongly connected, we have that $B d = \lambda d$ for the Perron-Frobenius eigenvalue $d$ of $P$, which is positive. Since now $\lambda <1$, this proves the ``only if'' statement of part 2.

		For part (3), suppose first that $P=L-D$ is continuous-time stable. Since $P=L-D$ has nonnegative off-diagonal elements, we can apply part (1) to observe that continuous time stability of $P$ is equivalent to existence of $d$ satisfying
		\[ (L-D)d \leq 0, ~~~d > 0 \] which is  equivalent to 
		\[ Ld \leq Dd, ~~~ d > 0. \]  We now to multiply both sides by $D^{-1}$ and obtain that there exists a $d$ satisfying
		\[ D^{-1} L d \leq d, ~~~ d > 0. \] Note that we used that $D^{-1}$ is nonnegative to multiply both sides by $D^{-1}$. Since $D^{-1} L$ is strongly connected, the last equation is, by part (2), exactly the statement that $B=D^{-1} L$ is discrete-time stable. 
		
		Conversely, suppose $D^{-1} L$ is discrete-time stable (note that we cannot simply reverse the above chain of implications since we do not assume  that $D$ is elementwise nonnegative; thus we can multiply a linear inequality by $D^{-1}$ but not necessarily by $D$).  Let $\lambda_{\rm max}, v_{\rm max}$ denote the Perron-Frobenius eigenvalue/eigenvector pair of $D^{-1} L$, guaranteed to exist by Lemma \ref{lemm:pf}; note that $v_{\rm max}$ is strictly positive. We then have 
		\begin{align*} L v_{\rm max} & = D D^{-1} L v_{\rm max} \\ 
			& = D \lambda_{\rm max} v_{\rm max} \\ 
			& \leq D v_{\rm max}, 
		\end{align*} where the last step used that $\lambda_{\rm max} \leq 1$. It follows that $P v_{\rm max} = (L-D) v_{\rm max} \leq 0$, and since, as already observed, $v_{\rm max}$ is positive, we obtain that $P$ is continuous-time stable by applying part (1). 
	\end{proof}
	
	\medskip
	
	We will later need to interchange the order of products while still preserving the condition of being strongly connected. To that end, the following lemma will be useful. 
	
	\medskip 
	
	\begin{lemma}\label{lemma:strong_connectivity}
		Suppose $U,V$ are two nonnegative $n \times n$ matrices  with no zero rows or columns such that $VU$ is strongly connected. Then $UV$ is strongly connected. 
	\end{lemma} 
	
	\begin{proof} Consider a directed bipartite graph $G$ on $2n$ vertices, with vertices $l_1, \ldots, l_n$ and $r_1, \ldots, r_n$ denoting the two sides of the bipartition, defined as follows. If $U_{ij} > 0$, then we put an edge from $l_j$ to $r_i$; and if $V_{ij} > 0$ we put an edge from $r_j$ to $l_i$. 
		Let $G_1$ be the graph on $l_1, \ldots, l_n$ where we put the directed edge from $l_i$ to $l_j$ if there is a path of length two from $l_i$ to $l_j$ in $G$. Likewise, let $G_2$ be the directed graph on $r_1, \ldots, r_n$ such that we put an edge from $r_i$ to $r_j$ whenever there is a path of length two in $G$. 
		
		Then the strong connectivity of $VU$ is equivalent to having $G_1$ be strongly connected: indeed, $(VU)_{ab} > 0$ if and only if there exists a link from $b$ to $a$ in $G_1$. Similarly,  the strong connectivity of $UV$ is equivalent to having $G_2$ be strongly connected. We will show that if $G_1$ is not strongly connected, neither is $G_2$. The converse is established via a similar argument. 
		
		Indeed, suppose $G_1$ is not strongly connected. That means there exists a proper subset of the vertices, say $L_1 = \{l_1, \ldots, l_k\}$, with no edges outgoing to $L_1^c = \{l_{k+1}, \ldots, l_n\}$. Let $r_1$ be the set of out-neighbors of $L_1$ in $G$. Then we must have that $r_1$ is a proper subset of $\{r_1, \ldots, r_n\}$ (for otherwise, the assumption that $V$ has no zero rows/columns would contradict no edges going from $L_1$ to $L_1^c$ in $G_1$) and the set of out-neighbors of $r_1$ in $G$ is contained in $L_1$. 
		
		But since (i) $r_1$ is a proper subset of the right-hand side (ii) the out-neighbors of $r_1$ in $G$ are contained in $L_1$ (iii) the out-neighbors of $L_1$ in $G$ are contained in $r_1$, we obtain that there are no edges leading from $r_1$ to $r_1^c$ in $G_2$. This proves $G_2$ is not strongly connected. 
	\end{proof} 
	
	\medskip
	
	With these preliminary lemmas in place, we now turn to core of our reduction. We begin with the network SIS dynamics, where we will reduce the task of finding an (unconstrained) optimal lockdown to the problem of matrix balancing defined earlier. The reduction will go through several ``intermediate'' problems, the first of which as follows. 
	
	\medskip

	\begin{definition} We will refer to the following  as the {\em stability scaling problem}: given a nonnegative strongly connected matrix $P$ and positive diagonal matrix $D$, find positive scalars $q_1, \ldots, q_n$ minimizing $\sum_{i=1}^n q_i^{-1}$ such that ${\rm diag}(q_1, \ldots, q_n) P - D$ is continuous-time stable. 
	\end{definition}
	
	\medskip
	
	The utility of this definition should become clear after the following lemma.

	\medskip

	\begin{lemma}  Suppose $A = CB^{\top}$ is strongly connected. The minimum cost lockdown problem for the SIS network can be written as stability scaling.  Under Assumption \ref{as:spread}, the constrained lockdown problem for the network SIS problem can also be written as stability scaling. \label{lemma:ss} 
	\end{lemma} 
	
	\begin{proof} We consider the unconstrained lockdown problem first. Recall that, in the SIS model, we are looking for a minimum cost positive vector $z$ such that 
		\begin{equation} \label{eq:sisz} C {\rm diag}(z) B^{\top} - (\gamma - \alpha) I 
		\end{equation} is continuous-time stable. 
		Since the nonzero eigenvalues of a product of two matrices do not change after we change the order in which we multiply them, this is the same as requiring that 
		\[ {\rm diag}(z) B^{\top} C - (\gamma - \alpha ) I \] is continuous-time stable. 
		This is exactly the stability scaling problem provided we have two additional conditions. The first condition is that $\alpha < \gamma$ (because the diagonal matrix subtracted needs to be positive). The second condition is that $B^{\top} C$ should be strongly connected. But by  Lemma \ref{lemma:strong_connectivity}, the second condition is true because we assumed that $A=CB^{\top}$ is strongly connected.

		Observe that the reduction resulted in an instance of stability scaling with matrix $P=B^{\top} C$. Since we have assumed $A=CB^{\top}$ is strongly connected, we can apply Lemma \ref{lemma:strong_connectivity} to obtain that $P$ is strongly connected as needed. Under the assumption $\alpha < \gamma$, the matrix $(\gamma - \alpha)I$ is a positive diagonal matrix as required. 
		
		We thus have the reduction we want, from optimal lockdown to stability scaling, assuming $\alpha < \gamma$. But what if $\alpha \geq \lambda$? In that case, we claim that the minimum cost lockdown problem does not have a solution. Indeed, since at the optimal solution we must have all $z_i^*>0$, we have that $C {\rm diag}(z^*) B^{\top} $ is an irreducible nonnegative matrix and its Perron-Frobenius eigenvalue is strictly positive by the Perron-Frobenius theorem \cite{gantmacher2005applications}. Consequently, $C {\rm diag}(z^*) B^{\top} - (\gamma-\alpha)I$ has a positive eigenvalue and cannot be continuous-time stable.

		Finally, we consider the constrained version. We argue that, under Assumption \ref{as:spread}, the optimal solution to the constrained lockdown problem will have $z_i^* \leq 1$ for all $i$, so we can simply drop the constraint. Indeed, first observe that we can assume $\alpha \leq \gamma$,  else the problem does not have a solution as explained above. Now suppose that $z_j^* > 1$; then ${\rm diag}(z^*) B^{\top} C - {\rm diag}(\gamma - \alpha) I$ has a nonnegative  $j$'th row with at least one positive entry in that row. Indeed, the off-diagonal entries in the $j$'th row are clearly nonnegative, while the diagonal entry is nonnegative by Assumption \ref{as:spread}. 
		
		By Lemma \ref{lemma:d}, ${\rm diag}(z^*) B^{\top} C - {\rm diag}(\gamma - \alpha) I$ cannot be continuous-time stable. As matrix $C {\rm diag}(z) B^{\top} - (\gamma - \alpha) I$ has the same nonzero eigenvalues as the matrix ${\rm diag}(z^*) B^{\top} C - {\rm diag}(\gamma - \alpha) I$, it can not be stable either. 
	\end{proof} 
	
	
	
	
	
	\medskip

	\begin{lemma} In the unconstrained case, the minimum lockdown model for COVID-19 dynamics can be reduced to stability scaling provided $A$ is strongly connected and $s(t_0)>0$. In the constrained case, the same holds under Assumption \ref{as:spread}. 
		\label{lemma:COVID}
	\end{lemma} 
	
	\begin{proof} 
		Let us begin by assuming that
		\begin{equation} \label{eq:alphabound} \alpha<\min(r^{\text s}, \epsilon + r^{\text a}).
		\end{equation} We will revisit this assumption later.  We  need to make the matrix 
		\begin{equation}\label{eq: A0 definition}
			A_0 := \left(
			\begin{array}{cc} 
				\beta^{\text a} {\rm diag}(s(t_0)) A_z - (\epsilon + r^{\text a}) +\alpha &  \beta^{\text s} {\rm diag}(s(t_0)) A_z \\
				\epsilon & - r^{\text s} + \alpha
			\end{array}
			\right)
		\end{equation}
		 stable, where we have introduced the notation  that $A_z = C {\rm diag} (z) B^{\top}$. Let us write 
		\[ A_0 = L - D,\] where 
		\[L = \left(\begin{matrix}\beta^{\text a} {\rm diag}(s(t_0))A_z & \beta^{\text s} {\rm diag}(s(t_0))A_z \\
			0 & 0 \end{matrix} \right), \]
		and
		\[ D = \left(\begin{matrix}\epsilon + r^{\text a} -\alpha & 0 \\
			-\epsilon & r^{\text s} - \alpha \end{matrix} \right) .\]
		
		We next  apply part 3 of Lemma \ref{lemma:d} to get that $A_0$ is continuous-time stable if and only if $D^{-1} L$ is discrete time stable. To do this, however, we need to verify that $D^{-1}$ is elementwise nonnegative, and that both $P$ and $D^{-1} L$ are strongly connected. This easily follows from observing that \[
		D^{-1} = \left (\begin{matrix} 
			\frac{1}{\epsilon + r^{\text a} -\alpha} & 0 \\
			\frac{\epsilon}{(\epsilon + r^{\text a} - \alpha)(r^{\text s} - \alpha)} & \frac{1}{r^{\text s} - \alpha},
		\end{matrix}\right )
		\] and recalling that $s(t_0) > 0$ as well as Eq. (\ref{eq:alphabound}). 
		
		To summarize, we have shown that equivalently we need to make sure that $B = D^{-1}L$ is discrete-time stable. Since the eigenvalues of a matrix which is the product of two matrices do not change after changing the order of the product, this is the same as having $L D^{-1}$ be discrete-time stable. But we have the simple expression 
		\begin{equation} \label{eq:ldeigs}
		L D^{-1} = \left(\begin{matrix}
			{\rm diag}(s(t_0))A_z b_1 & \frac{\beta^{\text s}}{r^{\text s} - \alpha}{\rm diag}(s(t_0))A_z \\
			0 & 0
		\end{matrix} \right),
		\end{equation}
		where $b_1 = \frac{\beta^{\text s} \epsilon +\beta^{\text a}(r^{\text s} - \alpha)}{(\epsilon + r^{\text a} -\alpha)(r^{\text s} - \alpha)}$. 
		
		But the eigenvalues of $LD^{-1}$ are just the eigenvalues of its top $n \times n$ block. Thus we obtain that $LD^{-1}$ is stable if and only if the matrix ${\rm diag}(s(t_0))Ab_1$ is discrete-time stable. Plugging in $A = C {\rm diag}(z)B^{\top}$, we now need that
		\begin{equation} \label{eq:im:stability} {\rm diag}(s(t_0))C {\rm diag}(z)B^{\top}b_1 \text{ is discrete-time stable}.\end{equation}
		
		We next appear to part (3) of Lemma \ref{lemma:d} again, using $D={\rm diag}(s(t_0))^{-1} b^{-1} I$ to obtain that we need 
		\begin{align}\label{eq:COVIDreduction}
			C {\rm diag}(z)B^{\top} - {\rm diag}(s(t_0))^{-1} b_1^{-1} I \text{ is continuous-time stable},   
		\end{align}
		We can apply part (3) of Lemma \ref{lemma:d} as strong connectivity follows from $z$ being elementwise positive and strong connectivity of $A=CB^{\top}$. But the last condition is equivalent to Equation (\ref{eq:sisz}), which we've already shown how to reduce to stability scaling.

		We conclude the reduction by observing that, once again, this results in an instance of stability scaling with the matrix $P=B^{\top} C$, which is strongly connected because $A=CB^{\top}$ is strongly connected, allowing us to apply Lemma \ref{lemma:strong_connectivity}. Further, the matrix ${\rm diag}(s(t_0))^{-1} b_1^{-1} I$ is a positive diagonal matrix due to the assumptions that $s(t_0)>0$ and the assumption of Eq. (\ref{eq:alphabound}).

		For the constrained case, we argue that under Assumption \ref{as:spread}, the optimal solution will have all $z_i^* \leq 1$, so we can equivalently consider the unconstrained case. Indeed, suppose e.g., that $z_j^* > 1$. In that case, the matrix ${\rm diag}(z^*) B^{\top} {\rm diag}(s(t_0)) C$ has its $(j,j)$'th entry at least $b_1^{-1}$. Since the $j$'th row of ${\rm diag}(z^*) B^{\top} {\rm diag}(s(t_0)) C $ is nonnegative, and since $B^{\top} {\rm diag}(s(t_0)) C$ is strongly connected by Lemma \ref{lemma:strong_connectivity}, we have that the $j$'th row of ${\rm diag}(z^*) B^{\top} {\rm diag}(s(t_0)) C $ is not zero. By by the same argument as we made in the SIS case, the matrix ${\rm diag}(z^*) B^{\top} {\rm diag}(s(t_0)) C  -  b_1^{-1} I$ cannot be continuous-time stable. This implies that Eq. (\ref{eq:COVIDreduction}) cannot be satisfied.

		Finally, we conclude the proof by revisiting the assumption we made in the very beginning, namely the assumption of Eq. (\ref{eq:alphabound}). We now argue the problem of optimal lockdown has no solution if that equation fails. Indeed, in that case, the matrix $A_0$  has either its first $n \times n$ block nonnegative, or its last $n \times n$ block nonnegative. We show that the problem has no solution in the first case; the other case has a similar proof. 
		
		We want to argue that having $d>0$ such that $A_0 d \leq 0$ cannot occur if the top $n \times n$ block of $A_0$ is nonnegative; by Lemma \ref{lemma:d} part (1) this is enough to show $A$ cannot be stable regardless of $z$. We can partition $d=[d_1, d_2]$, and we immediately see that we must have that $d_1=0$, since, by the strong connectivity of ${\rm diag}(s(t_0) ) A_z$, every entry of $d_1$ is multiplied by some entry in the top left $n \times n$ submatrix of $A_0$. Applying the same argument to the ``top right'' $n \times n$ block of $A_0$, get that $d_2=0$, and this is a contradiction. This concludes the proof. \end{proof}
	
	\medskip 
	
	To recap where we are: we have just reduced the unconstrained lockdown problem for the SIS and COVID-19 models to stability scaling; further, the constrained lockdown problem was reduced to the same under the high-spread condition. 
	
	We need to make the following remark: provided the matrix $A=CB^{\top}$ was strongly connected and $s(t_0)>0$, both of these reductions arrived at a stability scaling problems 
	Our next step is to reduce the stability scaling problem to a new problem, defined next, which has a somewhat simpler form.

	\begin{definition} Given a nonnegative matrix $A$, the problem of finding scalars $l_1, \ldots, l_n$ such that 
		$A - {\rm diag}(l_1, \ldots, l_n)$ is continuous-time stable while minimizing $\sum_{i=1}^n c_i l_i$ will be called the { \em diagonal subtraction problem. } 
	\end{definition} 
	
	\begin{lemma} Stability scaling can be reduced to diagonal subtraction. \label{lemma:ds}
	\end{lemma} 
	
	\begin{proof} Indeed, starting from 
		\[ \min_{q_1>0, \ldots, q_n>0} \sum_{i=1}^n c_i q_{i}^{-1} \] 
		\[ \mbox{ such that } {\rm diag}(q_1, \ldots, q_n) P - D \mbox{ is continuous time stable}, \] we apply Lemma \ref{lemma:d}, part (c) to obtain that the constraint  is the same as \[ D^{-1} {\rm diag}(q) P \mbox{ is discrete-time stable}. \] Here we used crucially that the matrix $P$ is nonnegative and strongly connected and that the diagonal matrix $D$ is positive, ensuring that the assumptions of Lemma \ref{lemma:d}, part 3 are satisfied. We next use the same Lemma \ref{lemma:d} part (c) again to obtain that this the last constraint is identical to \[ P - {\rm diag}(q_1, \ldots, q_n)^{-1} D \mbox{ is continuous time stable } \] For simplicity, it is convenient to introduce the notation $\delta_i = q_i^{-1}, i = 1, \ldots, n$. In terms of these new variables $\delta_i$, we have the problem 
		\[ \min_{\delta_1>0, \ldots, \delta_n>0} \sum_{i=1}^n c_i \delta_i \] 
		\[ \mbox{ subject to } P - {\rm diag}(d_1 \delta_1, \ldots, d_n \delta_n ) \mbox{ is continuous-time stable}, \] where $D={\rm diag}(d_1, \ldots, d_n)$. Finally, introducing variables $u_i=d_i \delta_i, i = 1, \ldots, n$, we have
		\[ \min_{u_1>0, \ldots, u_n>0} \sum_{i=1}^n c_i \frac{u_i}{d_i} \]
		\[ \mbox{ subject to } P - {\rm diag}(u_1, \ldots, u_n ) \mbox{ is continuous-time stable }, \] which is an instance of diagonal subtraction with cost coefficients $c_i/d_i$. 
	\end{proof}

	\medskip
	
	With all these preliminaries in place, we are now ready to turn to the proof of our main theoretical resu.t 
	
	\medskip

	\begin{proof}[Proof of Theorem \ref{thm:mainthm}, parts 1 and 2] In Lemma \ref{lemma:ss}, we have shown how to reduce the optimal lockdown  problem for SIS and COVID-19 dynamics to stability scaling. In the subsequent Lemma \ref{lemma:ds}, we showed how to reduce stability scaling to diagonal subtraction. We now prove Theorem \ref{thm:mainthm} by reducing diagonal subtraction to matrix balancing. 
		
		Indeed, we need to solve 
		\[  \min_{l_1 > 0, \ldots, l_n >0} \sum_{i=1}^n c_i l_i  \] 
		\[ \mbox{ such that } A -  {\rm diag}(l_1, \ldots, l_n)  \mbox{ is continuous time stable}, \] where $A$ is a nonnegative matrix. We use  part (a) of Lemma \ref{lemma:d} to write this ass
		\begin{equation} \min_{l_1, \ldots, l_n, d_1, \ldots, d_n} \sum_{i=1}^n c_i l_i \label{eq:first} \end{equation}
		\[ \mbox{ such that } Ad - {\rm diag}(l_1 d_1, \ldots, l_n d_n) \leq 0 \]
		\[ d > 0, l > 0  \]  The difficulty here is that we are optimizing over $l$ and $d$, and the constraint includes a product of the variables. We try to make this into a simpler problem by introducing the variables $\mu_i = l_i d_i, i = 1, \ldots, n$.  We change variables from $(l,d)$ to $(\mu,d)$ to obtain the equivalent problem: 
		\begin{equation} \label{eq:second}  \min_{\mu_1, \ldots, \mu_n, d_1, \ldots, d_n} \sum_{i=1}^n c_i \mu_i/d_i \end{equation}
		\[ \mbox{ such that } Ad  \leq  \mu \]
		\[ d > 0, \mu >0 \] 
		In this reformulation, the constraint is linear, and the nonlinearity has been moved to the objective. 
		
		Next, consider what happens when we fix some $d>0$ and consider the best $\mu$ for that particular  $d$. Because $c_i > 0$ for all $i$, and $d>0$ , we want to choose each element $\mu_i$ as small as possible. The best choice is clearly $\mu = Ad$; no component of $\mu$ can be lower than that by our constraints. Because $A$ is a nonnegative  matrix without zero rows by strong connectivity, this results in a feasible $\mu>0$. 
		
		Thus we can transform this into a problem which optimizes over just the variables $d_1, \ldots, d_n$: 
		\[ \min_{d_1>0, \ldots, d_n>0} \sum_{i=1}^n c_i [Ad]_i/d_i.\] Our next step is to change variables  one more time. Since $d$ is a positive vector, it is natural to write $d_i = e^{g_i}$. In terms of the new variables $g_1, \ldots, g_n$,  we just need to minimize the function $f(g_1, \ldots, g_n)$ defined as  
		\[ \min f(g_1, \ldots, g_n) := \min_{g_1, \ldots, g_n} \sum_{i=1}^n \sum_{j=1}^n c_i a_{ij} e^{g_j-g_i}. \] We now have an unconstrained minimization of a convex function. In particular, if we can find a point where $\nabla f(g_1, \ldots, g_n)=0$, we will have found the global optimum. 
		
		Let us consider the $k$'th component of the gradient of $f(g_1, \ldots, g_n)$. We have the equation 
		\begin{eqnarray} 0 = \frac{ \partial f}{\partial g_k} (g_1, \ldots, g_k)  & = &  \sum_{j \neq k}  -c_k a_{kj} e^{g_j-g_k} + \sum_{i \neq k} c_i a_{ik} e^{g_k - g_i} \nonumber \\
			& = & \sum_{j=1}^n  -c_k a_{kj} e^{g_j-g_k} + \sum_{i=1}^n c_i a_{ik} e^{g_k - g_i} \label{eq:zerograd}
		\end{eqnarray} 
		Let $D_g = {\rm diag}(e^{g_1}, \ldots, e^{g_n})$ and $D_c = {\rm diag}(c_1, \ldots, c_n)$.
		Observing that 
		\[ [D_g^{-1} D_c A D_g]_{uv} = e^{-g_u} c_u a_{uv} e^{g_v} = c_u a_{uv} e^{g_v - g_u},
		\] we have that Eq. (\ref{eq:zerograd}) can be written as 
		\[ - [D_{g}^{-1} D_c A D_g {\bf 1}]_k + [{\bf 1}^{\top} D_g^{-1} D_c A D_g ]_k =  0 \] 
		In other words, the matrix $D_g^{-1} D_c A D_g$ needs to have it's $k$'th column sum equals to its $k$'th row sum. Thus provided we can find a balancing of the matrix $D_c A$, we will have found the minimum of $f(g_1, \ldots, g_n)$ as needed. 
		
		However, a matrix $X$ can be balanced if and only if the underlying graph $G(X)$ is strongly connected \cite{kalantari1997complexity}. Since we have assumed that $c_i$ are all positive and $A$ is strongly connected, we have that $D_c A$ can be balanced. We conclude that the unique minimum of $f(g_1, \ldots, g_n)$ can be recovered from the balancing of this matrix. This concludes the reduction of diagonal subtraction to matrix balancing.

	\end{proof}

	\bigskip
	
	We next turn to the third part of Theorem \ref{thm:mainthm}. We will not be using any of the reductions used to prove the first two parts of the theorem. The first steps of our proof are very similar to the proof of the main result of \cite{birge2020controlling}, which gave a semidefinite formulation of a lockdown problem ensuring asymptotic stability in a model involving births and deaths. We diverge from \cite{birge2020controlling} when we write the problem as a covering semi-definite problem as described in Section \ref{sec:covering} and analyze its complexity. 
	
	\medskip

	\begin{proof}[Proof of Theorem \ref{thm:mainthm}, part 3] We can simply repeat the proof of Lemma \ref{lemma:ss} to reduce the constrained lockdown problem for network SIS and COVID-19 processes to stability scaling; except that, without Assumption \ref{as:spread}, we now  have to keep the constraint that $z \in [0,1]^n$. Indeed, note that the only place where Assumption \ref{as:spread} was used in those two proofs was to argue that we can  omit that constraint. 
		
		In this way, we can reduce the constrained lockdown problem to the following:
		\[\min_{z_1, \ldots, z_n} \sum_{i=1}^n c_i z_i^{-1} \]
		\[
		{\rm diag}(a) C {\rm diag}(z_1, \ldots, z_n) B^{\top} - qI \mbox{ is continuous time stable } \] 
		\[ z \in [0,1]^n. \] In this SIS case, $a = {\bf 1}$ while $q=\gamma - \alpha$. In the COVID-19 case, from Eq. (\ref{eq:COVIDreduction}) we have that $a =s(t_0)$ while $q = b_1^{-1}$. Note that this optimization problem is only over the variables $z_1, \ldots, z_n$; all other quantities appearing in it are parameters.

		Suppose next that matrices $C$ and $B$ are chosen according to Eq. (\ref{eq:bc}). In other words, we have the problem 
		\[  \min_{z_1, \ldots, z_n} \sum_{i=1}^n c_i z_i^{-1}  \]
		\[ {\rm diag}(a) \tau {\rm diag}(z_1, \ldots, z_n) D_1 \tau^{\top} D_2   - qI \mbox{ is continuous time stable }
		\]
		\[ z \in [0,1]^n, \] where, recall, the matrices $D_1,D_2$ are defined immediately after Eq. (\ref{eq:bc}). 
		
		Since the nonzero eigenvalues of a product of two matrices do not change when we flip the order in which they are multiplied, the second constraint is equivalent to 
		\[ 
		{\rm diag}(z_1, \ldots, z_n) D_1 \tau^{\top} D_2 {\rm diag}(a) \tau - qI \mbox{ is continuous time stable }. \]  Applying part (3) of Lemma \ref{lemma:d}, we get that this is the same as  
		\[ q^{-1} {\rm diag}(z_1, \ldots, z_n) D_1 \tau^{\top} D_2 {\rm diag}(a) \tau  \mbox{ is discrete time stable }. \] Here we used the strong connectivity and positive diagonal of the matrix $\tau$.  Next, applying the same lemma again, we further get the equivalence to 
		\[  \tau^{\top} D_2 {\rm diag}(a) \tau - D_1^{-1}  {\rm diag}(z_1, \ldots, z_n)^{-1} qI \mbox{ is continuous time stable }. \] This can be simplified further by observing that a symmetric matrix is stable if and only if it is non-positive-definite. 
		
		Indeed, let us introduce the notation $u_i = (D_1)_{ii}^{-1} q z_i^{-1}$, and $Q = \tau^{\top} D_2 {\rm diag}(a) \tau$. We can therefore write our problem as 
		\[ \min \sum_{i=1}^n \frac{c_i (D_1)_{ii}}{q} u_i \] 
		\[ 0 \succeq Q - {\rm diag}(u_1, \ldots, u_n) \] 
		\[ u_i \in [q (D_1)_{ii}^{-1}, +\infty) \mbox{ for all } i=1, \ldots, n. \] 
		Further defining $c_i' = c_i (D_1)_{ii}$ and the $2n \times 2n$ positive-semidefinite matrices $B_i = e_i e_i^{\top} + e_{n+i} e_{n+i}^{\top}$, we can write this as 
		\begin{eqnarray} 
			&& 
			\min \sum_{i=1}^n c_i' u_i \nonumber \\
			&& \sum_{i=1}^n u_i B_i \succeq \left(
			\begin{array}{cc} 
				Q & 0 \\ 
				0 & 0 
			\end{array}
			\right) 
			+ 
			\left(
			\begin{array}{cc} 
				0 & 0 \\ 
				0 & q {\rm diag}( (D_1)_{11}^{-1}, \ldots,  (D_1)_{nn}^{-1}) 
			\end{array}
			\right) \label{eq:covering}
		\end{eqnarray} 
		As mentioned earlier, this is known as a ``covering SDP.'' In the recent paper, \cite{jambulapati2020positive}, it was shown that it can be be solved, up to various logarithmic factors, in matrix multiplication time. We next discuss in detail the results in \cite{jambulapati2020positive} and how they are applicable to Eq. (\ref{eq:covering}).  
		
		Specifically, in \cite{jambulapati2020positive}, an algorithm for checking whether the system of equations 
		\[ \sum_{i} x_i C_i \succeq I, \sum_i x_i P_i \leq (1-\delta) I, x \geq  0, \] is feasible was given; here $C_i, P_i$ are arbitrary nonnegative definite matrices. To begin applying this to our problem, let us choose $\delta=1/2$ and $P_i = \frac{1}{2n} u^{-1} c_i' {\bf 1} {\bf 1}^{\top}$ for all $i$. In that case, we are checking the feasibility of 
		\begin{equation} \label{eq:simplecovering} \sum_i x_i C_i \succeq I, \sum_i c_i' x_i \leq u, x \geq 0. 
		\end{equation} 
		The running time of the algorithm from \cite{jambulapati2020positive} for checking whether this is feasible is 
		$\widetilde{O}(\sum_i {\rm nnz}(C_i) + n^{\omega})$, where the $\widetilde{O}$ hides a multiplicative factor which is a power of logarithm is $O \left( n^2  \max_i \lambda_{\rm max}(C_i)/\lambda_{\rm max}(P_i) \right)$ (see discussion under ``Main Claim'' in \cite{jambulapati2020positive}, and note that, in our case, the number of matrices and the dimension are proportional). In our case, we will choose $C_i = e_i e_i^{\top} + e_{n+i} e_{n+i}^{\top}$ so that $$\lambda_{\rm max}(C_i) \leq 1, \lambda_{\rm max}(P_i) = \frac{1}{2} \max_i c_i' u^{-1},$$ and plugging the definition of $c_i'$, we have that the factor being hidden in the tilde is a power of logarithm of  
		\begin{align*} n^2 \frac{1}{\max_i \lambda_{\rm max}(P_i)} & = 2n^2 \frac{1}{\max_i c_i' u^{-1}} \\
			& = 2n^2u \frac{1}{\max_i c_i'} \\ 
			& = 2n^2 u \frac{1}{ \max_i c_i(D_1)_{ii}} \\ 
			& \leq 2n^2 u^{-1}    \frac{1}{    \max_i c_i \left( \sum_{j} N_j \tau_{ji} \right)^{-1}}.
		\end{align*}

		To summarize, we have discussed how long it takes to check feasibility of Eq. (\ref{eq:simplecovering}). The conclusion is that the running time $\widetilde{O}(\sum_i {\rm nnz}(C_i) + n^{\omega})$, where the the tilde notation hides a factor that is logarithmic in the problem parameters $n, u^{-1}, \min_i  \sum_j c_i^{-1}  N_j \tau_{ji}$.
		
		Now let us transition from just checking feasibility of Eq. (\ref{eq:simplecovering}) to minimizing $\sum_{i=1}^n c_i' x_i$ subject to the constraint $\sum_i x_i C_i \succeq 0, x \geq 0$. We can, of course, find an $\epsilon'$-additive approximation of this problem with a lot of feasibility checks. 
		
		Let us suppose that we actually know that $x_i^* \geq l_i$ for some $l_i$ (we do, in fact, know this because Eq. (\ref{eq:covering}) forces 
		\begin{align*} x_i^* \geq q (D_1)_{ii}^{-1} & = q \sum_{k} N_{k} \tau_{ki} \\ & := q N_i^{\rm visit} \\
			:= l_i
		\end{align*} If $c^*$ is the optimal cost, then we need to do $O \left( \log (c^*/\sum_i c_i' l_i) +  \log \epsilon'^{-1} \right)$ feasibility checks to get an $\epsilon$ additive approximation to the optimal cost. Upper bounding $c^* \leq n c_{\rm max} \max_i x_i^*$, we see that we need to do 
		\[ O \left( \log \frac{n \max_i c_i \max_i x_i^*}{q \sum_i c_i' N_{i}^{\rm visit}} + \log \epsilon'^{-1}\right), \] feasibility checks.

		
		Now we cannot quite apply this directly to  Eq. (\ref{eq:covering}), since one issue remains. That issue is that  we have discussed an algorithm to handle the constraint $\sum_i u_i B_i \geq I$; but in Eq. (\ref{eq:covering}), the right-hand side is not the identity matrix. 
		
		However, the right-hand side of Eq. (\ref{eq:covering})  is symmetric, so it can be decomposed as $U \Lambda U^{\top}$ for orthogonal $U$ and diagonal $\Lambda$. Since $A \succeq B$ implies $Z A Z^{\top} \succeq Z B Z^{\top}$, by choosing $Z =  \Lambda^{-1/2} U^{\top} $ we can equivalently rewrite Eq. (\ref{eq:covering}) as 
		\[ \min \sum_{i=1}^n c_i' u_i \] 
		\[ 
		\sum_{i=1}^n u_i \Lambda^{-1/2} U^{\top} B_i U \Lambda^{-1/2} \succeq I. 
		\] Letting $C_i' = \Lambda^{-1/2} U^{\top} B_i U \Lambda^{-1/2}$, we thus obtain a problem to which the algorithm of Eq. (\cite{jambulapati2020positive}) is applicable. Unfortunately, in the transformation the number of nonzero entries in the matrix $C_i$ changes. Nevertheless,  as explained in \cite{jambulapati2020positive} (in footnote 6 in the arxiv version), despite the change of variables, the scaling in the running time remains with the number of nonzero matrices in the matrices $B_i$. However, we now need to add the complexity of computing the eigendecomposition of the right-hand side of Eq. (\ref{eq:covering}) to out complexity bounds. 
		
		We can now put together everything we have observed above. The total complexity scales as the complexity of matrix multiplication and computing the eigendecomposition of a symmetric matrix, which we takes $O(n^3)$ exact arithmetic operations \cite{pan1999complexity}. This is multiplied by a power of a logarithm in the quantities 
		\[ \epsilon'^{-1}, n, \max_i c_i, \max_i x_i^*, \frac{b_1}{ \sum_i c_i' N_i^{\rm visit}}, u^{-1}, \lambda_{\rm max}(C'), \min_i  \sum_j c_i^{-1}  N_j \tau_{ji} \] 
		All of these are listed in the theorem statement, except $\max_i x_i^*, u^{-1}, \frac{1}{\sum_i c_i' N_i^{\rm visit}}, \lambda_{\rm max}(C')$. Let's upper bound each of these four quantities. 
		
		Since After Eq. (\ref{eq:covering}), we switched from $u$ to $x$ as we began discusisng the results of \cite{jambulapati2020positive}, we have that 
		\[ \max_i x_i^* = \max_i u_i^* \leq q \max_i q (D_1)_{ii}^{-1} (\min_i z_i^*)^{-1} = b_1^{-1} \left( \max_a \sum_k N_k \tau_{ka} \right) (\min_i z_i^*)^{-1} \] 
		Next, the smallest $u$ will be is $$\sum_i c_i' l_i = q \sum_i c_i' N_{i}^{\rm visit}.$$ Thus we can ``kill two birds with one stone'' by first upper bounding $u^{-1}$ as 
		\[ \left( \sum_i c_i' l_i \right)^{-1} = \left( \sum_i c_i (D_1)_{ii}^{-1} q \sum_k N_{k} \tau_{ki} \right)^{-1} \] 
		which is a polynomial in $$\max_i c_i^{-1}, \max_i \left( \sum_k N_k \tau_{ki} \right)^{-1}, b_1,$$ and this also upper bounds $$\sum_i c_i' N_{i}^{\rm visit}$$ by a polynomial in the same quantities and $q^{-1}$. 
		Next, 
		\[ \lambda_{\rm max}(C_i') \leq \frac{1}{\min_i \Lambda_{ii}} \leq \frac{1}{\min(\lambda_{\rm min}(Q), q \min_a N_a^{\rm visit})}, \] which is in turn a polynomial in 
		\[ \frac{1}{\lambda_{\rm min}(Q)}, q^{-1}, \left( \min_a N_k \tau_{ka} \right)^{-1} \] Finally, since $Q = \tau^{\top} D_2 {\rm diag}(a) \tau$, we have that 
		\begin{align*} \lambda_{\rm min}(Q) &  \geq \left( \min_i (D_2)_{ii} a_i \right)  \lambda_{\rm min}(\tau^{\top} \tau) \\ 
			& = \lambda_{\rm min}(\tau^{\top} \tau) \left( \min_i \beta^{\text s} N_i s_i(t_0) \right).
		\end{align*} To summarize, the arithmetic complexity of computing an $\epsilon'$ additive approximation to the optimal solution is $O(n^3)$ times a polylog in the variables
		\begin{small}
			\[ \epsilon'^{-1}, n, \max_i c_i, \max_i c_i^{-1}, (\min_i z_i^*)^{-1}, b_1, b_1^{-1}, \max_a \sum_k N_k \tau_{ka}, \left(\min_a \sum_k N_k \tau_{ka} \right)^{-1},  \left( \lambda_{\rm min}(\tau^{\top} \tau) \right)^{-1} , \left(  \min_i \beta^{\text s} N_i s_i(t_0) \right)^{-1}.\] 
		\end{small} Finally, this is to compute and $\epsilon'$-approximate solution of the modified cost function in Eq. (\ref{eq:covering}) which omits a factor of $q^{-1}$. To account for this, we can simply divide $\epsilon'$ by a factor depending on $q$; but since $q=b_1^{-1}$ and both $b_1, b_1^{-1}$ are already in the list of quantities under the poly log term, this does not change anything. 
	\end{proof} 
	
	\section{Connection to the Reproduction Number\label{sec:repr}}

\aoa{We now make an explicit a connection between the constraint we use on the growth rate of the epidemic, namely the constraint $\lambda_{\rm max}(M(t_0)) \leq -\alpha$, and the basic reproduction number $R(t_0)$, defined as the expected number of secondary cases produced in a completely susceptible population by a typical individual at time $t_0$. In particular, we show that our constraint is 
completely equivalent to an upper bound on the reproduction number. }

\aoa{\begin{proposition} Given any $r \in [0,1]$, one can find an $\alpha$ in the domain $\alpha \in [0, \min (r^{\text s}, \epsilon + r^{\text a})]$ such that the constraints 
\[ R(t_0) \leq r \] and 
\[ \lambda_{\rm max}(M(t_0)) \leq - \alpha \]  
are equivalent. The converse is also true: given $\alpha$ in the above domain, one can find an $r \in [0,1]$ so that the above two constraints are equivalent.  
\end{proposition} }

\begin{proof} 

\aoa{The argument essentially reprises the proof of  Lemma \ref{lemma:COVID}. Indeed, in that lemma we considered the matrix
\[ M(t_0) = 
\begin{pmatrix} 
\beta^{\rm  a}  {\rm diag}(s(t_0)) A_z - \epsilon - r^{\rm a} & \beta^{\rm s} {\rm diag}(s(t_0)) A_z \\ 
\epsilon & -r^{\rm s} 
\end{pmatrix}.
\] In Eq. (\ref{eq:im:stability}), it was shown that the constraint 
\[ \lambda_{\rm max}(M(t_0) \leq \alpha \] 
was equivalent to the constraint 
\[ \lambda_{\rm max}({\rm diag}(s(t_0)) A_z b_1(\alpha) ) \leq 1.  \] Here we write the expression $b_1(\alpha)$, unlike in Eq. (\ref{eq:im:stability}),
this way to highlight the dependence on $\alpha$. 
Using the definition of $b_1(\alpha)$ from that lemma, the last condition can be rewritten as 
\begin{align} 
 \rho ( {\rm diag}(s(t_0)) A_z) & \leq   b_1(\alpha)^{-1}  \nonumber \\ 
 & =  \frac{(\epsilon + r^{\text a} -\alpha)(r^{\text s} - \alpha)}{\beta^{\text s} \epsilon +\beta^{\text a}(r^{\text s} - \alpha)} \nonumber \\ 
 & = \frac{\epsilon + r^{\text a} - \alpha}{\beta^{\text a} + \beta^{\text s} \epsilon / (r^{\text s} - \alpha)} \label{eq:fchoice} 
 \end{align} 
 We thus have that an upper bound the condition $\lambda_{\rm max}(M(t_0)) \leq -\alpha$ is equivalent to the upper bound $\rho ( {\rm diag}(s(t_0)) A_z) \leq f(\alpha)$,  where $f(\alpha)$ is a monotonically decreasing
 function of $\alpha$. 
 We next argue that a similar finding holds for the constraint $R(t_0) \leq r$. }

\aoa{To do this, we need to express the reproduction number in terms of the matrix $M(t_0)$. We use an expression from \cite{van2002reproduction} in the form discussed in the recent paper \cite{smith2021convex}. Indeed, Definition 3 of \cite{smith2021convex} shows\footnote{\aoa{Specifically, to apply those results here, the variable $y$ of \cite{smith2021convex} represents the susceptible agents,
while the variable $x$ of that paper collects all the asymptotmatic and symptomatic locations. The function $f(x,y)$ of \cite{smith2021convex}, representing the inflow of infections, is taken to be zero.}}, referring to \cite{van2002reproduction}, that if we split  
\[ M(t_0) = F + V \] where 
\[ F = \begin{pmatrix} 
\beta^{\rm  a}  {\rm diag}(s(t_0)) A_z  & \beta^{\rm s} {\rm diag}(s(t_0)) A_z \\ 
0 & 0 
\end{pmatrix}
\] and 
\[ V = 
\begin{pmatrix} 
-\epsilon -r^{\rm a} & 0 \\ 
\epsilon & -r^{\rm s}  
\end{pmatrix} 
\] then 
\[ R(t_0) = \rho(F V^{-1}). \] }

\aoa{However, we have already computed this spectral radius in Eq. (\ref{eq:ldeigs}):
\begin{align*} R(t_0)& =  \rho ( {\rm diag}(s(t_0)) A_z )  b_1(0) \\ 
& \leq  \rho ( {\rm diag}(s(t_0)) A_z ) \frac{\beta^s \epsilon + \beta^a r^s}{(\epsilon + r^a)r^s}
\end{align*} 
Thus the constraint 
\[ R(t_0) \leq r \] 
is equivalent to 
\begin{equation} \label{eq:gchoice} \rho({\rm diag}(s(t_0) A_z) \leq r \frac{(\epsilon + r^a)r^s}{\beta^s \epsilon + \beta^a r^s} \end{equation} We have thus shown that the constraint $R(t_0) \leq r$ is equivalent to the condition 
$\rho({\rm diag}(s(t_0) A_z) \leq g(r)$, where $g(r)$ is a monotonic function of $r$. Since both constraint on $R(t_0)$ and $\lambda_{\rm max}(M(t_0))$ are equivalent to constraints on 
$\rho({\rm diag}(s(t_0) A_z)$, they are equivalent to each other. }

\aoa{Finally,  inspecting Eq. (\ref{eq:fchoice}) and Eq. (\ref{eq:gchoice}), we see that $\alpha = 0$ corresponds to $r=1$, which is as one might expect. On the other hand, we see that as $\alpha$ approaches the endpoint
of its domain, $\min (r^{\text s}, \epsilon + r^{\text a})$, we have that the corresponding $r$ approaches zero. This concludes the proof}. \end{proof}

	\section{Empirical Data Analysis}\label{data_sources} \quad
	
	We now describe how our data was obtained and how our simulations were conducted in a higher level of detail compared to the discussion in the main body of the paper. 
	
	\medskip

	\textbf{Disease parameters.} The disease parameters ($\beta^{\text s}$, $\gamma$,  $r^{\text a}$,  $r^{\text s}$, $\epsilon$, $\beta^{\text a}$) are essential to derive the optimal lockdown rate $z$. To provide valid results and fully understand the role of the disease parameters, we use three different groups of estimates provided in recent literature \cite{birge2020controlling} \cite{giordano2020modelling} \cite{bertozzi2020challenges} to construct the disease parameters, respectively. However, none of those models quite match our two-state COVID-19 model, which is why one needs to be careful in re-using the parameters estimated in those models. 
	
	We directly use the recovery rate $\gamma$, $r^{\text a}$, and $r^{\text s}$ from these literature. The reason is that these quantities have the interpretation  that infected people recover at a rate $\gamma$ (or $r^{\text a}$, $r^{\text s}$ in the COVID-19 case) per unit time, which should carry over from model to model. Similarly, we will also reuse the parameter $\epsilon$ from literature \cite{birge2020controlling} \cite{giordano2020modelling} which represents the rate that an infected individual develops symptoms.
	
	In the SIS case, we  choose the constant $\zeta$ (from definition of matrix $D_2$ in Eq. (\ref{eq:bc})) to match the initial growth of the models in the literature rate; the ``initial growth rate'' is  $\lambda_{\text{max}}(A) -\gamma$. We do likewise in the SIR case. 
	
	In the case of COVID-19, we also need to choose the  transmission rates $\beta^{\text a}, \beta^{\text s}$, which we do as follows. 
	Our first step is to let  $\beta^{\text a} = \hat{\alpha} \beta^{\text s}$ and assume we can reuse $\hat{\alpha}$ from the existing literature, as this scalar measures the transmission rate difference of symptomatic individuals and asymptomatic individuals. Thus we only need to decide how to choose $\beta^{\text s}$. Our second step is to choose $\beta^{\text s}$ to match the initial growth rates of the models in the literature. This is  $\lambda_{\rm max}(M(t_0))$, where $M(t_0)$ is defined in \eqref{COVID_19}. Note that as a consequence of this procedure, the matrix $A$ used for in COVID-19 case may not match the matrix $A$ used in the SIS case.

	Finally, we remark that the initial growth rate of the models in literature \cite{birge2020controlling} \cite{giordano2020modelling} \cite{bertozzi2020challenges} can also be computed with the similar method. The resulting parameter values offered by literature \cite{birge2020controlling} \cite{giordano2020modelling} \cite{bertozzi2020challenges} are summarized in \stref{tab: model_parameters}. Our procedure is essentially the same as matching the $R_0$ of the models (see e.g., \cite{diekmann1990definition,heesterbeek2007type}). 
	
	One caveat is the in the model from \cite{bertozzi2020challenges}, the parameters $r^{\text a}$, $r^{\text s}$, and $\hat{\alpha}$ are missing. In our experiments using data from \cite{bertozzi2020challenges}, we simply choose $r^{\text a} = r^{\text s} = \gamma$, and let $\hat{\alpha} = 0.6754$ as in \cite{giordano2020modelling}.
	
	\medskip
	
	\textbf{Populations ($N_i$).} To define the populations of each node in the network, we adopt the 2010 Census Bureau data \cite{population} at the level of the counties in the New York state. These populations ($N_i$) will be used to construct the matrix $A$ as in \eqref{eq:bc} in the three models.
	
	\medskip
	
	\textbf{Economic coefficients ($c_i$).} In the cost function, $c_i$ represents the economic coefficients, which captures the cost of closing down site or location $i$. As the economic activity closely related to the number of employees, we let $c_i$ be proportional to the number of employees in location $i$. Specifically, suppose $e_i$ is the number of employees of node $i$, $e_{max}$ is the maximum of the number of the employees in the network, we define $c_i = e_i/e_{max}$. 
	We use LEHD Origin-Destination Employment Statistics (LODES) dataset \cite{employment} to obtain the number of employees in each county of NYS. Particularly, we use the Residence Area Characteristics (RAC) type of this dataset from \cite{employment}. This dataset is based on census at the  block geographic level; hence we aggregate the data to the county level.
	
	\medskip
	
	\textbf{Travel rate ($\tau_{ij}$).} To construct matrix $A$ for the three models, we need travel rate matrix $\tau$, where $\tau_{ij}$ represents the rate at which an individual travels from location $i$ to location $j$. We use the Social Distancing Metrics dataset \cite{travel} from SafeGraph to generate $\tau$. This dataset was collected using a panel of GPS pings from anonymous mobile devices, and it is based on Census Block Group levels. For each device/individual, the dataset identifies a ``home'' CBG, and the median daily home-dwell-time is provided for each CBG. Additionally, this dataset provides the daily number of trips that the people go from their home CBG to various destination CBGs. 
	
	In our empirical simulations, we only consider the network of New York State (i.e.,  we do not consider the trips to places outside the New York State). For each node, we aggregate the number of trips to the county level and obtain the number of trips from one node to another. We can also obtain the home-dwell-time of each node as the median of the home-dwell-time among all the CBGs (daily median home-dwell-time) in this county. Then, we define 
	\begin{equation}\label{eq: definition of tau}
		\tau_{ij} = \left(1 - \frac{h_i}{1440} \right)\cdot \frac{k_{ij}}{\sum_a k_{ia}},
	\end{equation}
	where $h_i$ is the home-dwell-time of node $i$ (measured in minutes), $k_{ia}$ is the number of trips from node $i$ to node $a$. We divide $h_i$ by $1440$ because the latter is the total number of minutes in a day. 
	
	\medskip 
	
	\textbf{Initial susceptible rate $s(t_0)$.} For the network SIR model and COVID-19 model, we need the initial susceptible rate $s(t_0)$ to derive the optimal lockdown rate. For each node of our network, we can get the cumulative confirmed cases on each day from the dataset provide in \cite{susceptible}. However, recent studies \cite{nishiura2020estimation} suggest that many people are infected with COVID-19 but not showing symptoms. To account for this, we divide the cumulative confirmed cases by a reporting rate to estimate the actual total infections. For the reporting rate, we use the estimate $0.14$ provided in recent literature \cite{birge2020controlling} \cite{hortaccsu2020estimating}. Then we obtain $s_i(t_0) = 1 - \frac{I_i(t_0)}{0.14N_i}$, where $I_i(t_0)$ represents the number of the cumulative confirmed cases of node $i$ at time point $t_0$ and, as before,  $N_i$ is the number of the population of node $i$.
	
	\medskip
	
	\textbf{Initial recovery rate and initial symptom rate.} To estimate the rate of active cases and the rate of cumulative cases for SIS model and SIR model, we need initialize the recovery rate. For COVID-19 model, besides the recovery rate, we also need initialize the rate of symptomatic individuals and the rate of asymptomatic individuals. Since we do not consider individuals who died for the pandemic in all three models, in our simulation, we simply regard these people as recovered patients. We get the number of cumulative death cases in each county of New York State on April 1st from New York Times \cite{death_rate}. For the truly recovered people of the pandemic, unfortunately, we can not find specific numbers for each county in New York State. However, we learn from \cite{total_recovery} that the total number of recovered cases in USA on April 1st, 2020 is 8878, and the total number of cumulative cases in USA on April 1st, 2020 is 215215. Therefore, we initialize the recovery rate of county $i$ as: $(D_i(t_0) + I_i(t_0) \frac{8878}{215215} )/(0.14 N_i)$, where $D_i(t_0)$ represents the number of cumulative death cases in county $i$ at time point $t_0$, we divide 0.14 as we also need consider the reporting rate. For the initial symptom rate, we use the estimation in literature \cite{birge2020controlling} \cite{hortaccsu2020estimating} again, we assume $86\%$ of active cases are asymptomatic individuals, and the remaining active cases are symptomatic individuals. The number of active cases are obtained as the difference of the cumulative cases and the recovered cases.
	
	\section{Numerical Calculations}\label{exp: extended}

	We now describe the details of the synthetic experiments we have performed. 
	Because in the empirical data all the variables of interest co-vary together, synthetic experiments are necessary to dis-entangle the effect of aspects of graph variation.
	
	\subsection{Implementation Details of Figure \ref{Fig: framework}e}\label{sec: implementation} 
	To clearly demonstrate the optimal lockdown issues we consider and compare our non-uniform lockdown policy and the uniform lockdown policy, we implemented a synthetic experiment on a simple three-nodes network as shown in Figure \ref{Fig: framework}b. This specific network consists of a city (A) with large population, and two suburbs (B, C) with small population (the population as well as other data we used are given in Supplementary Table \ref{tab: figure 1(d) data}). We will assume employment is proportional to the population so that the economic cost coefficient $c_i$ is also proportional to the population. Moreover, we assume there exists no direct trips between location B and location C. For the choice of the travel rate matrix $\tau$, we choose a matrix that is similar to the New York State data, but with rounder numbers; specifically,  we define $\tau_{ij} = (1- \frac{h_i}{1440})\cdot \frac{k_{ij}}{\sum_a k_{ia}}$, where $h_i$ is the home-dwell-time of node $i$, $k_{ia}$ is the number of trips from node $i$ to node $a$, and we let 
	\[k = \left[\begin{matrix} 8000 & 1000 & 2000\\ 2000 & 8500 & 0\\
		1500 & 0 & 8000\end{matrix} \right ],\]
	$h_i$ is given in Supplementary Table \ref{tab: figure 1(d) data}. The initial susceptible rates are chosen as $s(t_0) = [0.90, ~ 0.92, ~ 0.95]$, so that the epidemic mainly localizes at the city. The initial recovery rate and initial symptom rate are chosen as in \ssref{data_sources}, the detailed numbers are given in Supplementary Table \ref{tab: figure 1(d) data}. The disease parameters are taken from \cite{bertozzi2020challenges} where considers a SEIR model; we set the missing parameters as $\hat{\alpha} = 0.6754$ (where $\beta^{\text a} = \hat{\alpha} \beta^{\text s}$) and $r^{\text a} = r^{\text s} = \gamma$. We still choose the target  decay rate $\alpha = 0.0231$ that corresponds to halving every 30 days.
	
	It can be seen from Figure \ref{Fig: framework}e that our policy outperforms the uniform lockdown policy in all of the three locations. Besides, we can see that our optimal lockdown policy tends to shutdown the city more stringently than the suburbs even if the epidemic mainly happens at the city. We will discuss the details about this counter-intuitive phenomenon in Section \ref{applications}.

	\medskip
	
	\subsection{Disease Parameter Sensitivity Analysis} To understand the impact of various disease parameters, we implement sensitivity experiments. In each experiment, we vary the value of one parameter while fix the values of the others. The normal values of $\gamma$, $\epsilon$, and initial growth rate are chosen as in  \cite{bertozzi2020challenges}, and the normal value $\hat{\alpha}$ is chosen as in \cite{giordano2020modelling} ($\hat{\alpha}$ is not provide in \cite{bertozzi2020challenges}). When analyze the impact of $\gamma$, we let the decay rate $\alpha = 0.02$. In other cases, we let the decay rate $\alpha = 0.2 r^{\text s} = 0.04$, since it is impossible to achieve a decay rate better than $\gamma$, so we make it our goal to reach halfway there.   For simplicity, we assume $\gamma = r^{\text a}$, and $r^{\text s} = \hat{\gamma} r^{\text a}$, where the normal value of $\hat{\gamma} = 1$. To compare the economic cost of the obtained optimal lockdown policy and the uniform lockdown (decay matching) policy for each scenario. We define {\em efficiency} as the economic cost ratio of the optimal lockdown policy to the best uniform lockdown policy, i.e., 
	\[ 
	{\rm efficiency} = \frac{\sum_i c_i \left(\frac{1}{z_i^*} - 1\right)}{\sum_i c_i\left(\frac{1}{z_{\rm uni}^*} - 1\right)},\] where $z_{\rm uni}*$ represents the uniform lockdown rate. Here, the uniform lockdown (decay matching) implies the uniform policy which decays at a rate greater than equal to $\alpha$.  
	{\color{black}SI Fig} \sfref{Fig: sensitivity} shows the corresponding experimental results.

	Observe each column of {\color{black}SI Fig} \sfref{Fig: sensitivity}, it can be seen that the value of the optimal lockdown rate $z_i^*$ and the value of the efficiency is inversely proportional to each other. 
	In other words, the higher the level of the allowed economic  activities, the more effective the optimal lockdown policy (compare to the best uniform lockdown). Besides, it can be observed from {\color{black}SI Fig} \sfref{gamma_z}, \sfref{gamma_c} that the recovery rate $\gamma$ is the parameter that has the most significant impact on the value of $z_i^*$. In particular, when $\gamma$ ranges from $0.03$ to $0.1$, 
	the value of $z_i^*$ for all the three models increases more than $0.5$. Therefore, if we can make the patients recover faster, we can send the number of infections to 0 quickly while maintaining a high level of economic  activity. $\hat{\gamma}$ decides the value of the recovery rate of symptomatic individuals in the COVID-19 model, hence
	we allow a higher level of economic  activities when the value of $\hat{\gamma}$ increases, this can be seen in {\color{black}SI Fig} \sfref{gamma_hat_z}. Moreover, the value of the initial growth rate is closely related to the value of the transmission rate $\beta^{\text s}$, the larger the initial growth rate is, the larger $\beta^{\text s}$ will be for each model. As a consequence, we have to maintain a lower level of economic  activities to make sure the infections go to 0 quickly ({\color{black}SI Fig} \sfref{growth_z}).
	
	It can be seen from {\color{black}SI Fig} \sfref{alpha_hat_z} and \sfref{epsilon_z} that the value of $z_i^*$ is not quite sensitive to the symptom rate  $\epsilon$ and $\hat{\alpha}$. When the value of $\hat{\alpha}$ increases, the corresponding value of $z_i^*$ slightly increases. The reason may be that the increasing of $\hat{\alpha}$ can lead to the increasing of $\lambda_{\rm max}(M(t))$. As we fix the value of the initial growth rate $\lambda_{\rm max}(M(t_0))$, the obtained value of transmission rate $\beta^{\text s}$ can be smaller, which means we can allow a higher level of economic  activities.


	\medskip

	\subsection{Other Parameter Analysis.} Except the disease parameters, the structure of the data may also have a great impact on the value of the generated optimal lockdown rate $z_i^*$. To clearly investigate these relationships, we depict the optimal $z_i^*$ with respect to the centrality, the home-stay rate (i.e., daily home-dwell-time in minutes/total number of minutes in a day), the population and the employment for each county in {\color{black}SI Fig} \sfref{Fig: 2020_parameters}. 
	It can be observed that the value of $z_i^*$ for all three models increases when the home-stay rate increases. However, we do not observe any direct relationship between centrality, population, employment, and the value of $z_i^*$. To further verify this observation, we implemented random permutation experiments.

	\medskip
	
	\subsection{Random Permutation Analysis.} To fully understand the relationship of the
	value of $z_i^*$ and the data structure, we implemented random permutation experiments with respect to degree, home-stay rate, population, employment and initial susceptible rate. In a random permutation experiment, we randomly permute one parameter while keep all the others fixed, then we fit the models to the randomly permuted data and apply the proposed algorithms get the optimal lockdown rate. We repeat this process 100 times for each parameter, and compare the average values of the lockdown rate on permuted data with the original data.
	The disease parameters are chosen as in \cite{bertozzi2020challenges}. In order to do a random permutation of the degree of the nodes, we fix the network, and randomly permute all the other parameters.  We implemented such experiments based on data on different time points. The experimental results are reported in  {\color{black}SI Fig} \sfref{Fig: rand_perm1} and {\color{black}SI Fig} \sfref{Fig: rand_perm2}.
	It can be observed that the shapes of the histogram after randomly permuting the employment and initial susceptible rate are similar to the shape of the histogram of the original data. This implies the value of $z_i^*$ does not have a close relationship with employment and susceptible rate. However, the shapes of the histogram after randomly permuting the degree, home-stay rate, and population are quite different from the original data. In other words, the centrality of the node, home-stay rate, and population have a close relationship with the value of $z_i^*$. To find out what kind of relationship between them, we implemented further experiments on synthetic data.

	\medskip
	
	\subsection{Impact of Centrality.}\label{sec: centrality}
	To study the impact of centrality, we generated some geometric random graphs and then added additional ``hotspots'' with high degree. This was done by randomly choosing several nodes in the initially generated geometric random graphs, and letting the edges leading from these nodes to other nodes be present with probability 0.9.  
	
	We let the remaining parameters  be identical for all the nodes. 
	Specifically, we let the population of all the nodes be $4000$, the home-stay rate of all the nodes be 0.8, and we assume the number of employees of each node is proportional to the population. Since the SIS model and the SIR model will be identical if we let the initial susceptible rate of each node be the same constant (we choose $\beta^{\text s}$ to make the initial growth rate of each model be equal). Besides, according to our previous analysis, when the values of the initial susceptible rate of the nodes are close, it will not impact the value of the optimal lockdown rate of each node too much. As a consequence, we choose initial susceptible rate from interval $[0.8, 0.9]$ uniformly at random. Moreover, according to the data \cite{travel} provided by SafeGraph, people tend to spend more time in their own counties rather than travel to other counties when they are not stay at home. To make the synthetic data  similar to the real data, we 
	update the definition of the travelling matrix $\tau$ as following:
	
	$$  \tau_{ij} =   \begin{cases}
		0.8 \cdot (1 - h'_i) \cdot \frac{A_{ij}}{\sum_{a \neq i} A_{ia}} & i\neq j\\
		0.2 h'_i & i = j
	\end{cases},
	$$
	where $h_i'$ represents the home-stay rate of node $i$. In addition, the disease parameters of this experiment are set as in \cite{bertozzi2020challenges}, and the missing parameters are set as: $\hat{\alpha} = 0.6754$ ($\beta^{\text a} = \hat{\alpha} \beta^{\text s}$), $r^{\text a} = r^{\text s} = \gamma$. Similar to the experiments on real data, we let the decay rate $\alpha = 0.2 \gamma = 0.04$ to satisfy the assumption that $\alpha \leq \gamma$. 

	The results of this experiment are presented in Figure \ref{Fig: synthetic}a-b. It can be observed that centrality only matters for the value of $z_i^*$ when there exist hotspots (i.e., highly central nodes) in all the three models. Surprisingly, beyond such hotspots, the effect of centrality is essentially nonexistent. 
	
	To further study the impact of the centrality, we generated random graphs based on Barabási–Albert model, where there exist few nodes with unusually high degree compared to other nodes in the network. By varying the the number of the nodes, the number of links and the initial seed for the B-A model, we can generate various random graphs. Besides the adjacency matrix $A$, all the other parameters about the data (population, $\tau$, employment, initial susceptible rate, home-stay-rate), the disease parameters and the decay rate $\alpha$ are all chosen as in the experiment based on geometric graph. The experimental results are presented in {\color{black}SI Fig} \sfref{Fig: centrality_B_A}. From {\color{black}SI Fig} \sfref{Fig: centrality_B_A}, we can observe the similar phenomenon: the hotspots are assigned with smaller values of $z_i^*$ by the optimal lockdown policy in all the three models, however, such effect of the centrality is nonexistent for all the other nodes.
	
	
	Moreover, we also define a kind of random graphs to provide further evidence for the relationship of the centrality and the value of $z_i^*$. To generate graphs where the degree of each node can be different and can be decided by us, we define the adjacency matrix $A$ of the random graph as following: the upper-triangular element $A_{ij}$ will be 1 with probability $p_i$, and be 0 with probability $1-p_i$, where $p$ is the given probability vector; the values of the elements in the lower-triangular matrix are equal to the symmetry of the upper-triangular matrix. In other words, the edges leading from different nodes can be present with different probabilities.
	Similarly, the disease parameters, the decay rate $\alpha$, and all the other parameters about the data
	are chosen as in the experiment based on geometric random graphs. The experimental results are presented in {\color{black}SI Fig} \sfref{Fig: centrality_rand}. From {\color{black}SI Fig} \sfref{Fig: centrality_rand}, we can also observe that centrality only matters for the value of $z_i^*$ when there exist hotspots. Except the hotspots, the effect of centrality is centrality is nonexistent.
	
	\subsection{Impact of Population.} \label{sec: population}
	In the experiments about the population, we fix all other parameters and vary the population of each node from 100 to 20000. Since the employment is proportional to the population, and the economic cost $c_i$ of each node is decided by the corresponding number of employees, the economic cost of different nodes will also be different. To make sure the centrality of all the nodes are close, we experimented with random regular graph, Erdos-Renyi random graph, and the geometric random graph. All the other data parameters, disease parameters, and the decay rate $\alpha$ are set as in the experiments about the centrality. {\color{black}The experimental results are presented in Figure \ref{Fig: synthetic}c-d}, it can be observed that nodes with small populations are assigned with smaller values of $z_i^*$, surprisingly once the population is not very small, the effect is almost nonexistent. 
	
	\subsection{Impact of Home-Stay Rate.} \label{sec: home-stay rate}
	{\color{black}To study the impact of  home-stay rate, we fixed all other model parameters and tuned the home-stay rate of the nodes.} We chose random regular graph, Erdos-Renyi random graph, geometric random graph, and the 2d grid as our network as the centrality of the nodes on these graphs are similar. The disease parameters, the dacay rate $\alpha$, and all the other data parameters are set as in the experiments about the centrality. {\color{black}The experimental results are presented in {\color{black}SI Fig} \ref{fig: home_stay_rate}. We found that $z_l^*$ increases with increasing home-stay rate, which agrees well with our intuition.}

	\section{City-Suburb Model}\label{city_suburb} We now revisit the phenomenon we have observed in our analysis of NY, which is that the optimal lockdown tends to shutdown the outside of NYC harder than the NYC itself. To isolate this phenomenon in the simplest possible setting, we implement a simple synthetic experiment of a network with two nodes: node 1 will be referred to as ``the city'' while node 2 will be referred to as ``the suburb.'' The city will have a large number of population while the suburb will have a smaller population. We will assume employment  is proportional to the population so that the economic cost coefficient $c_i$ is also proportional to the population. Then we apply the proposed algorithms to design the optimal shutdown policy to this city-suburb model.

	The disease parameters are taken from \cite{bertozzi2020challenges} where considers a SEIR model; we set the missing parameters as $\hat{\alpha} = 0.6754$ (where $\beta^{\text a} = \hat{\alpha} \beta^{\text s}$) and $r^{\text a} = r^{\text s} = \gamma$. We still choose the target  decay rate $\alpha = 0.2 r^{\text s} = 0.04$ as before. For the choice of the travel rate matrix $\tau$, we choose a matrix that is similar to the New York State data, but with rounder numbers; specifically,  we define $\tau_{ij} = (1- \frac{h_i}{1440})\cdot \frac{k_{ij}}{\sum_a k_{ia}}$, where $h_i$ is the home-dwell-time of node $i$, $k_{ia}$ is the number of trips from node $i$ to node $a$, and we let 
	\[k = \left[\begin{matrix} 8000 & 200\\
		20 & 850 \end{matrix} \right ],  \]
	\[ h = [800 \quad 800].\]
	
	We will consider three different cases:
	\begin{itemize}
		\item Case 1: population = $[20,000 \quad 2,000]$, $s(t_0) = [0.7 \quad 0.95]$. 
		
		\item Case 2: population = $[200,000 \quad 2,000]$, $s(t_0) = [0.7 \quad 0.95]$.
		
		\item Case 3: population = $[200,000 \quad 2,000]$, $s(t_0) = [0.95 \quad 0.95]$.    
	\end{itemize}
	The experimental results are presented in \stref{tab: city_suburb}. We see the same phonemon as in our New York State experiments:  the optimal lockdown policy shutdown the suburb more stringently than the city even though, in cases 1 and 2,  the epidemic is mainly localized in the city. Comparing the results of Case 1 and Case 2, we  see that this trend gets stronger when the population difference between the city and the suburb increases. 
	
	\section{Beyond Eigenvalue Bounds}\label{two-parameters}
	In this section, we  further test our finding that the optimal stabilizing shutdown using New York State data is more stringent outside NYC. Our goal is to directly compare lockdowns by comparing the total number of infections. Unfortunately, we do not know of any way to optimize lockdowns efficiently to minimize the total number of infections; indeed, this difficulty is what motivates optimization of eigenvalue bounds in the first place. 
	
	We will compare the optimal stabilizing lockdown computed by our method, which we will say has cost $c^*$,  with lockdowns of the following structure: \begin{itemize}
		\item Two-parameters lockdown: shutdown the counties in NYC by $z_1$, and shutdown the other counties in New York State by $z_2$, where $z_2 > z_1$, such that the economic cost of the shutdown is at most $c^*$.
		\item Uniform lockdown (cost matching): shutdown all the counties in New York State by $z$, such that the economic cost of the shutdown is at most $c^*$.
	\end{itemize} Because these lockdowns are characterized by, respectively, two and one parameters, they can be found via direct search (i.e., discretizing the parameters and trying every possibility). To summarize, we will perform an exhaustive search, not only over uniform lockdowns, but overall lockdowns defined by two $z_i$, one inside NYC and one outside, with the former being smaller (corresponding to a more stringent lockdown). 
	
	We will apply the disease parameters from literature \cite{birge2020controlling}, \cite{bertozzi2020challenges}, and \cite{giordano2020modelling} respectively. Similarly, we let the decay rate $\alpha = 0.2 r^{\text s}$ as before. The experimental results are presented in {\color{black}SI Fig} \sfref{Fig: cumula_0_2}. These findings further substantiate our results: our  policy outperforms the best two-parameters lockdown as well as the best uniform lockdown in all scenarios.

	We conclude with a cautionary tale about what happens when the eigenvalues are pushed too far into the left-hand plane. We have already remarked, in the main body of the paper, that this could result in poor performance as the improvement in asymptotic rate starts to come at the expense of transient performance; we next demonstrate how this phenomenon occurs in a series of charts. We experimented with decay rate $\alpha = 0.5 r^{\text s}$, the experimental results are presented in {\color{black}SI Fig} \sfref{Fig: cumula_0_5}. Note that, in two of the three charts, this corresponds to an asymptotic rate with roughly 10\% decrease in the number of cases per day. This extremely aggressive rate results in an optimal shutdown which, while attaining this rate, does not perform well. In particular, in the figures with a 10\% decrease rate, we see that our optimal stabilizing shutdown underperforms the uniform shutdown and the two-parameters shutdown with the same cost. 
	\medskip
	
	\noindent {\bf Sensitivity analysis:}
	To further study the impact of decay rate as well as various disease parameters to the total number of infections, we implement sensitivity experiments. In each experiment, we vary the value of one parameter while fix the values of the others.
	The normal values of $\gamma$, $\epsilon$, and the initial growth rate are chosen as in  \cite{bertozzi2020challenges}, and the normal value of $\hat{\alpha}$ ($\beta^{\text a} = \hat{\alpha}\beta^{\text s}$) is set as in \cite{giordano2020modelling}. We let the normal value of the decay rate $\alpha = 0.2 r^{\text s} = 0.04$, since it is impossible to achieve a decay rate better than $\gamma$. For simplicity, we assume $\gamma = r^{\text a}$, and $r^{\text s} = \hat{\gamma} r^{\text a}$, where the normal value of $\hat{\gamma} = 1$. The experimental results are presented in {\color{black}SI Fig} \sfref{Fig: cumula_sensitivity} and \sfref{Fig: Covid_sensitivity}.
	
	As expected, it can be observed from {\color{black}SI Fig} \sfref{subfig: SIS_alpha}, \sfref{subfig: SIR_alpha}, and \sfref{subfig: Covid_19_alpha} that the total number of the infections of our policy is much better for small $\alpha$; but if we increase $\alpha$ too much, this starts to change and we no longer outperform. Additionally, it can be seen from {\color{black}SI Fig} \sfref{subfig: SIS_growth_rate}, \sfref{sufig: SIR_growth_rate}, and \sfref{subfig: Covid_19_growth_rate} the total number of infections of all the polices increases as the initial growth rate increases, again as expected; however, somewhat surprisingly, we can see from {\color{black}SI Fig} \sfref{subfig: SIS_gamma}, \sfref{subfig: SIR_gamma}, and \sfref{subfig: Covid_19_gamma} that the total number of infections increases when the recovery rate $\gamma$ increases. The reason for this counter-intuitive phenomenon is that we assume the initial growth rate is fixed in this experiment, so that an increase in the recovery rate $\gamma$ means an increase in the transmission rate $\beta^{\text s}$ to obtain the same initial growth rate, and the latter effect increases  the total number of infections.
	
	For COVID-19 model, besides $\alpha$, $\gamma$, and the initial growth rate, we also studied the impact of parameters $\epsilon$, $\hat{\alpha}$, and $\hat{\gamma}$. We can observe from {\color{black}SI Fig} \sfref{subfig: Covid_19_epsilon} that the total number of the infections is not sensitive to parameter $\epsilon$ for all the three polices. As the transmission rate of asymptomatic individuals $\beta^{\text a} = \hat{\alpha} \beta^{\text s}$, then the total number of infections increases as $\hat{\alpha}$ grows ({\color{black}SI Fig} \sfref{subfig: Covid_19_alpha_hat}). Parameter $\hat{\gamma}$ is closely related to the recovery rate of the symptomatic individuals, its effect is similar to $\gamma$, when the value of $\hat{\gamma}$ increases, the transmission rate $\beta^{\text s}$ increases and further lead to the increasing of the total number of the infections ({\color{black}SI Fig} \sfref{subfig: Covid_19_gamma_hat}).
	
	\section{Why Is The Shutdown Outside NYC Harder? An Intuitive Explanation} \label{sec: counter-intuitive}
	Our main finding in the empirical experiments -- that an optimal stabilizing lockdown will shut down the outside of NYC more stringently than NYC -- is counter-intuitive. At least part of the reason why this happens is that our costs are taken to be proportional to employment: we do not attempt to take into account either the larger salaries of workers in NYC or the importance of NYC to the national economy. Nevertheless, the finding remains counter-intuitive, and we are not aware of any previous literature pointing out that something like this can occur.

	We now provide an explanation which makes this finding more intuitive. Let us consider an idealized city-suburb model with two nodes. Imagine now that the coefficients $c_i$ are proportional to population, and imagine that, initially, there is no travel between city and suburb. For clarity, let us consider a related model where, instead of minimizing cost subject to a constraint on how fast infections decay, we instead put a price on each infection; this is closely related to what we do and slightly simpler to reason about.  In that case, the optimal shutdown should be invariant to scaling in population since doubling the population of a city will both double the cost of a shutdown as well as double the benefit in terms of number of infections averted. 
	
	Now let us consider further what happens when we change the system by stipulating that: (1) $1\%$ of trips now happen between city and suburb (2) the population $N$ of the city goes to infinity while the population of the suburb remains fixed at, say, $1000$. 
	
	In this case, whenever we shut down the suburb, the cost scales with the population $1000$; while the benefit also scales with $N$, the population of the city, since any shutdown in the suburb decreases the total number of infections in the city (since the epidemic can spread in the city through interactions with the suburb). Thus, even though only $1\%$ of the trips go between city and suburb in each direction, the benefit of shutting down the suburb will overwhelm the cost as $N \rightarrow +\infty$. The same argument does not apply to the city: the increased benefits from reducing the number of infections in the suburb is small relative to the cost which scales with $N$.  This provides an explanation why an optimal shutdown might choose to be more stringent in the suburb. 
	
	Of course, this is an idealized example. We note that in our empirical results, costs were taken to be proportional to employment, not population; so that we will never encounter a situation where the costs of shutting down a suburb do not scale with the size of the city (because if $1\%$ of the trips go between city and suburb, the suburb will also have employment that will scale with $N$ to provide services to visitors). Nevertheless, our empirical results suggest that this idealized example is not too far from what happens once we use real data on number of trips and employment in NYC and suburbs, as the optimal stabilizing shutdown does choose to be more stringent outside NYC. 
	
	\section{Extended Lockdown Cost Functions} \label{sec: extended economic cost functions}
	In this section, we explore the optimal lockdown  with additional  cost functions. Our goal is to show the counter-intuitive phenomenon we observed earlier -- that the optimal lockdown tends to shutdown the outside of NYC harder than the NYC itself -- is not due to the particular choice of the economic cost functions in Eq. \eqref{eq:lockdown_cost} in the main text. The choice of the new economic cost functions follow the similar rules in \eqref{eq:lockdown_cost}. 
	
	We thus consider other cost functions. First, we consider 
	\begin{align}\label{eq: cube}
		c(z_1, \ldots, z_n) = \sum_i^n c_i (\frac{1}{z_i^3} - 1).
	\end{align}
	By using the similar method as in \ssref{sec: lock down}, we can write the optimal lockdown issue with economic cost \eqref{eq: cube} as a constrained convex optimization problem as following
	\begin{eqnarray} 
		&& 
		\min \sum_{i=1}^n c_i' u_i^3 \nonumber \\
		&& \sum_{i=1}^n u_i B_i \succeq \left(
		\begin{array}{cc} 
			Q & 0 \\ 
			0 & 0 
		\end{array}
		\right) 
		+ 
		\left(
		\begin{array}{cc} 
			0 & 0 \\ 
			0 & q {\rm diag}( (D_1)_{11}^{-1}, \ldots,  (D_1)_{nn}^{-1}) 
		\end{array}
		\right). \label{eq: sdp with cube}
	\end{eqnarray}
	The constraints of \eqref{eq: sdp with cube} is convex SDP constraints, and the cost function is also convex. Hence, we can use the projected gradient descent (PGD) alg{}orithm to obtain the optimal solution of this problem. Each step of the projection involves solving a semi-definite program.
	
	Besides the cost \eqref{eq: cube}, we also consider the following economic cost functions:
	\begin{align*}
		c(z_1, \ldots, z_n) & = \sum_i^n c_i \left(\frac{1}{z_i^{1.5}}  - 1 \right), \\
		c(z_1, \ldots, z_n) & = \sum_i^n c_i \left(\frac{1}{z_i^2} - 1 \right), \\
		c(z_1, \ldots, z_n) & = \sum_i^n c_i \left(\min(\frac{1}{z_i}, 10) - 1\right), \\
		c(z_1, \ldots, z_n) & = \sum_i^n c_i \left(\min\left(\frac{1}{z_i}, 20\right) - 1\right), \\
		c(z_1, \ldots, z_n) & = \sum_i^n c_i \left(\min \left(\frac{1}{z_i}, 100 \right) - 1 \right). 
	\end{align*}
	Similarly, we can also write the optimal lockdown issues with these costs as the convex optimization problem with convex SDP constraints, which can be addressed by the PGD algorithm.
	
	To check if we can still observe the similar counter-intuitive phenomenon, we experimented with the data about COVID-19 break in NYS with these extended cost functions. In our simulations, we employed the toolbox from \cite{Lofberg2004} to solve the projection step in each step of PGD problem.
	The disease parameters are set as in \cite{giordano2020modelling}. The decay rate is chosen as $\alpha = 0.2 r^{\text s} = 0.0034$ so that $\alpha < \min (r^{\text a}, r^{\text s})$. The experimental results are presented in {\color{black}SI Fig} \ref{Fig: polynomial costs} and {\color{black}SI Fig} \ref{Fig: min costs}. 
	
	It can still observed that the value of $z_i^*$ for counties in NYC are relatively higher than other counties in NYS in each scenario, which implies we should shutdown the outside of NYC harder than itself. 
	Additionally, the median of $z_i^*$ is greater than or close to the value of the uniform lockdown in all the cases. In other words, compared to the uniform lockdown, the optimal lockdown policy not only leads to less economic losses but also allows the majority of the counties to have more economic activity.
	
	\section{Robustness Check  of Travel Rate Matrix $\tau$ } \label{Sec: Robustness}
	In this section, we check the robustness of one of our main findings, that the optimal lockdown policy should shutdown the outside of  NYC harder than itself, with respect to the travel rate matrix $\tau$. In our simulations of NY, we used the data from SafeGraph \cite{travel} to construct the travel rate matrix $\tau$. The data was collected by using a panel of GPS pings from anonymous mobile devices, and might be inaccurate due to various practical reasons. For instance, we do not know from how many of the minutes recorded as travel outside might have been spent alone in a vehicle, with no possibility for transmission to other individuals. To test the robustness of our finding against the uncertainty of the travel rates, we introduce two types of perturbation to the travel rate matrix $\tau$ by either adding noise to its entries or removing a fraction of its entries. Then we check if the optimal lockdown policy still suggests that we should shutdown the outside of  NYC harder. In both scenarios, we perturb the travel rates of each county while keep the home-stay-rate of each county be the same.
	
	In the first scenario, we add Gaussian noises to each entry of the matrix $\tau$. To do this, for each $\tau_{ij}$, we generate Gaussian noise $g_{ij}$ with mean $0$ and variance $\theta \tau_{ij}^2$, then we set
	\[\tau_{ij}^{\rm temp} = \max(\tau_{ij} + g_{ij}, 0).\]
	Next, we let
	\[ \tau_{ij}^{\rm new} = \beta_i \tau_{ij}^{\rm temp},\]
	where $\beta_i$ is chosen to make $\sum_j \tau_{ij}^{\rm new} = \sum_j \tau_{ij}$. In this case, the size of the noises will be proportional to $\tau_{ij}$, and we can tune the size of the noises by varying the parameter $\theta$. Meanwhile, recall the definition of the travel rate matrix \eqref{eq: definition of tau},  we still have $\sum_j \tau_{ij}^{\rm new} = \sum_j \tau_{ij} =  1- \frac{h_i}{1440} $, which means the home-stay-rate of county $i$ keeps the same.
	
	In our second scenario, we remove a $p$-fraction of entries from each row of matrix $k$ uniformly at random, and replace these values by $0$, where $k_{ij}$ is the daily number of trips from county $i$ to county $j$. This is to mimic missing data. We call the updated travel natrix as $k^{\rm new}$, then we generate $\tau$ with our original method, 
	\[\tau_{ij}^{\rm new} = (1- \frac{h_i}{1440})\cdot \frac{k_{ij}^{\rm new}}{\sum_a k_{ia}^{\rm new}}.\] 
	Similarly, we can tune the size of the missing data by simply varying $p$.
	
	We implemented the two different perturbation scenarios with the data about COVID-19 break in NY. The disease parameters are set as in \cite{birge2020controlling}, the decay rate $\alpha$ is chosen $0.0231$ that corresponds to halving every 30 days.  The simulation results are presented in {\color{black}SI Fig} \ref{robustness_check}.
	It can be observed that the average of $z_i^*$ for NYC cities is greater than the average of other counties for a wide range of the noise level or the fraction of missing data in the travel rate matrix. Thus our finding that the optimal lockdown should shutdown NYC harder is robust against considerable uncertainty in the travel rate matrix $\tau$.

	\section{Nonuniform Transmission Rate $\beta$: Potential Urban-Rural Differentials} \label{sec-nonuniform}
	 \aoa{We now consider the possibility that the spread of an epidemic, as measured by the parameters $\beta^{\rm s}$ and $\beta^{\rm a}$ depends on the location. Specifically, we take a simple model where transmission is proportional to a power $h$ of the {\em population density} of each county. }

	 \aoa{There are a number of reasons why that might be the case. In an epidemic where the primary mode of transmission is through outside contact, the transmission speed will naturally scale with population density. Additionally, higher density counties will on average have more extensive public transport, which could also 
	 facilitate transmission even in the presence of masking guidelines. {\em However, for COVID-19 in New York State, it is not clear that COVID-19 transmission substantially increases with density.} }
	 
	 \aoa{The data can be seen  in {\color{black}SI Fig} \ref{fig: density and Rt}. The bottom two graphs of that figure show population density vs $R_t$ (average and range). We do not in general see 
	 an upward curve. The first graph of that Figure presents a plot of $R_t$ for New York Country (approximately $70,000$ people per square mile) and Hamilton County (approximately $2.6$ per square mile); these are, respectively, the most and least dense counties in New York. Although $R_t$ fluctuates over time for 
	 both counties, {\color{black}SI Fig} \ref{fig: density and Rt} does not show any persistent differences, in spite of a $\times 30,000$ density differential. This rules out the possibility of a strong density dependence, e.g., $\beta$ being proportional to density, though it leaves open the possibility that population density could have some slighter influence which is
	 counteracted by the differences in other characteristics between these counties.  }
	 	 
	 \aoa{We thus consider what happens if we introduce some scaling with a power $h$ of population density into our models; we stress that $h$ here could be quite close to zero.  It is clear that, as we increase $h$, at some point our finding that NYC should be shut down 
	 less harshly than the rest of NYS will reverse: as infections comparatively spread faster and faster in NYC, the NYC shutdown should get more stringent, at some point becoming {\em more} stringent that for the remainder of the state.  We want to see how big $h$ has to be for this reversal to occur. As before, will find it helpful to frame this $h$ in terms of the differential between the most and least populated
	 counties. Thus we seek to answer the following question: {\em how much faster does the infection need to spread in New York County relative to Hamilton County for this reversal to occur?} }
	 	 
	 \aoa{More formally, suppose $p_l$ is the population density of location $l$.
	For COVID-19 model, we let the transmission rates associated with location $l$ be
	\begin{align*} \beta^{\rm s}_l & = k p_l^{h} \\ \quad\beta^{\rm a}_l & = \hat{\alpha}\beta^{\rm s}_l,
	\end{align*}
	where, recall,  $\hat{\alpha}$ is a fixed scalar which measures the transmission rate ratio between symptomatic individuals and asymptomatic individuals. The parameter $k$ is chosen to match the initial growth rate of the available COVID-19 models, just as before. In this case, we can write $A_0$ in Eq. \eqref{eq: A0 definition} as
\[A_0 := \left(
\begin{array}{cc} 
	\hat{\alpha}k{\rm diag}(p^{h}) {\rm diag}(s(t_0)) A_z - (\epsilon + r^{\text a}) +\alpha & k{\rm diag}(p^{h}){\rm diag}(s(t_0)) A_z \\
	\epsilon & - r^{\text s} + \alpha
\end{array}
\right).\]
We can then apply our approach to get the optimal stabilizing lockdown.
We do this, getting the population density data from \cite{densitySource}. The simulation results are presented in {\color{black}SI Fig} \ref{fig: density}. It can be observed that the value $z_i^*$ of NYC counties still larger than the other counties when $h\leq 0.1$; however, this phenomenon becomes
unintelligible by $h \geq 0.15$. Coming back to the New York County/Hamilton County divide, we find that our results still hold if $\beta$ corresponding to the former is roughly $2.8$ times the $\beta$ corresponding to the latter.} 

	\section{Impact of The Cases Decline Speed} \label{sec: two parameters compare}	

	In this section, we consider the cases where the lockdown sends the confirmed cases to zero with different rates. Our goal is to show the counter-intuitive phenomenon we observed earlier- that the optimal lockdown tends to shutdown the outside of NYC harder than NYC itself - is not due to the particular choice of the decay rate $\alpha$. In particular, we consider the cases when the confirmed cases are reduced very fast, with 10 \% (or even 20\%) daily decline speed.
	
	Suppose we aim to reduce the confirmed cases with a speed at least $a^*$ through lockdown, we will compare the optimal stabilizing lockdown computed by our method with the following two simple lockdown policies:
	 \begin{itemize}
		\item Best two-parameters (NYC) lockdown: shutdown the counties in NYC by $z_1$, and shutdown the other counties in New York State by $z_2$, where $z_2 > z_1$, such that the daily decline speed of the COVID-19 model is at least $a^*$ and the corresponding economic cost is minimal.
		\item Best two-parameters (outside) lockdown: shutdown the counties in NYC by $z_1$, and shutdown the other counties in New York State by $z_2$, where $z_1 < z_2$, such that the daily decline speed of the COVID-19 model is at least $a^*$ and the corresponding economic cost is minimal.
	\end{itemize}
	Since both of the two lockdown policies are characterized by two parameters, we will find them via grid search (i.e., discretizing the parameters and trying every possibility). To summarize, we will fix the decline speed of the epidemics and compare the economic cost of different lockdown policies.
	
	Similar to before, we will apply the disease parameters from literature \cite{bertozzi2020challenges,giordano2020modelling} and \cite{birge2020controlling} respectively. Since we need to ensure $\alpha < \min(r^{\rm a}, r^{\rm s})$, and the values of $r^{\rm a}, r^{\rm s}$ are different in these literatures, the range of $a^*$ we consider are different for different groups of parameters. The experimental results are presented in {\color{black}SI Fig} \ref{fig: two parameters compare}.  It can be observed, when the decline speed is fixed,  the best two-parameters (NYC) lockdown policy leads to more economic losses than the best two-parameterss (outside) lockdown and our method, even when the cases decline very fast. This finding provides further support for the counter-intuitive phenomenon we observed that the optimal lockdown tends to shutdown the outside of NYC harder.
	
	\section{Robustness Check: Activity Differential between the Symptomatic and Asymptomatic Individuals\label{sec:activitydecline}} 
	\aoa{In this section, we consider what happens if symptomatic and asymptomatic individuals have different activity levels. We consider a simple model where the travel rate of symptomatic individuals is a constant fraction of the asymptomatic individuals. This could happen because some symptomatic individuals may choose to isolate or quarantine themselves. }

\aoa{If the activity level of symptomatic individuals is zero, the COVID-19 model we have proposed here reduces to the standard SIR model, as people in the symptomatic class will not infect anyone. In general, however, the activity level of symptomatic individuals may not equal zero: some people may ignore regulations, others may ignore their symptoms, and still others may believe they are suffering from something other than COVID. Because we are not aware of any research allowing us to choose a specific activity reduction, we will consider a number of possibilities for how much symptomatic individuals  reduce activity relative to asymptomatic individuals, ranging from $10\%$ to $90\%$. }

\aoa{Specifically, we set
	\[ \tau_{ij}^{\rm s} = \kappa \tau_{ij}^{\rm a}.\]
	Equivalently, we may multiply our coefficient $\hat{\alpha}$ with $\frac{1}{\kappa}$ and obtain the same model, because, recall, \[\hat{\alpha} = \frac{\beta^{\rm a}}{\beta^{\rm s}},\] and we choose the value of transmission rate $\beta^{\rm s}$ to match the given initial growth rate.  In our simulations, we experimented with $k = 0.1,~ 0.5, ~ 0.9,~ 1.0$ respectively. The simulation results are presented in {\color{black}SI Fig}  \ref{fig: activity level different }. We find that the optimal stabilizing lockdowns $z_i^*$ vary when $\kappa$ takes on different values, however, the patterns of $z_i^*$ are quite similar. Specifically, we still can observe that the optimal stabilizing lockdown tends to shutdown the outside of NYC harder than itself.}

	\clearpage
	
	\section{Supplementary Figures}\label{extended_results_figs}

	\begin{figure}[!htb]
		\centering
		\begin{minipage}[b]{0.95\linewidth}\label{fig: active_SIS}
			\centering
			\includegraphics[width=1.0\linewidth]{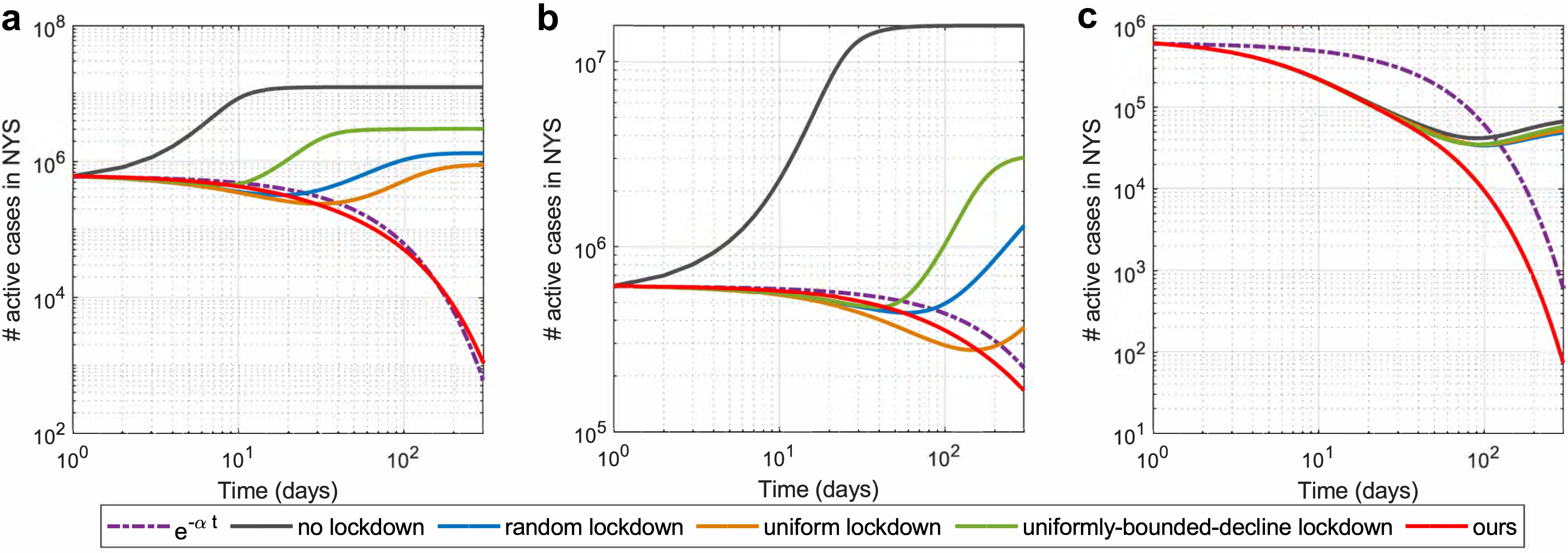}			
		\end{minipage}\hfill
		
		\vspace{3mm}
		
		\begin{minipage}[b]{0.95\linewidth}\label{fig: acc_SIS}
			\centering
			\includegraphics[width=1.0\linewidth]{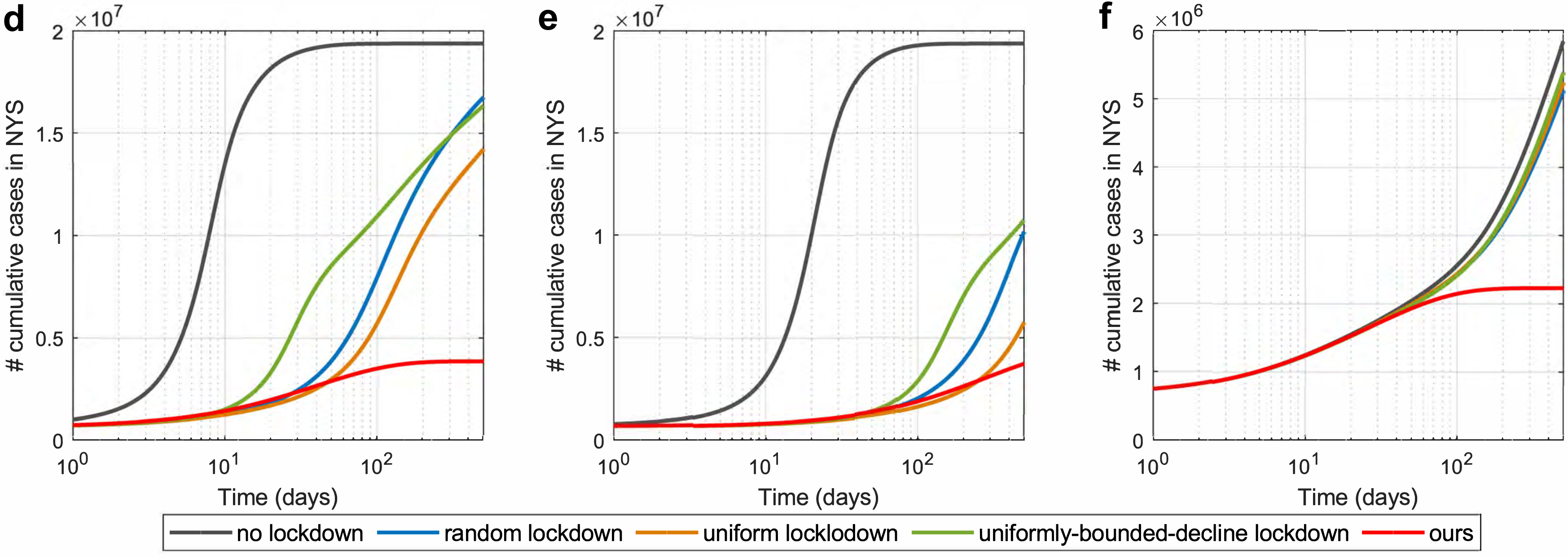}
		\end{minipage}%
		\centering
		\vspace{2mm}
		\caption{Experimental results on real data: {\bf the estimated number of active casesand cumulative cases for SIS model} by applying different lockdown policies based on available data about COVID-19 outbreak in NY on April 1st, 2020.  {\bf a-c},  the estimated number of active cases in NY {\color{black}from Apr. 1st, 2020 to Jan. 26th, 2021 (or June 10th, 2022)}. {\bf d-f}, the estimated cumulative cases in NY {\color{black}from Apr. 1st, 2020 to Aug. 14th, 2021 (or May 10th, 2024)}. In {\bf a}, {\bf d}, the disease parameters are set as in \cite{bertozzi2020challenges}, the decay rate $\alpha$ is chosen as 0.0231 which corresponds to halving every 30 days. In {\bf b}, {\bf e}, the disease parameters are set as in \cite{giordano2020modelling}, the decay rate is chosen as $\alpha = 0.2 r^{\text s} = 0.0034$ so that $\alpha < \min (r^{\text a}, r^{\text s})$. In {\bf c}, {\bf f}, the disease parameters are set as in \cite{birge2020controlling}, the decay rate $\alpha$ is chosen $0.0231$ that corresponds to halving every 30 days. Uniform lockdown, random lockdown, and uniformly-bouded-decline lockdown are defined as in Fig \ref{Fig: framework}.  It can be observed that our policy outperforms all the other lockdown polices.
		}\label{Fig: active_acc_SIS}
	\end{figure}
	
	\begin{figure}[!htb]
		\centering
		\begin{minipage}[b]{0.95\linewidth}\label{fig: active_SIR}
			\centering
			\includegraphics[width=1.0\linewidth]{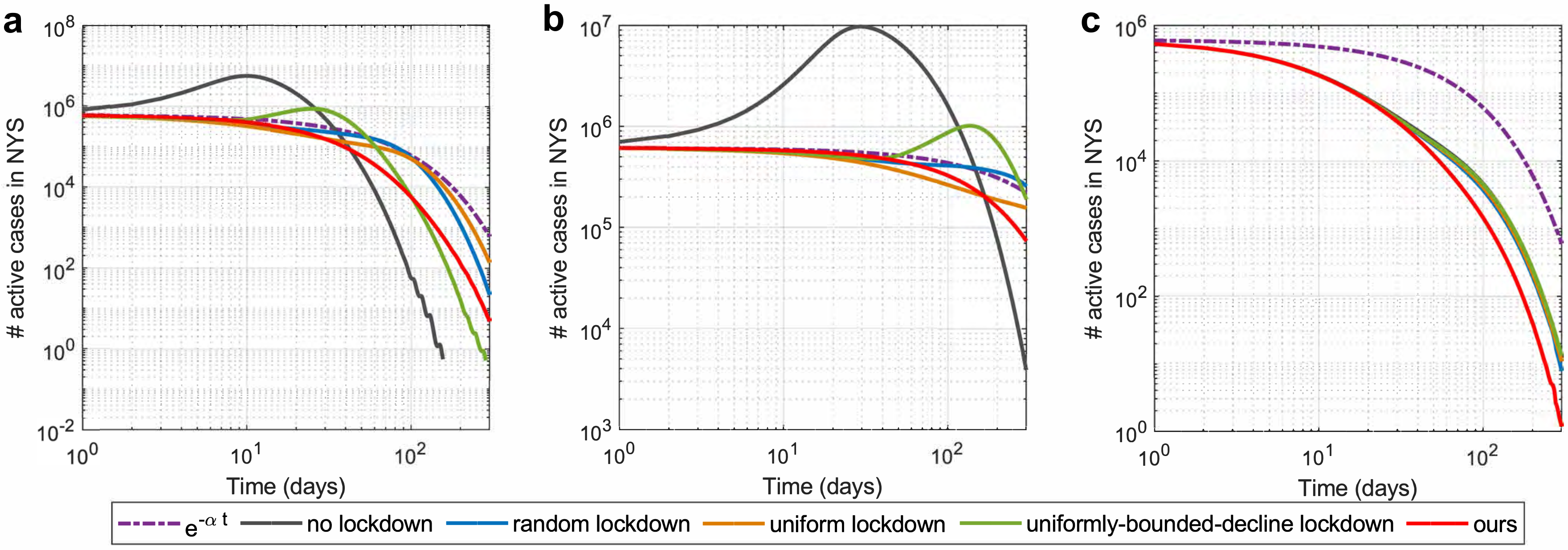}			
		\end{minipage}\hfill
		
		\vspace{3mm}
		
		\begin{minipage}[b]{0.95\linewidth}\label{fig: acc_SIR}
			\centering
			\includegraphics[width=1.0\linewidth]{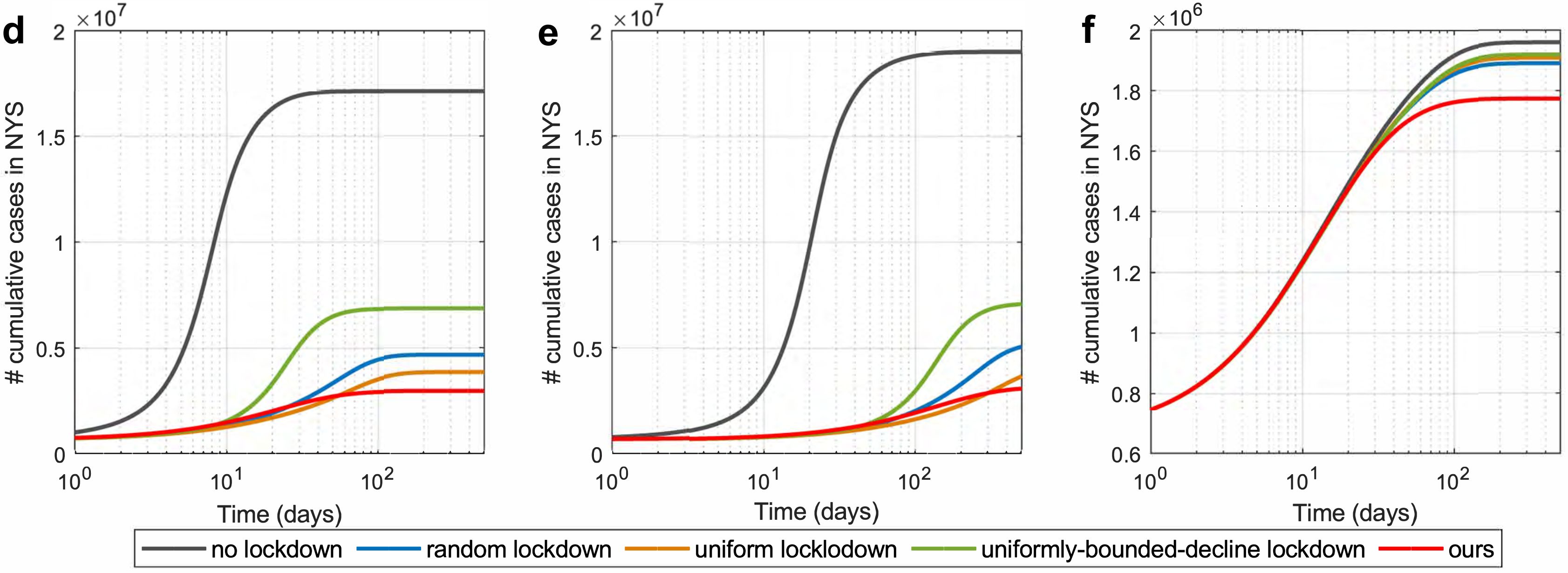}
		\end{minipage}%
		\centering
		\vspace{-1mm}
		\caption{Experimental results on real data: {\bf the estimated number of active cases and cumulative cases for SIR model} by applying different lockdown policies based on available data about COVID-19 outbreak in NY on April 1st, 2020.  {\bf a-c},  the estimated number of active cases in NY {\color{black}from Apr. 1st, 2020 to Jan. 26th, 2021 (or June 10th, 2022)}. {\bf d-f}, the estimated cumulative cases in NY {\color{black}from Apr. 1st, 2020 to Aug. 14th, 2021 (or May 10th, 2024)}. In {\bf a}, {\bf d}, the disease parameters are set as in \cite{bertozzi2020challenges}, the decay rate $\alpha$ is chosen as 0.0231 which corresponds to halving every 30 days. In {\bf b}, {\bf e}, the disease parameters are set as in \cite{giordano2020modelling}, the decay rate is chosen as $\alpha = 0.2 r^{\text s} = 0.0034$ so that $\alpha < \min (r^{\text a}, r^{\text s})$. In {\bf c}, {\bf f}, the disease parameters are set as in \cite{birge2020controlling}, the decay rate $\alpha$ is chosen $0.0231$ that corresponds to halving every 30 days. Uniform lockdown, random lockdown, and uniformly-bouded-decline lockdown are defined as in Fig \ref{Fig: framework}.  It can be observed that our policy outperforms all the other lockdown polices.
		}\label{Fig: active_acc_SIR}
	\end{figure}
	
	\clearpage

	\begin{figure}[!htb]
		\centering
		\begin{minipage}[b]{0.9\linewidth}
			\centering
			\includegraphics[width=1.0\linewidth]{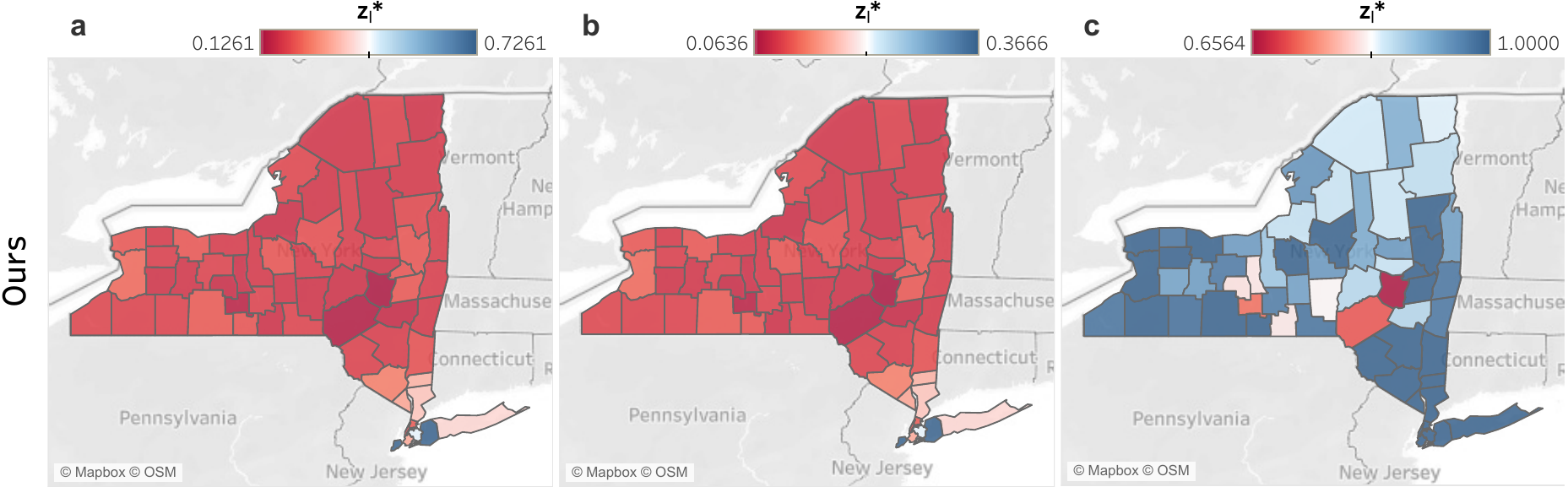}
		\end{minipage}

		\begin{minipage}[b]{0.9\linewidth}
			\centering
			\includegraphics[width=1.0\linewidth]{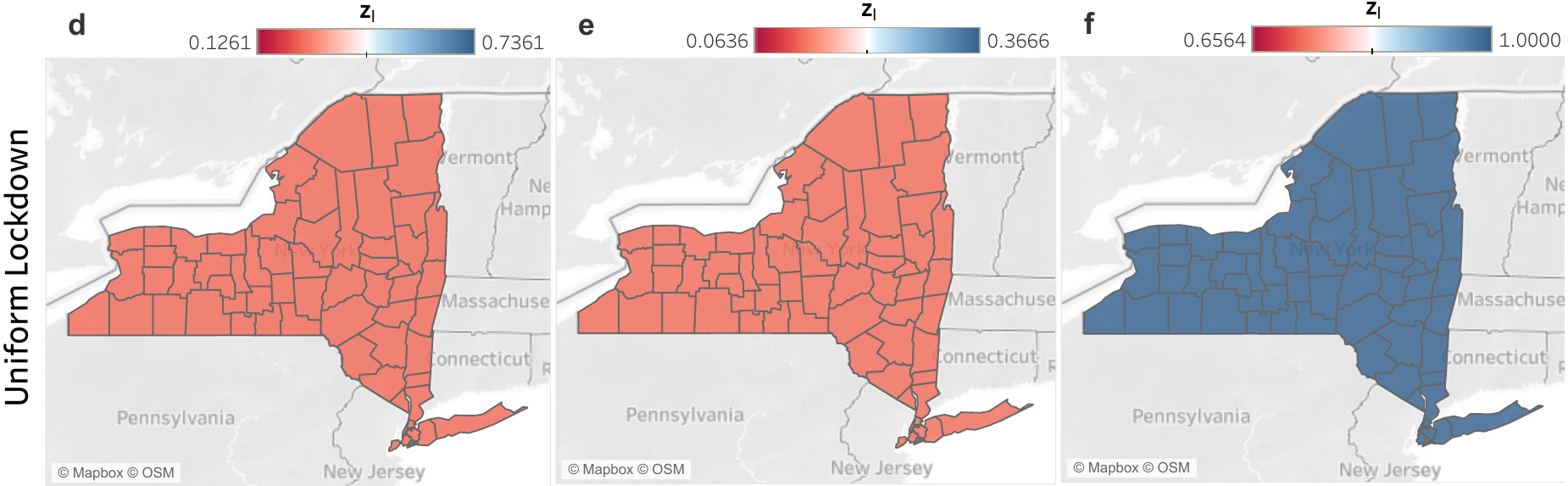}
		\end{minipage}%

		\begin{minipage}[b]{0.9\linewidth}
			\centering
			\includegraphics[width=1.0\linewidth]{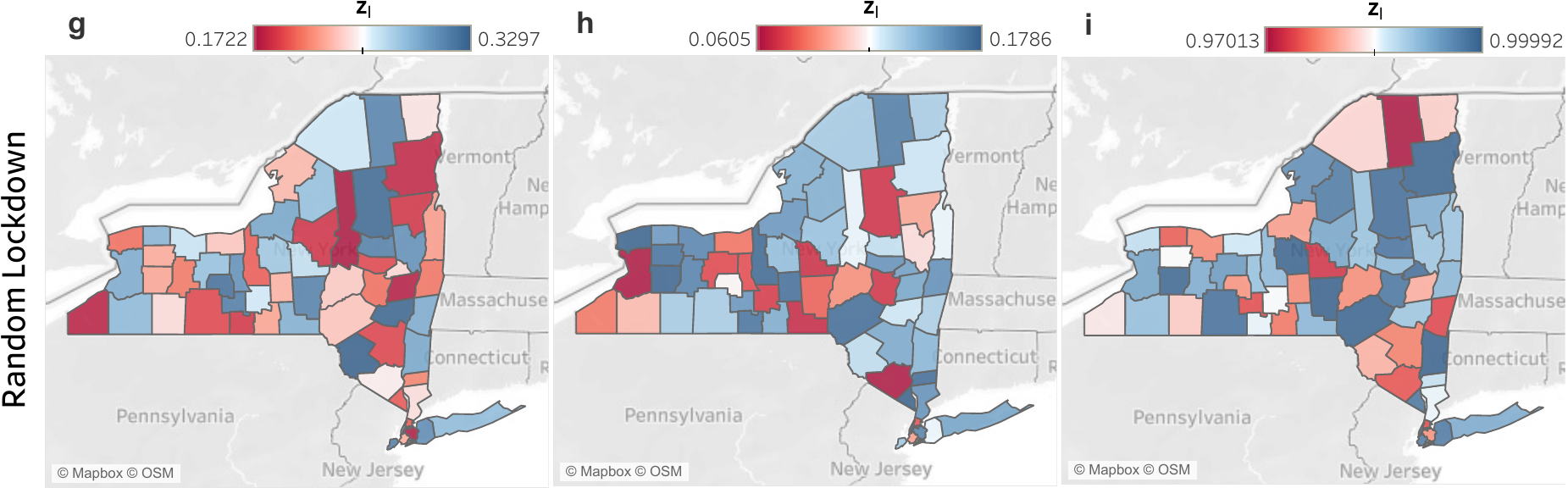}
		\end{minipage}%
		
		\begin{minipage}[b]{0.9\linewidth}
			\centering
			\includegraphics[width=1.0\linewidth]{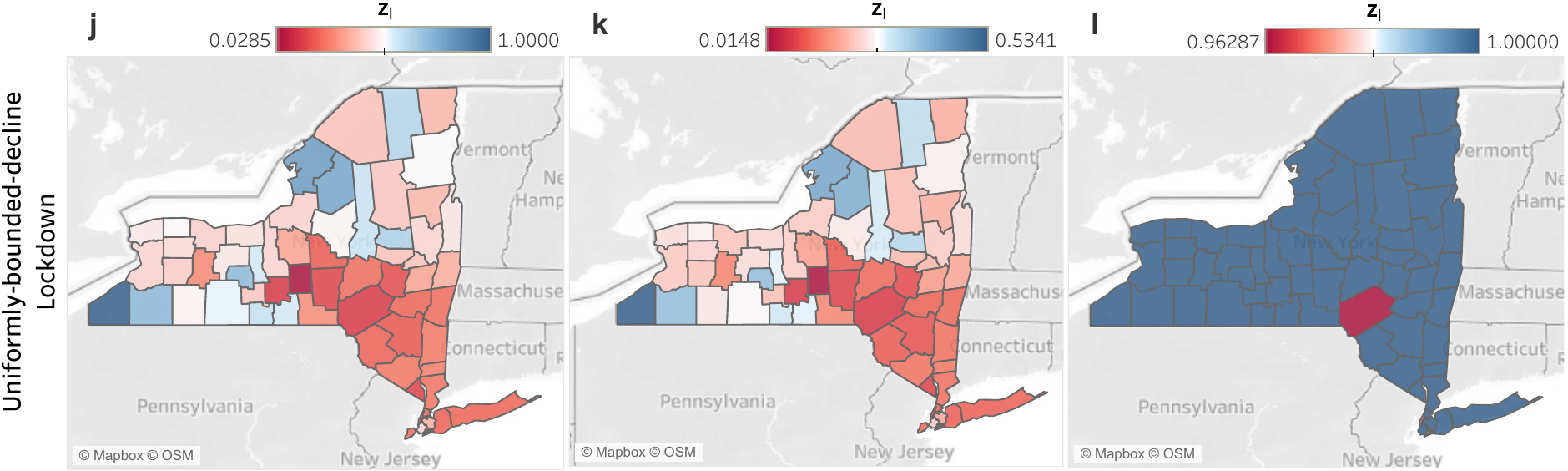}
		\end{minipage}%
		\caption{Experimental results of real data: {\bf lockdown rate of each county given by different polices for SIS model} based on available data about COVID-19 outbreak in NY on April 1st, 2020. {\bf a-c}, optimal lockdown rate $z_i^*$ given by our method . {\bf d-f}, uniform lockdown rate $z_i$. {\bf g-i}, random lockdown rate $z_i$. {\bf j-l}, uniformly-bounded-decline lockdown rate $z_i$. Uniform lockdown, random lockdown, and uniformly-bouded-decline lockdown are defined as in Fig \ref{Fig: framework}.
		In {\bf a}, {\bf d}, {\bf g}, {\bf j}, the disease parameters are set as in \cite{bertozzi2020challenges}, the decay rate $\alpha$ is chosen as 0.0231 which corresponds to halving every 30 days. In {\bf b}, {\bf e}, {\bf h}, {\bf k} the disease parameters are set as in \cite{giordano2020modelling}, the decay rate is chosen as $\alpha = 0.2 r^{\text s} = 0.0034$ so that $\alpha < \min (r^{\text a}, r^{\text s})$. In {\bf c}, {\bf f}, {\bf i}, {\bf l}, the disease parameters are set as in \cite{birge2020controlling}, the decay rate $\alpha$ is chosen $0.0231$ that corresponds to halving every 30 days.  It can observed from {\bf a-c} that the value of $z_i^*$ for counties in NYC are relatively higher than other counties in New York State, which implies we should shutdown the outside of NYC harder than itself. Besides, it can be seen that such counter-intuitive phenomenon does not appear in any other lockdown polices.
	}\label{Fig: SIS_each_county_uni_ours_decline}
	\centering
	\vspace{-1mm}	
	\end{figure}

	\begin{figure}[!htb]
		\centering
		\begin{minipage}[b]{0.9\linewidth}
			\centering
			\includegraphics[width=1.0\linewidth]{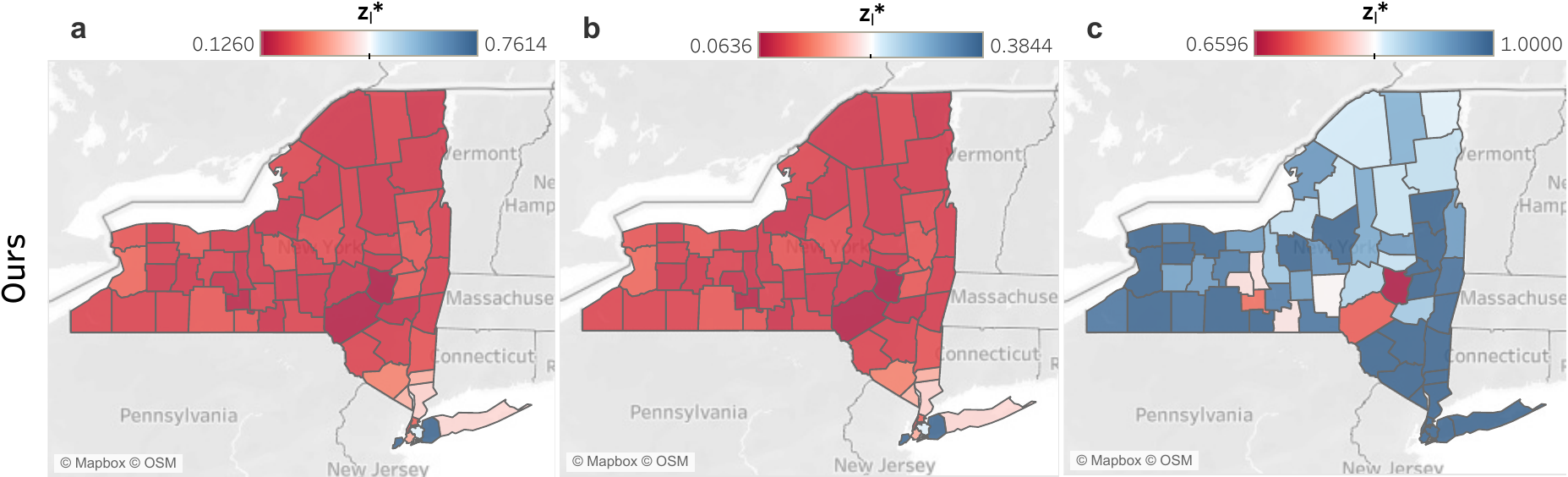}
		\end{minipage}
		
		\begin{minipage}[b]{0.9\linewidth}
			\centering
			\includegraphics[width=1.0\linewidth]{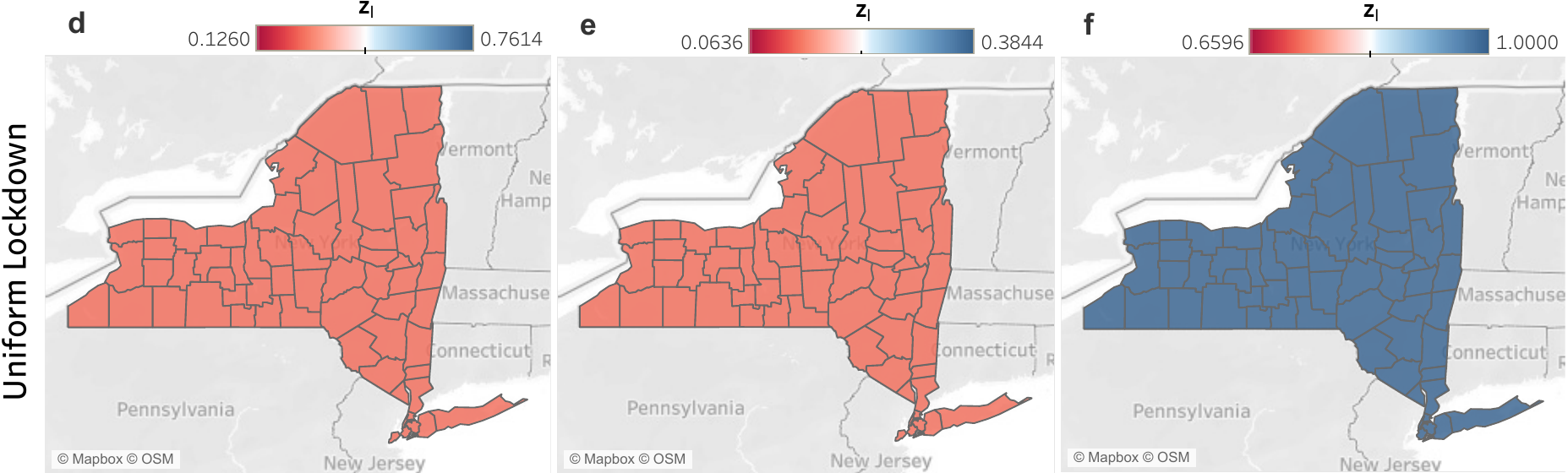}
		\end{minipage}%
		
		\begin{minipage}[b]{0.9\linewidth}
			\centering
			\includegraphics[width=1.0\linewidth]{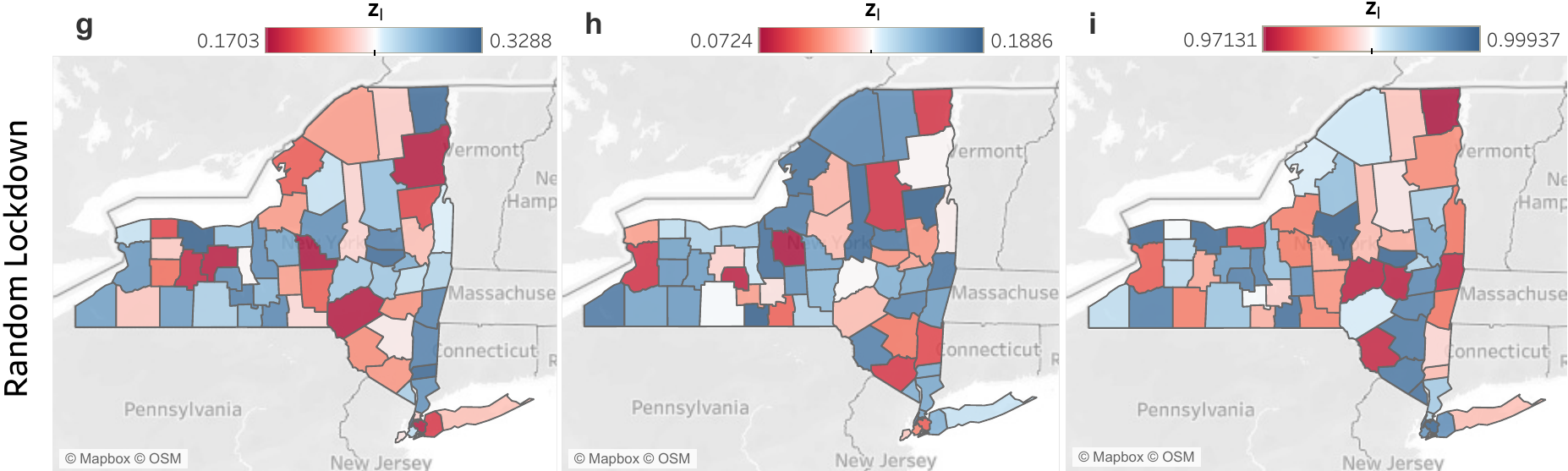}
		\end{minipage}%
		
		\begin{minipage}[b]{0.9\linewidth}
			\centering
			\includegraphics[width=1.0\linewidth]{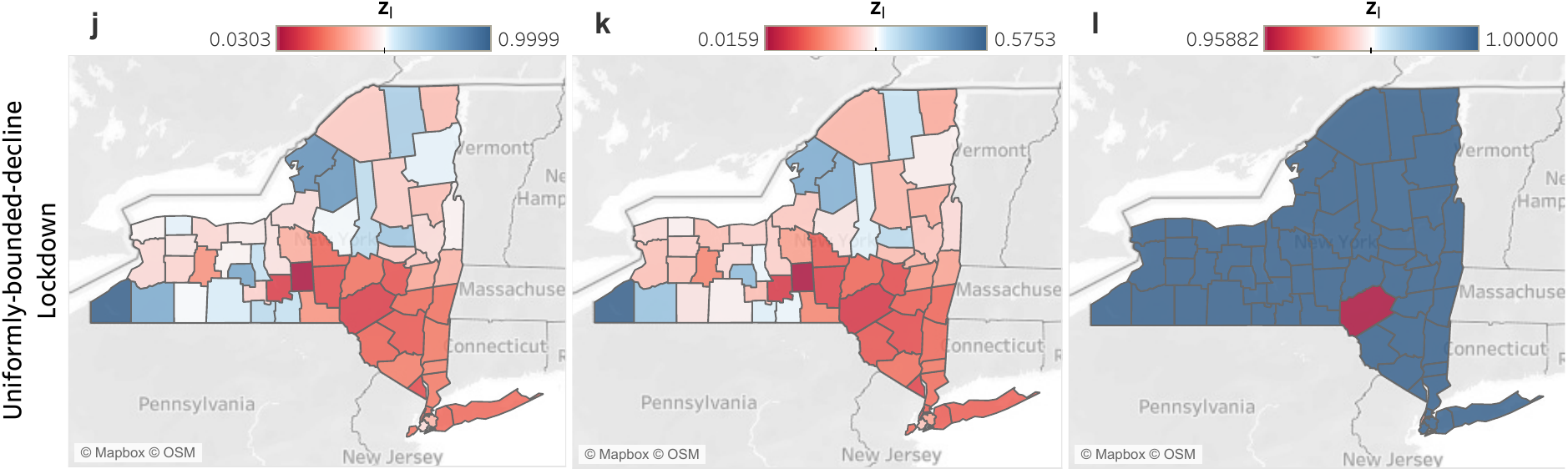}
		\end{minipage}%
		\caption{Experimental results of real data: {\bf lockdown rate of each county given by different polices for SIR model} based on available data about COVID-19 outbreak in NY on April 1st, 2020. {\bf a-c}, optimal lockdown rate $z_i^*$ given by our method . {\bf d-f}, uniform lockdown rate $z_i$. {\bf g-i}, random lockdown rate $z_i$. {\bf j-l}, uniformly-bounded-decline lockdown rate $z_i$.  Uniform lockdown, random lockdown, and uniformly-bouded-decline lockdown are defined as in Fig \ref{Fig: framework}.
	In {\bf a}, {\bf d}, {\bf g}, {\bf j}, the disease parameters are set as in \cite{bertozzi2020challenges}, the decay rate $\alpha$ is chosen as 0.0231 which corresponds to halving every 30 days. In {\bf b}, {\bf e}, {\bf h}, {\bf k} the disease parameters are set as in \cite{giordano2020modelling}, the decay rate is chosen as $\alpha = 0.2 r^{\text s} = 0.0034$ so that $\alpha < \min (r^{\text a}, r^{\text s})$. In {\bf c}, {\bf f}, {\bf i}, {\bf l}, the disease parameters are set as in \cite{birge2020controlling}, the decay rate $\alpha$ is chosen $0.0231$ that corresponds to halving every 30 days.  It can observed from {\bf a-c} that the value of $z_i^*$ for counties in NYC are relatively higher than other counties in New York State, which implies we should shutdown the outside of NYC harder than itself. Besides, it can be seen that such counter-intuitive phenomenon does not appear in any other lockdown polices.
}\label{Fig: SIR_each_county_uni_ours_decline}
\centering
\vspace{-1mm}				
	\end{figure}

	\begin{figure}
		
		\includegraphics[width=1.0\linewidth]{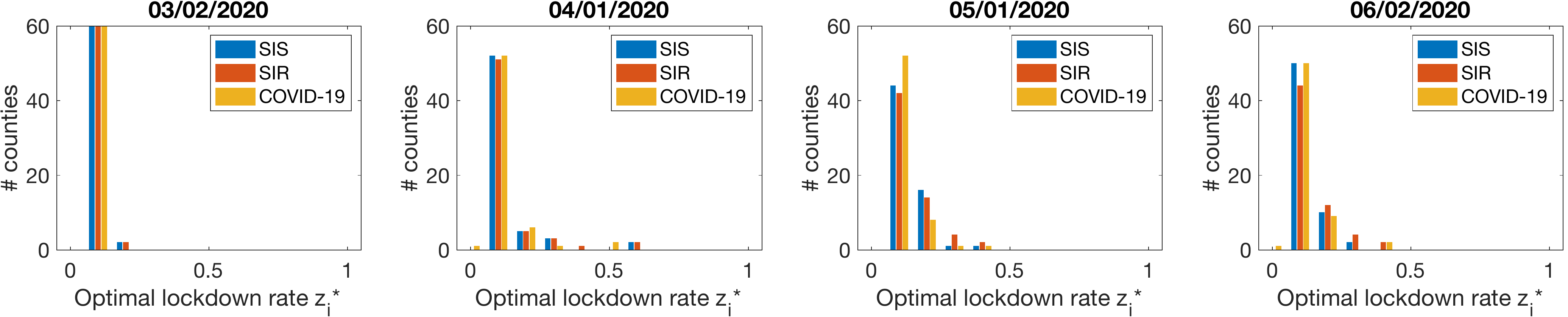}
		\centering
		\caption{{\bf Experimental results of on extended dates}: optimal lockdown rate $z_i^*$ given by Theorem \ref{thm:mainthm} for SIS, SIR, and COVID-19 model based on available data about COVID-19 outbreak in New York State on March 2nd, April 1st, May 1st and June 2nd. The disease parameters are set as in \cite{bertozzi2020challenges}. The decay rate is chosen as $\alpha = 0.2 r^{\text s} = 0.04$ so that $\alpha < \min(r^{\text a}, r^{\text s})$. It can be seen that the value of $z_i^*$ tends to increase from March to June, as people travel less frequently due the impact of COVID-19.
		}\label{Fig: diff_dates_res}
	\end{figure}

	\clearpage

	%

	\begin{figure}[!htb]
		\centering
		\subfigure[]{\label{gamma_z}
			\begin{minipage}[b]{0.195\linewidth}
				\centering
				\includegraphics[width=1.0\linewidth]{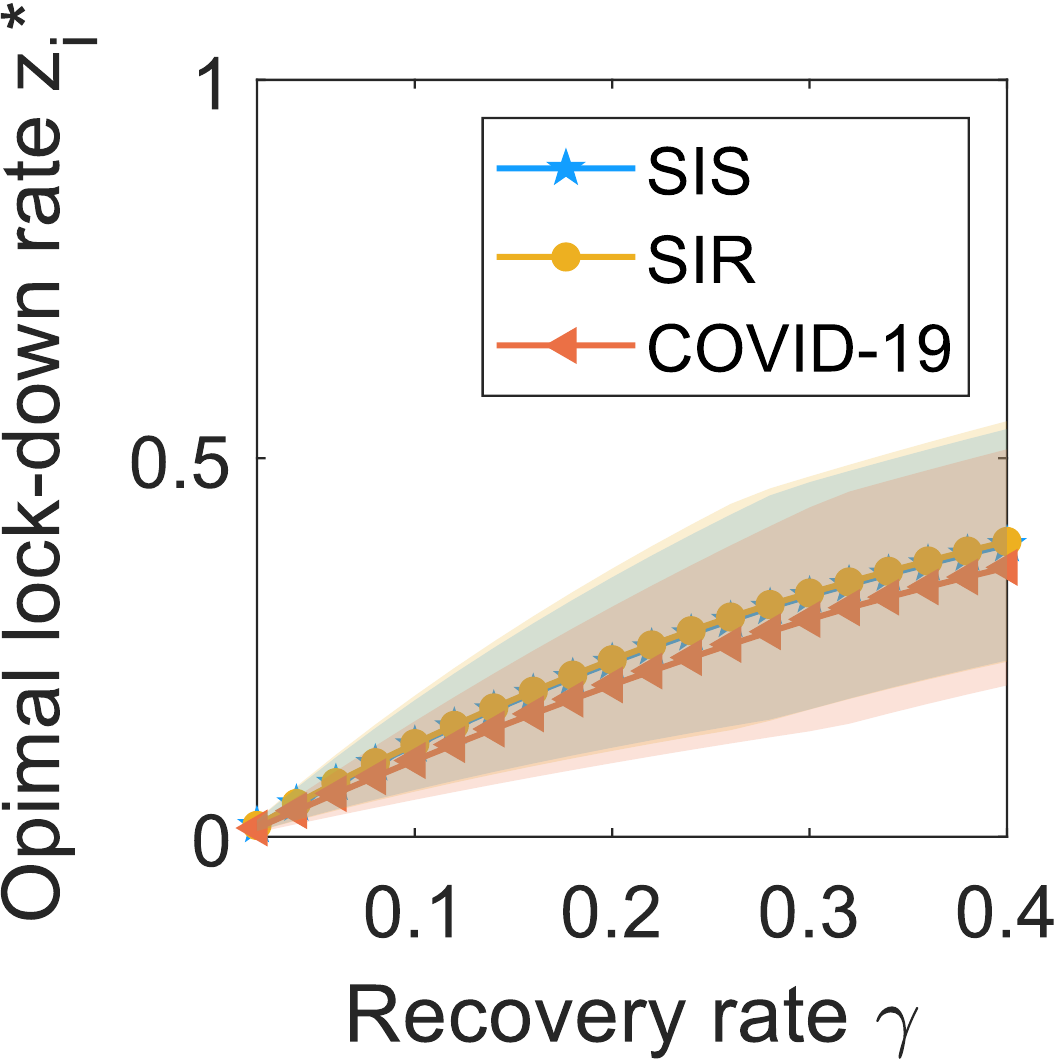}
			\end{minipage}%
		}%
		\subfigure[]{\label{growth_z}
			\begin{minipage}[b]{0.19\linewidth}
				\centering
				\includegraphics[width =1.0\linewidth]{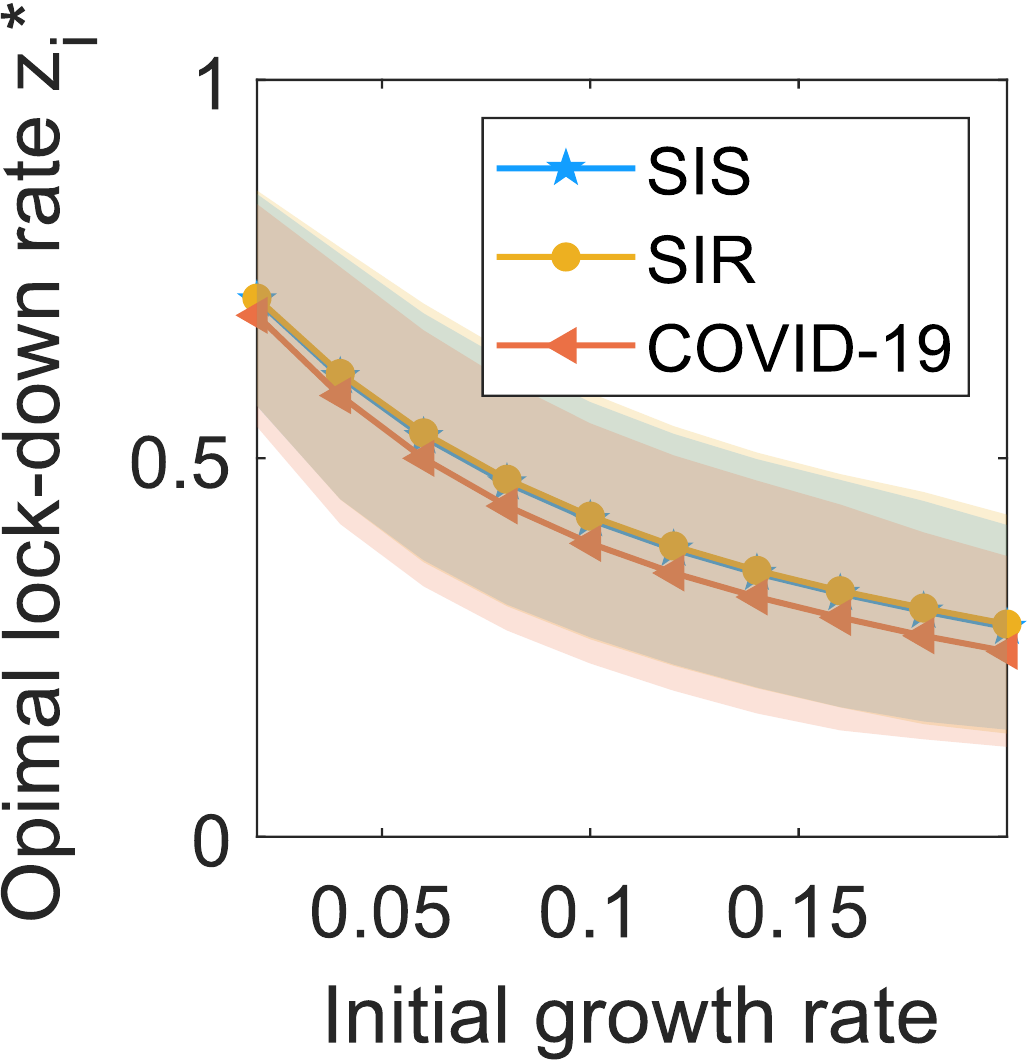}
			\end{minipage}%
		}%
		\subfigure[]{\label{epsilon_z}
			\begin{minipage}[b]{0.19\linewidth}
				\centering
				\includegraphics[width=1.0\linewidth]{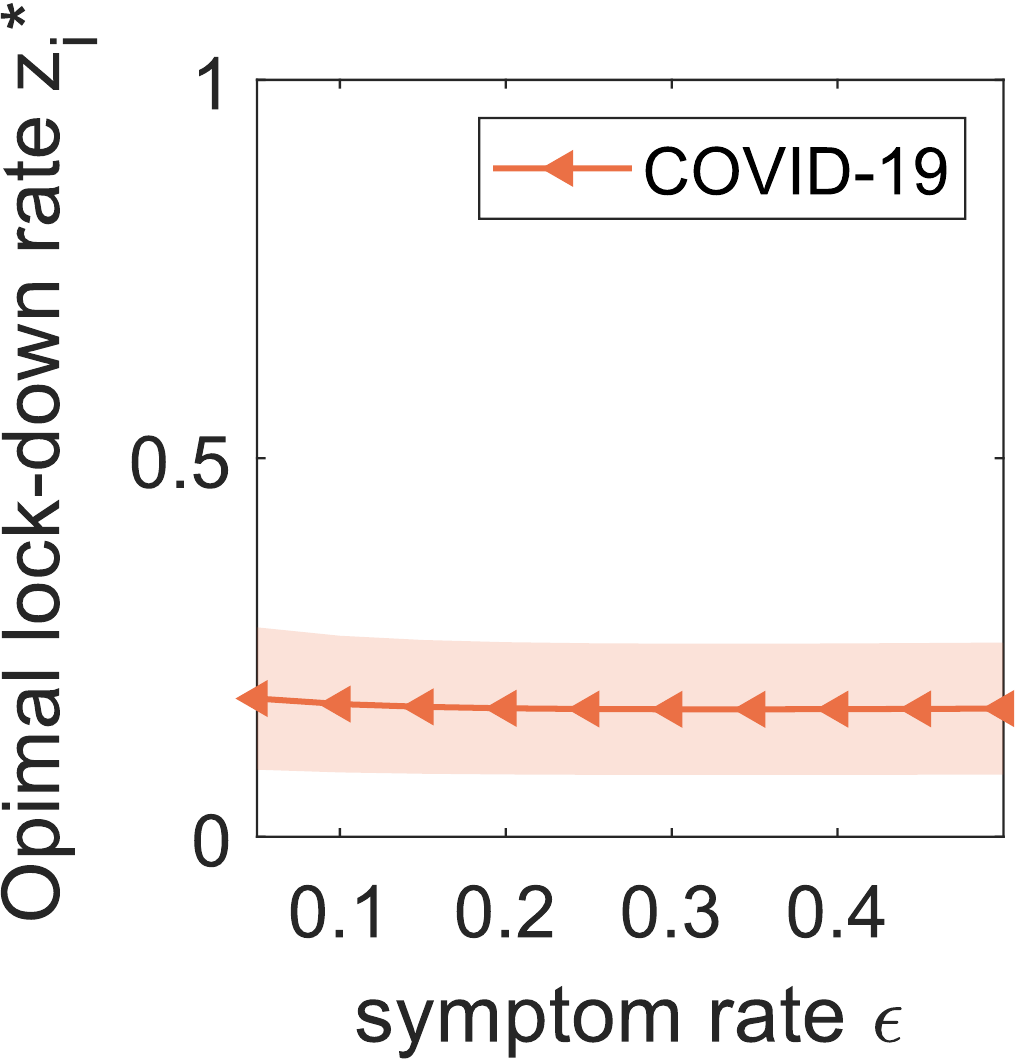}
			\end{minipage}%
		}%
		\subfigure[]{\label{alpha_hat_z}
			\begin{minipage}[b]{0.195\linewidth}
				\centering
				\includegraphics[width=1.0\linewidth]{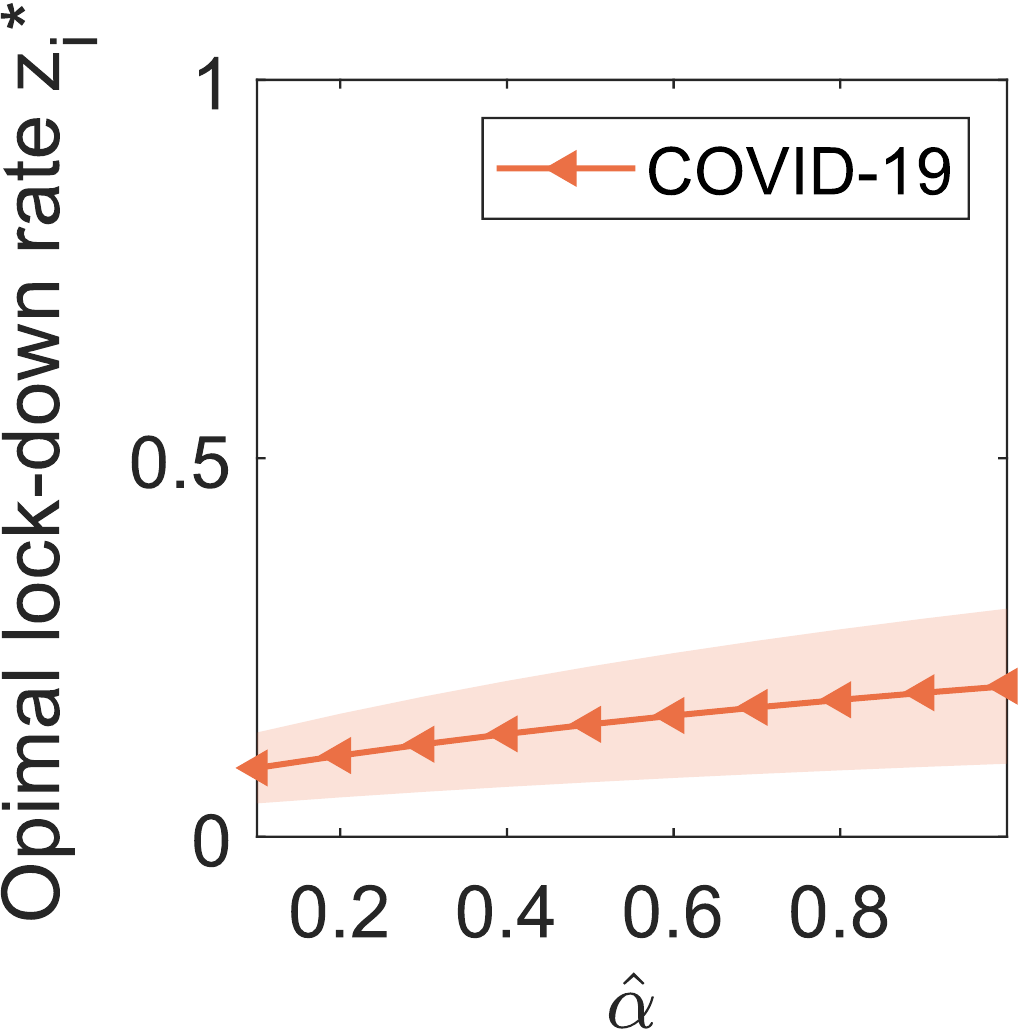}
			\end{minipage}%
		}%
		\subfigure[]{\label{gamma_hat_z}
			\begin{minipage}[b]{0.195\linewidth}
				\centering
				\includegraphics[width=1.0\linewidth]{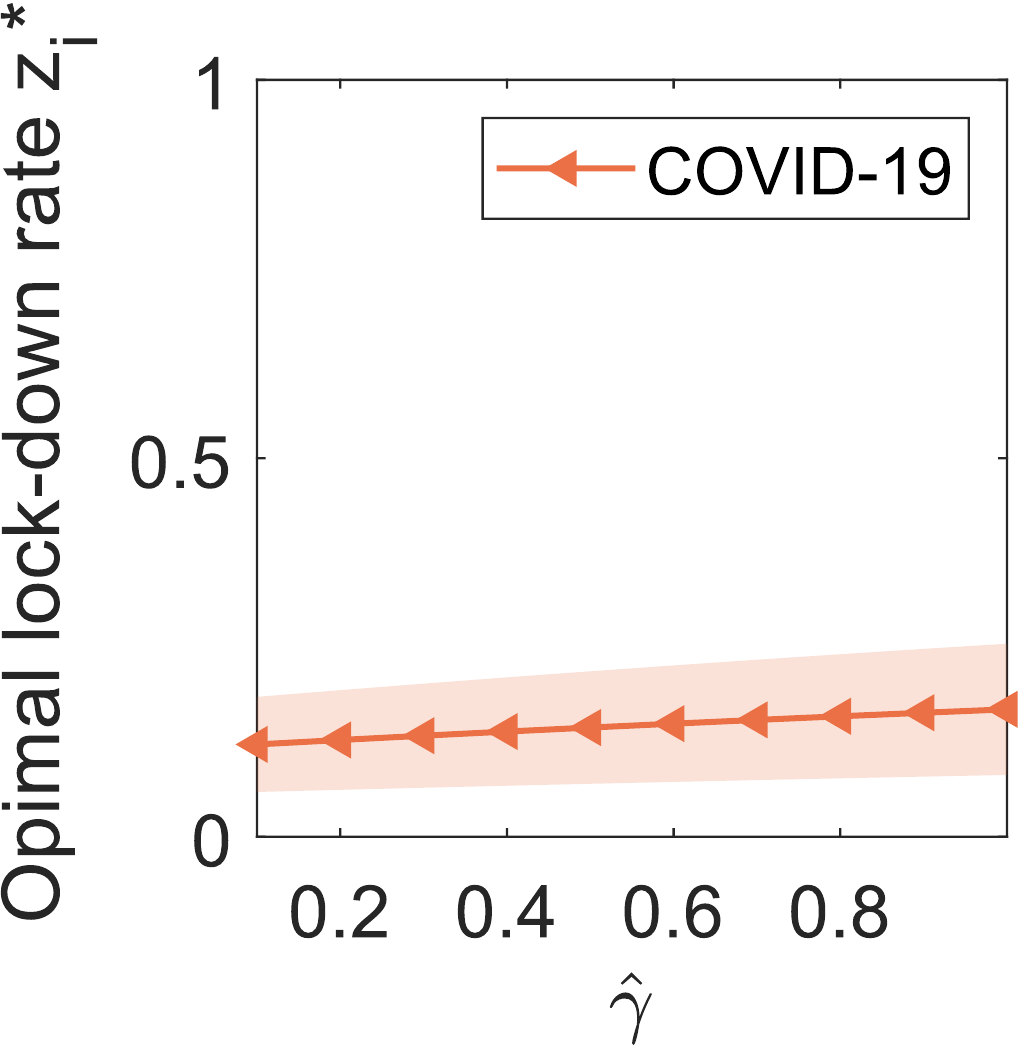}
			\end{minipage}%
		}%
		
		\vspace{-2.5mm}
		\subfigure[]{\label{gamma_c}
			\begin{minipage}[b]{0.195\linewidth}
				\centering
				\includegraphics[width=1.0\linewidth]{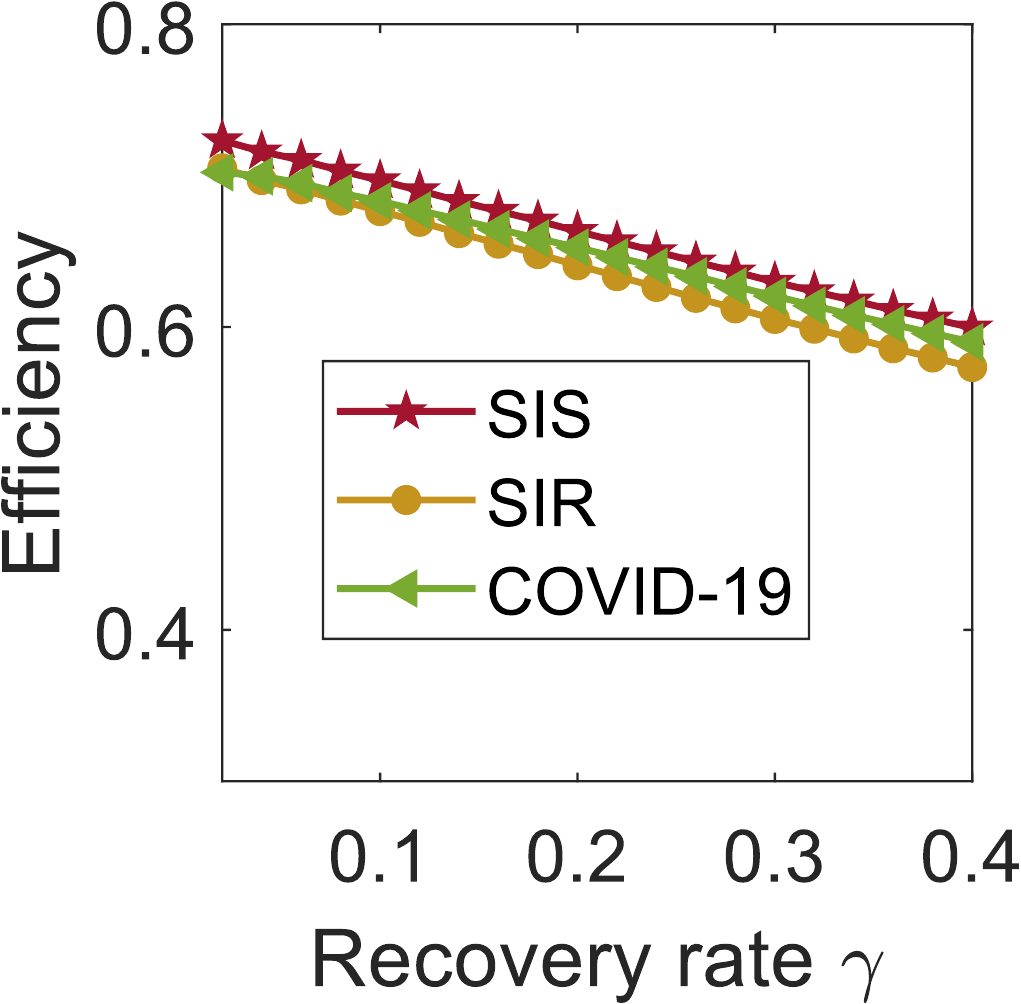}
			\end{minipage}%
		}%
		\subfigure[]{\label{growth_c}
			\begin{minipage}[b]{0.195\linewidth}
				\centering
				\includegraphics[width=0.98\linewidth]{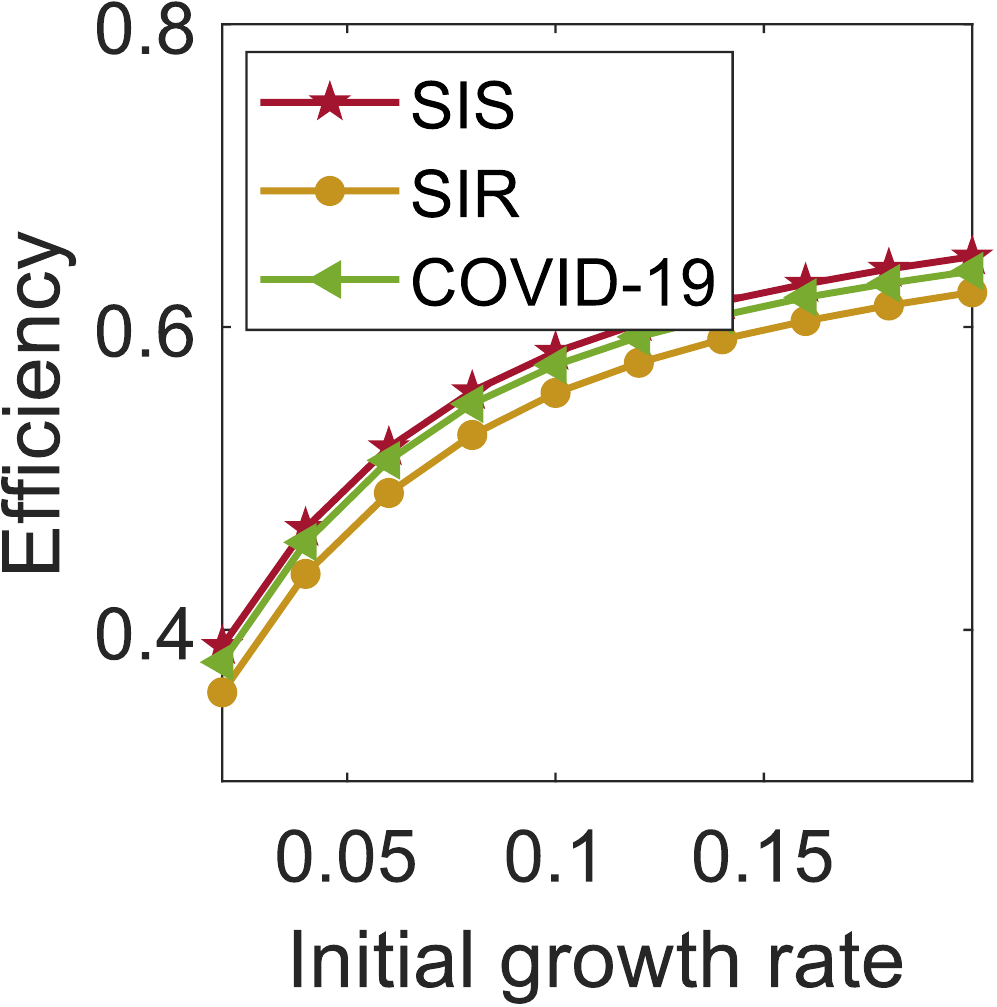}
			\end{minipage}%
		}%
		\subfigure[]{\label{epsilon_c}
			\begin{minipage}[b]{0.19\linewidth}
				\centering
				\includegraphics[width=1.0\linewidth]{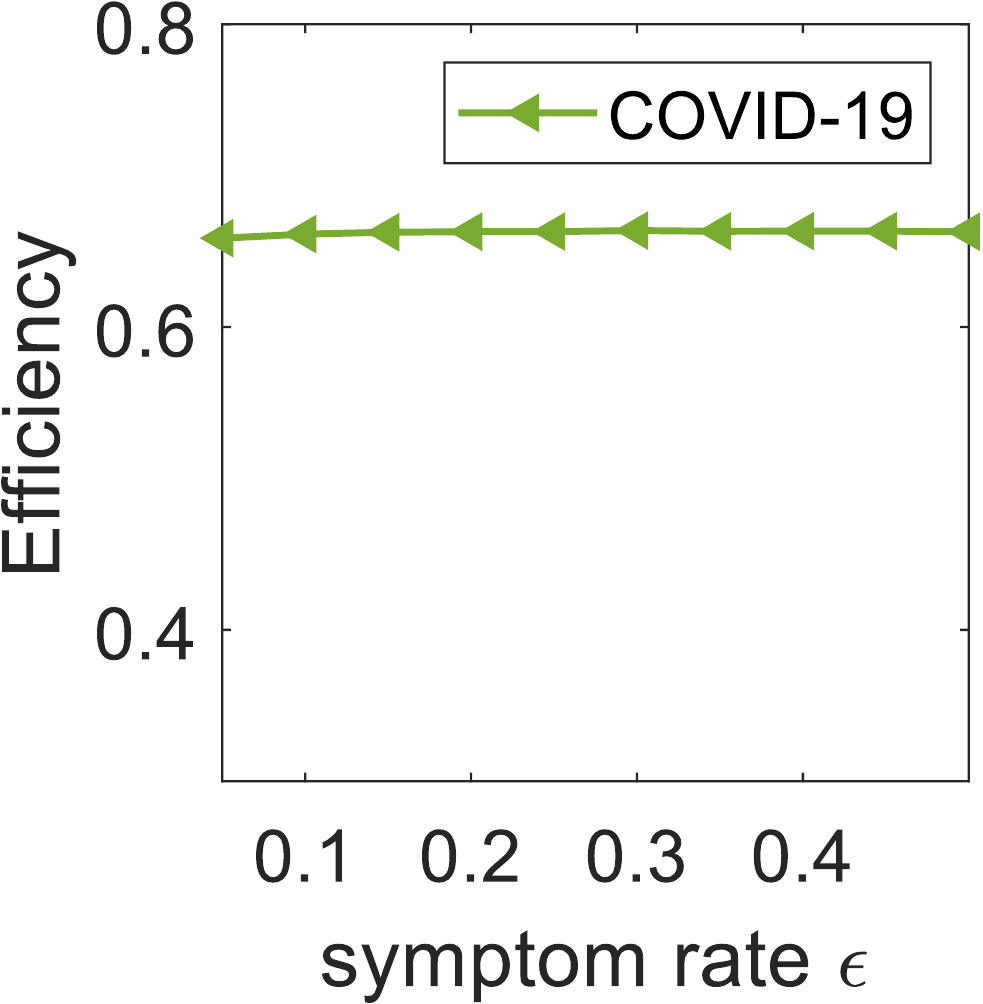}
			\end{minipage}%
		}%
		\subfigure[]{\label{alpha_hat_c}
			\begin{minipage}[b]{0.197\linewidth}
				\centering
				\includegraphics[width=1.0\linewidth]{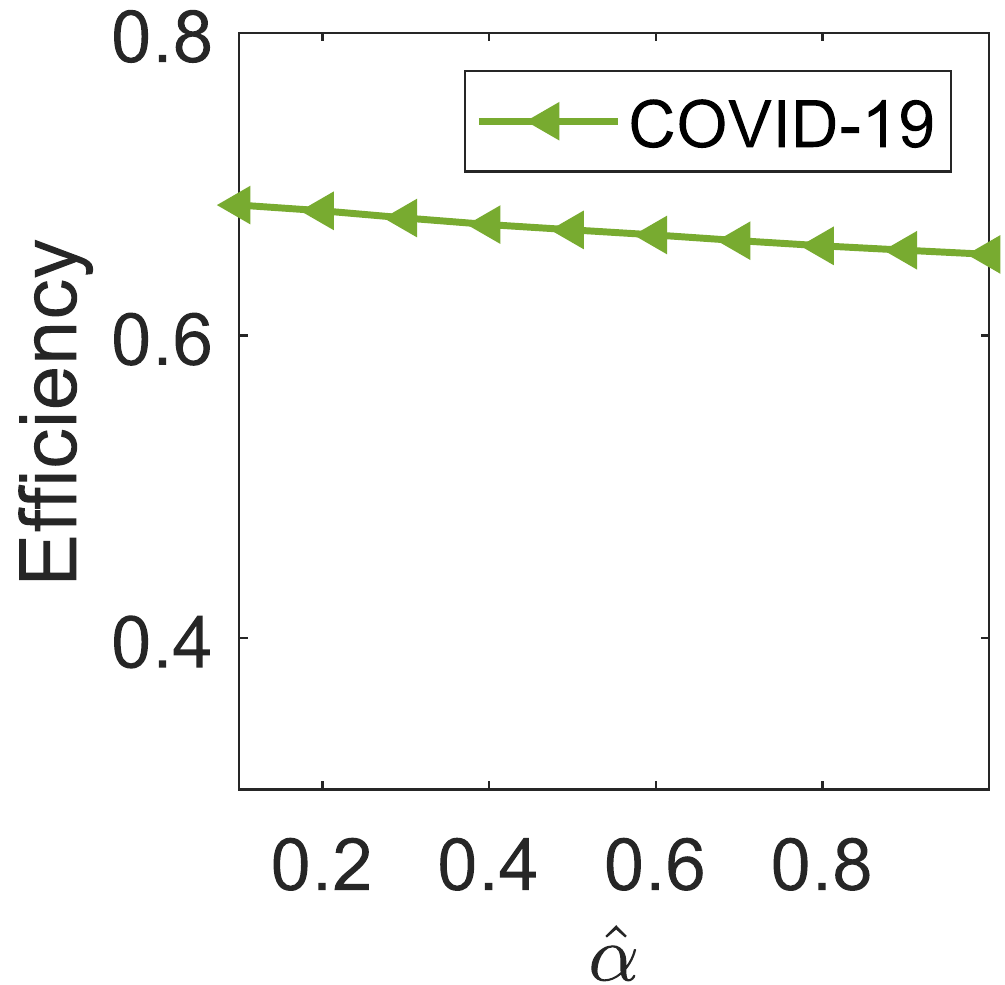}
			\end{minipage}%
		}%
		\subfigure[]{\label{gamma_hat_c}
			\begin{minipage}[b]{0.19\linewidth}
				\centering
				\includegraphics[width=1.0\linewidth]{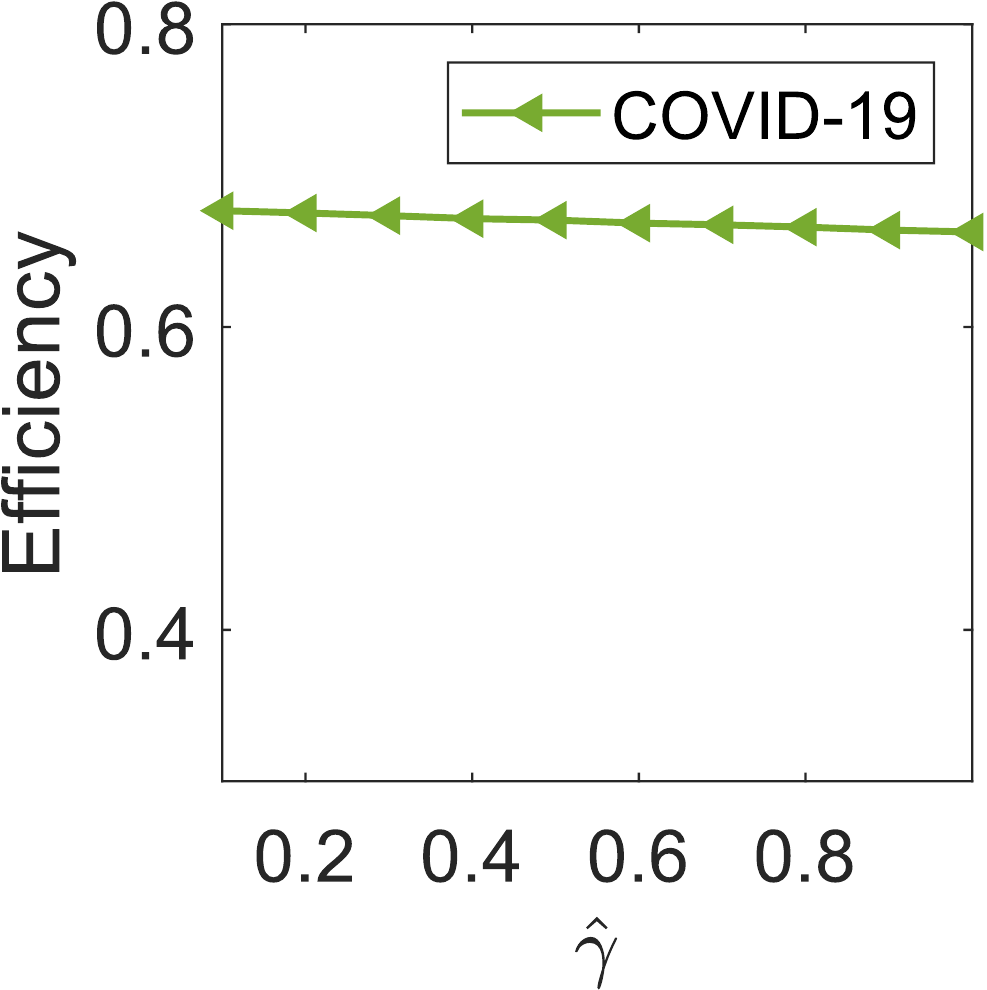}
			\end{minipage}%
		}%
		\centering
		\vspace{-1mm}
		\caption{{\bf Disease parameter sensitivity analysis with respect to $z_i^*$ and efficiency}: optimal lockdown rate $z_i^*$ and corresponding efficiency given by Theorem \ref{thm:mainthm} for SIS, SIR, and COVID-19 model based on available data about COVID-19 outbreak in New York State on April 1st. In {\bf a-e}, the shaded region represents the standard deviation of $z_i^*$ for all 62 counties, the solid line is the mean value of $z_i^*$ for all 62 counties. Parameter $\epsilon$, $\hat{\alpha}$, $\hat{\gamma}$ only appears in COVID-19 model, hence {\bf c-e} and {\bf h-j} only show the results of COVID-19 model. It can be seen that the value of $z_i^*$ and the corresponding economic cost are sensitive to recovery rate and the initial growth rate but not to other parameters.}\label{Fig: sensitivity}
		\vspace{-1mm}
	\end{figure}

	\begin{figure}[!htb]
		\centering
		\subfigure[]{\label{Fig: PNAS_rela}
			\begin{minipage}[b]{0.9\linewidth}{}
				\centering
				\includegraphics[width=1.0\linewidth]{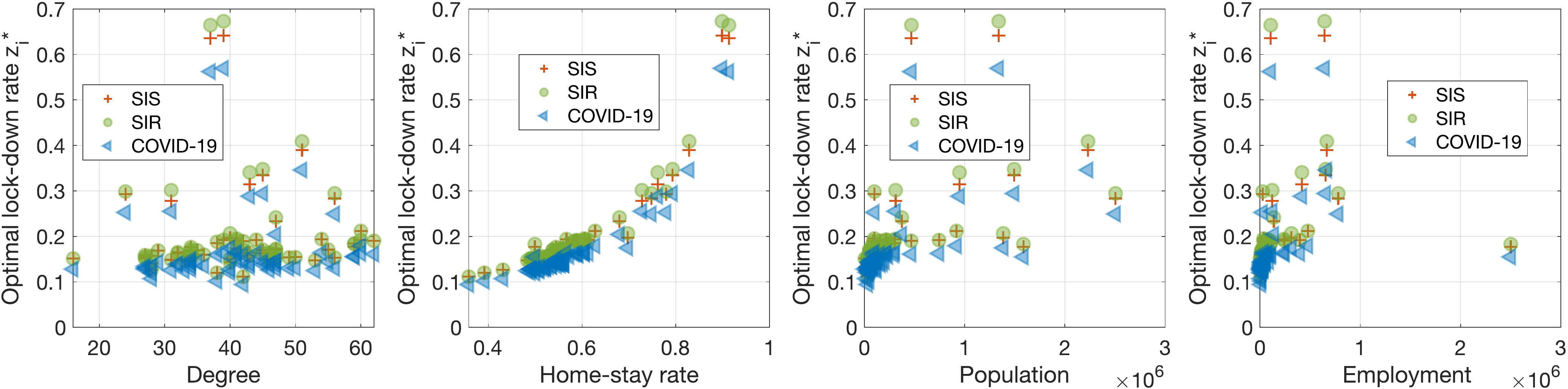}
			\end{minipage}%
		}%
		
		\subfigure[]{\label{Fig: nature_medicine_rela}
			\begin{minipage}[b]{0.9\linewidth}
				\centering
				\includegraphics[width =1.0\linewidth]{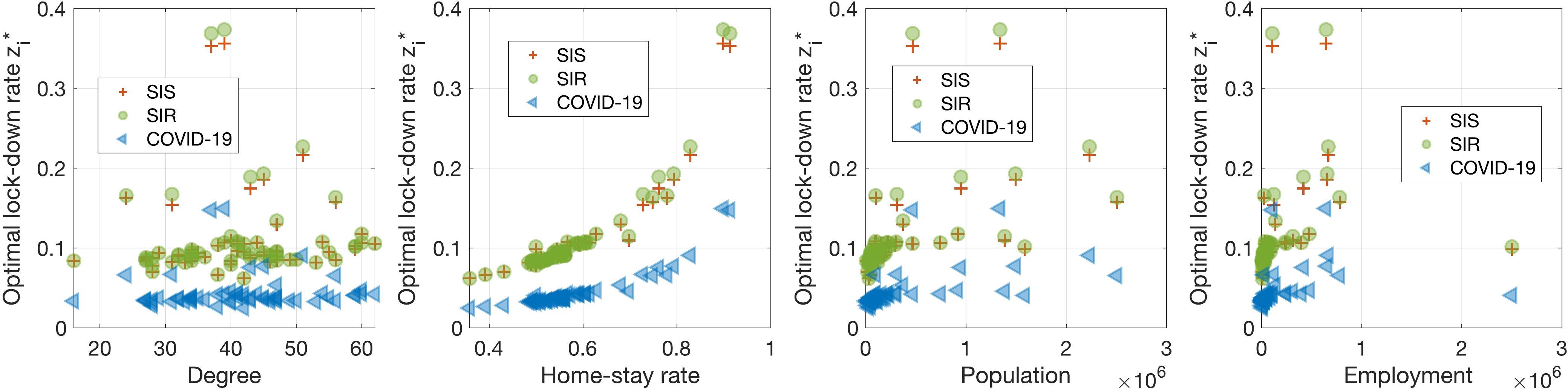}
			\end{minipage}%
		}%
		
		\subfigure[]{\label{Fig: Chicago_rela}
			\begin{minipage}[b]{0.9\linewidth}
				\centering
				\includegraphics[width=1.0\linewidth]{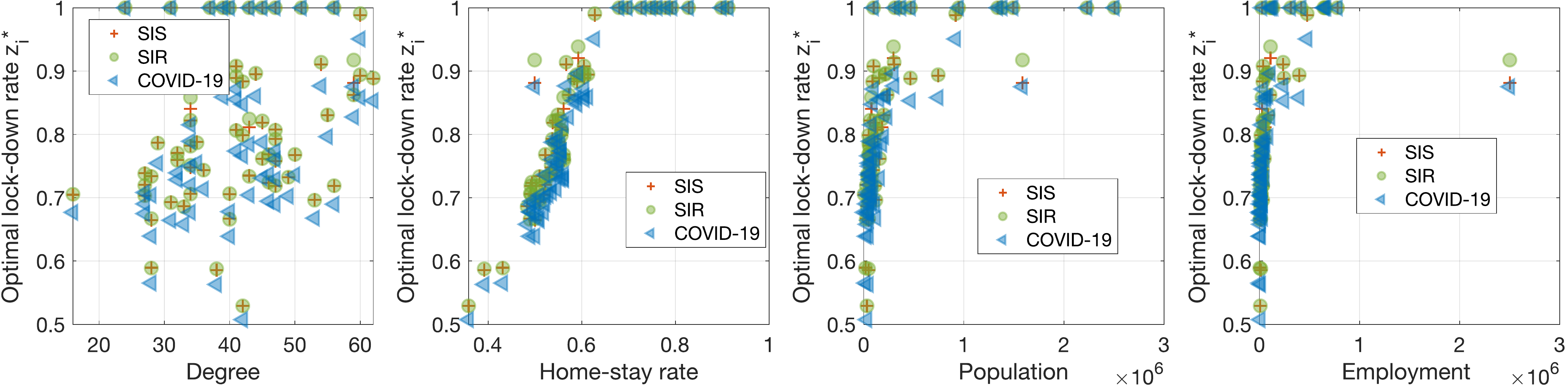}
			\end{minipage}%
		}%
		\centering
		\vspace{-1mm}
		\caption{Other parameter analysis: {\bf the relationship of the optimal lockdown rate $z_i$ and the centrality, home-stay rate, population and employment}. $z_i^*$ are produced by Theorem \ref{thm:mainthm} based on available data about COVID-19 outbreak about COVID-19 outbreak in New York State on April 1st, 2020. 
			The disease parameters are set as in \cite{bertozzi2020challenges}. In {\bf a}, the disease parameters are set as in \cite{bertozzi2020challenges}, the decay rate $\alpha$ is chosen $0.0231$ that corresponds to halving every 30 days. In {\bf b}, the disease parameters are set as in \cite{giordano2020modelling}, the decay rate is chosen as $\alpha = 0.2 r^{\text s} = 0.0034$ so that $\alpha < \min(r^{\text a}, r^{\text s})$. In {\bf c}, the disease parameters are set as in \cite{birge2020controlling}, the decay rate $\alpha$ is chosen $0.0231$ that corresponds to halving every 30 days.
			Degree closely related to centrality, and the number of employees decides the economic cost coefficients $c_i$. Each point represents a county in New York State. It can be observed that the value of $z_i^*$ increases as the home-stay rate of the node increases, the impacts of the degree, population and the employment to the value of $z_i^*$ are not obvious.}\label{Fig: 2020_parameters}
		\vspace{-1mm}
	\end{figure}
	
	\begin{figure}[!htb]
		\centering
		\subfigure[]{\label{March_rand_perm}
			\begin{minipage}[b]{0.9\linewidth}
				\centering
				\includegraphics[width =1.0\linewidth]{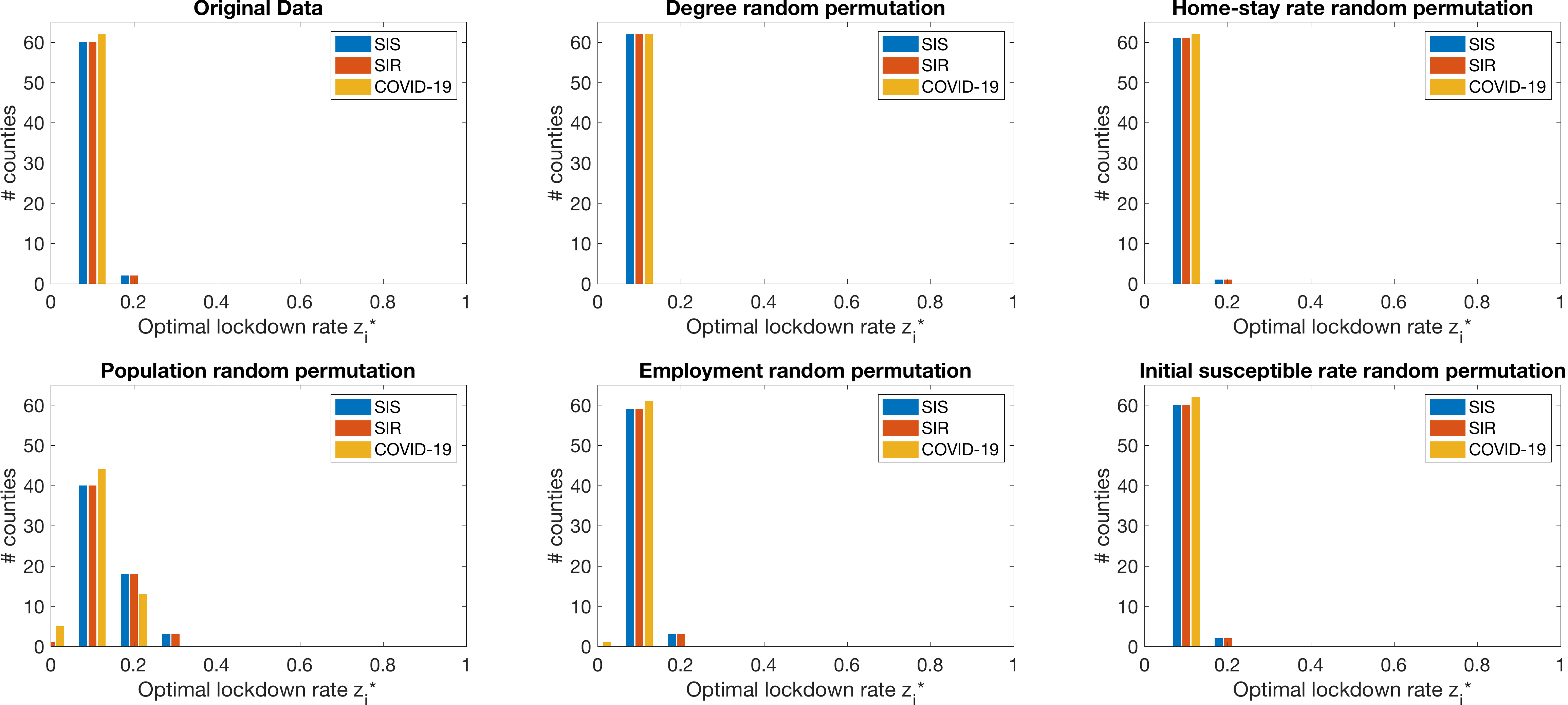}
			\end{minipage}%
		}%
		
		\subfigure[]{\label{April_rand_perm}
			\begin{minipage}[b]{0.9\linewidth}
				\centering
				\includegraphics[width=1.0\linewidth]{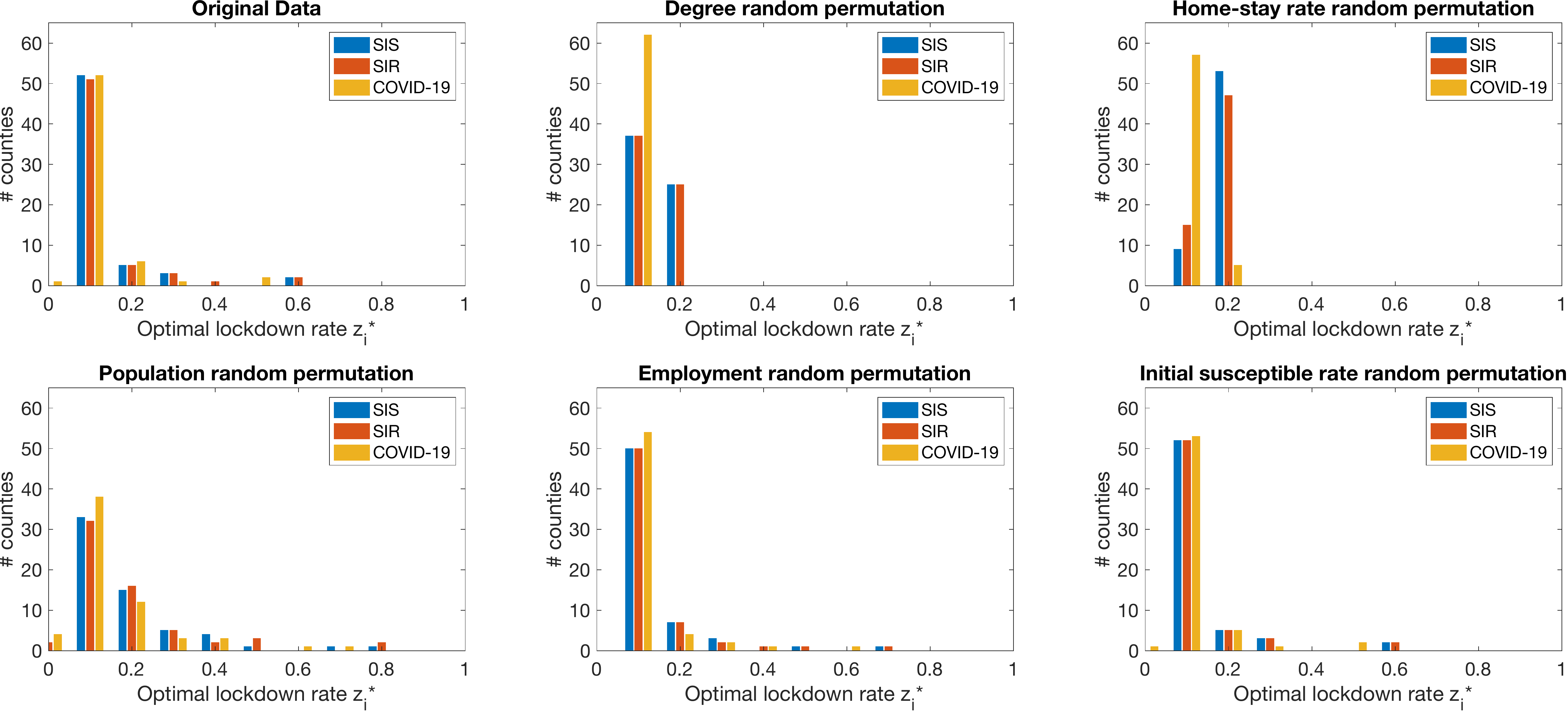}
			\end{minipage}%
		}%
		\centering
		\vspace{-1mm}
		\caption{{\bf Experimental results of random permutation}: optimal lockdown rate $z_i^*$ after doing random permutation with respect to degree, home-stay rate, population, employment, and initial susceptible rate. {\bf a},
			shows the results based on available data about COVID-19 outbreak in NYS on March 2nd, 2020. {\bf b},  the results based on available data about COVID-19 outbreak in NYS on April 1st, 2020. The decay rate is chosen as $\alpha = 0.2 r^{\text s} = 0.04$ so that $\alpha < \min(r^{\text a}, r^{\text s})$. The results are average of 100 repeat. It can be observed that the distribution of $z_i^*$ can be very strongly affected by permutations of centrality, population, and the home stay rate. On the other hand, the distribution of $z_i^*$ is not altered much by permuting employment and the initial susceptible rate.}\label{Fig: rand_perm1}
		\vspace{-1mm}
	\end{figure}

	\begin{figure}[!htb]
		\centering
		\subfigure[]{\label{May_rand_perm}
			\begin{minipage}[b]{0.9\linewidth}
				\centering
				\includegraphics[width =1.0\linewidth]{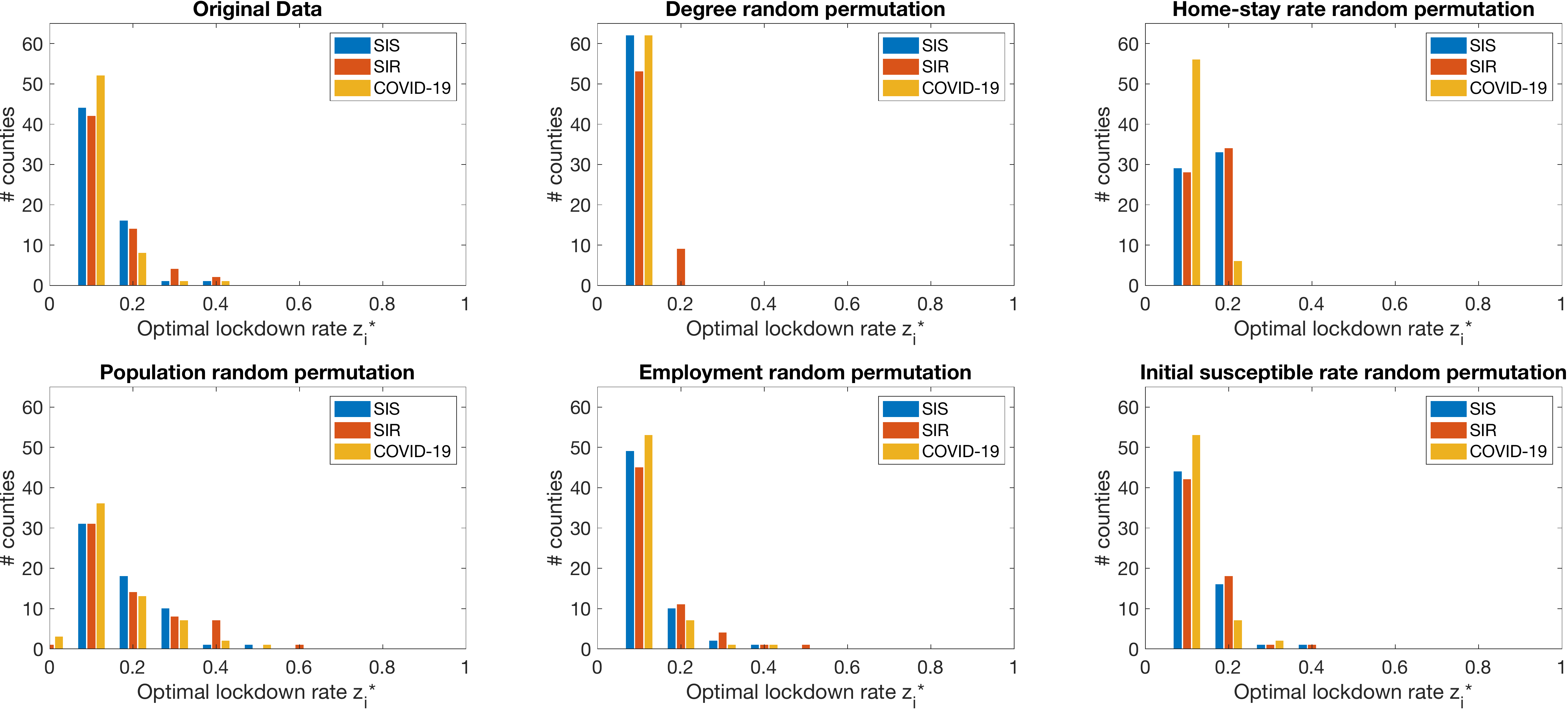}
			\end{minipage}%
		}%
		
		\subfigure[]{\label{June_rand_perm}
			\begin{minipage}[b]{0.9\linewidth}
				\centering
				\includegraphics[width=1.0\linewidth]{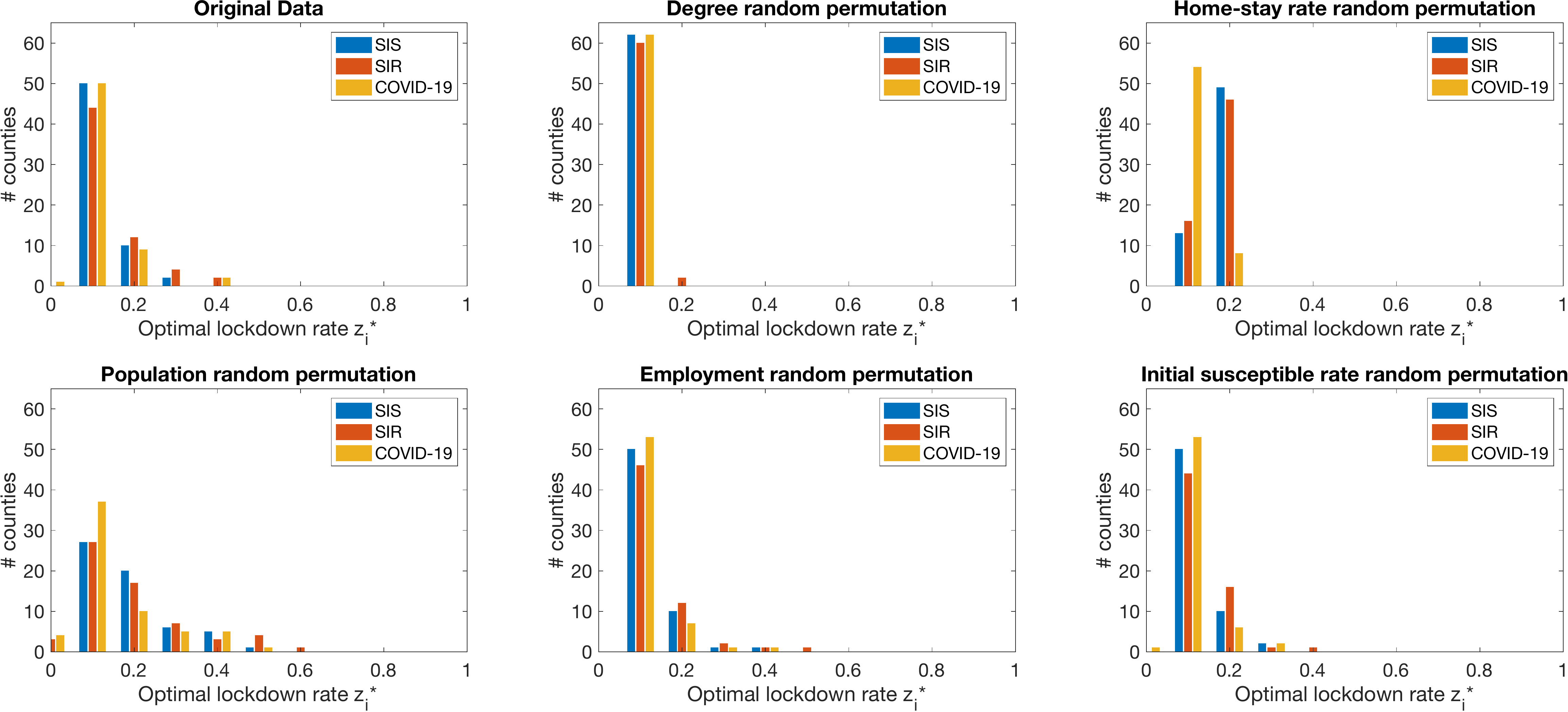}
			\end{minipage}%
		}%
		\centering
		\vspace{-1mm}
		\caption{{\bf Experimental results of random permutation}: optimal lockdown rate $z_i^*$ after doing random permutation with respect to degree, home-stay rate, population, employment, and initial susceptible rate. {\bf a}, the results based on available data about COVID-19 outbreak in NYS on May 1st, 2020. {\bf b}, the results based on available data about COVID-19 outbreak in NYS on June 2nd, 2020. The disease parameters are set as in \cite{bertozzi2020challenges}. The decay rate is chosen as $\alpha = 0.2 r^{\text s} = 0.04$ so that $\alpha < \min(r^{\text a}, r^{\text s})$.The results are average of 100 repeat. It can be observed that the distribution of $z_i^*$ can be very strongly affected by permutations of centrality, population, and the home stay rate. On the other hand, the distribution of $z_i^*$ is not altered much by permuting employment and the initial susceptible rate.
		}\label{Fig: rand_perm2}
		\vspace{-1mm}
	\end{figure}



	\begin{figure}[!htb]
		\centering
		\subfigure[]{\label{Fig: centrality_B_A}
			\begin{minipage}[b]{0.9\linewidth}
				\centering
				\includegraphics[width =1.0\linewidth]{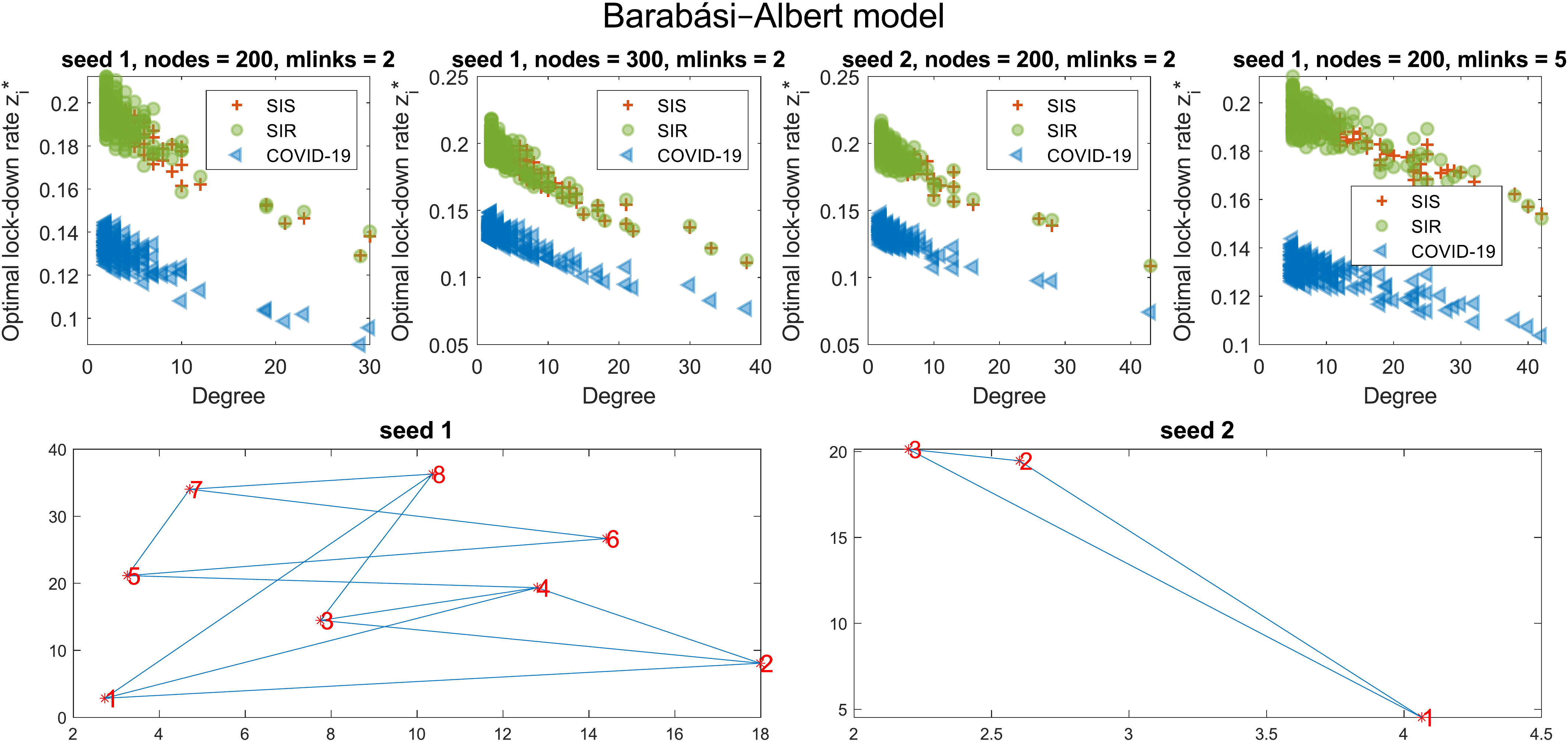}
			\end{minipage}%
		}%
		
		\subfigure[]{\label{Fig: centrality_rand}
			\begin{minipage}[b]{0.9\linewidth}
				\centering
				\includegraphics[width=1.0\linewidth]{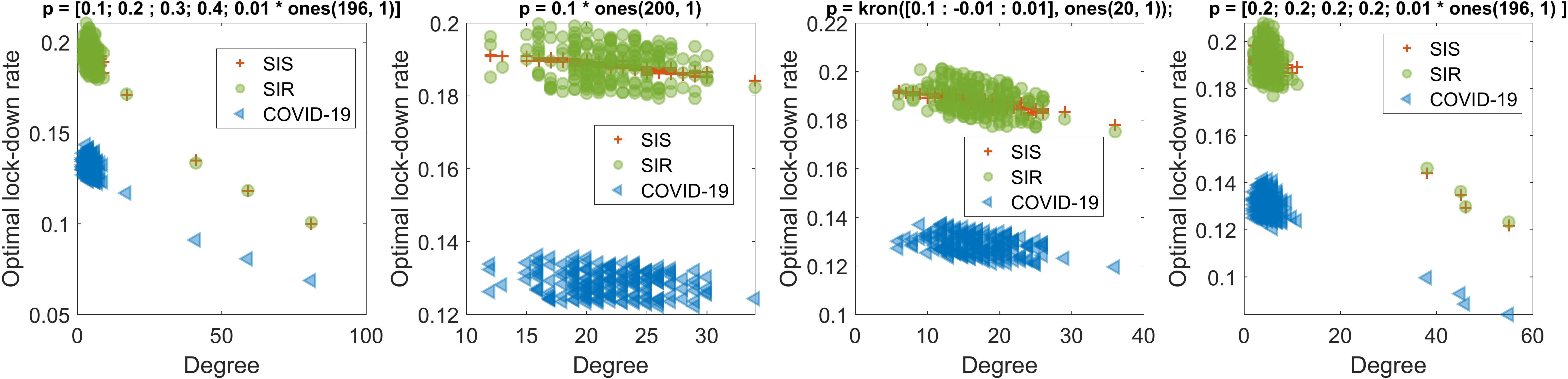}
			\end{minipage}%
		}%
		\centering
		\vspace{-1mm}
		\caption{Experimental results of synthetic data: {\bf the impact of centrality of the network} to the value of optimal lockdown rate $z_i^*$. {\bf a}, the results based on Barabási–Albert model. {\bf b}, the results based on a kind of generated random graph.
			Each point represents a node in the network. It can be observed that centrality only matters for the value of $z_i^*$ when there exist hotspots in all the three models. 
			Surprisingly, beyond such hotspots, the effect of centrality is essentially nonexistent.}
		\vspace{-1mm}
	\end{figure}

	\clearpage	
	\begin{figure}[!htb]
				\centering
			\begin{minipage}[b]{0.75\linewidth}
				\includegraphics[width=0.95\linewidth]{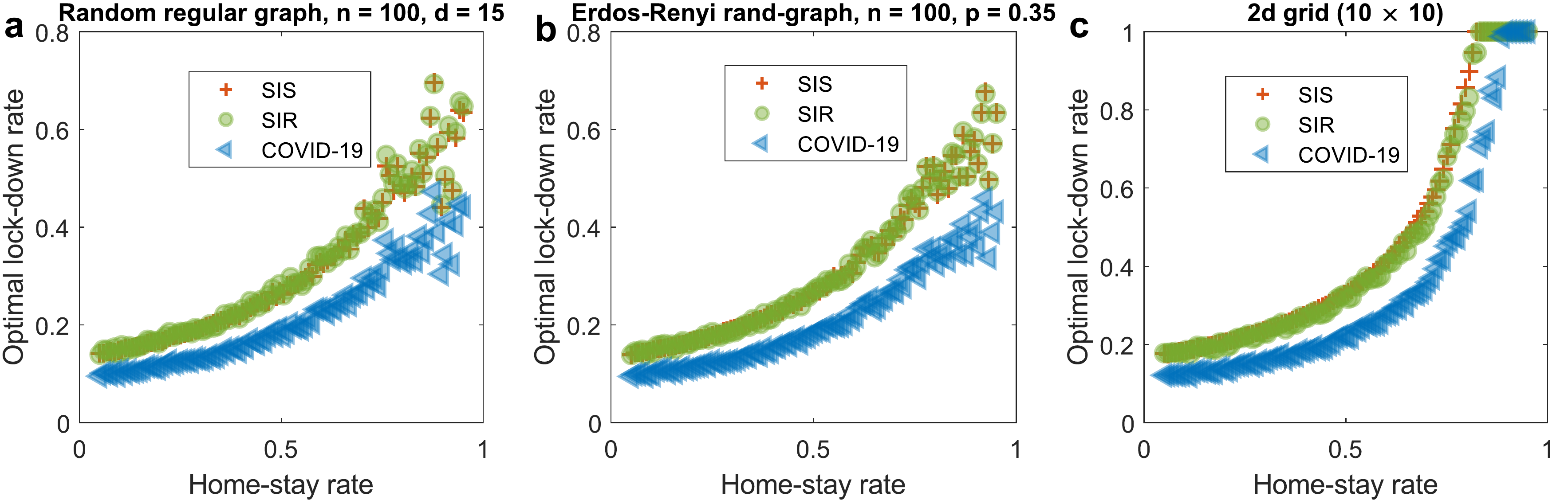}
			\end{minipage}%
					\raggedleft\begin{minipage}[b]{0.235\linewidth}
			\includegraphics[width=0.95\linewidth]{plots/home_rate_new.pdf}
		\end{minipage}%
		\centering
		\vspace{-1mm}
		\caption{{\color{black}Experimental results of synthetic data: {\bf the impact of home-stay rate} to the value of optimal lockdown rate $z_i^*$.  Each point represents a node in the network. It can be observed that $z_i^*$increases as the home-stay rate increases, and in fact the home-stay rate has by far the biggest influence on $z_i^*$ compared to the other parameters we consider.}}\label{fig: home_stay_rate}
		\vspace{-1mm}
	\end{figure}

	\begin{figure}[!htb]
		\centering
		\subfigure[]{\label{PNAS_0_2}
			\begin{minipage}[b]{0.9\linewidth}
				\centering
				\includegraphics[width =1.0\linewidth]{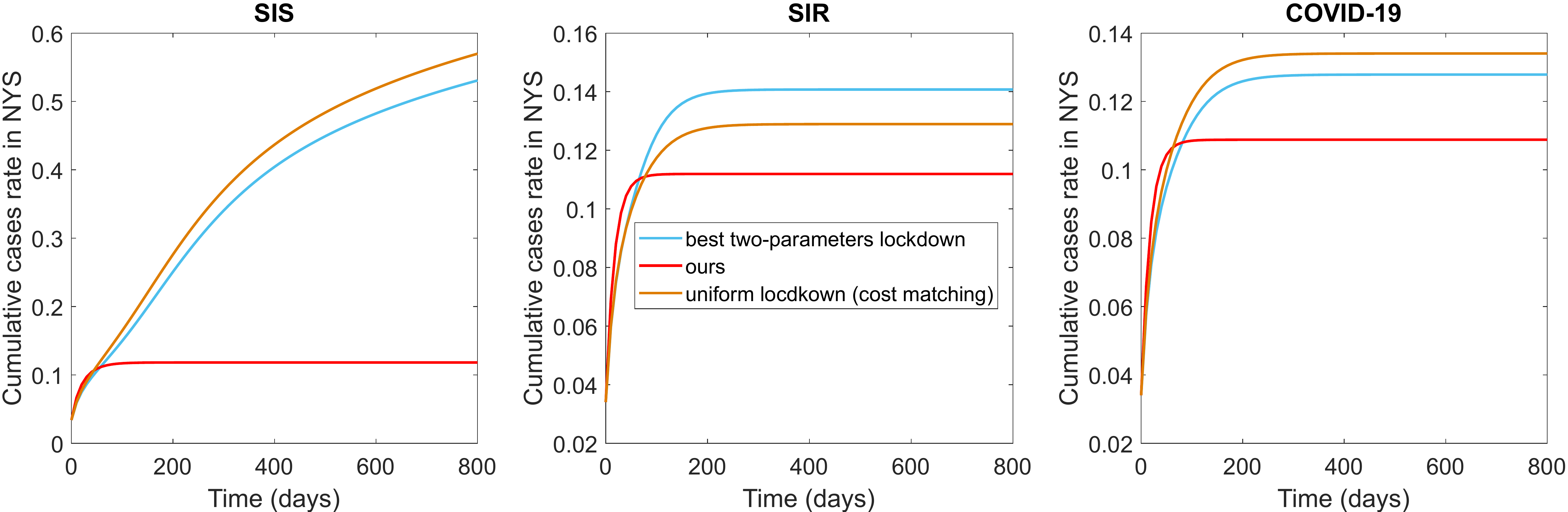}
			\end{minipage}%
		}%
		
		\subfigure[]{\label{nature_medicine_0_2}
			\begin{minipage}[b]{0.9\linewidth}
				\centering
				\includegraphics[width=1.0\linewidth]{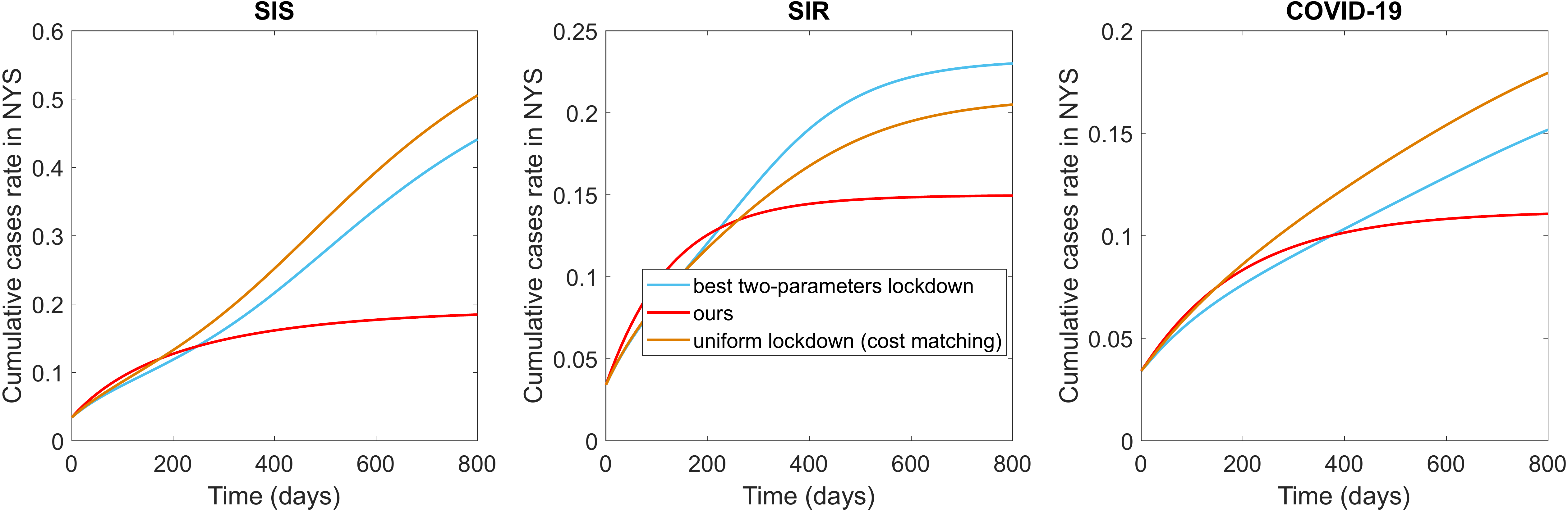}
			\end{minipage}%
		}%
		
		\subfigure[]{\label{Chicago_0_2}
			\begin{minipage}[b]{0.9\linewidth}
				\centering
				\includegraphics[width=1.0\linewidth]{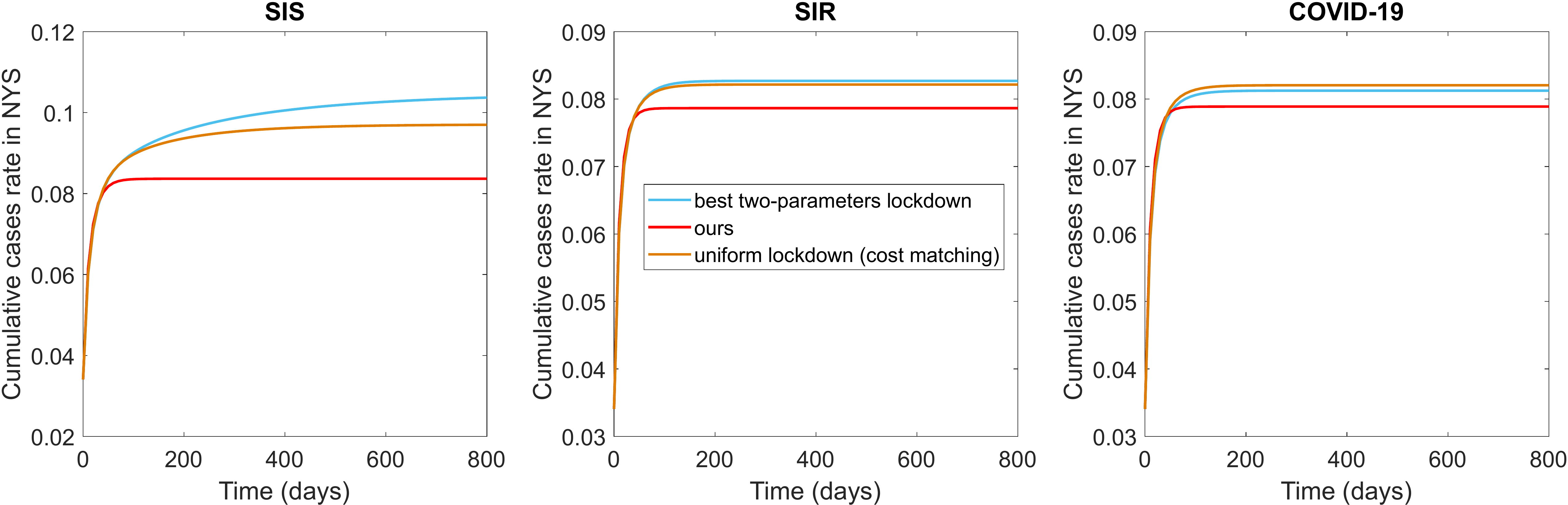}
			\end{minipage}%
		}%
		\centering
		\vspace{-1mm}
		\caption{Experimental results of real data: {\bf the estimated rate of accumulative cases in the total population} by applying different lockdown policies based on available data about COVID-19 outbreak in New York State on April 1st, 2020. It can be observed our policy outperforms the best two-parameters lockdown as well as the uniform lockdown in all the three models. In {\bf a}, the disease parameters are set as in \cite{bertozzi2020challenges}, and the decay rate $\alpha = \mathbf{0.2} r^{\text s} = 0.04$. In {\bf b}, the disease parameters set as in \cite{giordano2020modelling}, and the decay rate $\alpha = \mathbf{0.2} r^{\text s} = 0.0034$. In {\bf c}, the disease parameters set as in \cite{birge2020controlling}, and the decay rate $\alpha = \mathbf{0.2} r^{\text s} =0.058$.
		}\label{Fig: cumula_0_2}
		\vspace{-1mm}
	\end{figure}

	\begin{figure}[!htb]
		\centering
		\subfigure[]{\label{PNAS_0_5}
			\begin{minipage}[b]{0.9\linewidth}
				\centering
				\includegraphics[width =1.0\linewidth]{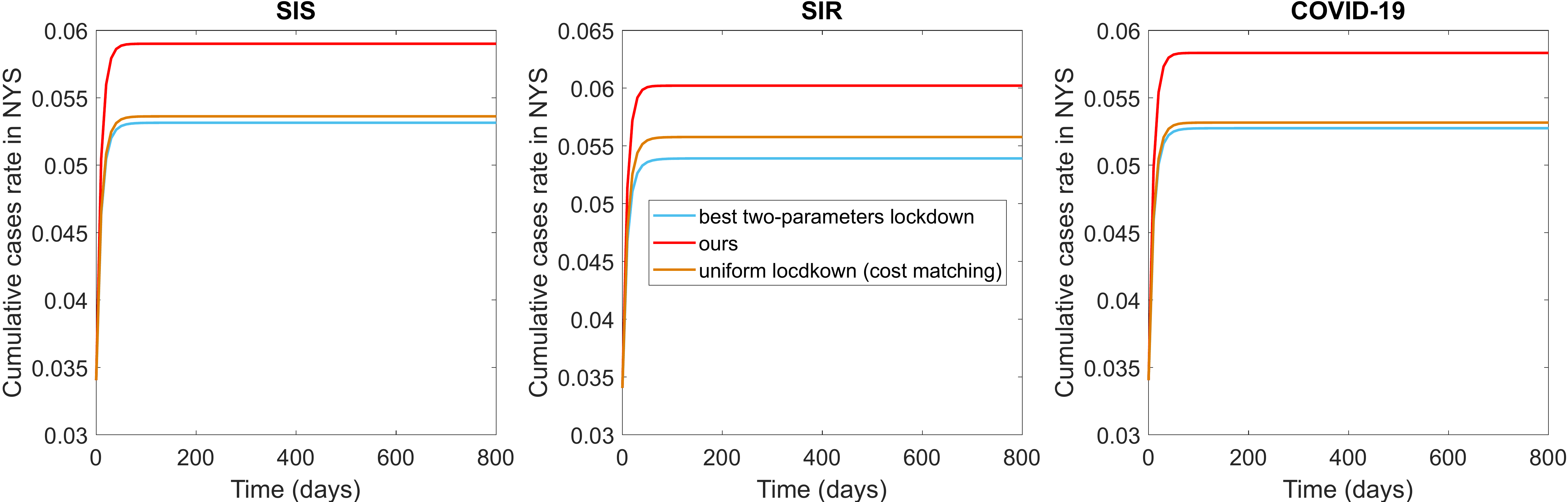}
			\end{minipage}%
		}%
		
		\subfigure[]{\label{nature_medicine_0_5}
			\begin{minipage}[b]{0.9\linewidth}
				\centering
				\includegraphics[width=1.0\linewidth]{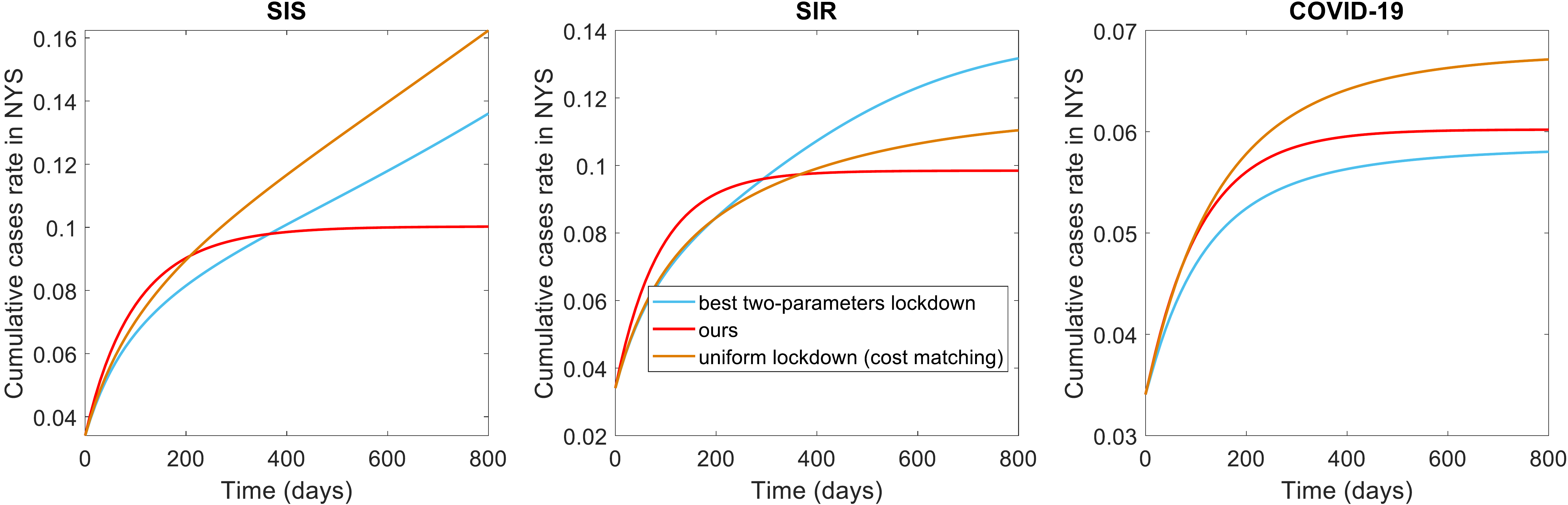}
			\end{minipage}%
		}%
		
		\subfigure[]{\label{Chicago_0_5}
			\begin{minipage}[b]{0.9\linewidth}
				\centering
				\includegraphics[width=1.0\linewidth]{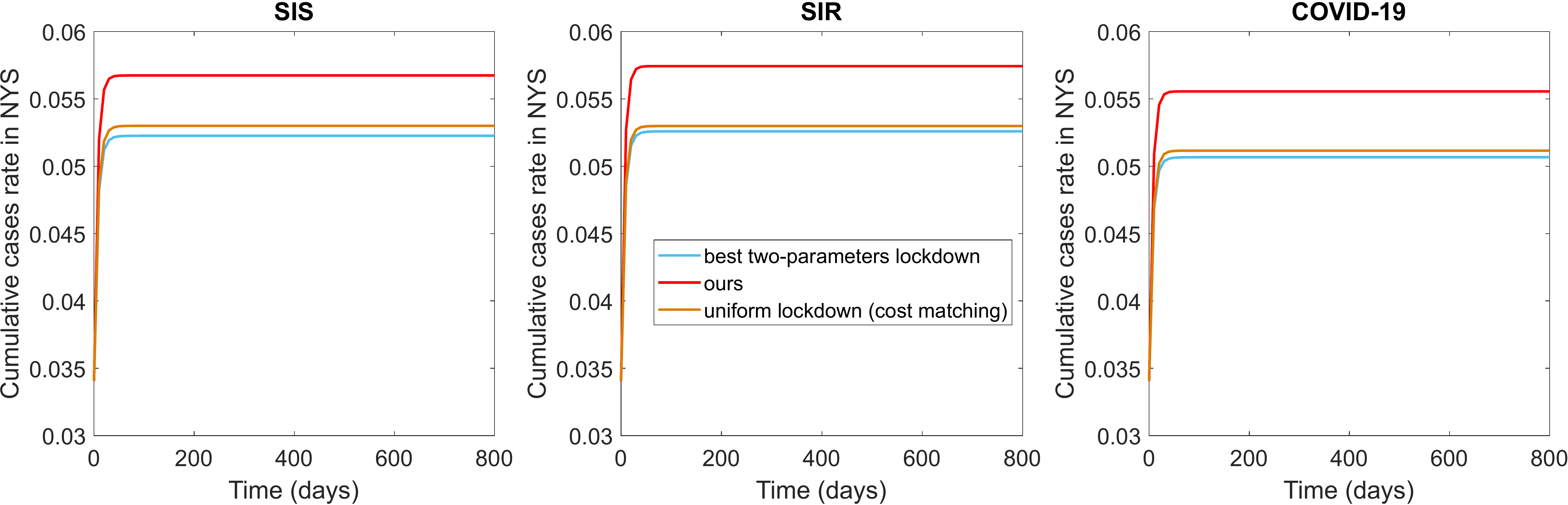}
			\end{minipage}%
		}%
		\centering
		\vspace{-1mm}
		\caption{Experimental results of real data: {\bf the estimated rate of accumulative cases in the total population} by applying different lockdown policies based on available data about COVID-19 outbreak in New York State on April 1st, 2020. In {\bf a}, the disease parameters set as in \cite{bertozzi2020challenges}, and the decay rate $\alpha = \mathbf{0.5}  r^{\text s} = 0.1$. In {\bf b}, the disease parameters set as in \cite{giordano2020modelling}, and the decay rate $\alpha = \mathbf{0.5} r^{\text s} = 0.0085$. In {\bf c}, the disease parameters set as in \cite{birge2020controlling}, and the decay rate $\alpha = \mathbf{0.5}  r^{\text s} = 0.145$. In {\bf b}, it can be seen that our policy outperforms the best two-parameters lockdown as well as the uniform lockdown in SIS model and SIR model. In {\bf a}, {\bf c}, and the COVID-19 model in {\bf b}, it can be seen that our policy underperforms the best two-parameters lockdown as well as the uniform lockdown, the reason for such phenomenon is provided in  \ssref{two-parameters}. }\label{Fig: cumula_0_5}
		\vspace{-1mm}
	\end{figure}

	\begin{figure}[!htb]
		\centering
		\textbf{SIS}\par\medskip
		\vspace{-1mm}
		\centering
		\subfigure[]{\label{subfig: SIS_alpha}
			\begin{minipage}[b]{0.3\linewidth}
				\centering
				\includegraphics[width=1.0\linewidth, height= 1.018\linewidth]{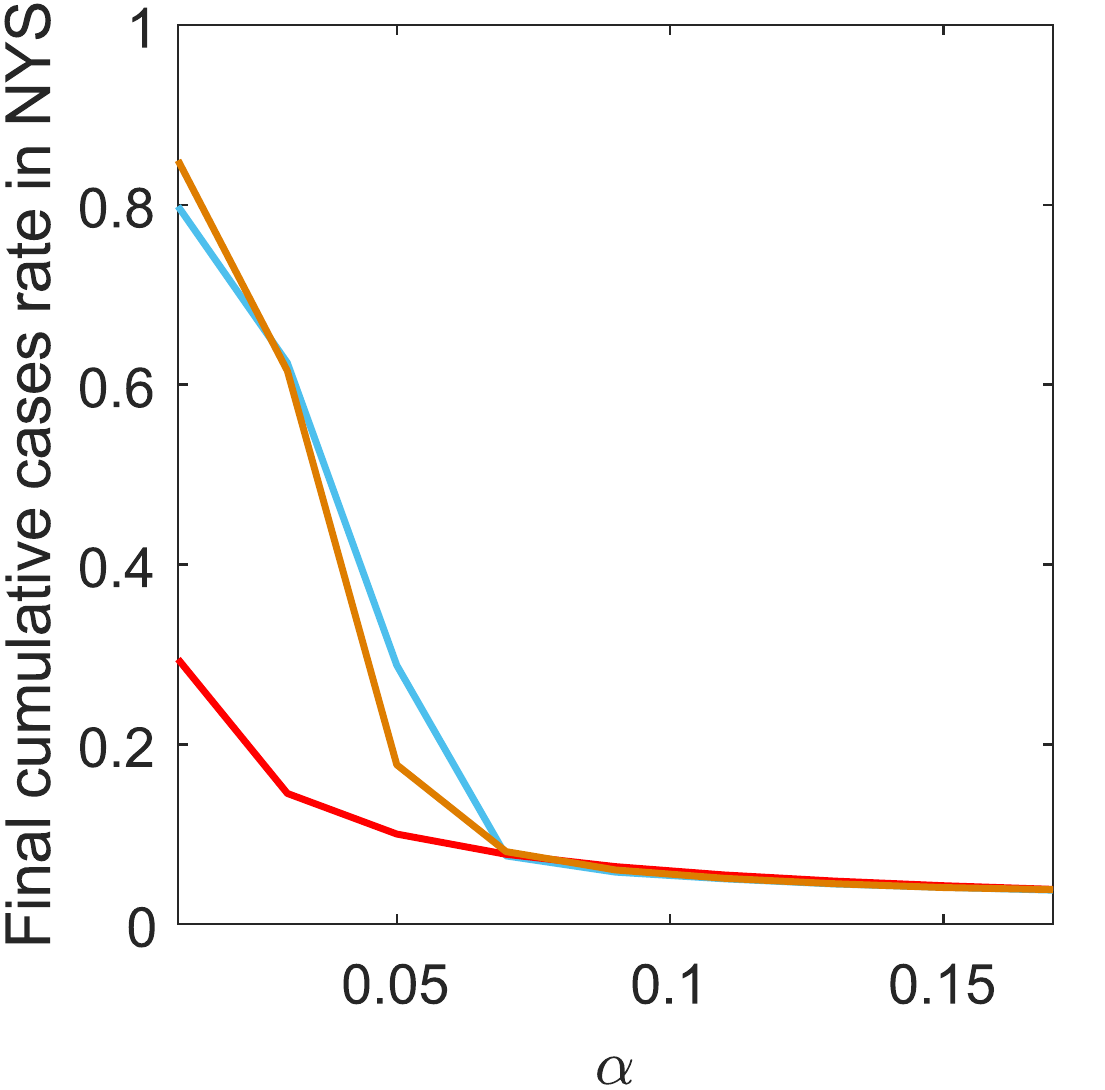}
			\end{minipage}%
		}%
		\subfigure[]{\label{subfig: SIS_gamma}
			\begin{minipage}[b]{0.3\linewidth}
				\centering
				\includegraphics[width=1.0\linewidth]{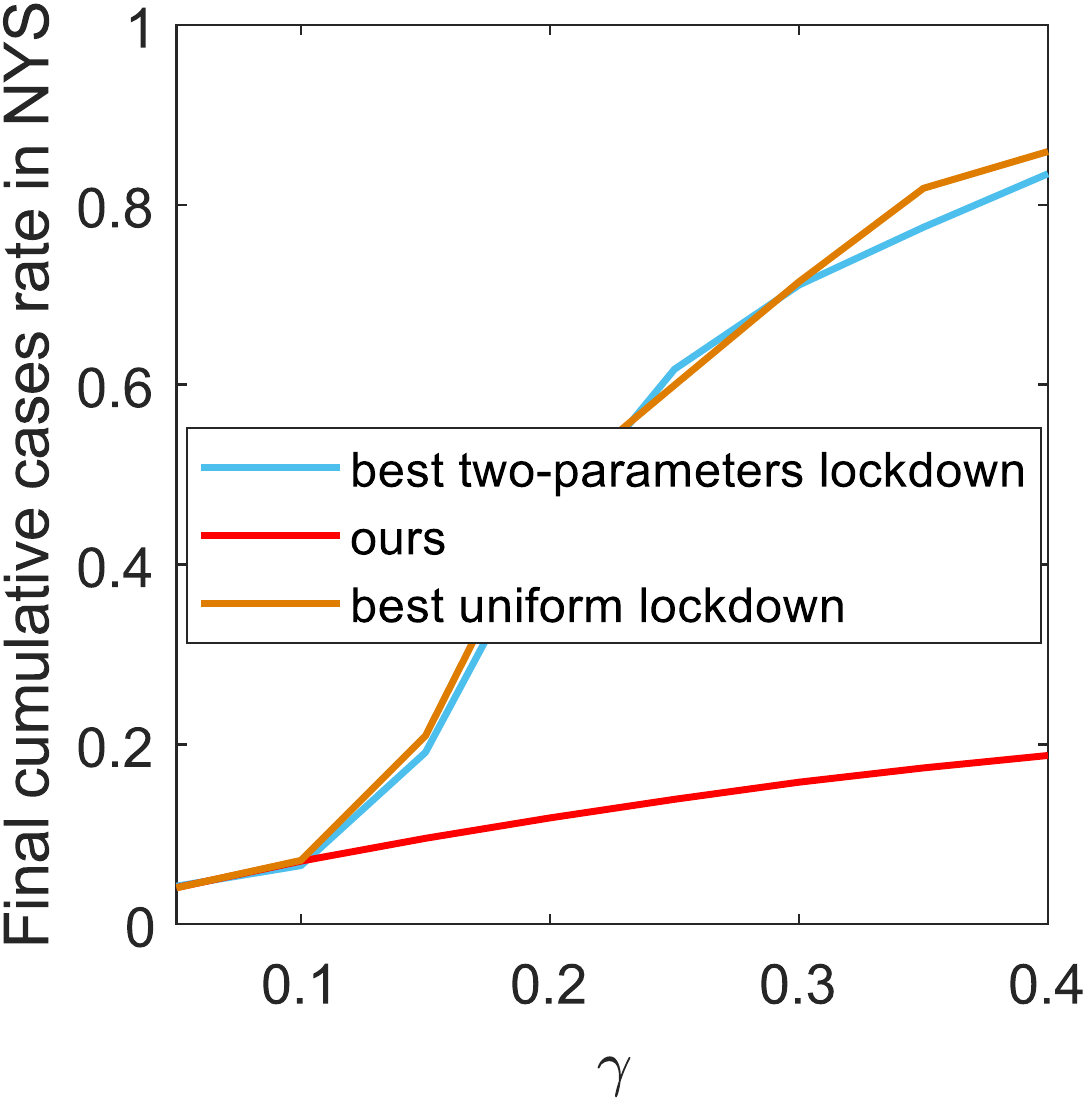}
			\end{minipage}%
		}%
		\subfigure[]{\label{subfig: SIS_growth_rate}
			\begin{minipage}[b]{0.3\linewidth}
				\centering
				\includegraphics[width=1.0\linewidth]{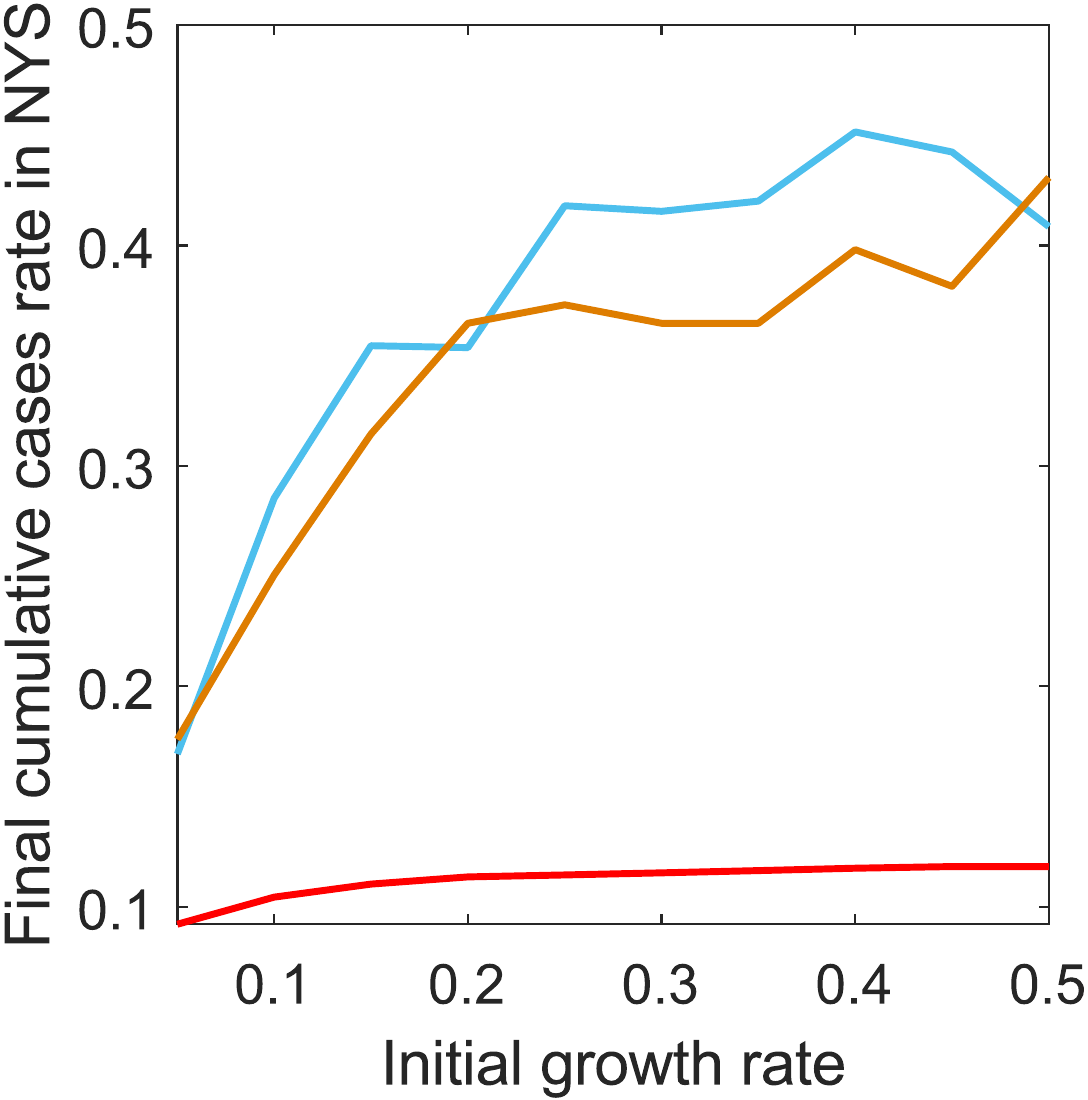}
			\end{minipage}%
		}%

		\centering
		\textbf{SIR}\par\medskip
		\vspace{-1mm}
		\centering
		\subfigure[]{\label{subfig: SIR_alpha}
			\begin{minipage}[b]{0.3\linewidth}
				\centering
				\includegraphics[width=1.0\linewidth]{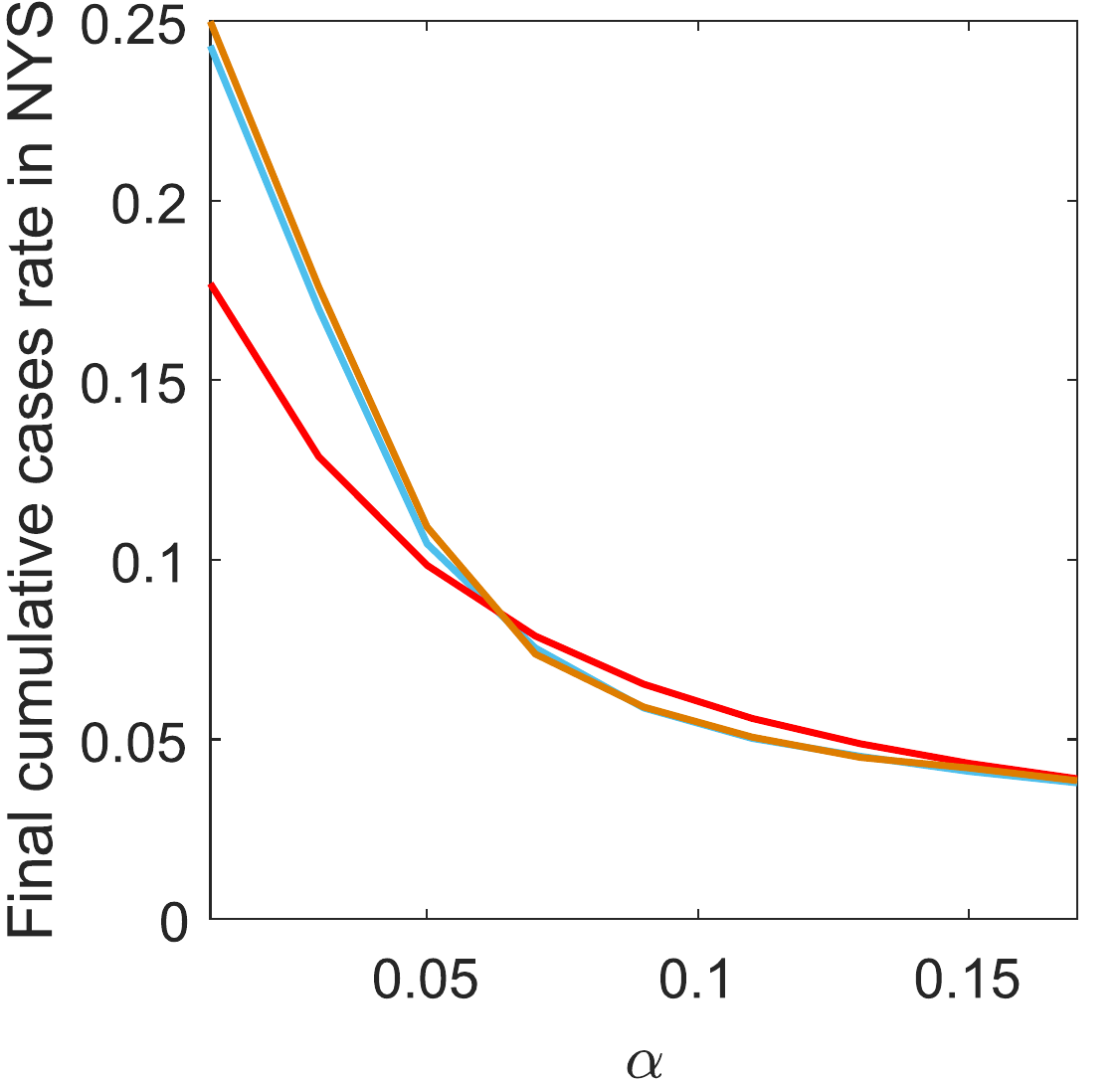}
			\end{minipage}%
		}%
		\subfigure[]{\label{subfig: SIR_gamma}
			\begin{minipage}[b]{0.3\linewidth}
				\centering
				\includegraphics[width =1.0\linewidth]{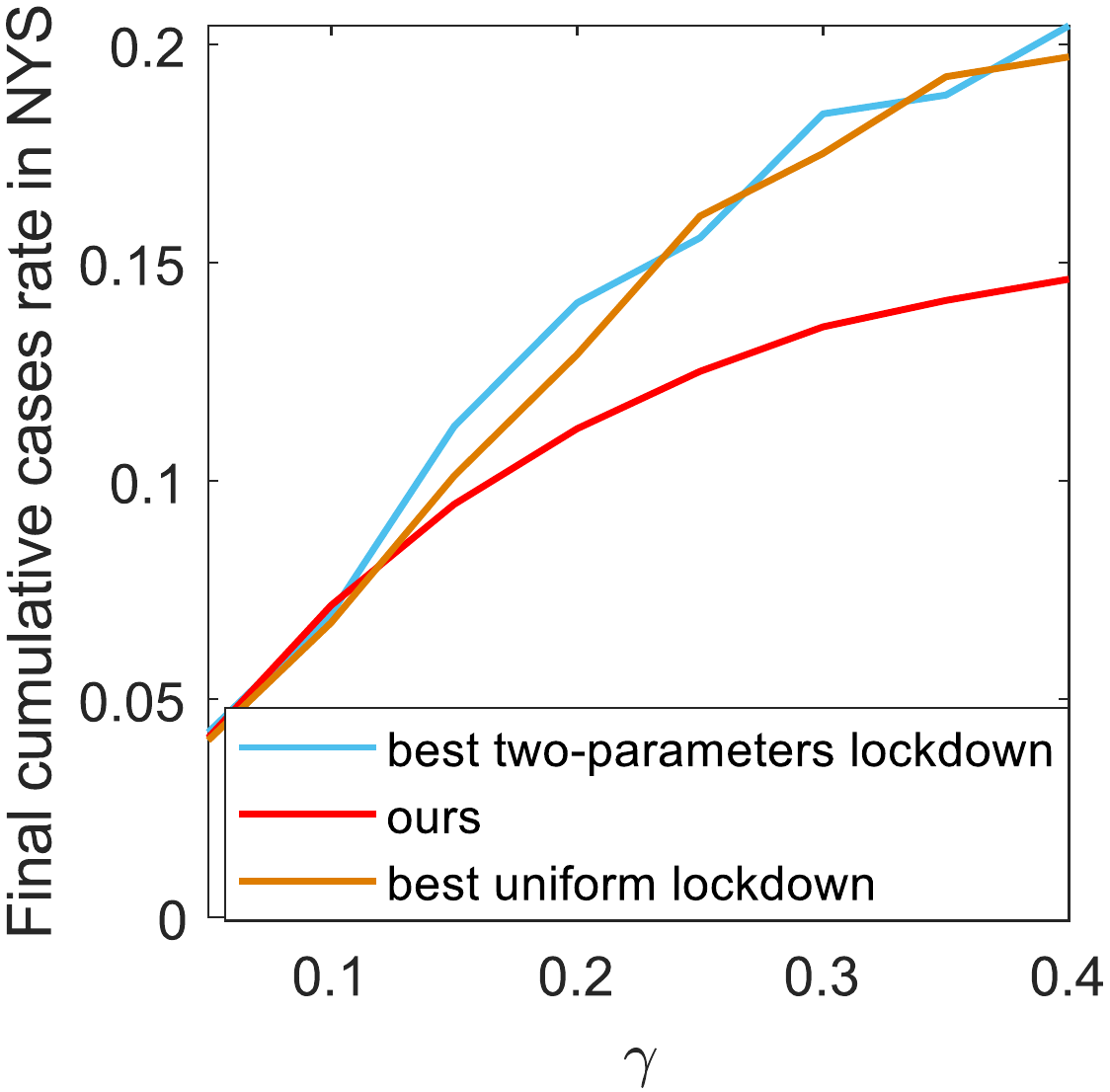}
			\end{minipage}%
		}%
		\subfigure[]{\label{sufig: SIR_growth_rate}
			\begin{minipage}[b]{0.3\linewidth}
				\centering
				\includegraphics[width=1.0\linewidth]{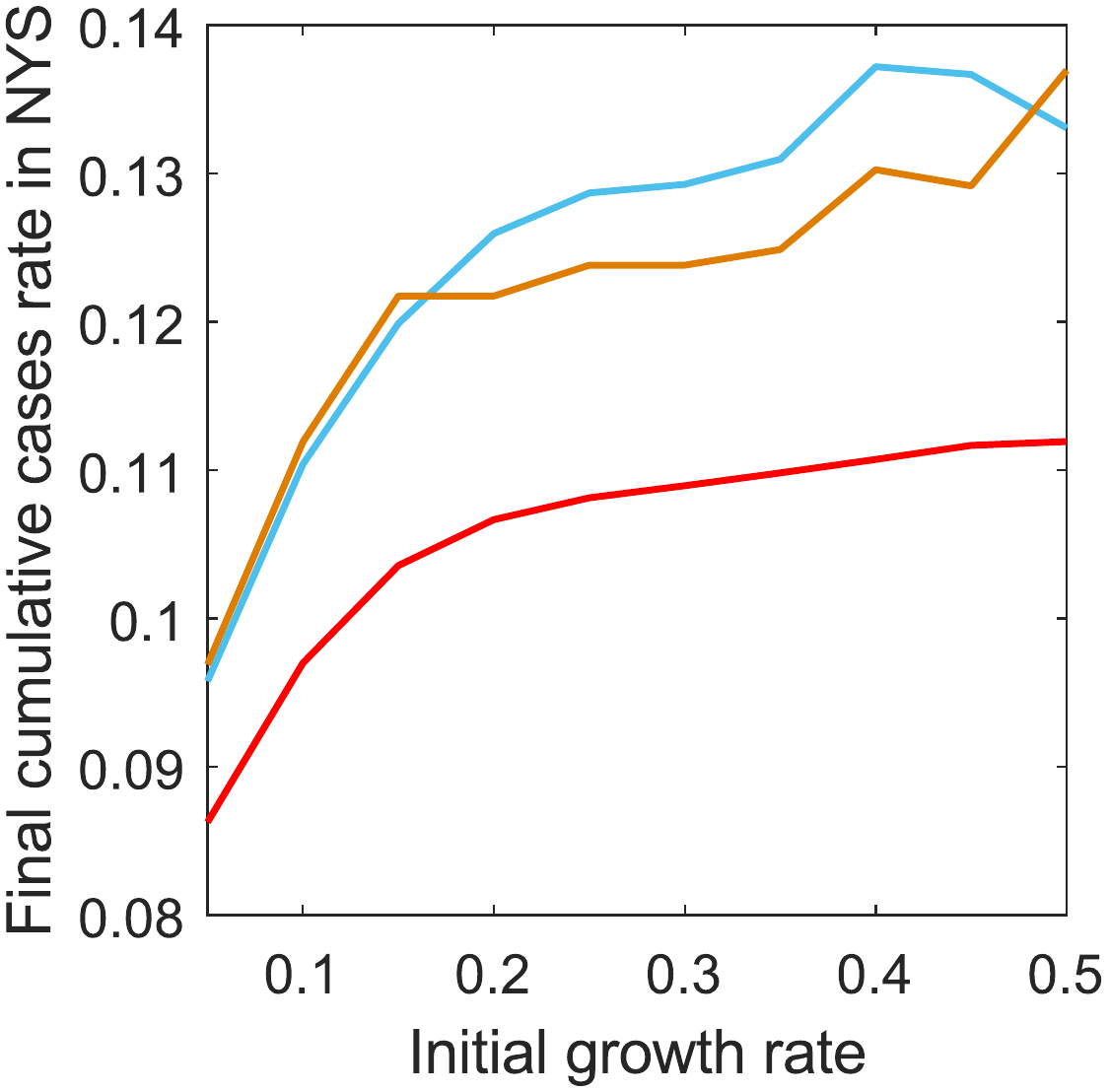}
			\end{minipage}%
		}%
		\centering
		\vspace{-2mm}
		\caption{{\bf Sensitivity analysis of disease parameters with respect to final cumulative cases}: the estimated rate of final accumulative cases (at day 500) in the total population for {\bf SIS model and SIR model}, respectively based on available data about COVID-19 outbreak in New York State on April 1st, 2020. It can be observed from {\bf a} and {\bf d} that the total number of the infections of our policy is much better for small $\alpha$; but if we increase $\alpha$ too much, this starts to change and we no longer outperform. Besides, from {\bf c} and {\bf f} we can see that the total number of infections of all the polices increases as the initial growth rate increases. In addition, it can be observed from {\bf b} and {\bf e} that
			the total number of infections increases when the recovery rate $\gamma$ increases, the reason for this counter-intuitive phenomenon is provided in \ssref{two-parameters}.}\label{Fig: cumula_sensitivity}
		\vspace{-1mm}
	\end{figure}

	\begin{figure}[!htb]
		
		\centering
		\textbf{COVID-19}\par\medskip
		\vspace{-1mm}
		\centering
		\subfigure[]{\label{subfig: Covid_19_alpha}
			\begin{minipage}[b]{0.3\linewidth}
				\centering
				\includegraphics[width=1.0\linewidth, height = 1.0\linewidth]{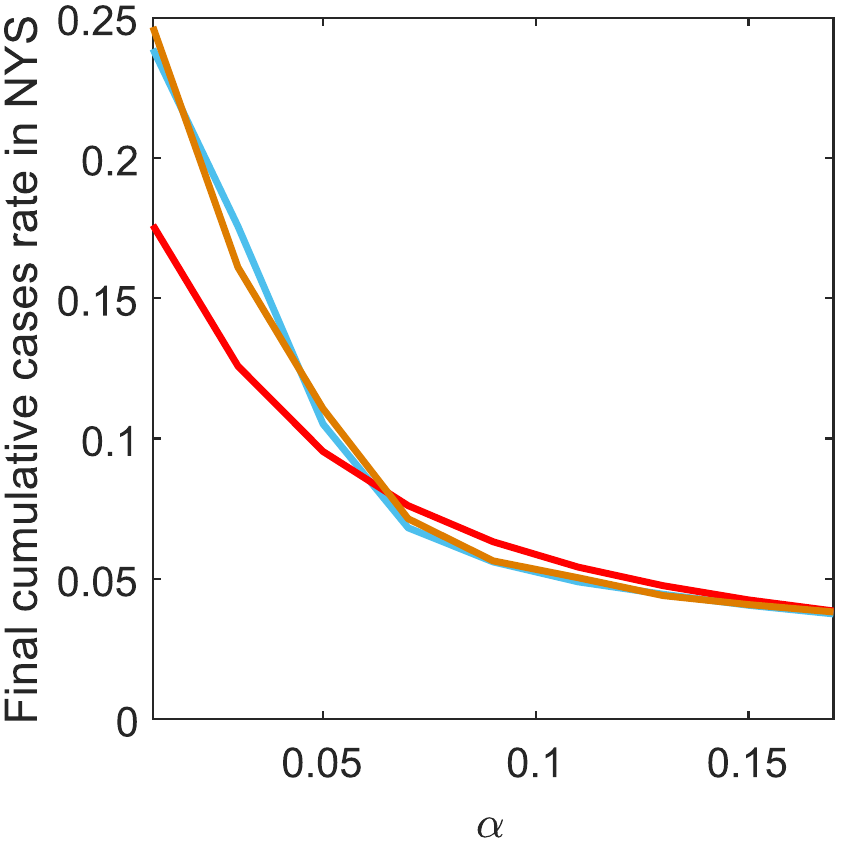}
			\end{minipage}%
		}%
		\subfigure[]{\label{subfig: Covid_19_growth_rate}
			\begin{minipage}[b]{0.3\linewidth}
				\centering
				\includegraphics[width=1.0\linewidth, height = 1.0\linewidth]{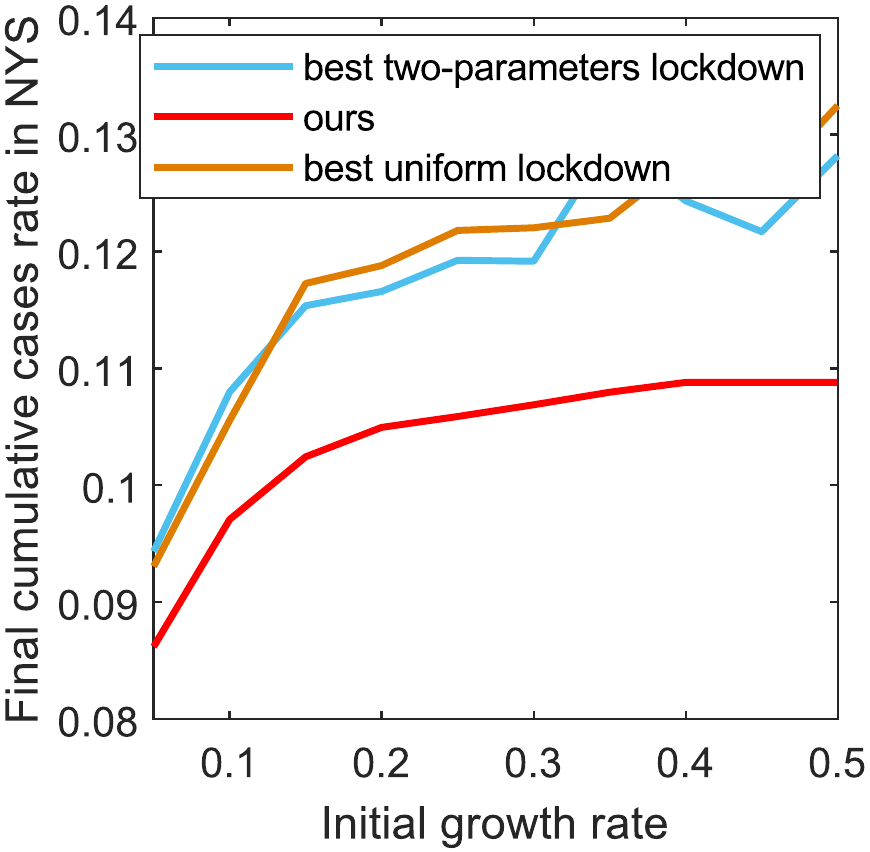}
			\end{minipage}%
		}%
		\subfigure[]{\label{subfig: Covid_19_gamma}
			\begin{minipage}[b]{0.3\linewidth}
				\centering
				\includegraphics[width =1.0\linewidth, height = 1.0\linewidth]{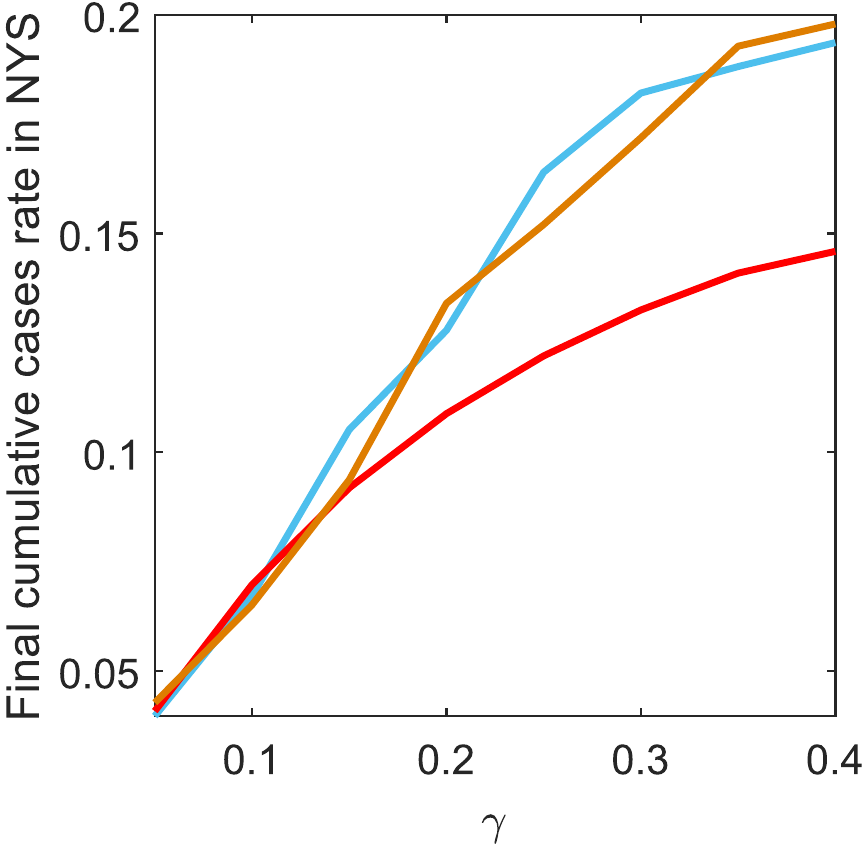}
			\end{minipage}%
		}%

		\subfigure[]{\label{subfig: Covid_19_epsilon}
			\begin{minipage}[b]{0.3\linewidth}
				\centering
				\includegraphics[width=1.0\linewidth, height = 1.0\linewidth]{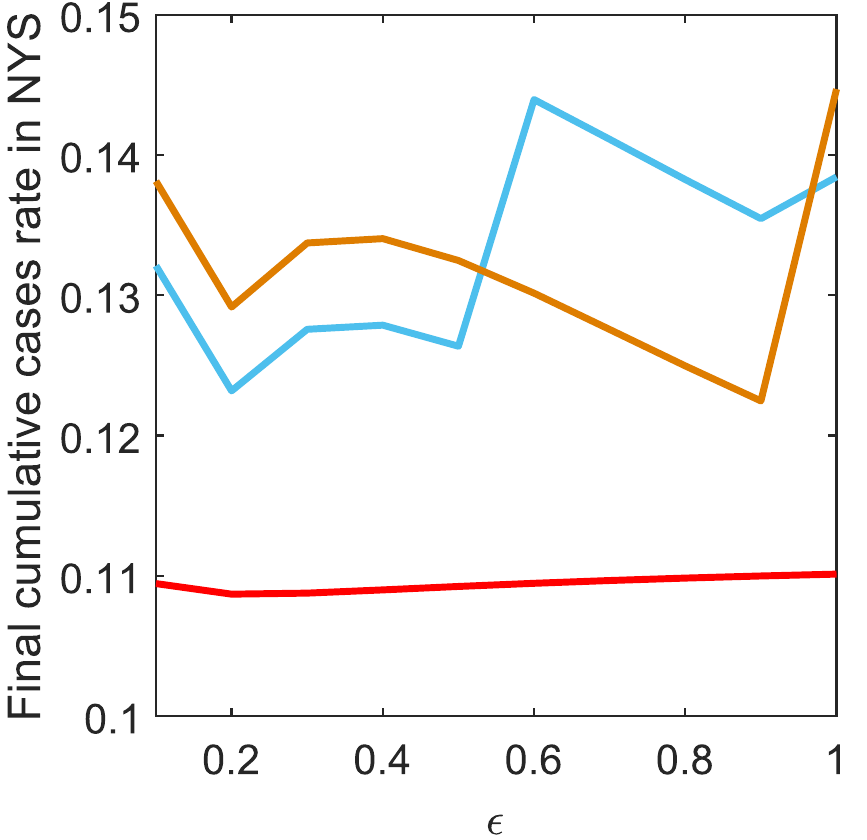}
			\end{minipage}%
		}%
		\subfigure[]{\label{subfig: Covid_19_alpha_hat}
			\begin{minipage}[b]{0.3\linewidth}
				\centering
				\includegraphics[width=1.0\linewidth, height = 1.0\linewidth]{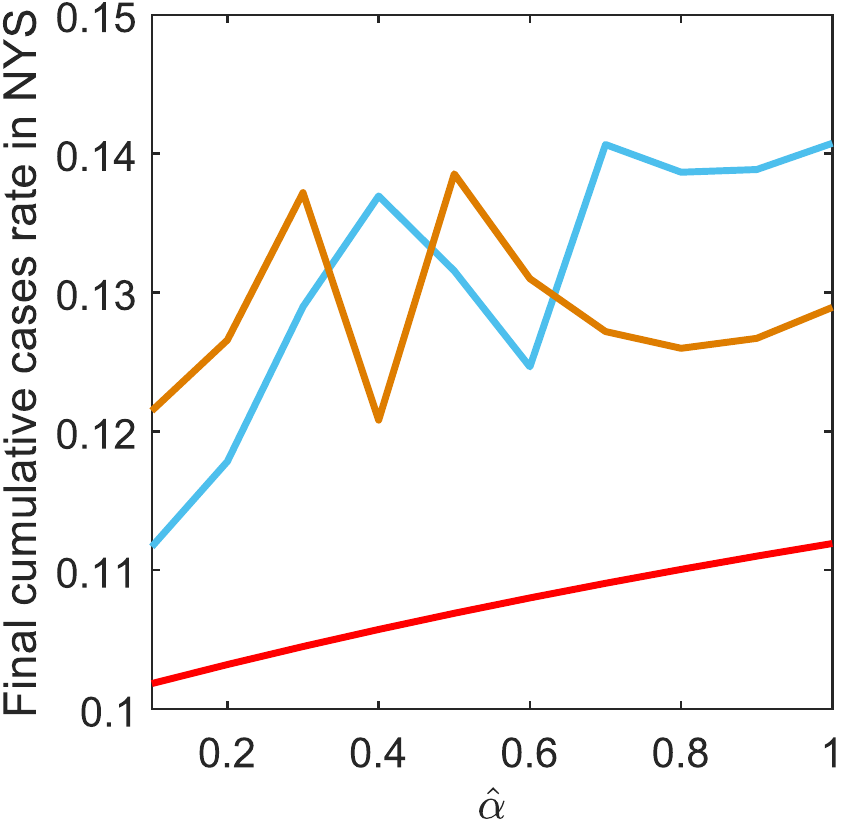}
			\end{minipage}%
		}%
		\subfigure[]{\label{subfig: Covid_19_gamma_hat}
			\begin{minipage}[b]{0.3\linewidth}
				\centering
				\includegraphics[width=1.0\linewidth, height = 1.0\linewidth]{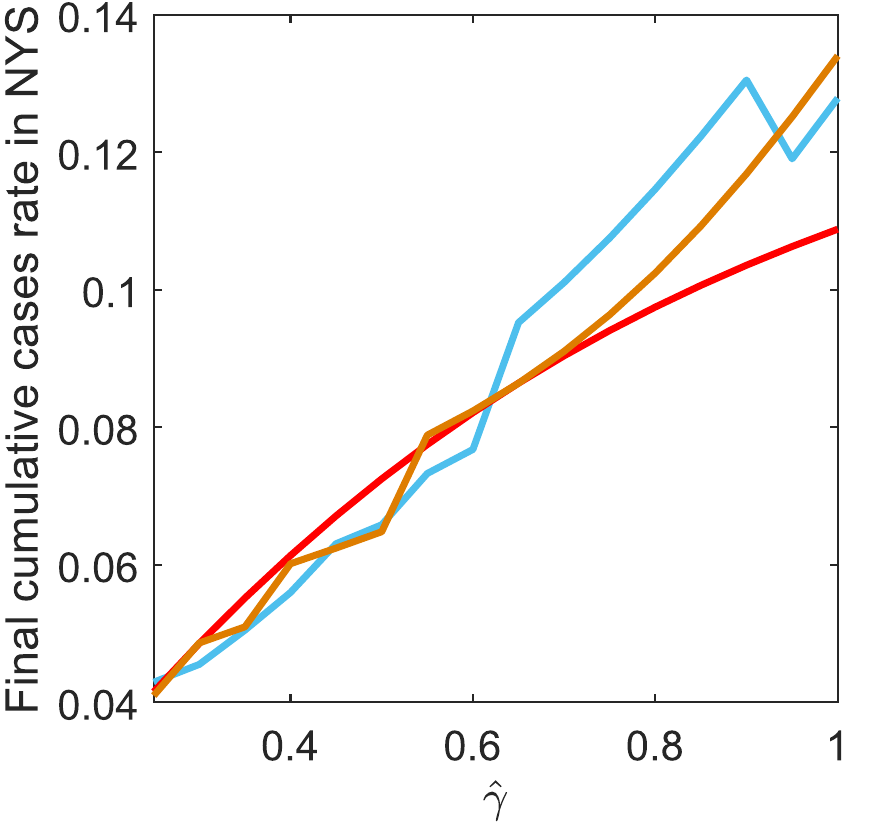}
			\end{minipage}%
		}%
		
		\centering
		\vspace{-2mm}
		\caption{{\bf Sensitivity analysis of disease parameters with respect to final cumulative cases}: the estimated rate of final accumulative cases (at day 500) in the total population for {\bf COVID-19 model} based on available data about COVID-19 outbreak in New York State on April 1st, 2020. It can be observed from {\bf a} that the total number of the infections of our policy is much better for small $\alpha$; but if we increase $\alpha$ too much, this starts to change and we no longer outperform. Besides, from {\bf b}, we can see that the total number of infections of all the polices increases as the initial growth rate increases. In addition, it can be observed from
			{\bf c} that
			the total number of infections increases when the recovery rate $\gamma$ increases, the reason for this counter-intuitive phenomenon is provided in \ssref{two-parameters}. The analysis for {\bf d-f} can also be found in \ssref{two-parameters}.}\label{Fig: Covid_sensitivity}
		\vspace{-1mm}
	\end{figure}
	
	\begin{figure}[!htb]
		\centering
		\subfigure[]{
			\begin{minipage}[b]{0.9\linewidth}\label{fig: one_half}
				\centering
				\includegraphics[width=1.0\linewidth]{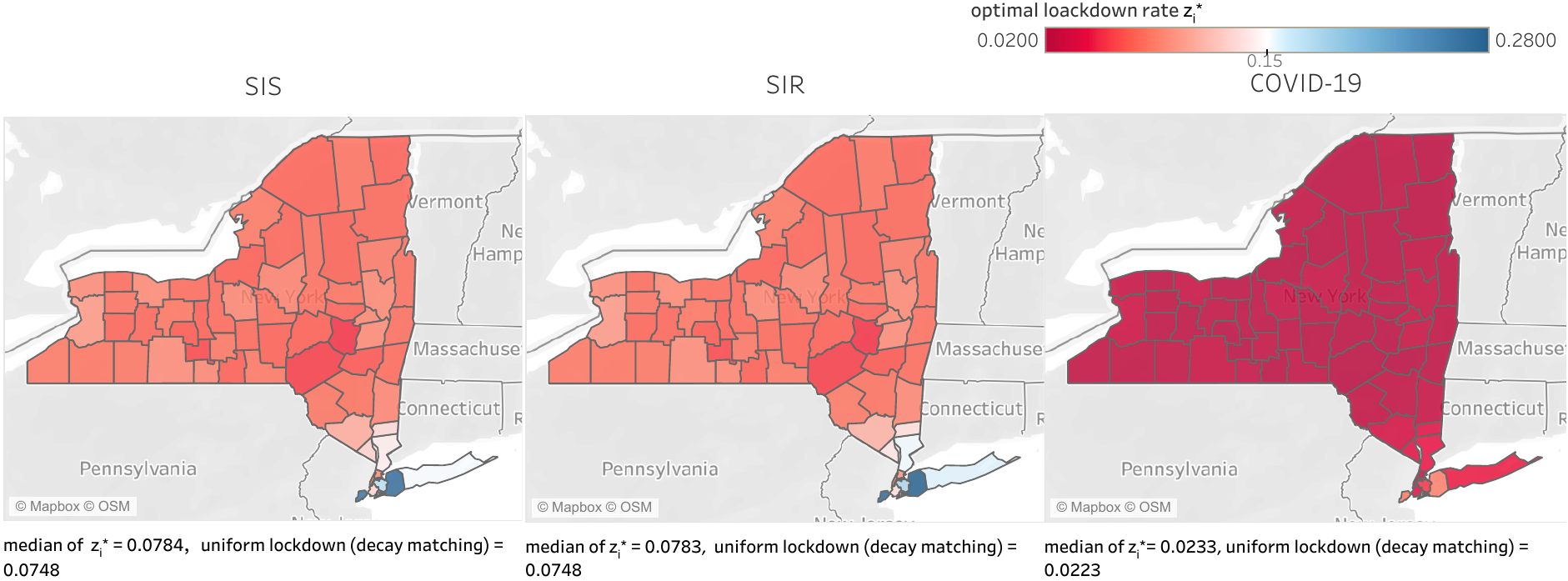}
			\end{minipage}%
		}%
		
		\subfigure[]{
			\begin{minipage}[b]{0.9\linewidth}\label{fig: square}
				\centering
				\includegraphics[width=1.0\linewidth]{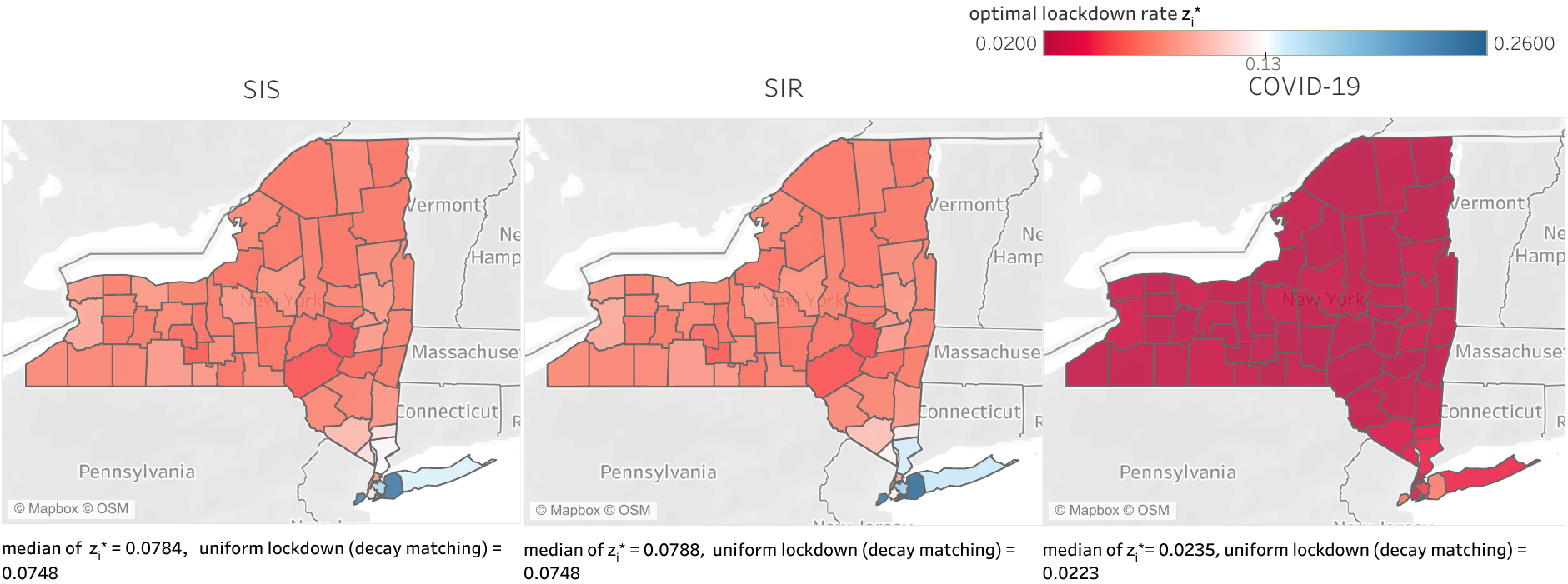}
			\end{minipage}%
		}%
		
		\subfigure[]{
			\begin{minipage}[b]{0.9\linewidth}\label{fig: cube}
				\centering
				\includegraphics[width=1.0\linewidth]{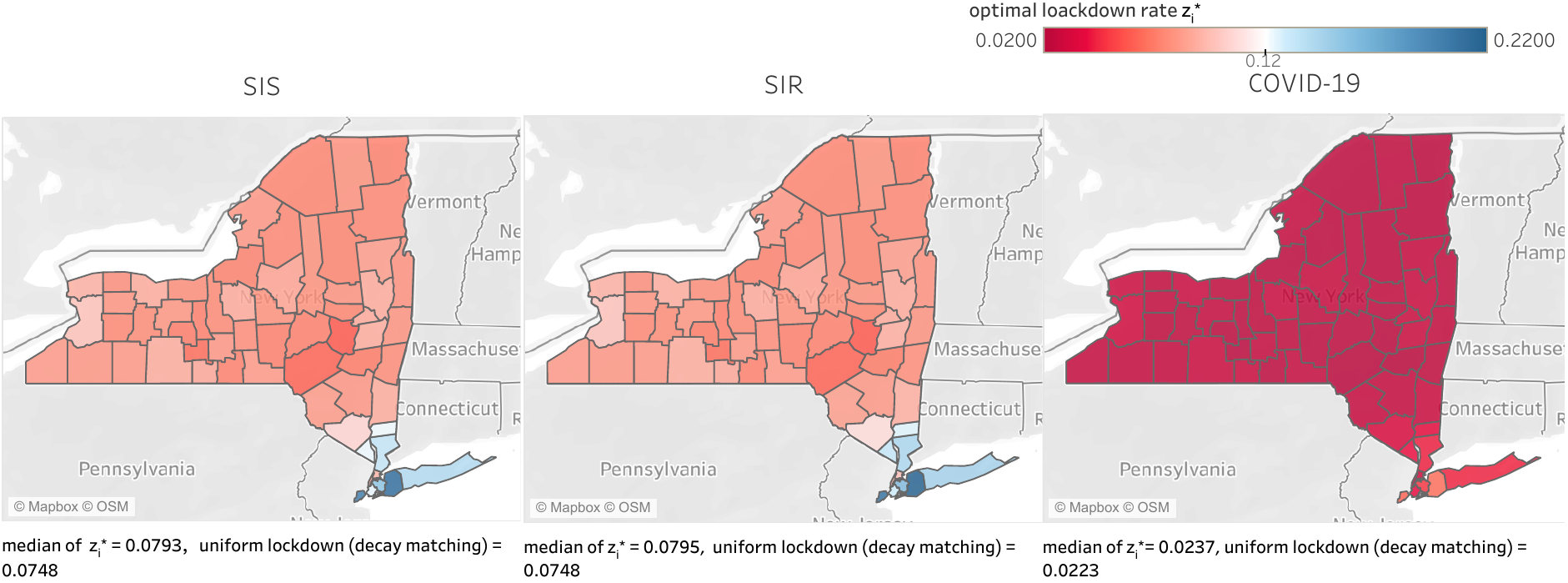}
			\end{minipage}%
		}%
		\centering
		\caption{{\bf Experimental results with extended lockdown cost function}: optimal lockdown rate $z_i^*$ given by Theorem \ref{thm:mainthm} for SIS, SIR, and COVID-19 model based on available data about COVID-19 outbreak in New York State on April 1st, 2020. The disease parameters are set as in \cite{giordano2020modelling}. The decay rate is chosen as $\alpha = 0.2 r^{\text s} = 0.0034$ so that $\alpha < \min (r^{\text a}, r^{\text s})$. In {\bf a}, the cost function is chosen as $\mathbf{\sum_i{c_i (\frac{1}{z_i^{1.5}} - 1)}}$.
			In {\bf b}, the cost function is chosen as $\sum_i{c_i(\frac{1}{z_i^{2}} - 1)}$. In {\bf c}, the cost function is chosen as $\sum_i{c_i (\frac{1}{z_i^{3}} - 1)}$. It can observed that the value of $z_i^*$ for counties in NYC are relatively higher than other counties in New York State, which implies we should shutdown the outside of NYC harder than itself. Besides, the median of $z_i^*$ is greater than the value of the uniform lockdown in all the cases.
		}\label{Fig: polynomial costs}
		\vspace{-1mm}
	\end{figure}
	\clearpage

	\vspace{\fill}
	\clearpage
	
	\begin{figure}[!htb]
		\centering
		\subfigure[]{
			\begin{minipage}[b]{0.9\linewidth}\label{fig: min_10}
				\centering
				\includegraphics[width=1.0\linewidth]{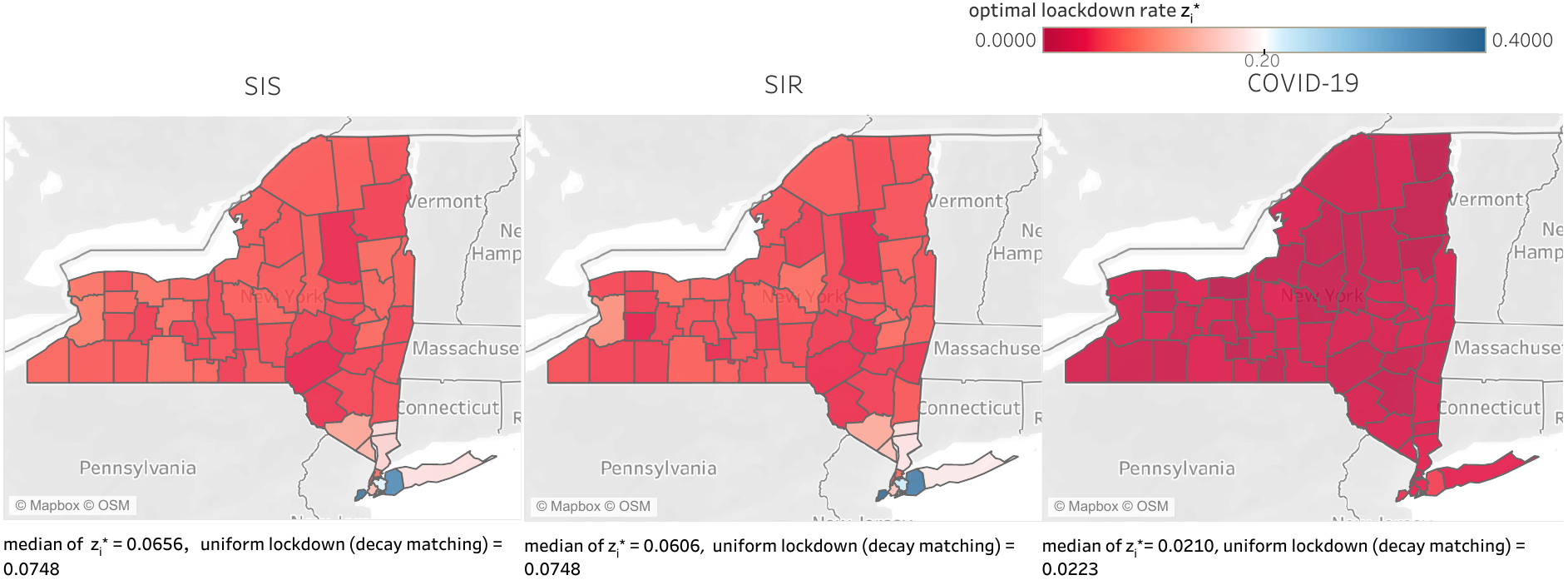}
			\end{minipage}%
		}%
		
		\subfigure[]{
			\begin{minipage}[b]{0.9\linewidth}\label{fig: min_20}
				\centering
				\includegraphics[width=1.0\linewidth]{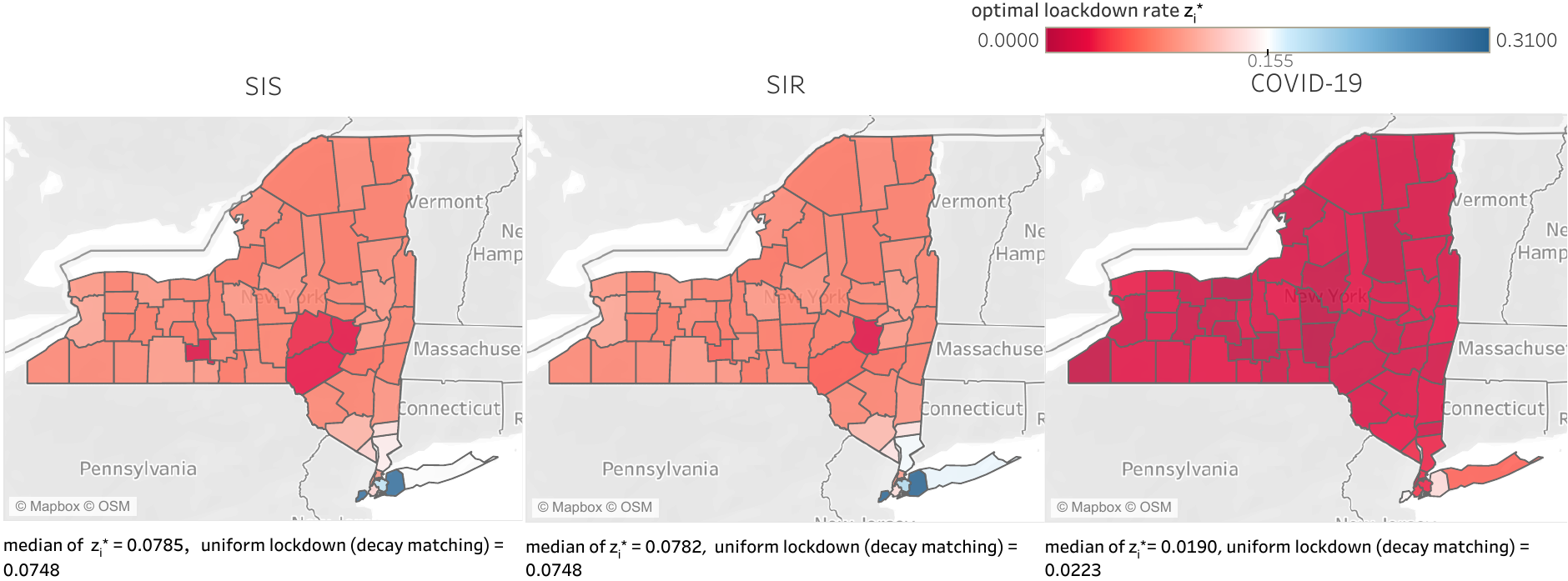}
			\end{minipage}%
		}%
		
		\subfigure[]{
			\begin{minipage}[b]{0.9\linewidth}\label{fig: min_100}
				\centering
				\includegraphics[width=1.0\linewidth]{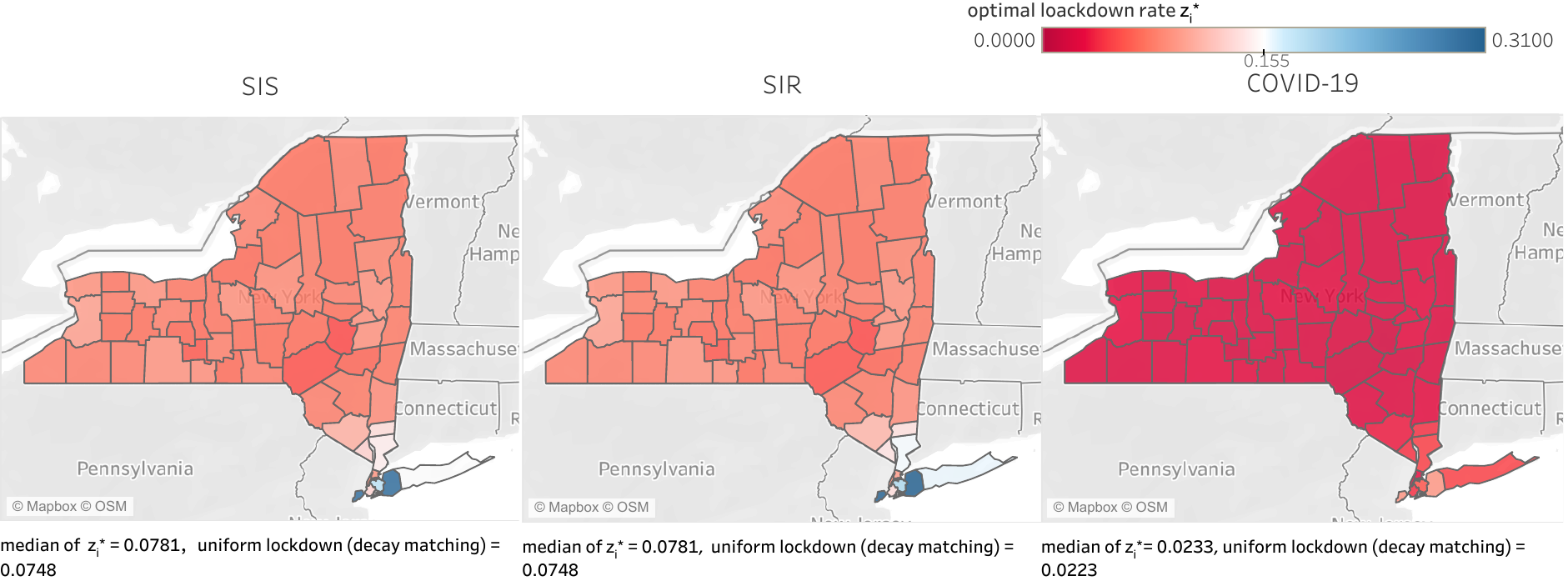}
			\end{minipage}%
		}%
		\centering
		\caption{{\bf Experimental results with extended lockdown cost function}: optimal lockdown rate $z_i^*$ given by Theorem \ref{thm:mainthm} for SIS, SIR, and COVID-19 model based on available data about COVID-19 outbreak in New York State on April 1st, 2020. The disease parameters are set as in \cite{giordano2020modelling}. The decay rate is chosen as $\alpha = 0.2 r^{\text s} = 0.0034$ so that $\alpha < \min (r^{\text a}, r^{\text s})$. In {\bf a}, the cost function is chosen as $\sum_i{c_i (\min(\frac{1}{z_i}, 10) - 1)}$.
			In {\bf b}, the cost function is chosen as $\sum_i{c_i (\min(\frac{1}{z_i}, 20) - 1)}$. In {\bf c}, the cost function is chosen as $\sum_i{c_i(\min(\frac{1}{z_i}, 100) - 1)}$. It can observed that the value of $z_i^*$ for counties in NYC are relatively higher than other counties in New York State, which implies we should shutdown the outside of NYC harder than itself. Besides, the median of $z_i^*$ is greater than or close to the value of the uniform lockdown in all the cases.
		}\label{Fig: min costs}
		\vspace{-1mm}
	\end{figure}

	\begin{figure}[!htb]
		\centering
		\begin{minipage}[b]{0.95\linewidth}
			\centering
			\includegraphics[width=1.0\linewidth]{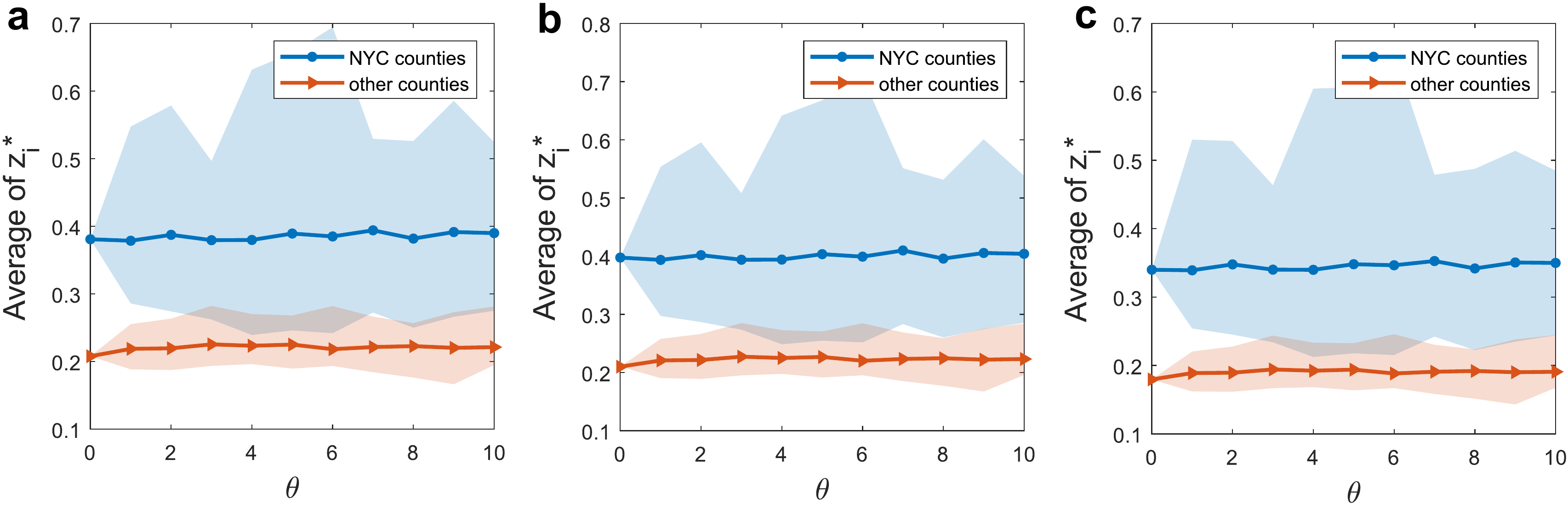}			
		\end{minipage}\hfill
		
		\vspace{3mm}
		
		\begin{minipage}[b]{0.95\linewidth}
			\centering
			\includegraphics[width=1.0\linewidth]{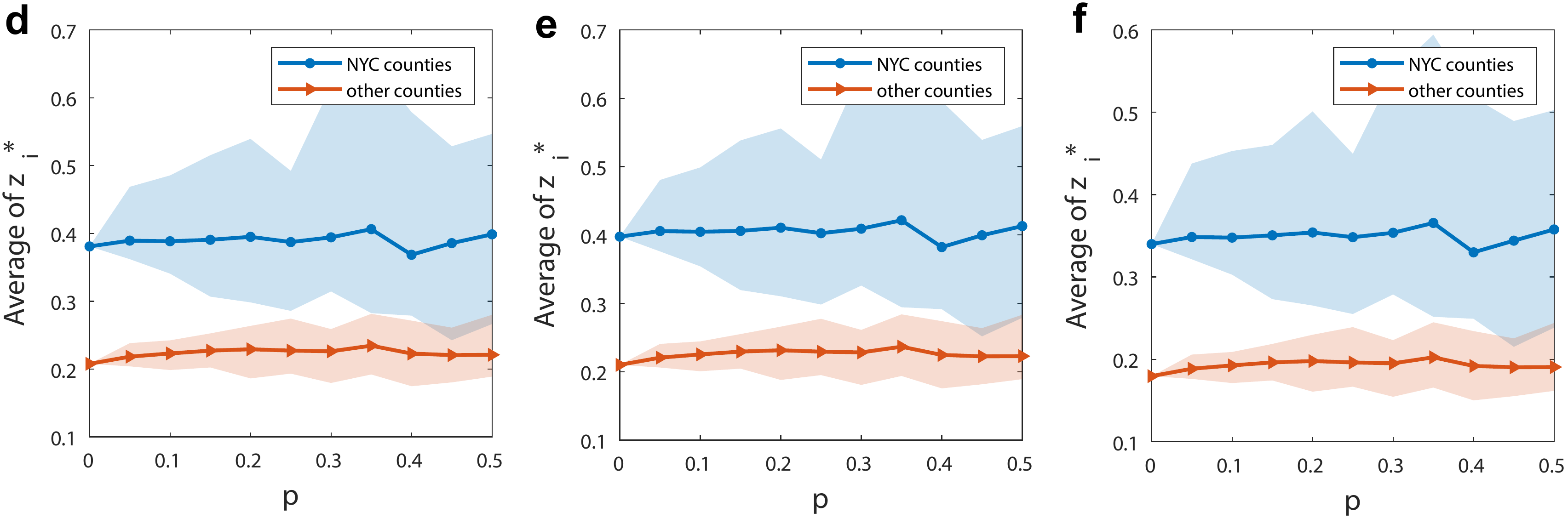}
		\end{minipage}%
		\centering
		\vspace{-1mm}
		\caption{{\bf Robustness check with respect to the travel rate matrix $\tau$}: the average values of the optimal lockdown rate $z_i^*$ calculated by our method for the counties in NYC and other counties in NY after introducing noises or removing data from the travel rate matrix $\tau$.  {\bf a-c}, results after introducing gaussian noises to each entry of $\tau$,  $\theta$ is proportional to the magnitude of the noises we add, details are presented in SI Sec. \ref{Sec: Robustness}. {\bf d-f}, results after removing  a $p$-fraction of travelling data for each county (mimic the situation where a $p$-fraction of travelling data are missing). {\bf a, d} are results based on SIS model, {\bf b, e} are results based on SIR model, {\bf c, f} are results based on COVID-19 model. 
			The disease parameters are set as in \cite{birge2020controlling}, the decay rate $\alpha$ is chosen $0.0231$ that corresponds to halving every 30 days. Solid lines show the average of 50 runs, the shaded regions show the min-max interval of the 50 runs. It can be observed that the average of $z_i^*$ for NYC cities is greater than the average of other counties for a wide range of the noise level or the fraction of missing data. This suggests that our finding that the optimal lockdown should shutdown NYC harder is robust against the uncertainty of the travel rate matrix $\tau$.
		}\label{robustness_check}
	\end{figure}

\begin{figure}[!htb]
	\centering
	\begin{minipage}[b]{0.9\linewidth}
		\centering
		\includegraphics[width=1.0\linewidth]{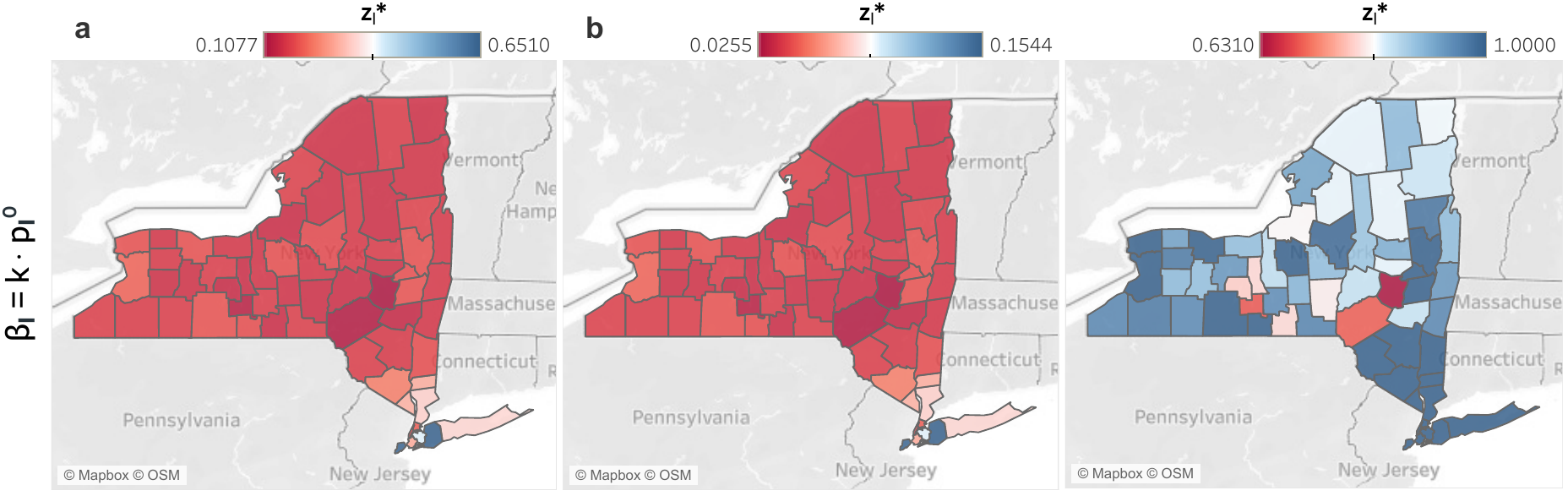}
	\end{minipage}	
		\begin{minipage}[b]{0.9\linewidth}
			\centering
			\includegraphics[width=1.0\linewidth]{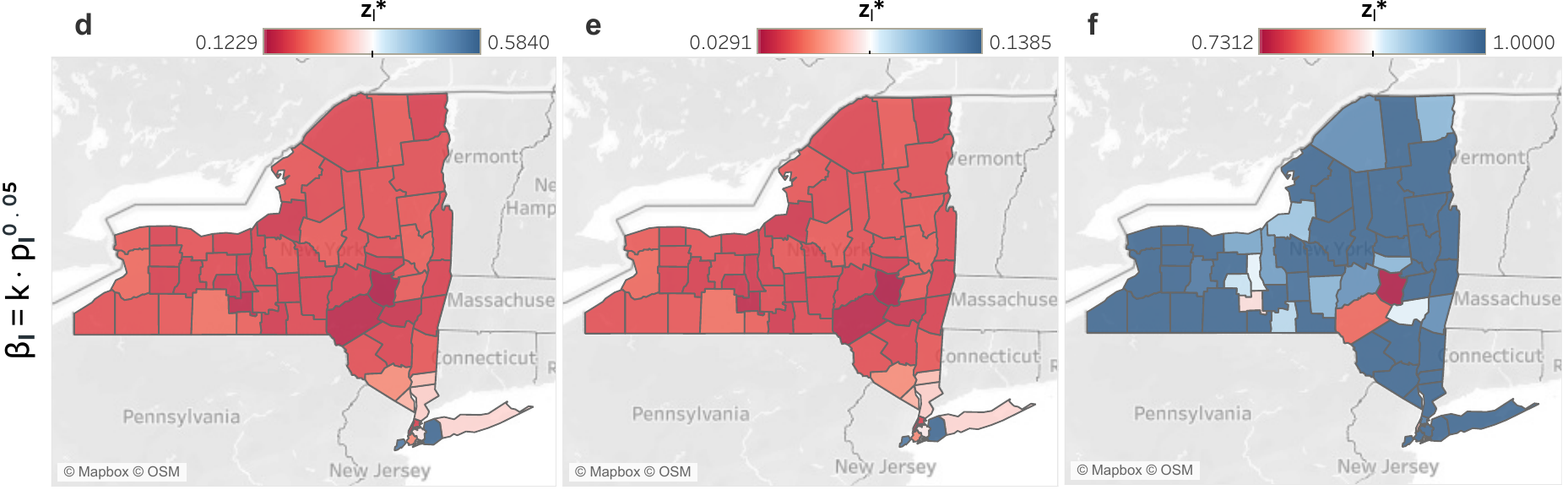}
		\end{minipage}%

		\begin{minipage}[b]{0.9\linewidth}
			\centering
			\includegraphics[width=1.0\linewidth]{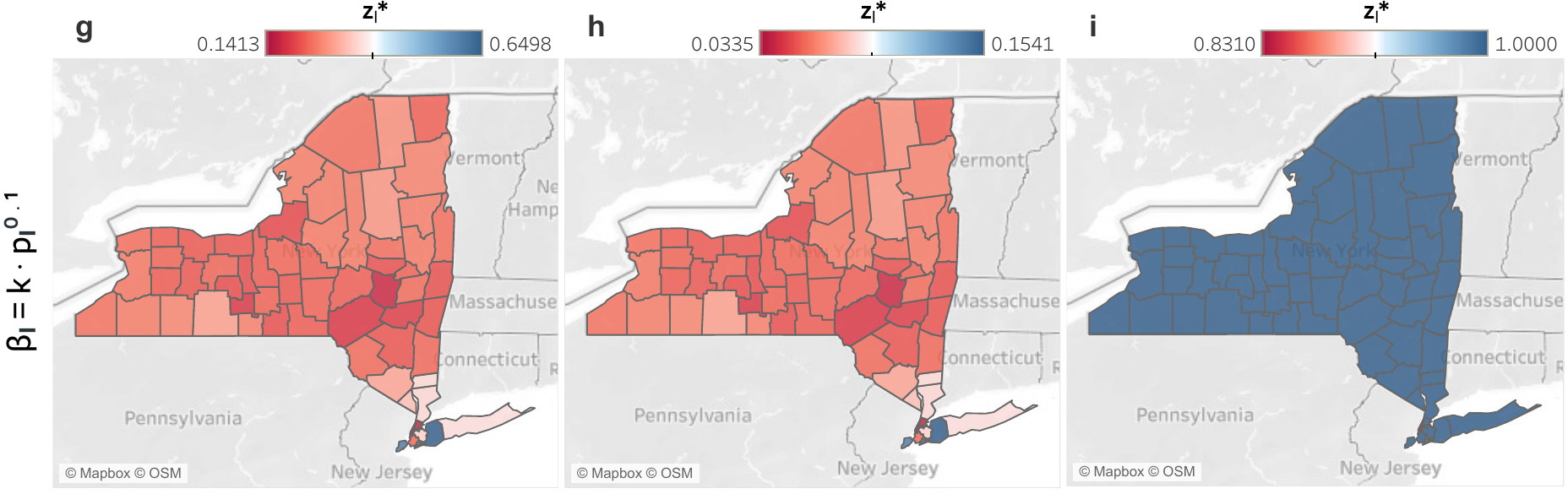}
		\end{minipage}%
		
		\begin{minipage}[b]{0.9\linewidth}
			\centering
			\includegraphics[width=1.0\linewidth]{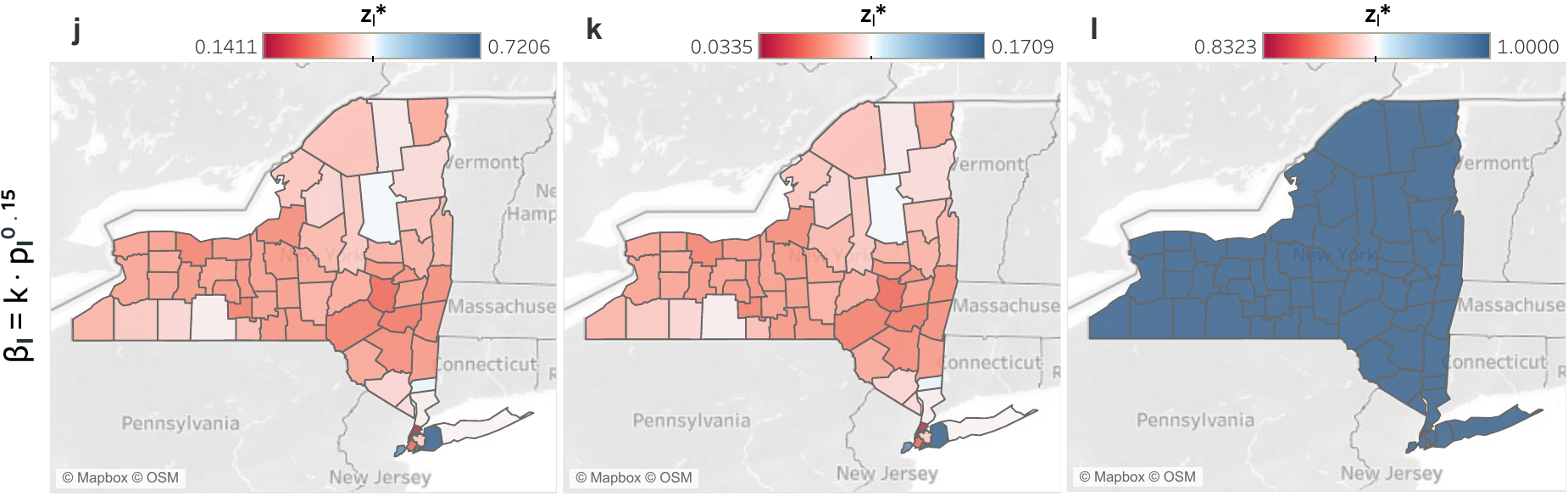}
		\end{minipage}%
	
	\vspace{-3mm}
	\caption{Optimal lockdown rate of each county calculated by our method for COVID-19 model when {\bf the transmission rate $\beta$ is nonuniform}. In the first row, we assume the transmission rate $\beta$ at each county are the same. In the second row, we assume the transmission rate $\beta_l$ is proportional to $p_l^{0.05}$, where $p_l$ represents the population density at county $l$. In the third row, we assume the transmission rate $\beta_l$ is proportional to $p_l^{0.1}$. In the last row, we assume the transmission rate $\beta_l$ is proportional to $p_l^{0.15}$.
		In the first column, the disease parameters are set as in \cite{bertozzi2020challenges}, the decay rate $\alpha$ is chosen as 0.0231 which corresponds to halving every 30 days. In the second column, the disease parameters are set as in \cite{giordano2020modelling}, the decay rate is chosen as $\alpha = 0.2 r^{\text s} = 0.0034$ so that $\alpha < \min (r^{\text a}, r^{\text s})$. In the third column, the disease parameters are set as in \cite{birge2020controlling}, the decay rate $\alpha$ is chosen $0.0231$ that corresponds to halving every 30 days. The data used is about COVID-19 outbreak in NY on April 1st, 2020.  
	}\label{fig: density}
	\centering
	\vspace{-1mm}
\end{figure}

\begin{figure}[!htb]
	\centering
	\begin{minipage}[b]{0.9\linewidth}
		\centering
		\includegraphics[width=1.0\linewidth]{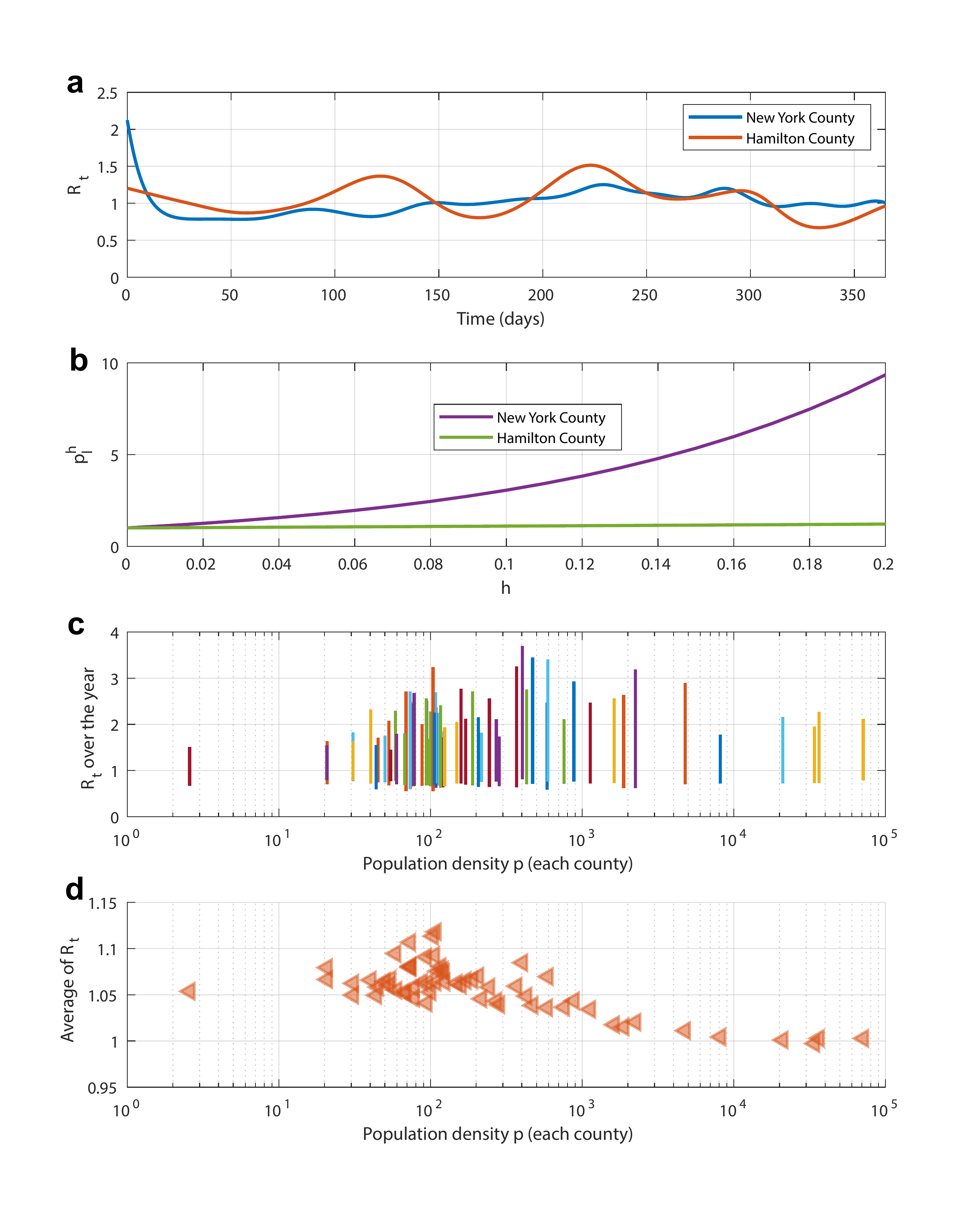}
	\end{minipage}	
	\vspace{-3mm}
	\caption{{\bf a}, the estimated value of the effective reproduction number {\bf $R_t$ for New York County and Hamilton County} from March 16th, 2020 to March 16th, 2021.  The data comes from \cite{NewYorkRt}. {\bf b}, the constant power of the population density $p_l^h$ for New York County and Hamilton County when $h$ ranges in the interval $[0, 0.2]$. {\bf c}, $R_t$ range of each county in New York State over the year (March 16th, 2020 - March 16th, 2021) and the population density of each county. {\bf d} the average of $R_t$ for each county in New York State over the year (March 16th, 2020 - March 16th, 2021) and the population density of each county.
	}\label{fig: density and Rt}
	\centering
	\vspace{-1mm}
\end{figure}

\begin{figure}[!htb]
	\centering
	\begin{minipage}[b]{0.9\linewidth}
		\centering
		\includegraphics[width=1.0\linewidth]{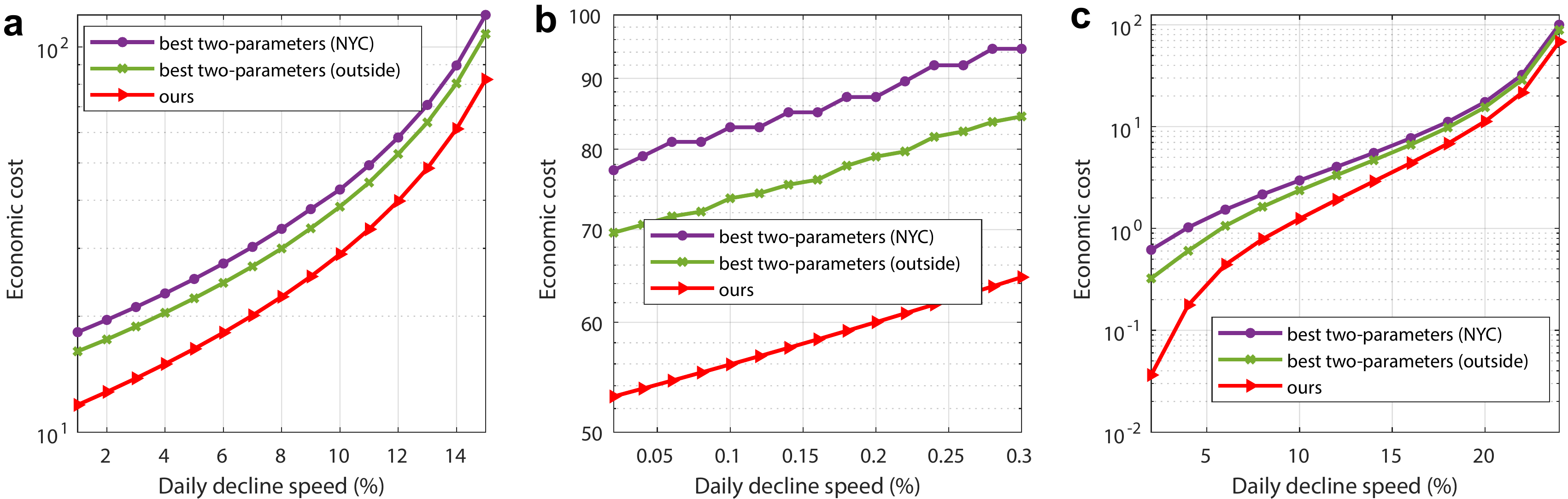}
	\end{minipage}	
	\vspace{-3mm}
	\caption{The economic cost of different lockdown policies when {\bf the daily decline speed of the confirmed cases is fixed}. In {\bf a}, the disease parameters are set as in \cite{bertozzi2020challenges}, the daily cases decline speed ranges between $[1\%, ~15\%]$. In {\bf b}, the disease parameters set as in \cite{giordano2020modelling}, the daily cases decline speed ranges between $[0.01\%, ~0.3\%]$. In {\bf c}, the disease parameters set as in \cite{birge2020controlling}, the daily cases decline speed ranges between $[2\%, ~24\%]$. The best two-parameters (NYC) lockdown and the best wo-parameters (outside) lockdown are defined in SI Sec \ref{sec: two parameters compare}.
	}\label{fig: two parameters compare}
	\centering
	\vspace{-1mm}
\end{figure}

\begin{figure}[!htb]
	\centering
	\begin{minipage}[b]{0.9\linewidth}
		\centering
		\includegraphics[width=1.0\linewidth]{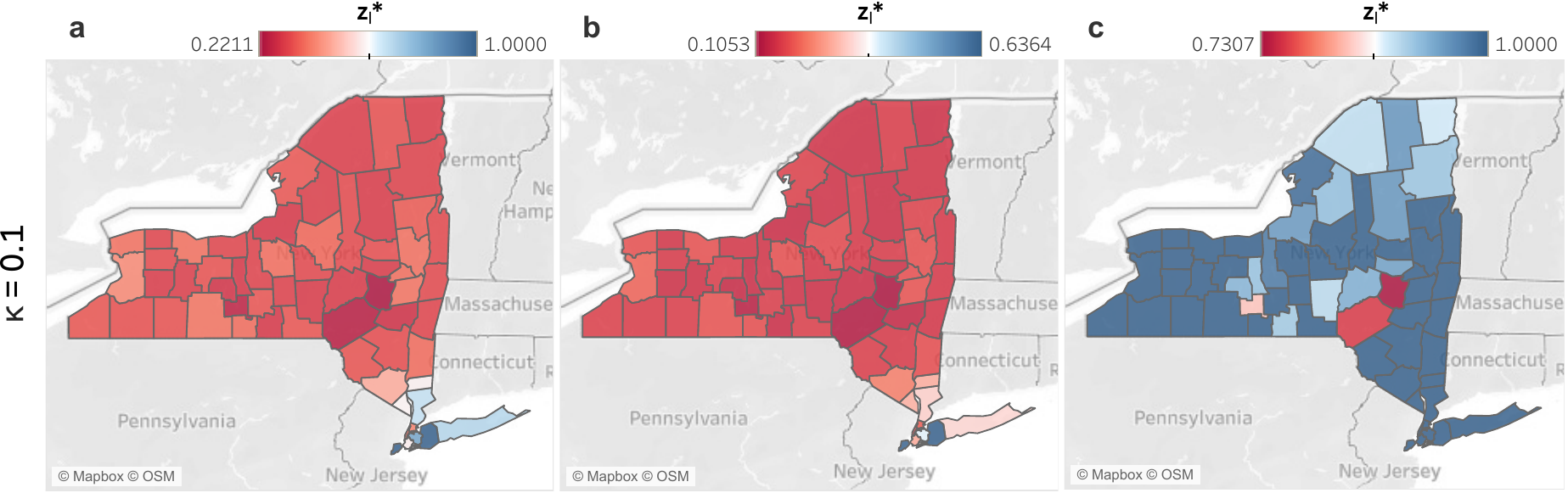}
	\end{minipage}	
	\begin{minipage}[b]{0.9\linewidth}
		\centering
		\includegraphics[width=1.0\linewidth]{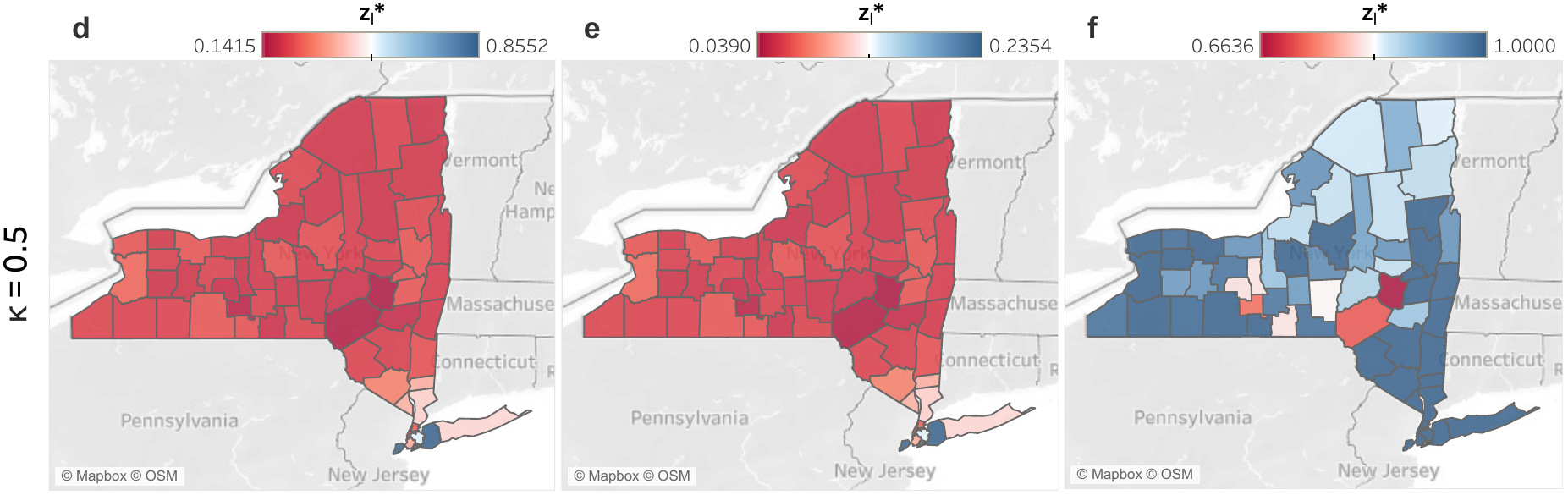}
	\end{minipage}%

	\begin{minipage}[b]{0.9\linewidth}
		\centering
		\includegraphics[width=1.0\linewidth]{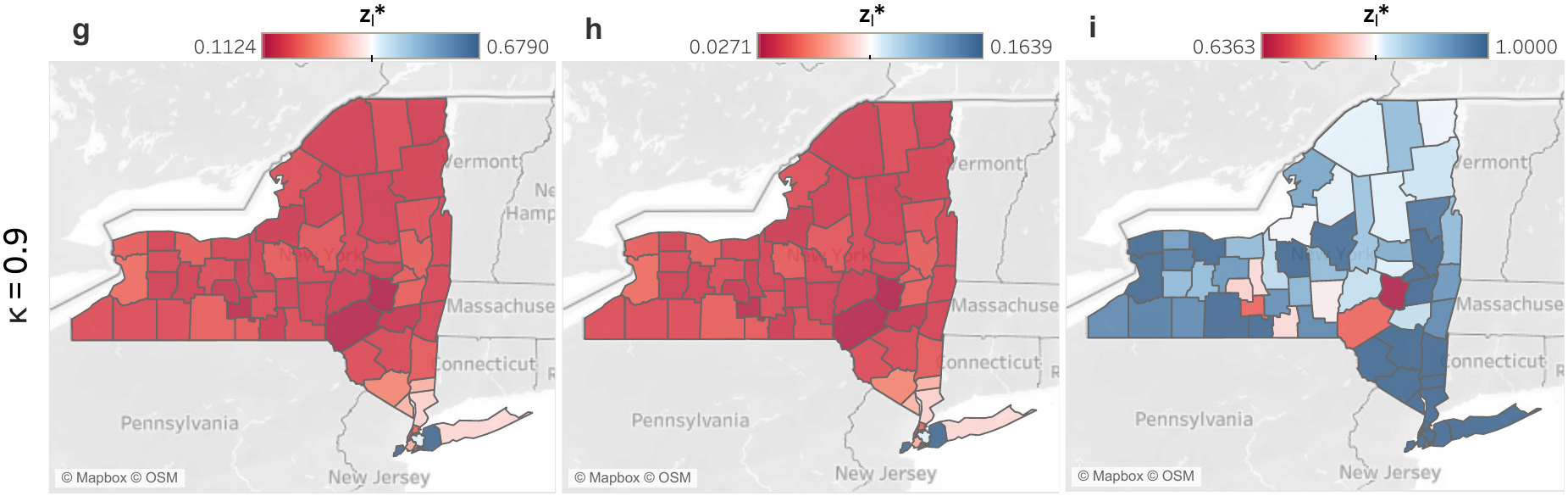}
	\end{minipage}%
	
	\begin{minipage}[b]{0.9\linewidth}
		\centering
		\includegraphics[width=1.0\linewidth]{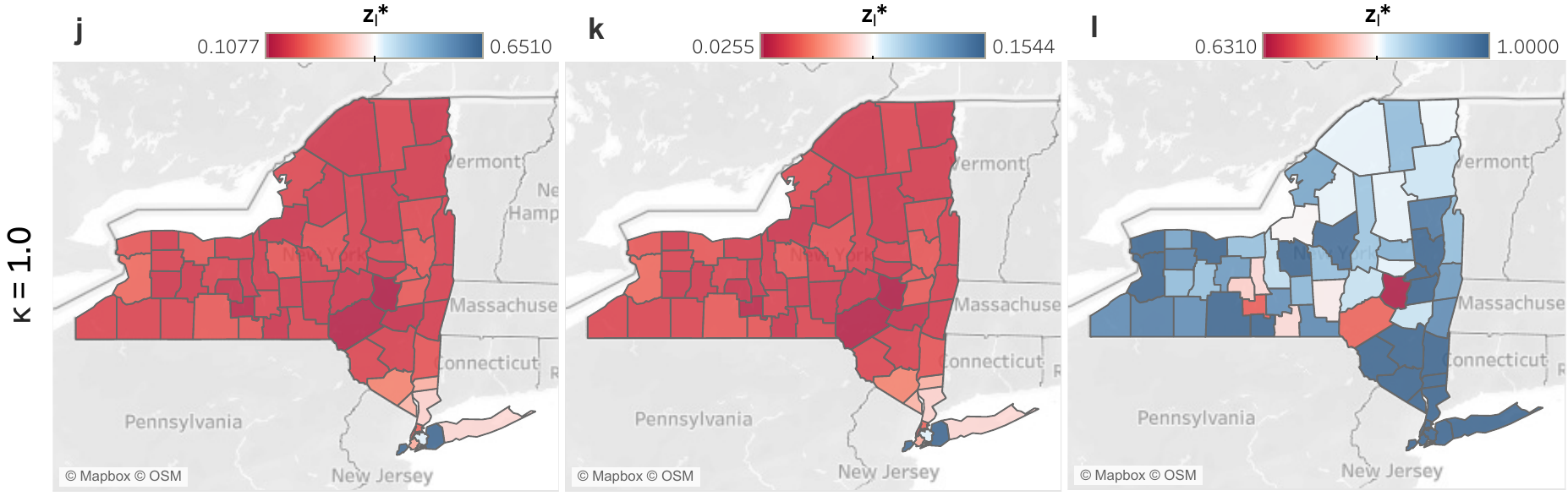}
	\end{minipage}%
	
	\vspace{-3mm}
	\caption{Optimal lockdown rate of each county calculated by our method for COVID-19 model when {\bf the activity level of symptomatic and asmptomatic individuals are different}. In {\bf a-c}, the travel rate $\tau_{ij}$ of symptomatic individuals is $10\%$ of the asymptomatic individuals.  In {\bf d-f}, the travel rate $\tau_{ij}$ of symptomatic individuals is $50\%$ of the asymptomatic individuals. In {\bf g-i}, the travel rate $\tau_{ij}$ of symptomatic individuals is $90\%$ of the asymptomatic individuals. 
	In {\bf j-l}, the travel rate $\tau_{ij}$ of symptomatic individuals are equal to the asymptomatic individuals. 
		In the first column, the disease parameters are set as in \cite{bertozzi2020challenges}, the decay rate $\alpha$ is chosen as 0.0231 which corresponds to halving every 30 days. In the second column, the disease parameters are set as in \cite{giordano2020modelling}, the decay rate is chosen as $\alpha = 0.2 r^{\text s} = 0.0034$ so that $\alpha < \min (r^{\text a}, r^{\text s})$. In the third column, the disease parameters are set as in \cite{birge2020controlling}, the decay rate $\alpha$ is chosen $0.0231$ that corresponds to halving every 30 days. The data used is about COVID-19 outbreak in NY on April 1st, 2020.  
	}\label{fig: activity level different }
	\centering
	\vspace{-1mm}
\end{figure}

	\clearpage

	\section{Supplementary Tables}\label{extended_results_table}
	
	\begin{table}[!htb]
		\centering
		\caption{Data used in Figure \ref{Fig: framework}d \label{tab: figure 1(d) data}
		}
		\scalebox{0.85}{
			\begin{tabular}{c  p{0.8cm}<{\centering} p{1.0cm}<{\centering} p{1.0cm}<{\centering} p{1.0cm}<{\centering} p{1.5cm}<{\centering} p{0.8cm}<{\centering}}
				\toprule
				Locations &  $s_i(t_0)$ & $r_i(t_0)$ & $x^{\text a}(t_0)$ & $x^{\text s}(t_0)$ & population & $h_i$ \\
				\midrule  
				A & 0.90 & 0.0041 & 0.0825 & 0.0134 & 200,000 & 800 \\
				B & 0.92 & 0.0033 & 0.0660 & 0.0107 & 2000 & 800\\
				C & 0.95 & 0.0021 & 0.0412 & 0.0067 & 4000 & 800\\
				\bottomrule
			\end{tabular}
		}
	\end{table}
	
	\begin{table}[!htb]
		\centering
		\caption{disease parameters from references \label{tab: model_parameters}
		}
		\scalebox{0.85}{
			\begin{tabular}{c  p{0.8cm}<{\centering} p{0.8cm}<{\centering} p{0.8cm}<{\centering} p{0.8cm}<{\centering} p{1.0cm}<{\centering}}
				\toprule
				Sources &  $\gamma$ & $r^{\text a}$ & $r^{\text s}$ & $\epsilon$ & $ \hat{\alpha}$ \\
				\midrule  
				\cite{birge2020controlling}  &  0.29 & 0.29 & 0.29 & 0.14 & 0.55 \\
				\cite{giordano2020modelling} &  0.034 & 0.034 & 0.017 & 0.125 & 0.6754\\
				\cite{bertozzi2020challenges} &  0.20 & – & – & 0.32 & – \\
				\bottomrule
			\end{tabular}
		}
	\end{table}

	\begin{table}[!htb]
		\centering
		\caption{Optimal lockdown rate $z_i^*$ for the City-suburb model. The first node is the city and the second one is the suburb. It can be observed that the optimal lockdown policy shutdown the suburb more stringently than the city even though, in cases 1 and 2,  the epidemic is mainly localized in the city. \label{tab: city_suburb}
		}
		\scalebox{0.85}{
			\begin{tabular}{m{3cm}<{\centering} m{3cm}<{\centering} m{3cm}<{\centering} m{3cm}<{\centering} }
				\toprule Model & Case 1 & Case 2 & Case 3
				\\
				\midrule  
				SIS &  [0.195, 0.189] & [0.196, 0.120] & [0.170, 0.104]  \\
				SIR &  [0.216, 0.164] & [0.197, 0.113] & [0.170, 0.104]\\
				COVID-19 & [0.185, 0.141] & [0.169, 0.098] & [0.145, 0.089] \\
				\bottomrule
			\end{tabular}
		}
	\end{table}
	\clearpage

\end{document}